\documentclass[11pt,a4paper,twoside]{article}
\usepackage[latin1]{inputenc}
\usepackage{cprotect}
\usepackage{latexsym}
\usepackage{a4wide}
\usepackage{mathrsfs}   %nices in the fonts
\usepackage{amsthm}
\usepackage{amssymb}
\usepackage{amsmath}
\usepackage{enumerate}
\usepackage{amsfonts}
\usepackage[textsize=small]{todonotes}
\usepackage{mathrsfs}   %nice fonts
\usepackage{tikz}       % for the picture
\usepackage{graphicx}
\usepackage{float}

\usepackage{hyperref}

\usetikzlibrary{arrows,automata}
\usepackage{tabularx}
\usepackage{multirow}
\usepackage{color}

\usepackage{ifthen} %Paket einbinden
 \newboolean{longversion} %Deklaration einer boolschen Variable
 \setboolean{longversion}{true}

\newcommand{\e}{{\mathrm e}}

\usepackage{fancyhdr}

\newcommand{\Cbold} {{\mathbb C}}

\newcommand{\Bcal}   {\mathcal{B}}
\newcommand{\Ccal}   {\mathcal{C}}

\newcommand{\Fcal}   {\mathcal{F}}

\newcommand{\nin}{\not\in}

\def\1{{\mathchoice {1\mskip-4mu\mathrm l}      % Blackboard bold 1
{1\mskip-4mu\mathrm l}
{1\mskip-4.5mu\mathrm l} {1\mskip-5mu\mathrm l}}}
\newcommand{\indic}[1]{\1_{\{#1\}}}
\newcommand{\indicwo}[1]{\1_{#1}}

\DeclareMathSymbol{\expect}        {\mathord}{AMSb}{"45}
\DeclareMathSymbol{\expec}        {\mathord}{AMSb}{"45}
\DeclareMathSymbol{\prob}        {\mathord}{AMSb}{"50}
\DeclareMathSymbol{\Ibold}        {\mathord}{AMSb}{"49}
\DeclareMathSymbol{\Nbold}        {\mathord}{AMSb}{"4E}
\DeclareMathSymbol{\Rbold}        {\mathord}{AMSb}{"52}
\DeclareMathSymbol{\Zbold}        {\mathord}{AMSb}{"5A}
\newcommand{\ver}{\Z^d}

\newcommand{\C}{\mathbb C}

\newcommand{\Z}{\Zbold}
\newcommand{\Zd}{\Zbold^d}

\newcommand{\sss}   { \scriptscriptstyle }

\newcommand{\conn}{\longleftrightarrow}

\newcommand{\dbc}{\Longleftrightarrow}
\newcommand{\ct}[1]     { \stackrel{#1}{\conn} }
\newcommand{\connLe}[1]     { \stackrel{#1}{\conn} }

\newcommand{\ctx}[1]     {\leftarrow\shift\!\xrightarrow{#1}}
\newcommand{\nc}        { \conn {\hspace{-2.8ex} /} \hspace{1.8ex}   }

\newcommand{\bb}{\underline{b}}
\newcommand{\tb}{\overline{b}}

\newcommand{\eqalign}[1]
    {\begin{align}#1\end{align}
    }
\newcommand{\eqarray}   {\begin{eqnarray}}
\newcommand{\enarray}   {\end{eqnarray}}
\newcommand{\lbeq}[1]  {\label{e:#1}}
\newcommand{\refeq}[1] {\eqref{e:#1}}
\newcommand{\eq}{\begin{equation}}
\newcommand{\en}{\end{equation}}
\newcommand{\ben}{\begin{enumerate}}
\newcommand{\een}{\end{enumerate}}
\newcommand{\eqn}[1]{\begin{equation} #1 \end{equation}}
\newcommand{\eqan}[1]{\begin{align} #1 \end{align}}

\newcommand{\shift}{\!\!\!\!}
\newcommand{\nn}{\nonumber}
\newcommand{\nnb}{\nonumber\\}

\renewcommand{\to}{\rightarrow}

\def\Zd{\mathbb{Z}^d}

\def\mA[#1]{ {\bf A}(#1)}
\def\miA[#1]{ {\bf A}^{-1}(#1)}

\def\mD[#1]{{\bf \hat D}(#1)}
\def\mE[#1]{{\bf \hat E}(#1)}
\def\mM[#1]{{\bf \hat M}_1(#1)}
\def\mMa[#1]{{\bf \hat M}_2(#1)}
\def\mE[#1]{{\bf E}_{#1}}
\def\ve[#1]{ {e}_{#1}}

\def\v1{{\vec 1}}
\def\mJ{{\bf J}}
\def\mI{{\bf I}}

\newcommand{\dmin}{11}

\def\mPi[#1]{\hat {\bf \Pi}_z(#1)}
\def\mPiwoz[#1]{\hat {\bf \Pi}(#1)}
\def\mPiM[#1]{\hat {\bf \Pi}_{{\sss M}}(#1)}

\def\vXi[#1]{\vec {\hat \Xi}(#1)}
\def\vXiz[#1]{\vec {\hat \Xi}_z(#1)}
\def\vXiM[#1]{\vec {\hat \Xi}_{\sss M}(#1)}
\def\vRM[#1]{\vec {\hat R}_{\sss M}(#1)}
\def\vPsi[#1]{\vec {\hat \Psi}(#1)}
\def\vPsiT[#1]{\vec {\hat \Psi}^T(#1)}
\def\vPsiz[#1]{\vec {\hat \Psi}_z(#1)}
\def\vPsiMT[#1]{\vec {\hat \Psi}^T_{\sss M}(#1)}
\def\vPsiM[#1]{\vec {\hat \Psi}_{\sss M}(#1)}
\def\hatPhiM[#1]{\hat {\Phi}_{\sss M}(#1)}

\def\betaaa{\beta_{\sss \mu}}
\def\betaaalow{\underline \beta_{\sss \mu}}

\def\aap{\mu_p}
\def\aabp{\bar {\mu}_p}

\def\ssss[#1]{{\sss \text{\rm #1}}}
\def\ssc[#1]{{\sss( \text{\rm #1})}}
\def\genC[#1]{\hat {C}_{\mu_z}(#1)}
\def\genG[#1]{\hat {G}_{z}(#1)}
\def\bvtheta[#1]{\vec  {\boldsymbol \theta}_{#1}}

\def\diagRepulsiveLetter{\mathscr}

\newcommand{\LTsecondcolomnwidth}{20mm}
\newcommand{\LTthirdcolomnwidth}{20mm}
\newcommand{\LTfourthcolomnwidth}{32mm}
\newcommand{\LTfivethcolomnwidth}{38mm}

\newcommand{\fourcolomntablePer}[3]{
\begin{table}[ht!]
\caption{#1}
\label{#2}
{\small
\begin{tabular}
{|
>{\centering\arraybackslash} m{\LTsecondcolomnwidth-9mm}|
>{\centering\arraybackslash} m{\LTthirdcolomnwidth+1mm}|
>{\centering} m{\LTfourthcolomnwidth+2mm}|
>{\centering\arraybackslash} m{\LTfivethcolomnwidth+3mm}|
}
  \hline
  Parameter & Condition & Diagram & Definition \\
  \hline
#3
  \hline
\end{tabular}
}
\end{table}
}

\newcommand{\threecolomntablePer}[3]{
\begin{table}[h!]
\caption{#1}
\label{#2}
{\small
\begin{tabular}
{|
>{\centering\arraybackslash} m{\LTsecondcolomnwidth+10mm}|
>{\centering} m{\LTfourthcolomnwidth+\LTthirdcolomnwidth-10mm}|
>{\centering\arraybackslash} m{\LTfivethcolomnwidth}|
}
  \hline
  Parameter &Diagram & Definition \\
  \hline
#3
  \hline
\end{tabular}
}
\end{table}
}

\def\projIndexsetdir[#1]{{\mathsf {IK} }(#1)}
\def\projIndexsetPoints[#1]{{\mathsf {AB} }(#1)}
\def\projIndexsetNumber[#1]{{\mathsf {NM} }(#1)}
\def\projIndexsetPointsTwo[#1]{{\mathsf {X} }(#1)}

\newcommand{\ii}{{\mathrm i}}

\newtheorem{theorem}{Theorem}[section]
\newtheorem{corollary}[theorem]{Corollary}
\newtheorem{lemma}[theorem]{Lemma}
\newtheorem{prop}[theorem]{Proposition}

\newtheorem{ass}[theorem]{Assumption}
\newtheorem{definition}[theorem]{Definition}
\newtheorem{remark}[theorem]{Remark}
\numberwithin{equation}{section}
\numberwithin{theorem}{section}

\title{Mean-field behavior for nearest-neighbor percolation in $d>10$\iflongversion : Extended version \fi}

\author{Robert Fitzner\thanks{Department of Mathematics and
        Computer Science, Eindhoven University of Technology,
        5600 MB Eindhoven, The Netherlands.
        {\tt math@fitzner.nl},{\tt rhofstad@win.tue.nl}}
        \and
        Remco van der Hofstad$^*$
    }

    	\date{November 1, 2016}

\begin{document}
\maketitle

%=================================================IntroAbstract=================================================================
%==========================================================================================================================

\begin{abstract}
We prove that nearest-neighbor percolation in dimensions $d\geq\dmin$ displays mean-field behavior by proving that the infrared bound holds, in turn implying the finiteness of the percolation triangle diagram. The finiteness of the triangle implies the existence and mean-field values of various critical exponents, such as $\gamma=1, \beta=1, \delta=2$. We also prove sharp $x$-space asymptotics for the two-point function and the existence of various arm exponents.  Such results had previously been obtained in unpublished work by Hara and Slade for nearest-neighbor percolation in dimension $d\geq 19$, so that we bring the dimension above which mean-field behavior is rigorously proved down from $19$ to $\dmin$. Our results also imply sharp bounds on the critical value of nearest-neighbor percolation on $\Z^d$, which are provably at most $1.306\%$ off in $d=\dmin$. We make use of the general method analyzed in \cite{FitHof13b}, which proposes to use a lace expansion perturbing around non-backtracking random walk. This proof is {\em computer-assisted}, relying on (1) rigorous numerical upper bounds on various simple random walk integrals as proved by Hara and Slade \cite{HarSla92a}; and (2) a verification that the numerical conditions in \cite{FitHof13b} hold true. These two ingredients are implemented in two Mathematica notebooks that can be downloaded from the website of the first author.

The main steps of this paper are (a) to derive a non-backtracking lace expansion for the percolation two-point function; (b) to bound the non-backtracking lace expansion coefficients, thus showing that the general methodology of \cite{FitHof13b} applies, and (c) to describe the numerical bounds on the coefficients.

\iflongversion
In the appendix of this extended version of the paper, we give additional details about the bounds on the NoBLE coefficients that are not given in the article version.
\fi
\end{abstract}

%\tableofcontents
%=================================================IntroFirst=================================================================
%==========================================================================================================================
\section{Introduction}
\label{sec-intro}
\subsection{Motivation}
\label{sec-motiv}
The \emph{lace expansion} was invented by Brydges and Spencer in 1985 \cite{BrySpe85} to prove mean-field behavior for weakly self-avoiding walk. Thereafter, it was extended to self-avoiding walks (SAW), percolation, and lattice trees and animals \cite{HarSla90a,HarSla90b,HarSla92a,Slad87}, and has become an indispensable tool to prove mean-field behavior of statistical mechanical models above the so-called upper critical dimension. More recent extensions include oriented percolation \cite {HofSla02, NguYan93,NguYan95}, the contact process \cite{HofSak04,Saka01}, and the Ising model \cite{Saka07}.

Being a \emph{perturbative method} in nature, applications of the lace expansion typically necessitate a small parameter. This small parameter tends to be the degree of the underlying base graph. There are two possible approaches to obtain a small parameter. The first is to work in a so-called \emph{spread-out} model, where long- but finite-range connections over a distance $L$ are possible, and we take $L$ large. This approach has the advantage that the results hold, for $L$ sufficiently large, all the way up to the critical dimension of the corresponding model. The second approach applies to the simplest and most often studied nearest-neighbor version of the model. For the nearest-neighbor model, the degree of a vertex is $2d$ which then has to be taken large in order to prove mean-field results. Thus, we need to take the dimension large, and therefore obtain suboptimal results in terms of the dimension above which the results hold. A seminal exception is SAW, where Hara and Slade \cite{HarSla92b} have proved that $d\geq 5$ is sufficient for their perturbation analysis to hold, using a computer-assisted method. For SAW, mean-field results are expected to be false in dimension $d=4$. See the work using the renormalization group to identify the logarithmic corrections to mean-field behavior by Bauerschmidt, Brydges and Slade in \cite{BauBrySla15b} and the references therein. Here, the Green's function does not have logarithmic corrections \cite{BauBrySla15b}, while, e.g., the susceptibility does  \cite{BauBrySla15a}.

For percolation, on the other hand, this methodology was proved to apply in the nearest-neighbor setting only for $d\geq 19$ (see \cite{HarSla90a, HarSla94}), and makes use of similar computer-assisted methods as used for SAWs in \cite{HarSla92a}. These computations were never published. Through private communication with Takashi Hara, the authors have learned that he recently obtained a further improvement to $d\geq 15$. Hara's proof is restricted to $d\geq 15$, as it assumes that the heptagon is finite.

These seemingly suboptimal results, where the results are proved to hold for $d\geq 19$, while they are expected to hold for $d>6$, have a reason that is quite deep. Indeed, it is well-known that for mean-field behavior to hold for percolation, it is sufficient that the so-called \emph{triangle condition} holds (see e.g., \cite{AizNew84,BarAiz91}). As we explain in more detail below, the triangle condition states that a certain sum called the triangle diagram is \emph{finite}. However, the current methodology of the lace expansion only applies when the triangle diagram is \emph{sufficiently small}. Thus, we can think of $d\geq 19$ to be sufficient for the triangle diagram to be sufficiently small, rather than being finite, and there previously was no way to prove that the triangle diagram is finite rather than small. In this paper, we take a first step to prove such a result, by proving that the triangle diagram is finite, but in the proof working with different diagrams that need to be small. We are able to do so, since the diagrams that we obtain in our analysis contain loops of at least four bonds, while the classical lace expansion gives rise to loops that can also contain two bonds. This allows us to reduce the dimension above which the infrared bound holds from 19 to $\dmin$.

We extend the proof by Hara and Slade so that it applies to $d\geq \dmin$, by using several novel ideas. The main innovations in our proof are that (i) we perturb around {\em non-backtracking random walk,} rather than simple random walk, so that the lace-expansion coefficients are significantly smaller than in the classical lace expansion as used by Hara and Slade in \cite{HarSla90a}; (ii) our bounds on the lace-expansion coefficients are {\em matrix-based,} so as to profit maximally from the fact that loops consists of at least four bonds in our expansion; (iii) we use and provide Mathematica notebooks that implement the bounds and that can be downloaded from the first author's website. As a side result, our proof gives the best bounds on the percolation threshold available in the literature, that are of independent interest.

Our results prove that the percolation triangle is finite, and thus prove that many critical exponents exist and take on their mean-field values such as the ones related to the percolation function ($\beta=1$), to the expected subcritical cluster size ($\gamma=1$) and to the critical cluster-tail distribution ($\delta=2$). Further, we prove the sharp asymptotics of the critical two-point function in $x$-space, using the results of Hara \cite{Hara08}, which in turn implies the existence of several arm-exponents as proved by Kozma and Nachmias \cite{KozNac08, KozNac11}, as well as the existence of the {\em incipient infinite cluster} \cite{HeyHofHul14a, HofJar04}. An overview about recent results on high-dimensional percolation can be found in \cite{HeyHof15}.\\
Also our proof is computer-assisted, and relies on the following two key ingredients:
\begin{enumerate}
\item[(I)] Rigorous upper bounds on various simple random walk integrals, as proved by Hara and Slade in \cite{HarSla92a}, where they also served as a key ingredient in the proof. This part of the analysis is {\em unchanged} compared to the Hara-Slade proof. The crucial reason why we can use these integrals is that the non-backtracking random walk Green's function can be explicitly described in terms of the simple random walk Green's function. Our analysis requires us to compute 112 such integrals, corresponding to convolutions of random walk Green's functions with itself at various values in $\Z^d$. We further need to compute the number of simple random walks of lengths up to 10 ending at various values in $\Z^d$, as well as related self-avoiding walks and bond-self-avoiding walks.  These bounds are performed in one Mathematica notebook;
\item[(II)] Two other Mathematica notebooks, a first that implements the computations in our general approach to the non-backtracking lace expansion (NoBLE) in \cite{FitHof13b}, as well as a notebook that computes the rigorous bounds on the lace-expansion coefficients provided in the present paper. These notebooks do nothing else than implement the bounds proved here and in \cite{FitHof13b}, and rely on nothing but many multiplications, additions as well as diagonalizations of two three-by-three matrices. These computations {\em could} be performed by hand, but the use of the notebooks tremendously simplifies them.
\end{enumerate}

The fact that our Mathematica notebooks are made publicly available maximizes the {\em transparancy} for the entire community about the status of the proof.  Indeed, for one, the community can verify that the computations performed really {\em are} the ones provided in this paper and in \cite{FitHof13b}, for second, anyone interested can optimize constants so as to improve bounds on various percolation parameters.

We next introduce the nearest-neighbor percolation model that we investigate, and state our main results.

%=================================================IntroModelResults=================================================================
%==========================================================================================================================
\renewcommand{\Ccal}{\mathscr{C}}
\subsection{Model}
\label{sec-mod}
In nearest-neighbor percolation, we set each bond $\{x,y\}\in\Z^d\times\Z^d$, with $x$ and $y$ nearest-neighbors, \emph{occupied}, independently of all other bonds, with probability $p$ and \emph{vacant} otherwise. The corresponding product measure is denoted by $\prob_p$ with corresponding expectation $\expec_p$. We write $\{x\conn y\}$ for the event that there exists a path of occupied bonds from $x$ to $y$. When the event $\{x\conn y\}$ occurs we call the vertices $x$ and $y$ \emph{connected}. For $x\in\Z^d$, the set $\Ccal(x):=\{y\in\Z^d\mid y\conn x\}$ of vertices connected to $x$ is called the \emph{cluster} of $x$. It is the size and geometry of these clusters that we are interested in.

Clearly, for $p$ small, $\Ccal=\Ccal(0)$ is $\prob_p$-a.s.\ finite, whereas for $d\ge2$ and large $p$, the percolation probability
    \eqn{
    \lbeq{def-Theta}
    \theta(p)=\prob_p(|\Ccal|=\infty),
    }
i.e., the probability that the cluster $\Ccal$ is infinite, is strictly positive. Hence, there exists some critical value where this probability turns positive (see e.g.\ \cite{Grim99}). As it turns out, it is convenient for us to use a different definition of the critical value, as we explain now. For this, we define the \emph{two-point function} $\tau_p\colon \Z^d\times \Z^d \to [0,1]$ by
    \eqn{
    \lbeq{def-tau}
    \tau_p(x,y)=\prob_p(x\conn y).
    }
By translation invariance, $\tau_p(x,y)=\tau_p(0,y-x)\equiv \tau_p(y-x)$. We further introduce the \emph{susceptibility} as
    \eqn{
    \lbeq{def-Chi}
    \chi(p)=\sum_{x\in\Z^d}\tau_p(x).
    }
We define $p_c$, the critical value of $p$, as
    \eqn{
    \lbeq{def-pc-alt}
    p_c(d)=\sup\,\left\{p\,|\,\chi(p)<\infty\right\}.
    }
Thus, $p_c$ is characterized by the explosion of the expected cluster size. Menshikov \cite{Mens86}, as well as Aizenman and Barsky \cite{AizBar87}, have proved that this characterization coincides with the critical value described below \refeq{def-Theta}.
See also the recent and very short proof of this fact by Duminil-Copin and Tassion \cite{DumTas15}.

We now discuss the notion of \emph{critical exponents}. It is predicted that
    \eqn{
    \lbeq{beta-def}
    \theta(p) \sim (p-p_c)^{\beta}\qquad \text{as} \qquad p\searrow p_c,
    }
for some $\beta>0$. The symbol $\sim$ in \refeq{beta-def} can have several meanings, and we shall always mean that the critical exponent exists in the \emph{bounded-ratio form}, meaning that there exist $0<c_1<c_2<\infty$ such that, uniformly for $p\geq p_c$,
    \eqn{
    \lbeq{beta-def-bounded-ratios}
    c_1(p-p_c)^{\beta}\leq \theta(p)\leq c_2(p-p_c)^{\beta}.
    }
The existence of a critical exponent is \emph{a priori} unclear, and needs a mathematical proof. Indeed, the existence of the critical exponent $\beta>0$ is stronger than the continuity of $p\mapsto \theta(p)$, which is unknown in general, and is arguably the holy grail of percolation theory. More precisely, $p\mapsto \theta(p)$ is clearly continuous on $[0,p_c)$, and it is also continuous (and even infinitely differentiable) on $(p_c,1]$ by the results of \cite{BerKea84} (for infinite differentiability of $p\mapsto \theta(p)$ for $p\in (p_c,1]$, see \cite{Russ78}). Thus, continuity of $p\mapsto \theta(p)$ is equivalent to the statement that $\theta(p_c(d))=0$. The critical exponent $\gamma$ for the expected cluster size is given by
    \eqn{
    \lbeq{gamma-def}
    \chi(p) \sim (p_c-p)^{-\gamma},\qquad p\nearrow p_c,
    }
while the exponent $\delta\geq 1$ measures the power-law exponent of the critical cluster tail, i.e.,
    \eqn{
    \lbeq{delta-def}
    \prob_{p_c}(|\Ccal(0)|\geq n) \sim n^{-1/\delta}, \qquad n\rightarrow \infty,
    }
the assumption that $\delta\geq 1$ following from the fact that $\chi(p_c)=\infty$.

\subsection{Results}
\label{sec-res}
Our analysis makes heavy use of Fourier analysis. Unless specified otherwise, $k$ always denotes an arbitrary element from the Fourier dual of the discrete lattice, which is the torus $[-\pi,\pi]^d$. The Fourier transform of a function $f\colon\Z^d\to\C$ is defined by
    \eqn{
    \lbeq{def-FourTrans}
    \hat f(k)=\sum_{x\in\Zd}f(x)\,\e^{\ii k\cdot x}.
    }
For two summable function $f,g\colon \Zd \mapsto \Rbold$, we let $f\star g$ denote their \emph{convolution}, i.e.,
    \eqn{
    \lbeq{definition-convolution}
    (f\star g)(x)= \sum_{x\in\Zd} f(y)g(x-y).
    }
We note that the Fourier transform of $f\star g$ is given by the product of $\hat f$ and $\hat g$. In particular, let $D(x)=\indic{|x|=1}/(2d)$ be the nearest-neighbor random walk transition probability, so that
    \eqn{
    \lbeq{def-Dhat}
    \hat{D}(k)=\frac{1}{2d}\sum_{x:|x|=1} \e^{\ii k\cdot x} = \frac{1}{d} \sum_{i=1}^d \cos(k_i).
    }
The main result of this paper is the following \emph{infrared bound:}

\begin{theorem}[Infrared bound]
\label{thm-IRB}
For nearest-neighbor percolation with $d\geq \dmin$, there exist constants $A_1(d)$ and $A_2(d)$ such that
    	\eqn{
    	\lbeq{IRB}
    	\hat{\tau}_p(k)\leq \frac{A_1(d)}{\chi(p)^{-1}+p[1-\hat{D}(k)]}
	\quad\qquad
	\text{and}
	\quad\qquad
	\hat{\tau}_p(k)\leq \frac{A_2(d)}{1-\hat{D}(k)},
    	}
uniformly for $p\leq p_c(d)$.
\end{theorem}
\medskip

Our methods require a detailed analysis of both the critical value as well as the
amplitudes $A_1(d)$ and $A_2(d)$. As a result, we obtain the following bounds:
\begin{theorem}[Bounds on critical value and amplitude]
\label{thm-bds-crit}
For nearest-neighbor percolation with $d\geq \dmin$,
the following upper bounds hold:
{\rm \begin{center}
\begin{tabular}{|c|  c|c | c|c |c|c|}
  \hline
  $d$             & 11& 12   & 13 & 14  & 15  & 20 \\
  \hline
  $(2d-1)p_c(d)\leq$ &1.01306 & 1.00857 &1.006244  & 1.00484 &  1.0039  & 1.001777  \\
   $A_2(d)\leq$      &1.02393 & 0.9947  & 0.986    & 0.98237 &  0.98085 & 0.981136   \\
 \hline
\end{tabular}
\end{center}
}
\end{theorem}
\bigskip
Remarkably, the bound on $d\mapsto A_2(d)$ is not decreasing, as is usually the case. The explanation of this may be quite simple. Indeed, for NBW, $A_2(d)=(2d-2)/(2d-1)$, which is {\em increasing}. Unfortunately, our method does not allow us to get very close to this value, particularly for {\em low} dimensions. This explains why first $d\mapsto A_2(d)$ decreases (as we get closer to the NBW constant), after which it increases, being very close to its NBW counterpart.
\medskip

In the literature, the following numerical values given in Table \ref{tab-num-pc}, have appeared for the percolation critical value. These values indicate that the approximation $p_c(d)\approx 1/(2d-1)$ is already quite good for $d\geq 7$, being at most 2\% off the reported numerical values. Also, our estimate for $p_c(11)$ is only around 0.62\% off the reported numerical value, the one for $p_c(13)$ only 0.16\%.

\begin{center}
\begin{table}[h]
{\small \begin{tabular}{|c|c|c|c|c|c|c|c|}
  \hline
  $d$                               & 7                 & 8                     & 9			&10             &11		&12		&13\\
  \hline
  $p_c(d)\approx$                   	&0.078675  &0.06771               &0.05947	&0.0531      	&0.04795	&0.04373	&0.040188\\
  $\!\!(2d-1)p_c(d)\approx$\!\!             	& 1.02278     &1.015626               &1.011433       &1.00876		&1.00694	&1.00565	&1.0047\\
  \hline
\end{tabular}
}
\caption{Numerical values of $p_c(d)$ taken from \cite[Table I]{Gras03}. Related numerical values can be found in \cite[Table 3.6]{Hugh96} and \cite{AdlMeiAhaHar90}.}
\label{tab-num-pc}
\end{table}
\end{center}

We next report some consequences of the infrared bound in Theorem \ref{thm-IRB}. In \cite{AizNew84}, it was proved that $\gamma=1$ in the bounded-ratio sense when the so-called \emph{triangle condition}, a condition on the percolation model, holds. The triangle condition states that
    \eqn{
    \lbeq{triangle-cond-def}
    \bigtriangleup(p_c)=\sum_{x,y\in \Z^d} \tau_{p_c}(0,x)\tau_{p_c}(x,y)\tau_{p_c}(y,0)
    <\infty.
    }
In \cite{BarAiz91}, it was shown that, under the same condition, $\beta=1$ and $\delta=2$ in the bounded-ratio sense. Since
    	\eqn{
    	\bigtriangleup(p_c)=(\tau_{p_c}\star\tau_{p_c}\star\tau_{p_c})(0)
	=\lim_{p\nearrow p_c} (\tau_{p}\star\tau_{p}\star\tau_{p})(0)
	=\lim_{p\nearrow p_c} \int_{(-\pi,\pi)^d} \frac {dk} {(2\pi)^d} \hat \tau^3_{p}(k),
    	}
the infrared bound in Theorem \ref{thm-IRB} (which is uniform in $p<p_c$) immediately implies that the triangle condition holds, and therefore that $\gamma=\beta=1, \delta=2$:
\begin{corollary}[Triangle condition and critical exponents]
\label{cor-TC-crit-exp}
For nearest-neighbor percolation with $d\geq\dmin$, the triangle condition holds. Therefore the critical exponents $\gamma, \beta$ and $\delta$ exist in the bounded-ratio sense, and take on the mean-field values $\gamma=\beta=1, \delta=2$.
\end{corollary}
\medskip

We next investigate the asymptotics in $x$-space of $\tau_{p_c}(x)$ for $x$ large, using the results of Hara in \cite{Hara08}:

\begin{theorem}[Two-point function asymptotics]
\label{thm-x-space}
For nearest-neighbor percolation with $d\geq \dmin$, there exists a constant $A(d)$ such that, as $|x|\rightarrow \infty$,
    	\eqn{
    	\lbeq{x-space}
    	\tau_{p_c}(x)=\frac{a_dA(d)}{|x|^{d-2}}(1+O(|x|^{-2/d})),
	\qquad
	\text{with}
	\qquad
	a_d=\frac{d\Gamma(d/2-1)}{2\pi^{d/2}}.
    }
\end{theorem}
\medskip

The proof of Theorem \ref{thm-x-space} follows by verifying that the conditions that Hara poses in \cite{Hara08}, which are formulated in terms of the classical lace expansion, are satisfied. In particular, these conditions imply that $\bigtriangleup(p_c)$ is quite close to 1, which is the contribution in \refeq{triangle-cond-def} of $x=y=0$. We have decided to state Theorem \ref{thm-x-space} explicitly, as it has major consequences, such as the existence of the incipient infinite cluster (IIC) and the behavior of random walks on it. We next state these results.

Let $\Fcal$ denote the $\sigma$-algebra of events. A {\em cylinder event} is an event given by conditions on the states of finitely many bonds only. We denote the algebra of cylinder events by $\Fcal_0$. We define
	\eqn{
	\lbeq{def-P_x}
  	\prob_x(F)= \prob(F\mid 0 \conn x)
           = \frac{1}{\tau_{p_c}(x)} \prob(F,\, 0 \conn x),
         \quad F \in \Fcal.
	}
Then, the results on the existence of the IIC in \cite{HofJar04, HeyHofHul14a} imply that the following theorem holds:

\begin{theorem}[Existence of the IIC]
\label{thm-IIC-lim}
Let $d \geq \dmin$ and $p = p_c$. Then, the limit
	\eqn{
	\lbeq{IIC-lim}
  	\prob_{\sss \infty}(F) = \lim_{|x| \to \infty} \prob_x(F)
	}
exists for any cylinder event $F$. Also, $\prob_{\sss \infty}$ extends uniquely from $\Fcal_0$ to a probability measure on $\Fcal$.
\end{theorem}
\medskip

There is quite some literature investigating the existence of IICs and proving that different limiting schemes produce the same limiting IIC measure. We refer to \cite{HeyHofHul14a, HofHolSla98, HofJar04, Jara03b, Jara03a,Kest86a} for more details. We also note that the existence of the IIC measure as well as Theorem \ref{thm-x-space}  allow for a proof that the Alexander-Orbach conjecture holds for nearest-neighbor percolation and $d\geq \dmin$, see \cite{KozNac08} for more details.

We close this section by identifying two arm-exponents that follow from Theorem \ref{thm-x-space}, using recent proofs by Kozma and Nachmias \cite{KozNac08, KozNac11}. For this, we start by introducing some notation. Fix $p=p_c$ and let $B_r(x)$ denote the ball of intrinsic radius $r$ from $x\in \Z^d$. Thus, $y\in B_r(x)$ when there exists a path of at most $r$ occupied bonds connecting $y$ to $x$. We further write $\partial B_r(x)=B_r(x)\setminus B_{r-1}(x)$. Kozma and Nachmias \cite{KozNac08, KozNac11} have proved that Theorem \ref{thm-x-space} implies that arm-probabilities decay as inverse powers of $r$, and have identified the corresponding critical exponents:

\begin{theorem}[One-arm exponents]
\label{thm-one-arm}
Let $d \geq \dmin$ and $p = p_c$. Then,
	\eqn{
	\lbeq{one-arm}
  	\prob_{p_c}(\partial B_r(0)\neq \varnothing)\sim r^{-1},
	\qquad
	\text{and}
	\qquad
	\prob_{p_c}(0\conn Q_r^c)\sim r^{-2},
	}
in the bounded-ratios sense, where $Q_r=\{x\in \Z^d\colon |x|\leq r\}$.
\end{theorem}
\medskip

\noindent
In the next section, we give an overview of the proof of Theorem \ref{thm-IRB}.

%=================================================IntroOv=================================================================
%==========================================================================================================================
\section{Overview of the proof}
\label{sec-overview}
In this section, we give a brief overview of how we derive our main results. We start by describing the philosophy behind our proof.

\subsection{Philosophy of the proof}
\label{sec-phil-proof}
In this section, we reduce the proof of Theorem \ref{thm-IRB} to three key propositions and a computer-assisted proof. These ingredients involve
	\begin{enumerate}[(a)]
  	\item the derivation of the non-backtracking lace expansion (NoBLE) in Proposition \ref{prop-LE};
  	\item the diagrammatic bounds on the NoBLE coefficients in Proposition \ref{prop-bds-LEC};
  	\item the analysis presented in \cite{FitHof13b} to obtain the infrared bound in Theorem \ref{thm-IRB} for all $p\leq p_c$ for all $d\geq\dmin$, as stated in Proposition \ref{prop-analysis-is-success}; and
	\item a computer-assisted proof to verify the numerical conditions arising in the analysis in \cite{FitHof13b}.
	\end{enumerate}
\label{parts-proof}
These parts are discussed in Sections \ref{sec-part-a}-\ref{sec-part-d}, respectively.

In Sections \ref{secExp} and \ref{secBoundsPerc} we prove parts (a) and (b), respectively. In Section \ref{sec-part-c}, we explain how we obtain part (c) using the analysis of \cite{FitHof13b}, the computer-assisted proof \cite{FitNoblePage} of part (d) and the results of this paper.
For the analysis in the generalized setting \cite{FitHof13b} we state assumptions, which we verify for percolation in  Sections \ref{sec-part-c}, \ref{subsec-prop-coef} and \ref{secSummaryBounds}, respectively. Part (d) is explained in detail in Section \ref{sec-part-d}, where we describe how the necessary computations are performed in several Mathematica notebooks. The mathematics behind the notebooks is explained in \cite{FitHof13b}. The notebooks also include a routine that verifies whether the numerical assumption on the expansions are satisfied and thereby verifies whether the analysis of \cite{FitHof13b} yields the infrared bounds for percolation or not for a given dimension. Thus, \cite{FitHof13b} and the notebooks together prove Proposition \ref{prop-analysis-is-success}. See also Figure \ref{Struct-NoBLE} for a visual description of the proof of Theorem \ref{thm-IRB}.  We prove Theorems \ref{thm-x-space}, \ref{thm-IIC-lim} and \ref{thm-one-arm} in Section \ref{sec-proof-rel-res} and close this section with a discussion of our method and results.
\vskip-1cm

\begin{center}
\begin{figure}[ht]
\begin{tikzpicture} [scale=0.85]

  \draw [->,line width=2] (2.5,0) to (6.25,0);
  \draw [->,line width=2] (3.75,-4.75) to (6.25,-4.75);
  \draw [<-,line width=2] (10,-1.25) to (10,-3);
  \draw [->,line width=2] (9.5,-1.25) to (9.5,-3);
  \draw [->,line width=2] (0,-1.25) to (0,-3.25);

  \draw [<-,line width=2] (9.75,-5.75) to (9.75,-6.75);

  \node[above]   at(4.4,0.1)     {General relation};
  \node[below]   at(4.6,-5)   {Bounds in};
  \node[below]   at(4.6,-5.5)   { form of diagrams};
%  \node[left]   at(0,-2)        {\begin{tabular}{c} Coefficients to \\ describe the perturbation \end{tabular}};
  \node[left]   at(0,-1.5)        {Coefficients to};
  \node[left]   at(0,-2)        {describe the};
  \node[left]   at(0,-2.5)        {perturbation};
  \node[right]   at(10,-1.75)      {Bound on the};
  \node[right]   at(10,-2.25)      {perturbation};

  \node[left]   at(9.5,-1.5)      {Prior bounds};
  \node[left]   at(9.5,-2)       {on the two-point };
  \node[left]   at(9.5,-2.5)      {function};
  \node[right]   at(10,-6.2)      {Numerical values};

\node [draw] {
\begin{tabular}{c}
{\bf \large Expansion} \\[1mm]
\hline
%\vspace{4mm}
{\large Proposition \ref{prop-LE}}
\end{tabular}};

\node [draw]  at (10,0){
\begin{tabular}{c}
{\bf \large Analysis/ Bootstrap } \\[1mm]
\hline
\vspace{2mm}
{\large in accompanying paper \cite{FitHof13b}}\\[-2mm]
{\normalsize Model independent}
\end{tabular}};

\node [draw]  at (10,-4.5){
\begin{tabular}{c}
{\bf \large Numerical bounds } \\[1mm]
\hline
\vspace{2mm}
{\large in accompanying}\\
{\large Mathematica notebooks \cite{FitNoblePage}}
\end{tabular}};

\node [draw]  at (10,-8){
\begin{tabular}{c}
{\bf \large Numerical computation} \\[1mm]
\hline
{\large in accompanying}\\
{\large Mathematica notebooks \cite{FitNoblePage}}
\end{tabular}};

\node [draw]  at (0,-4.5){
\begin{tabular}{c}
{\bf \large Diagrammatic bounds } \\[1mm]
\hline
{\large Propostion \ref{prop-bds-LEC} }
\end{tabular}};
\end{tikzpicture}
\caption{Structure of the non-backtracking lace expansion.}
\label{Struct-NoBLE}
\end{figure}
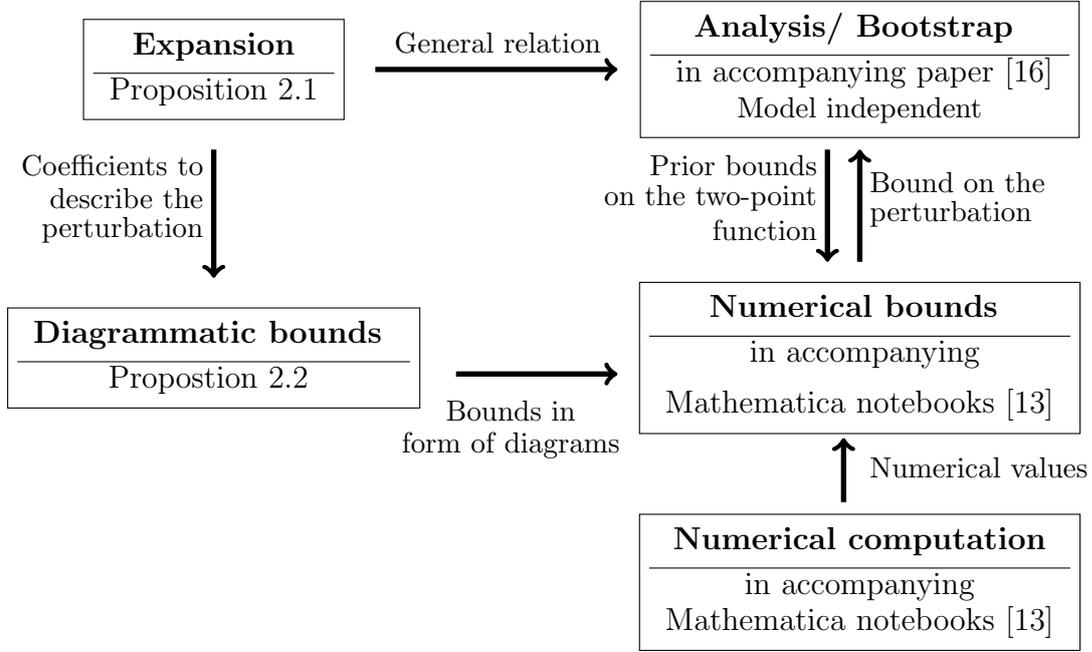
\end{center}

\vskip-1cm

In the following, we explain the philosophy behind the non-backtracking lace expansion (NoBLE), and start by describing simple random walk and non-backtracking walk.

\paragraph{Simple random walk and non-backtracking walk.}
In the expansion we derive that the percolation two-point function $\tau_p$ can be viewed as a perturbation of the non-backtracking walk two-point function. We now define simple and non-backtracking walk to be able to formalise this connection.

An $n$-step \emph{nearest-neighbor simple random walk} (SRW) on $\Zd$ is an ordered $(n+1)$-tuple $\omega=(\omega_0,\omega_1,\omega_2,\dots, \omega_n)$, with $\omega_i\in\Zd$ and $\|\omega_i-\omega_{i+1}\|_1=1$, where $\|x\|_1=\sum_{i=1}^d |x_i|$. Unless stated otherwise, we take $\omega_0=(0,0,\dots,0)$.\\
We define $p_n(x)$ to be the number of $n$-step SRWs with $\omega_n=x$. Then, for $n\geq 1$,
	\begin{eqnarray}
	\lbeq{SRWRecScheme}
	p_n(x) =\sum_{y\in\Zd} 2d D(y)p_{n-1}(x-y) =2d (D \star p_{n-1})(x)
	=(2d)^{n} D^{\star n}(x),
	\end{eqnarray}
where $D$ is the one-step transition probability, see also \refeq{def-Dhat}, and $f^{\star n}$ denotes the $n$-fold convolution of a function $f$.  The SRW two-point function is given by the generating function of $p_n$, i.e., for $|z|<1/(2d)$,
	\begin{eqnarray}
	\lbeq{genSRW}
	C_z(x)&=&\sum_{n=0}^\infty p_n(x)z^n,\qquad\text{and}
	\qquad \hat C_z(k) =\frac {1}{1-2dz\hat D(k)}
	\end{eqnarray}
in $x$-space and $k$-space, respectively. We denote the SRW \emph{susceptibility} by
	\begin{eqnarray}
	\chi^{\sss\rm SRW}(z)= \hat C_z(0)= \frac {1} {1-2dz},
	\end{eqnarray}
for $|z|<z_c$, with \emph{critical point} $z_c=1/(2d)$.

If an $n$-step SRW $\omega$ satisfies $\omega_i\not=\omega_{i+2}$ for all $i=0,1,2,\dots,n-2$, then we call $\omega$ \emph{non-backtracking}. In order to analyze non-backtracking walk (NBW), we derive an equation similar to \refeq{SRWRecScheme}. The same equation does not hold for NBW as it neglects the condition that the walk does not revisit the origin after the second step. %For this reason, we introduce the condition that a walk should not go in a certain direction $\iota$ with its first step.

We exclusively use the Greek letters $\iota$ and $\kappa$ for values in $\{-d,-d+1,\dots,-1,1,2,\dots,d\}$ and denote the unit vector in direction $\iota$ by $\ve[\iota]\in\Zd$, e.g.\ $(\ve[\iota])_i=\text{sign}(\iota) \delta_{|\iota|,i}$.

Let $b_n(x)$ be the number of $n$-step NBWs with $\omega_0=0,\omega_n=x$. Further, let $b^{\iota}_{n}(x)$ denote the number of $n$-step NBWs $\omega$ with $\omega_n=x$ and $\omega_1 \not =\ve[\iota]$. Summing over the direction of the first step\footnote{Bear in mind that the first step is to $-\ve[\iota]$.} we obtain, for $n\geq 1$,
	\begin{eqnarray}
	\lbeq{NBWRecScheme1}
	b_{n}(x) &=& \sum_{\iota\in\{\pm1,\dots,\pm d\}} b^{\iota}_{n-1}(x+\ve[\iota]).
	\end{eqnarray}
Further, we distinguish between the case that the walk visits $-\ve[\iota]$ in the first step or not to obtain, for $n\geq 1$,
	\begin{eqnarray}
	\lbeq{NBWRecScheme2}
	b_{n}(x)&=&b^{-\iota}_{n}(x) + b^{\iota}_{n-1}(x+\ve[\iota]).
	\end{eqnarray}
The NBW two-point functions $B_z$ and $B^{\iota}_{z}$ are defined as the generating functions of $b_n$ and $b_n^{\iota}$, respectively, i.e.,  for $|z|<1/(2d-1)$,
	\begin{eqnarray}
	B_{z}(x)&=&\sum_{n=0}^{\infty} b_n(x)z^n,\qquad \qquad
	B^{\iota}_{z}(x)=\sum_{n=0}^{\infty} b^{\iota}_n(x)z^n.
	\end{eqnarray}
In this paper, we use $\Cbold^{2d}$-valued and $\Cbold^{2d}\times\Cbold^{2d}$-valued quantities.
For a clear distinction between scalar-, vector- and matrix-valued quantities, we always write $\Cbold^{2d}$-valued functions with a vector arrow (e.g.\ $\vec v$) and matrix-valued functions with bold capital letters (e.g.\ ${\bf M}$). We do not use $\{1,2,\dots,2d\}$ as the index set for
the elements of a vector or a matrix, but use $\{-d,-d+1,\dots,-1,1,2,\dots,d\}$ instead.
Further, for $k\in[-\pi,\pi]^d$ and negative index $\iota\in\{-d,-d+1,\dots,-1\}$, we write $k_\iota=-k_{|\iota|}$.

We denote the identity matrix by $\mI\in\Cbold^{2d\times 2d}$ and the all-one vector by $ \v1 =(1,1,\dots,1)^T\in\Cbold^{2d}$. Moreover, we define the matrices $\mJ,\mD[k]\in\Cbold^{2d\times 2d}$ by
	\begin{eqnarray}
  \lbeq{J-D-matrix-def}
	(\mJ)_{\iota,\kappa}=\delta_{\iota,-\kappa}\qquad\text{ and }\qquad
	(\mD[k])_{\iota,\kappa}=\delta_{\iota,\kappa} \e^{\ii k_\iota}.
	\end{eqnarray}
We define the vector $\vec {\hat B}_{z}(k)$ with entries $(\vec {\hat B}_{z}(k))_\iota=\vec {\hat B}^{\iota}_{z}(k)$ and rewrite \refeq{NBWRecScheme1}-\refeq{NBWRecScheme2} to
	\eqn{
	\lbeq{NBWScheme1}
	\hat B_{z}(k)=1+z \v1^T \mD[-k] \vec {\hat B}_{z}(k),%\\
	%\lbeq{NBWScheme2}
	\qquad\qquad
	\hat B_{z}(k)\v1 =\mJ \vec {\hat B}_{z}(k) +  z\mD[-k]\vec {\hat B}_{z}(k),
	}
%so that
Then, as derived in detail in \cite[Section 1.2.2]{FitHof13b}, %we have
%	\eqn{
%	\vec {\hat B}_{z}(k)=\hat B_{z}(k)\left[\mI+z\mD[k]\mJ\right]^{-1}\v1.
%	}
%This implies that
	\begin{eqnarray}
	\lbeq{NBWGen}
	\hat B_{z}(k)&=& \frac {1} {1 -z\v1^T\left[\mD[k]+z \mJ\right]^{-1}\v1}.
	\end{eqnarray}
In turn, using that $\left[\mD[k]+z \mJ\right]^{-1}=\frac 1 {1-z^2} \left(\mD[-k]-z \mJ\right)$, this is equivalent to
	\eqan{
	\lbeq{NBWGenSolved}
	\hat B_{z}(k)&= \frac {1}{1-2d z\frac {\hat D(k)-z}{1-z^2}}
	= \frac {1-z^2} {1+(2d-1)z^2-2dz\hat D(k)}\nn\\
	&=\frac {1-z^2}{1+(2d-1)z^2} \cdot \hat C_\frac{z}{1+(2d-1)z^2}(k).
	}
The NBW susceptibility is $\chi^{\rm\sss NBW}(z)=\hat B_{z}(0)$ with critical point $z_c=1/(2d-1)$.
The NBW and SRW critical two-point functions are related by
	\eqn{
	\lbeq{SRWtoNBWlink}
	\hat B_{1/(2d-1)}(k)=\frac {2d-2}{2d-1}\hat C_{1/2d}(k)
	=\frac {2d-2}{2d-1} \cdot \frac 1 {1-\hat D(k)}.
	}
This link allows us to compute values for the NBW two-point function in $x$- and $k$-space, using the SRW two-point function.
A detailed analysis of the NBW including a proof that the NBW, when properly
rescaled, converges to Brownian motion can be found in \cite{FitHof13a}.

\subsection{Part (a): Non-Backtracking Lace Expansion (NoBLE)}
\label{sec-part-a}
In this section, we explain the shape of the Non-Backtracking Lace Expansion (NoBLE), which is a perturbative expansion of the two-point function.
The aim of the NoBLE for percolation is to derive equations alike \refeq{NBWScheme1} for the percolation two-point function $\tau_p(x)$. This is motivated by the fact that a large part of the interaction present in $\tau_p(x)$ is due to percolation paths not backtracking. We next explain how we can set this expansion up, and explain how our proof follows from three main propositions.

Next to the usual two-point function \refeq{def-tau}, we use a slight
adaptation of it. For a direction $\iota\in\{\pm1,\pm2,\dots, \pm d\}$, we define
    \eqn{
    \lbeq{tau-z-def}
    \tau^\iota_p(x)=
    %\prob_p(0\conn x \text{ occurs when all bonds containing $\ve[\iota]$ are made vacant})=
    \prob^{e_{\iota}}_p(0\conn x),
    }
where we write
    \eqan{
    \lbeq{conditional-Point-Prob-def}
    \prob^{y}_p(E)=\prob_p(E \text{ occurs when all bonds containing $y$ are made vacant})
    }
for $y\in\Zd$ and all events $E$.

Our analysis relies on two expansion identities relating $\tau_p(x)$ and $\tau_p^\iota(x)$, which are perturbations of \refeq{NBWScheme1}. %-\refeq{NBWScheme2}.
In the following, we drop the subscript $p$ when possible, and write, e.g., $\tau(x)=\tau_p(x)$ and $\tau^{\iota}(x)=\tau^{\iota}_p(x)$.
The NoBLE is formulated in the following proposition:

\begin{prop}[Non-backtracking lace expansion]
\label{prop-LE}
For every $x\in \Z^d$, $\iota,\kappa\in \{\pm1,\pm2, \dots, \pm d\}$, and $M\geq 1$, the following recursion relations hold:
    \eqan{
    \tau(x)&=\delta_{0,x}+\aap\sum_{y\in\Zd, \kappa\in\{\pm 1,\dots,\pm d\}} (\delta_{0,y}+\Psi_{\sss M}^\kappa(y))\tau^{\kappa}(x-y+\ve[\kappa])
        +\Xi_{\sss M}(x),
    \lbeq{taux-M-in-ex-1}\\
        \tau(x)&=\tau^\iota(x)+\aap \tau^{-\iota}( x-\ve[\iota])
        +\sum_{y\in\Zd, \kappa\in\{\pm 1,\dots,\pm d\}} \Pi_{\sss M}^{\iota,\kappa}(y) \tau^{\kappa}(x-y+\ve[\kappa])+\Xi_{\sss M}^{\iota}(x),
    \lbeq{taux-M-in-ex-2}
    }
where
    \eqan{
    \lbeq{PhiPsialt-M}
    \Pi_{\sss M}^{\iota,\kappa}(y)
    &= \sum_{N=0}^{M} (-1)^N
    \Pi^{\ssc[N],\iota,\kappa}(y),
    \qquad
    \Xi_{\sss M}(x)= R_{\sss M}(x)+\sum_{N=0}^{M} (-1)^N
    \Xi^{\ssc[N]}(x),\\
   \lbeq{PhiPsialt-M-2}
    \Psi_{\sss M}^\kappa(x)&=\sum_{N=0}^{M} (-1)^N
    \Psi^{\ssc[N],\kappa}(y),
    \qquad
    \Xi_{\sss M}^\iota(x)= R_{\sss M}^\iota(x)+ \sum_{N=0}^{M} (-1)^N
    \Xi^{\ssc[N],\iota}(x),\\
    \aap &=   p\prob(\ve[1]\nin \Ccal(0) \mid (0,\ve[1])\text{ vacant}),
   \lbeq{Definition-aap}
    }
with
    \eqan{
    R_{\sss M}(x)&\leq \aap\sum_{\stackrel{ y\in\Zd}{ \kappa\in\{\pm 1,\dots,\pm d\}}} \Psi^{\ssc[M],\kappa}(y)\tau^{\kappa}(x-y+\ve[\kappa]),\lbeq{RM-bd}\\
    R_{\sss M}^\iota(x)&\leq \sum_{\stackrel{ y\in\Zd}{ \kappa\in\{\pm 1,\dots,\pm d\}}} \Pi^{\ssc[M],\iota,\kappa}(y)\tau^{\kappa}(x-y+\ve[\kappa]).
	\lbeq{RMiota-bd}
    }
Explicit formulas for the lace-expansion coefficients in \refeq{PhiPsialt-M}--\refeq{PhiPsialt-M-2} are given in Section \ref{sec-compl}.
\end{prop}

Of course, the precise formulas for the lace-expansion coefficients are crucial for our analysis to succeed. However, at this stage, we refrain from stating their forms explicitly, and refer to Section \ref{sec-compl} instead. We continue by discussing how to bound these coefficients.

\subsection{Part (b): Bounds on the NoBLE}
\label{sec-part-b}
In this section, we explain the strategy of proof for bounds on the lace-expansion coefficients of the NoBLE. We start by rewriting the equations \refeq{taux-M-in-ex-1} and \refeq{taux-M-in-ex-2} to obtain an explicit equation for $\hat{\tau}(k)$.

\paragraph{The NoBLE-equation.}
Using the  NoBLE expansion of Proposition \ref{prop-LE} we now derive the NoBLE form, which rewrites $\hat \tau(k)$ in a form that is a perturbation of \refeq{NBWGen}.  We take the Fourier transforms of \refeq{taux-M-in-ex-1} and \refeq{taux-M-in-ex-2} to obtain
	\begin{align}
    	\hat{\tau}(k)&=1+\hat{\Xi}_{\sss M}(k)+\aap\sum_{\kappa} (1+\hat{\Psi}_{\sss M}^\kappa(k))	\e^{-\ii k\cdot \ve[\kappa]} \hat{\tau}^{\kappa}(k),
    	\lbeq{taux-M-in-ex-1-Four}\\
        \hat{\tau}(k)&=\hat{\tau}^\iota(k)+\aap\e^{\ii k\cdot \ve[\iota]}\hat{\tau}^{-\iota}(k)
        +\sum_{\kappa} \hat{\Pi}_{\sss M}^{\iota,\kappa}(k)\e^{-\ii k\cdot \ve[\kappa]}\hat{\tau}^{\kappa}(k)+\hat \Xi_{\sss M}^{\iota}(k).
    	\lbeq{taux-M-in-ex-2-Four}
	\end{align}
We write $\vec {\hat \tau}(k)\in\Rbold^{2d}$ for the (column-)vector with entries
    \eqn{
    (\vec {\hat \tau}(k))_\iota=(\hat{\tau}^{\iota}(k)).
    }
and note that, by $\mD[k]\mJ=\mJ\mD[-k]$ (see \refeq{J-D-matrix-def}) and $k_{-\iota}=-k_\iota$,
   \eqn{
    \e^{\ii k\cdot \ve[\iota]}\hat{\tau}^{-\iota}(k)=\big(\mD[k]\mJ\vec {\hat \tau}(k) \big)_\iota.
    }
Defining the vectors $\vec {\hat \Psi}(k),\vXiM[k]$ and the matrix $\mPiM[k]$, with entries
    \eqn{
    (\vec {\hat \Psi}(k))_\kappa=\hat \Psi^{\kappa}(k),\qquad
    (\vXiM[k])_\iota=\hat \Xi^{\iota}_{\sss M}(k), \qquad     (\mPiM[k])_{\iota,\kappa}=\hat \Pi^{\iota,\kappa}_{\sss M}(k),
    }
we can rewrite \refeq{taux-M-in-ex-2-Four} as
	\begin{align}
    	\hat{\tau}(k) {\v1}=\vec {\hat \tau}(k)+
        \aap\mD[k]\mJ  \vec {\hat \tau}(k)
        +\mPiM[k]\mD[-k] \vec {\hat \tau}(k)+ \vXiM[k],
	\end{align}
so that
	\begin{align}
    	\vec {\hat \tau}(k)=\mD[k]\big[\mD[k]+\aap \mJ +\mPiM[k]\big]^{-1} \times
        \big(\hat{\tau}(k) {\v1}- \vXiM[k]\big).
	\end{align}
In turn, by \refeq{taux-M-in-ex-1-Four}, the above gives rise to the relation
	\begin{align}
    	\hat{\tau}(k)=&1+\hat{\Xi}_{\sss M}(k)+\aap(\v1+\vPsiM[k])^T\mD[-k]\vec {\hat \tau}(k)\nnb
        =& 1+\hat{\Xi}_{\sss M}(k)-	\aap(\v1+\vPsiM[k])^T\big[\mD[k]+\aap\mJ  +\mPiM[k]\big]^{-1} \vXiM[k]\nnb
        &+\hat{\tau}(k)\aap(\v1+\vPsiM[k])^T\big[\mD[k]+ \aap\mJ+\mPiM[k]\big]^{-1}\v1.
	\end{align}
%with
%	\begin{align}
%    	\lbeq{hat-Phi-def}
 %   	\shift\hat{\Phi}_{{\sss M},p}(k)= \hat{\Xi}_{\sss M}(k)-	\aap(\v1+\vPsiM[k])^T\big[\mD[k]+\aap\mJ  +\mPiM[k]\big]^{-1} \vXiM[k].
%	\end{align}
Thus, we can solve the above equation as
	\begin{align}
    	\lbeq{lace-exp-eq}
    	\hat{\tau}(k)=\frac{1+\hat{\Xi}_{\sss M}(k)-\aap(\v1+\vPsiM[k])^T\big[\mD[k]+\aap\mJ  +\mPiM[k]\big]^{-1} \vXiM[k]}
    {1-\aap(\v1 +\vPsiM[k])^T\big[\mD[k]+\aap \mJ+\mPiM[k]\big]^{-1}{\v1}}.
	\end{align}
Equation \refeq{lace-exp-eq} is the NoBLE equation, and is the workhorse behind our proof. The goal of the NoBLE is now to show that \refeq{lace-exp-eq} is indeed a perturbation of \refeq{NBWGen}. This amounts to proving that $\hat{\Xi}_{\sss M}(k)$,$\vXiM[k]$, $\vPsiM[k]$ and $\mPiM[k]$ are small, which will only be true in sufficiently high dimensions.

We will show that, for every $p<p_c(d)$, the remainder terms $R_{\sss M}(x), R_{\sss M}^\iota(x)\rightarrow 0$ as $M\rightarrow \infty$. The content of the second key proposition is that the NoBLE coefficient can be bounded by combinations of simple diagrams. Simple diagrams are combinations of two-point functions, alike the triangle defined in
\refeq{triangle-cond-def} and the following examples:
	\begin{align}
	(2dp)^2 (\tau_p\star D\star D\star \tau_p)(e_\iota),   \qquad
	\sup_{x\in\Zd\colon \ \|x\|_2^2>2}  \sum_{y\in\Zd} \|y\|_2^2\tau_p(y) (\tau_p\star D)(x-y).
	\end{align}
In Section \ref{secBoundsPerc} we prove that we can derive bounds on the NoBLE coefficients:

\begin{prop}[Diagrammatic bound on the NoBLE coefficients]
\label{prop-bds-LEC}
For each $N\geq 0$ the NoBLE coefficients
$ \Pi^{\ssc[N],\iota,\kappa}(y),$ $\Xi^{\ssc[N]}(x),$ $\Psi^{\ssc[N],\kappa}(y)$ and $\Xi^{\ssc[N],\iota}(x)$  can be bounded by a finite combination of sums and products of simple diagrams.
\end{prop}
\noindent
The explicit form of the bounds in Proposition \ref{prop-bds-LEC} is given in Section \ref{sec-bounds-N0} for $N=0$, in Section \ref{sec-bounds-N=1} for $N=1$, and in Section \ref{secPercostatingtheBounds} for $N\geq 2$.

\iflongversion
We will only informally present the proof of the bounds on the NoBLE coefficients.  A detailed proof can be found in the thesis of the first author \cite{Fit13} that can be download from \cite{FitNoblePage}. In Appendix \ref{app-bounds}, we give the precise form of our bounding diagrams on the NoBLE coefficients that we use to prove the result for $d\geq \dmin$.
\else
As the complete proof of the bounds on the NoBLE coefficients is quite elaborate, we only sketch the proof. In this paper, we only informally define the building blocks that we use to bound the coefficients. A complete definition and more details can be found in the extended version \cite{FitHof13d-ext}.
\fi

\begin{remark}[Matrix-valued bounds]
\rm
The lace-expansion coefficients describe contributions created by multiple mutually intersecting paths, which we call {\em loops}.
In the NoBLE, these loops require at least $4$ bonds by design, as direct reversals are excluded. Lines can be part of two loops.
To optimally use the information that loops contain at least 4 bonds, we distinguish three cases for the length of lines shared by two loops.
Then, we bound the contribution of a loop to the lace-expansion diagram in terms of the lengths of the lines shared with the previous and preceding loops.
We explain this in detail in Section \ref{secBoundsPercDefinitionBlocks}, see especially Figures \ref{fig-Form-Xi4} and \ref{fig-Form-Xi4-Decomposed}.
This gives rise to a bound on the NoBLE coefficients in terms of a matrix product, as given in Section \ref{secPercostatingtheBounds} below.
In Section \ref{secProofBoundsNOne}, we explain how these matrix-valued bounds arise, see especially \refeq{lemmapercboundXi1-1-summation}.
For example, our proof yields that
	\begin{align*}
	\hat \Xi^{\ssc[N]}_{p}(0)\leq& \vec P^{\sss \rm S} ({\bf B}^\iota)^{N-1} {\bf \bar  A}^\iota \vec P^{\sss \rm E},
	\end{align*}
for $N\geq 2$, see \refeq{BoundXiPercNTwo}, for certain vectors $P^{\sss \rm S}, P^{\sss \rm E}$ and 3 by 3 matrices ${\bf \bar  A}^\iota, {\bf B}^\iota$. We will give an interpretation to the elements in these vectors and matrices in Section \ref{secBoundsPercDefinitionBlocks}.
For our analysis we require a bound on this when summed over $N$. To compute this bound numerically, we perform an eigenvector decomposition of $\vec P^{\sss \rm S}$, in terms of the eigenvectors $(\vec v_i)_{i=1}^3$ of ${\bf B}^{\iota}$ with corresponding eigenvalues $(\lambda_i)_{i=1}^3$. In this decomposition, we write $\vec P^{\sss \rm S}=\sum_{i=1}^3 \vec v_i,$ so that the eigenvectors used are not normalized.\footnote{We do not account for the possibility that ${\bf B}^{\iota}$ is not diagonalizable, as numerically this has never occurred in our applications.}
Then,
	\begin{align*}
	\hat \Xi^{\ssc[N]}_{p}(0)\leq& \sum_{i=1}^3 \vec v_i \lambda^{N-1}_i {\bf \bar  A}^\iota \vec P^{\sss \rm E}.
	\end{align*}
The sum of this over $N$ is computed using the geometric sum, see \cite[Section 5.4]{FitHof13b} for more details.
The order of this bound is to a large extent given by the largest eigenvalue of ${\bf B}^{\iota}$. For example, in $d=11$,
  	\begin{align*}
	{\bf B}^{\iota}=\begin{pmatrix}
  	0.0134202&	0.0112907&	0.0257405 \\
	0.0127527&	0.0108018&	0.0338533\\
	0.028009&	0.0260537&	0.0401418
	\end{pmatrix}.
  	\end{align*}
with largest eigenvalue $\lambda_1=0.073$. In the classical lace expansion also bounds on the $N$th lace-expansion diagram are present that decay exponentially in $N$, where the base of the exponent, roughly corresponding to $\sum_{i,j}{\bf B}^{\iota}_{i,j}$, is bounded in terms of a non-trivial triangle, which we can bound by $0.281$. This shows the power of the NoBLE, as well as the gain achieved by using matrix-valued bounds.
\end{remark}

\subsection{Part (c): The NoBLE analysis}
\label{sec-part-c}
We start by defining what our so-called {\em bootstrap functions} are.

\paragraph{Bootstrap functions.}
%Consider $n=0$ and $n=1$. The latter is related to the fact that for percolation the triangle diagram needs to be finite. See Section \ref{sec-disc-res} for more details.
For the bootstrap, we use the following functions:
	\begin{eqnarray}
	\lbeq{defFunc1}
	f_1(p)&:=&\max\left\{(2d-1)p,c_\mu (2d-1)\aap\right\},\\
	\lbeq{defFunc2}
	f_2(p)&:=& \sup_{k\in(-\pi,\pi)^d}
	\frac{\hat \tau_p(k)}{\hat B_{1/(2d-1)}(k)} = \frac{2d-1}{2d-2}\sup_{k\in(-\pi,\pi)^d} [1-\hat D(k)]\ \hat \tau_p(k),\\
	\lbeq{defFunc3}
	f_3(p)&:=& \max_{\{n,l,S\}\in \mathcal{S}} \frac{\sup_{x\in S}
	\sum_{y}\|y\|_2^2\tau_p(y)(\tau_p^{\star n}\star D^{\star l})(x-y)}{c_{n,l,S}},
	\end{eqnarray}
where $c_\mu>1$ and $c_{n,l,S}>0$ are some well-chosen constants and $\mathcal{S}$ is some finite set of indices. Let us now start to discuss the choice of these functions.

The functions $f_1$ and $f_3$ can been seen as the combinations of multiple functions.  We group these functions together as they play a similar role and are analyzed in the same way. We do not expect that the value of the bound on the individual functions constituting $f_1$ and $f_3$ are comparable. This is the reason that we introduce the constants $c_\mu$ and $c_{n,l,S}$.

The value of $n$ is model-dependent.  For SAW, we would use only $n=0$. For percolation, we use $n=0,1$, while $n=0,1,2$ for LT and LA. This can intuitively be understood as follows. By the $x$-space asymptotics in \refeq{x-space}, and the fact that $(f\star f)(x)\sim \|x\|_2^{4-d}$ when $f(x)\sim \|x\|_2^{2-d}$, we have that $\|y\|_2^2\tau_p(y)\sim (\tau_p\star \tau_p)(y)$. As a result, this suggests that
	\eqn{
	\lbeq{n=1-triangle}
	\sum_{y}\|y\|_2^2\tau_p(y)(\tau_p^{\star n}\star D^{\star l})(x-y)
	\sim \sum_{y}(\tau_p\star \tau_p)(y)(\tau_p^{\star n}\star D^{\star l})(x-y)
	=\big(\tau_p^{\star (n+2)}\star D^{\star l}\big)(x),
	}
so that finiteness of $\sum_{y}\|y\|_2^2\tau_p(y)(\tau_p^{\star n}\star D^{\star l})(x-y)$ is related to
finiteness of the bubble when $n=0$, of the triangle when $n=1$ and of the square when $n=2$.

The choices of point-sets $S\in \mathcal{S}$ improve the numerical accuracy of our method. For example, we obtain much better estimates in the case when $x=0$, since this leads to closed diagrams, than for $x\neq 0$. For $x$ being a neighbor of the origin, we can use symmetry to improve our bounds significantly. To obtain the infrared bound for percolation in $d\geq \dmin$, we use
	\eqn{
	\lbeq{calS-def}
  	\mathcal{S}=\big\{ \{0,0,\mathcal{X}\},\{1,0,\mathcal{X}\},\{1,1,\mathcal{X}\},\{1,2,\mathcal{X}\},
	\{1,3,\mathcal{X}\},\{1,6,\{0\}\} \big\},
	}
with $\mathcal{X}=\{x\in\Zd\colon \|x\|_2>1\}$. This turns out to be sufficient for our main results.

We apply a \emph{forbidden region or bootstrap argument} that is based on three claims:
\begin{enumerate}
\item[(i)] $p\mapsto f_i(p)$ is \emph{continuous} for all $p\in [1/(2d-1), p_c)$ and $i=1,2,3$; \item[(ii)] $f_i(1/(2d-1))\leq \gamma_i$ holds for $i=1,2,3$; and
\item[(iii)] if $f_i(p)\leq \Gamma_i$ holds for $i=1,2,3$, then, in fact, also $f_i(p)\leq \gamma_i$ holds for every $i=1,2,3$, where $\gamma_i<\Gamma_i$ for every $i=1,2,3$.
\end{enumerate}
Together, these three claims imply that $f_i(1/(2d-1)\leq \gamma_i$ holds for every $i=1,2,3$ and all $p\in [1/(2d-1), p_c)$. This in turn implies the statement of Theorem \ref{thm-IRB}  for all $p\in [1/(2d-1), p_c)$.

The continuity in Claim (i) is proven in \cite{FitHof13b} under some assumption that we explain and prove  below. The proofs of the initialization of the bootstrap in Claim (ii) as well as the improvement of the bounds in Claim (iii) use the following relations that are also sketched in Figures \ref{fig-Heuristic-bootstrap-initial} and \ref{fig-Heuristic-bootstrap}, where we write $p_{I}=1/(2d-1)$:
\begin{enumerate}[(1)]
\item simple diagrams can be bounded by a combination of two-point functions,\\ see \cite[Section 4]{FitHof13b};
\item the NoBLE coefficients can be bounded by a combination of simple diagrams,\\ see Section \ref{secBoundsPerc};
\item bounds on the NoBLE coefficients imply bounds on the two-point function,\\ see \cite[Section 2]{FitHof13b}.
\end{enumerate}
Thus, whenever we have numerical bounds on simple diagrams, or on NoBLE coefficients, or on the two-point function, we can also conclude bounds on the other two quantities.

Using that $\tau_{p_I}(x)\leq  B_{p_I}(x)$ with $p_{I}=1/(2d-1)$ and that we can compute $B_{p_I}(x)$ numerically, we verify the initialization of the bootstrap in Claim (ii) (i.e., $f_i(p_I)\leq \gamma_i$ for $i=1,2,3$) numerically, see Figure \ref{fig-Heuristic-bootstrap-initial}.

\begin{figure}
 \begin{center}
\begin{tikzpicture}[line width=1pt,auto,scale=0.7]
 \node   at ({3*cos(153)},{3*sin(153)})     {$f_i(p_I)\leq \gamma_i$};
 \node   at ({3*cos(0)},{3*sin(0)})      {Bounds on simple diagrams};
 \node   at ({3*cos(230)},{3*sin(230)+0.2})      {Bounds on coefficients};
 \draw [->,very thick] ({3*cos(215)-0.2},{3*sin(215)})  arc  (215:165:3);
% \draw [->,very thick] ({3*cos(70)},{3*sin(70)})  arc  (70:10:3);

\draw [->,very thick] ({3*cos(350)},{3*sin(350)})  arc  (350:240:3);
 \draw [->,very thick] (4,2.2) to (3,0.5);
 \color{black}
 \node  at(4,2.5)      {$\tau_{p_I}(x)\leq B_{p_I}(x)$};
 \node[left]   at(-3.1,0)   {Conclude a bound};
\end{tikzpicture}
\end{center}
 \caption{Initialization of the bootstrap: proof that $f_i(p_I)\leq \gamma_i$ holds for $i=1,2,3$. Here $\gamma_1,\gamma_2,\gamma_3$ are appropriately and carefully chosen constants.}
 \label{fig-Heuristic-bootstrap-initial}
\end{figure}
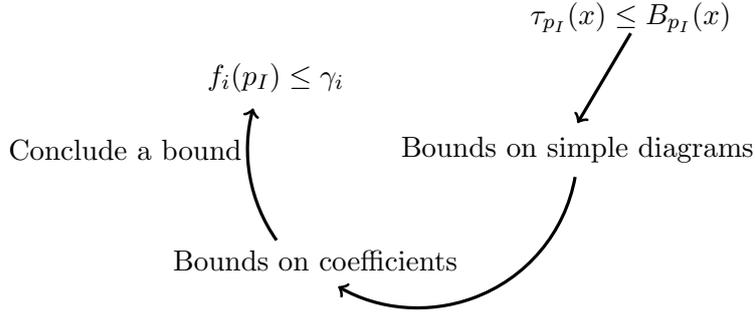

The proof of Claim (iii) is the most elaborate step of our analysis. Its structure is shown in Figure \ref{fig-Heuristic-bootstrap}. We start from the assumption that $f_i(p)\leq \Gamma_i$ holds for every $i=1,2,3$. The function $f_1$ gives a bound on $p$ and $f_2$ allows us to bound the  two-point function in Fourier space by $\hat B_{p_I}(k)$, which we can integrate numerically to obtain numerical bounds on simple diagrams. These, in turn, imply bounds on the NoBLE coefficients, which we use to compute bounds on the bootstrap functions.

\emph{In the case that} the computed bounds are small enough, we can conclude that $f_i(p)\leq \gamma_i$ holds and thereby that the improvement of the bounds in Claim (iii) holds. Whether we can indeed prove that Claim (iii)  holds depends on the dimension we are in, the quality of the bounds and the analysis used to conclude bounds for the bootstrap function. In high dimensions (e.g.\ $d\geq 1000$) the perturbation is rather small so that it is
relatively easy to prove Claim (iii). Proving the claim in lower dimension is only possible  when the bounds on the lace-expansion coefficients and the analysis are sufficiently sophisticated. It is here that it pays off to apply the NoBLE compared to the classical lace expansion.

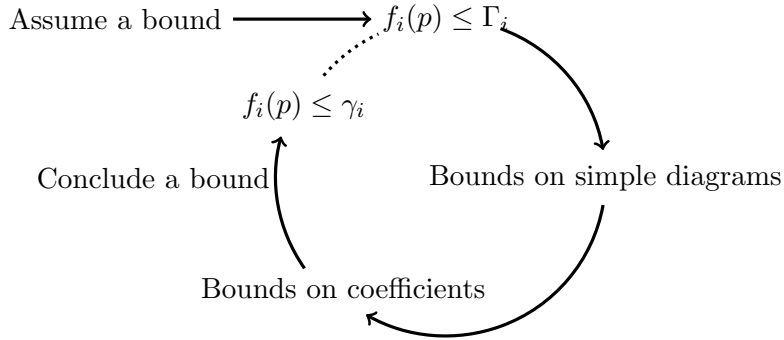
\begin{figure}
 \begin{center}
\begin{tikzpicture}[line width=1pt,auto,scale=0.7]
\node   at ({3*cos(90)},{3*sin(90)})     {$f_i(p)\leq \Gamma_i$};
 \node   at ({3*cos(153)},{3*sin(153)})     {$f_i(p)\leq \gamma_i$};
 \node   at ({3*cos(0)},{3*sin(0)})      {Bounds on simple diagrams};
 \node   at ({3*cos(230)},{3*sin(230)+0.2})      {Bounds on coefficients};
 \draw [->,very thick] ({3*cos(215)-0.2},{3*sin(215)})  arc  (215:165:3);
 \draw [->,very thick] ({3*cos(70)},{3*sin(70)})  arc  (70:10:3);
 \draw [dotted,very thick] ({3*cos(140)},{3*sin(140)})  arc  (140:115:3);
\draw [->,very thick] ({3*cos(350)},{3*sin(350)})  arc  (350:240:3);
 \draw [->,very thick] (-4,3) to ({3*cos(90)-1.4},{3*sin(90)}) ;
  \color{black}
 \node[left]   at(-4,3)      {Assume a bound};
 \node[left]   at(-3.1,0)   {Conclude a bound};
\end{tikzpicture}
\end{center}
 \caption{ Proof of claim (iii): $f_i(p)\leq \Gamma_i$ for $i=1,2,3$ implies that $f_i(p)\leq \gamma_i$ for $i=1,2,3$.}
 \label{fig-Heuristic-bootstrap}
\end{figure}

The third step in the proof of Theorem \ref{thm-IRB} is formalized in the following proposition:

\begin{prop}[Successfull application of NoBLE analysis]
\label{prop-analysis-is-success}
For nearest-neightbor percolation in $d\geq \dmin$, the NoBLE analysis of \cite{FitHof13b} applies and proves the infrared bound in Theorem \ref{thm-IRB}. In particular, there exist constants $\Gamma_1,\Gamma_2,\Gamma_3$ and $\gamma_1,\gamma_2,\gamma_3$ such that, for every $p<p_c$, the bounds $f_i(p)\leq \Gamma_i$ for $i=1,2,3$ imply that $f_i(p)\leq \gamma_i$ for $i=1,2,3$.
\end{prop}

As shown in Figure \ref{fig-Heuristic-bootstrap}, Proposition \ref{prop-analysis-is-success} is proved using the results of Propositions \ref{prop-LE}-\ref{prop-bds-LEC}, the analysis of \cite{FitHof13b} and the computer-assisted proof performed in the Mathematica notebook that can be found on \cite{FitNoblePage}.  To apply the general NoBLE analysis of \cite{FitHof13b} for percolation in $d\geq \dmin$, we need to prove that the assumptions formulated in \cite{FitHof13b} hold. We recall these assumptions when we prove them. For example, we now verify the assumptions on the two-point function $\tau_{p}$. In Section \ref{subsec-prop-coef}, we verify the symmetry assumptions on the NoBLE coefficients.
The assumption that is most difficult to prove is the existence of the bounds on the NoBLE coefficients. We prove this in Sections \ref{secBoundsPerc}-- \ref{secProofBounds}. We give an overview of where to find the bounds on the NoBLE coefficients stated in \cite[Assumptions 4.3]{FitHof13b} in Section \ref{secSummaryBounds}.

\paragraph{Verification of the assumptions for the general NoBLE analysis in \cite{FitHof13b}.}
%\label{subsec-Prop-tau}
%To make use of the analysis of \cite{FitHof13b} we need to verify a number of assumption made there.
In this section, we verify the assumptions in \cite{FitHof13b} that are independent of the NoBLE. Namely, we prove that \cite[Assumptions 2.2, 2.3 and 2.4]{FitHof13b} hold for percolation:

\paragraph{\cite[Assumption 2.2]{FitHof13b}: At the initial point $p_I$.}
Whenever $\{0\conn x\}$ occurs, there exists a path of occupied bonds connecting $0$ and $x$.
Since such a connecting path is a non-backtracking walk (NBW),
	\begin{align}
	\tau_{p}(x)&\leq \sum_{n=0}^\infty b_n(x) p^n= B_{p}(x),
	\end{align}
for all $p\leq 1/(2d-1)$, which implies \cite[Assumption 2.2]{FitHof13b}.

\paragraph{\cite[Assumption 2.3]{FitHof13b}: Growth of the two-point function.}
We need to show that for every $x\in \Zd$, the two-point functions $p\mapsto \tau_p(x)$ and $p\mapsto \tau^\iota_p(x)$
are non-decreasing, differentiable in $p\in(0,p_c)$. Further, we need to show that for all $\varepsilon>0$, there exists a constant $c_{\varepsilon}\geq 0$ such that for all $p\in(0,p_c-\varepsilon)$ and $x\in\Zd\setminus\{0\}$,
	\begin{eqnarray}
	\lbeq{assGzDiffBound}
	\frac d {dp} \tau_p(x)\leq c_{\varepsilon} (\tau_p\star D\star \tau_p)(x),
	\quad \text{ and therefore }\quad \frac d {dp} \hat \tau_p(0)\leq c_{\varepsilon} \hat \tau_p(0)^2.
	\end{eqnarray}
Finally, we need to show that for each $p\in(0,p_c)$, there exists a constant $K(p)<\infty$ such that $\sum_{x\in\Zd} \|x\|_2^2 \tau_{p}(x)<K(p)$. We will do this now.\\

We recall that $\tau_p(x)=\prob_p(0\conn x)$,
%is the probability that the origin is connected to $x$ by a path of occupied bonds in a percolation configuration in which each bond is independently occupied with probability $p$. Thus,
so that $\tau_p(x)$ is non-decreasing in $p$ as occupying a bond can only increase the probability that a path of occupied bonds from $0$ to $x$ exists.
The same clearly also holds for  $\tau^\iota_p(x)$ in \refeq{tau-z-def}.
The differentiability of $p\mapsto\tau_p(x)$ is well known for $p\in(0,p_c)$ and the bound on the derivative \refeq{assGzDiffBound} is obtained using Russo's Formula (\cite[Lemma 3]{Russ81} or \cite[Theorem 2.25]{Grim99}) and the BK-inequality \cite{BerKes85}.
As these are standard arguments in percolation theory we will not comment further on them. The argument for $\tau_p^\iota(x)$ is identical, by considering percolation on the base graph $\Z^d\setminus \{e_{\iota}\}$ instead.\\[2mm]
%\todo[inline]{How to prove the same for $\tau^\iota$ in a simple way.}
To prove the bound on $\sum_{x\in\Zd} \|x\|_2^2 \tau_{p}(x)$, by \cite[Theorem 6.1]{Grim99},
	\begin{align}
	\tau_p(x)\leq \e^{-\sigma(p) \|x\|_\infty},
	\end{align}
where $\|x\|_\infty=\max_{i=1}^d |x_i|$, with $\sigma(p)>0$ for every $p<p_c$. From this, we conclude that $\sum_x \|x\|_2^2 \tau_p(x)\leq K(p)<\infty$,
%	\begin{align}
%	\sum_x \|x\|_2^2 \tau_p(x)\leq &\sum_x \|x\|_2^2 \e^{-\sigma(p) \|x\|_\infty}
%	\leq d\sum_x \|x\|_\infty^2 \e^{-\sigma(p) \|x\|_\infty}\nnb
%	\leq &d^2\sum_{n=1}^\infty n^2 \e^{-\sigma(p) n} \sum_{x:\|x\|_\infty=n}1
%	\leq d^3\sum_{n=1}^\infty n^{d+1} \e^{-\sigma(p)n}
%	:=K(p)<\infty.
%	\end{align}
which completes the proof of \cite[Assumption 2.3]{FitHof13b}.

\paragraph{\cite[Assumption 2.4]{FitHof13b}: Continuity of $\aabp$ and $\aap$.}
In the application of the analysis in \cite{FitHof13b}, we define $\aabp=p$ and $\aap=p\prob_p(\ve[1]\nin \Ccal(0) \mid (0,\ve[1])\text{ vacant})$. Thus, $p\mapsto \aabp$ clearly is continuous in $p$.
For $\aap$, we note that the percolation two-point function on the lattice $\Zd$ where the edge $\{0,\ve[\iota]\}$ is deleted, is also continuous. This can e.g.\ be seen by modifying the proof of \cite[Assumptions 2.3]{FitHof13b}, given above. Thus, also $p\mapsto\aap$ is continuous.

\subsection{Part (d): Numerical analysis}
\label{sec-part-d}
In this section, we explain how the numerical computations are performed using Mathematica notebooks that are available from the first author's homepage.
%\RF{--$>$ Remco:Did not like the old text, please check and remove old part or combine with new version. Blue old, red new}
%\color{blue}The procedure starts by evaluating the notebooks \verb|SRW| or \verb|SRW_basic|. This produces two files, \verb|SRWCountData.nb| and \verb|SRWIntegralsData.nb|, containing counts of SRWs of a given number of steps ending at various locations in $\Zd$, and numerical values for SRW integrals, respectively. Running these programs takes several hours. \color{red}

\paragraph{Simple-random walk computations.}
The procedure starts by evaluating the notebook \verb|SRW|. The file computes the number of SRWs and SAWs of a given number of steps ending at various locations in $\Zd$, using a combinatorial analysis, as well as numerical values for SRW integrals based on numerical integration of certain Bessel functions.
These computations are performed in \cite[Appendix B]{HarSla92a}, and are explained in detail in \cite[Section 5]{FitHof13b}. The SRW integrals provide {\em rigorous numerical upper bounds} on various convolutions of SRW Green's functions with themselves, evaluated at various $x\in \Z^d$. For the analysis in $d=\dmin$, we rely on 112 of such integrals.

Running these programs takes several hours. For this reason, once computed, the results are saved in two files, \verb|SRWCountData.nb| and \verb|SRWIntegralsData.nb| and are loaded automatically when the notebooks are evaluated a second time for the same dimension. %\color{black}
Alternatively, these two files can also be downloaded directly from the home page of the first author, and put in one's own home directory.\footnote{In Mathematica, the command {\tt \$InitialDirectory} will tell you what this directory is.}

\paragraph{Implementation of the NoBLE analysis for percolation.}
After having computed all the simple random walk ingredients, we evaluate the notebook \verb|General|, that implements the bounds of the NoBLE analysis \cite{FitHof13b}. After this, we are ready to perform the NoBLE analysis for percolation by evaluating the notebook \verb|Percolation|. In the percolation notebook, we implement all the bounds proved in this paper. The computations in \verb|General| and \verb|Percolation| merely implements the bounds proved in this paper and in \cite{FitHof13b}, and rely on many multiplications and additions, as well as the diagonalization of 2 three-by-three matrices. These computations could in principle be done by hand (even though we prefer a computer to do them).

\paragraph{Output of Mathematica Notebooks.} After having evaluated the Mathematica notebooks, we can verify whether the analysis has worked with the chosen constants $\Gamma_1,\Gamma_2, \Gamma_3$. See Figure \ref{fig-percolation-successful} for the first output after evaluating the \verb|Percolation| notebook. Let us now explain this output in more detail. The green dots mean that the bootstrap has been successful for the parameters as chosen. When evaluating the notebook, it is possible that some red dots appear, and this means that these improvements were not successful. The first 3 dots in the first table are the verifications that $f_i(1/(2d-1))\leq \gamma_i$ for $i=1,2,3$. The next three dots show that the improvement has been successful for all $p<p_c(\dmin)$. The values for $\Gamma_1, \Gamma_2,\Gamma_3$ are indicated in the first few lines.
For example, $\Gamma_1=1.01306$ means that $(2d-1)p\leq 1.01306$. In the verification of the bootstrap improvement, it turns out that $\gamma_1$ can be taken to be $1.0130591$. Since this it true for all $p<p_c(\dmin)$, we obtain that $(2d-1)p_c(\dmin)\leq 1.0130591$. This explains the value in the table in Theorem \ref{thm-bds-crit}.
Similarly, $\Gamma_2=1.076$. This implies that $A_2(\dmin)\leq 1.07513 \times 20/21= 1.02393$. Anyone interested in obtaining improved bounds on $p_c(d)$ or $A_2(d)$ for $d\geq \dmin$ can play with the notebook to optimize them. The second table in Figure \ref{fig-percolation-successful} gives more details on the improvement of $f_3(p)$, which, as indicated in \refeq{defFunc3}, consists of several contributions, over which the maximum is taken. The assumed bound correspond to the constants $c_{n,l,S}$, with $S\in {\mathcal S}$ in \refeq{calS-def}. The notebook \verb|Percolation| also includes a routine that optimizes the choices of $\Gamma_i$ and $c_{n,l,S}$. This makes it easier to find values for which the analysis works (when these exist).

\begin{figure}
\begin{center}
\includegraphics[width=0.75\textwidth]{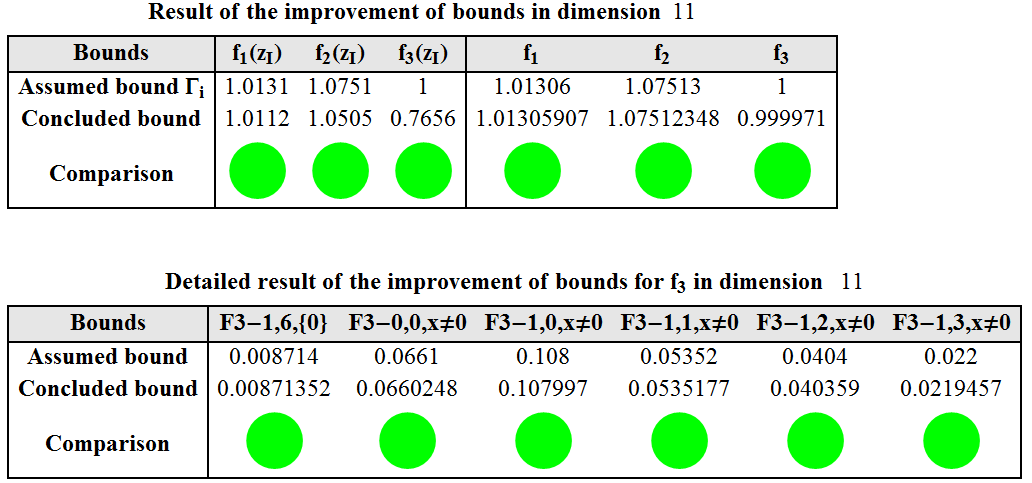} %''last minute change''
\end{center}
 \cprotect\caption{\label{fig-percolation-successful} Output of the Mathematica notebook {\verb|Percolation|}.}
\end{figure}

\subsection{Structure of the NoBLE proof and related results}
\label{sec-summ-proof}
\paragraph{Summary of the proof of the infrared bound in Theorem \ref{thm-IRB}.}
We have explained the proof of the infrared bound in Theorem \ref{thm-IRB}. When reviewing the analysis, we have already seen how delicately the four parts of the proof described on page \pageref{parts-proof} are intertwined. The expansion in part (a) gives us a characterisation of $\hat \tau_p(k)$ as a perturbation of $\hat B_\mu(k)$ involving the NoBLE coefficients. The analysis in part (c) allows us to compute bounds on $\hat \tau_p(k)$ provided that numerical bounds on the coefficients are available.  To obtain such bounds we need to derive diagrammatic bounds, as formulated in part (b), that bound the NoBLE coefficients by simple diagrams.  However, we rely on bounds on $\tau_p$ to bound such simple diagrams. Thus, we obtain a circular reasoning.

Using the bootstrap argument we can indeed complete the circle, see Figures \ref{fig-Heuristic-bootstrap-initial}-\ref{fig-Heuristic-bootstrap}, to obtain a bound on $\hat \tau_p(k)$ for all $p\in [1/(2d-1), p_c)$. For the bootstrap argument, we need to show that $f_i(p_I)\leq \gamma_i$, as well as the fact that $f_i(p)\leq \Gamma_i$ implies that $f_i(p)\leq \gamma_i$, for appropriately chosen $\gamma_i$ and $\Gamma_i$ for all $p\in( 1/(2d-1), p_c)$. The verification whether $f_i(p_I)\leq \gamma_i$ holds for $i=1,2,3$. Whether we can conclude from $f_i(p)\leq \Gamma_i$ for $i=1,2,3$ that also $f_i(p)\leq \gamma_i$ holds for $i=1,2,3$ requires a computer-assisted proof as indicated in Section \ref{sec-part-d}. Starting from $\tau_{p_I}(x)\leq  B_{1/(2d-1)}(x)$, $f_i(p)\leq \Gamma_i$ for $i=1,2,3$ and explicit computations of $B_{1/(2d-1)}(x)$, we obtain numerical bounds on simple diagrams.  These are then used to obtain numerical bounds on the NoBLE coefficients, which we in turn use to verify whether we can actually conclude from $f_i(p)\leq \Gamma_i$ for $i=1,2,3$ that $f_i(p)\leq \gamma_i$ for $i=1,2,3$ holds.

Combining these steps yields the required results for $p\in [1/(2d-1), p_c)$. We obtain the statement also for $p=p_c$ by using that $\hat \tau_p(k)/\hat B_{1/(2d-1)}(k)$ and the NoBLE-coefficients are continuous and uniformly bounded for $p\in [1/(2d-1), p_c)$ and left-continuous in $x$-space at $p_c$. We explain this in more detail in Section \ref{subsec-prop-coef}.

\paragraph{The numerical bounds in Theorem \ref{thm-bds-crit}.} As can be observed from Figure \ref{fig-percolation-successful}, after running the notebook \verb|Percolation|, we obtain numerical estimates on $f_1(p)$ and $f_2(p)$ that are uniform in $p\leq p_c(d)$, which will provide the bounds in Theorem \ref{thm-bds-crit}.

\paragraph{Proof of related results and the classical lace expansion.} The strategy behind the proof of our related results in Theorems \ref{thm-x-space} and \ref{thm-IIC-lim} is that we show that the {\em classical} lace expansion actually also converges, and we prove sufficient bounds on the clasical lace-expansion coefficients to deduce Theorems \ref{thm-x-space} and \ref{thm-IIC-lim} from the literature. Remarkably, we thus see that for $d=11$, we cannot {\em directly} prove that the classical lace expansion converges, but we can prove it after we have obtained sharp estimates on the two-point function in $k$-space and on $p_c(11)$ using the NoBLE. Theorem \ref{thm-one-arm} follows directly from Theorem \ref{thm-x-space}.

\subsection{Relations to the literature}
\label{sec-disc-res}
\paragraph{Trigonometric approach.} The improvement of the bootstrap function $f_3(p)$ is the most delicate of the general analysis. In the bootstrap function $f_3(p)$, the most important parameter is $n\geq 0$.  In \refeq{n=1-triangle}, we have explained that we can think of $f_3(p)$ as bounding various triangle diagrams.

In \cite{BorChaHofSlaSpe05b, HeyHofSak08, Slad06}, the use of {\em trigonometric functions} has been used successfully to simplify the traditional lace-expansion analysis. In the trigonometric approach, the analysis is performed directly in $k$-space, by using
	\eqn{
	\tilde{f}_3(p)=\sup_{k,l\in (-\pi,\pi]^d} \frac{|\hat\tau_p(k+l)+\hat\tau_p(k-l)-2\hat\tau_p(k)|}{\hat U(k,l)},
	}
where
	\eqn{
	\hat U(k,l)=[1-\hat{D}(l)]\big[\hat{C}(k+l)\hat{C}(k-l)+\hat{C}(k+l)\hat{C}(k)+\hat{C}(k-l)\hat{C}(k)\big],
	}
and related objects, instead of $f_3(p)$ in \refeq{defFunc3}.
We have compared both approaches using the NoBLE in the thesis of the first author \cite{Fit13}. This comparison shows that the $x$-space approach that we describe in this paper is numerically superior, and therefore we have decided not to describe the competing trigonometric approach.

The analysis that we derive in \cite{FitHof13b} is powerful and flexible. Both the bounding and the analysis could be further improved to reduce the  dimension even further. However, we have decided that the necessary effort would not  be in relation to the possible gain.
%Even for the bound for $d\geq \dmin-2$, we decided not to state the full bounds.
%\iflongversion
The ideas underlying these bounds are explained in Sections 3-5.
\iflongversion The precise definitions are in Appendix \ref{app-bounds}.
\else
The precise definitions are in \cite[Appendix A]{FitHof13d-ext}.
\fi
It turns out that our methods no longer work in dimension $d=10$. The main reason is that the improvement of $f_3(p)$ becomes problematic. Particularly the weighted open bubbles become rather large, and as a result, the perturbations become unmanageable.
%The complete definition of bounds can only be found in the Mathematica notebooks, implementing them.
%\else
%In the extended version of this file we explain the ideas underlying these bounds.
%%The complete definition of bounds for $d\geq \dmin-2$ can only be found in the Mathematica notebooks, implementing them.
%\fi

\paragraph{Relations to the work of Hara and Slade.}
We close this discussion section by relating our methods to those of Hara and Slade, which have been explained in full detail for SAW in  \cite{HarSla92b}. Takashi Hara has been so kind to explain us how it has been precisely implemented for percolation. The crucial estimates involve the triangle diagram. We can bound this using $f_2(p)$, but typically the constant $\Gamma_2$ that is used to bound the two-point function in $k$-space is rather large, and thus limits the numerical accuracy of the method. Therefore, both here as well as in the analysis by Hara and Slade, such bounds are being improved. The main method of Hara and Slade is to improve the bounds on the two-point function by bounding
	\eqn{
	\lbeq{tau-NBW}
	\tau_p(x)\leq \tau_{1/(2d-1)}(x)+(p-p_{I})\frac{d}{dp}\tau_p(x).
	}
We can obtain a good start by noting that  $\tau_{1/(2d-1)}(x)\leq B_{1/(2d-1)}(x)$, which can be numerically computed and is independent of $\Gamma_2$. By Russo's formula,
	\eqn{
	\lbeq{Russo}
	\frac{d}{dp}\tau_p(x)\leq (\tau_p\star \tau_p\star D)(x),
	}
which can then be bounded using $f_2(p)$. Since this yields a small factor $p-p_{I}$ in front of these terms, these bounds are smaller than those obtained by using $f_2(p)$ immediately. However, it does mean that $\tau_p$ is bounded in terms of $\tau_p\star \tau_p$, which turns a triangle into a square. For the best possible results, which apply to $d\geq 15$, this bound is used repeatedly leading to heptagons. Since heptagons are only finite for $d\geq 15$, this method cannot be used for $d=\dmin$.

Let us comment on the main differences of our approach compared to that of Hara and Slade. Our NoBLE expansion perturbs around non-backtracking random walk, and thus explicitly takes immediate reversals into account. As a result, loops arising in the lace-expansion coefficients consist of at least four bonds. We use a matrix-based approach to bound the lace-expansion coefficients taking the number of bonds on loops explicitly into account. This is much more efficient, as it removes the dominant contribution from the lace-expansion coefficients. In particular, when explicitly taking the length of paths into account, we bound
	\eqn{
	\lbeq{extraction-step}
	\tau_p(x)-\delta_{0,x}\leq 2dp D(x)+2dp (D\star (\tau_p-\delta_0))(x).
	}
Terms containing $D$ can be computed explicitly, and the factor $D$ in the second term also significantly reduces the bound. This bound is further improved by noting that the paths leading to triangles are often mutually disjoint, thus leading to self-repellent triangles. Also, we extract longer paths than the single-step path in \refeq{extraction-step}, and use that these paths can be taken to be mutually disjoint. Finally, the accuracy is significantly improved due to the NoBLE expansion, which ensures that all closed paths contain at least $4$ steps, so that our triangles contain more steps than those in the Hara-Slade approach. Apart from these differences, our method crucially relies on the ideas of Hara-Slade in \cite{HarSla92b}, in that spatial estimates have been used, the SRW Green's functions are computed in the same way, etc.
Thus, our work could not have been possible without theirs.

Recently, Chen, Handa, Heydenreich, Kamijima and Sakai \cite{CheHeySak15} have started to investigate percolation on the high-dimensional body-centered cubic lattice. Here, the bonds are given by $\{x,y\}$, where $|x_i-y_i|=1$ for every $i\in \{1,\ldots,d\}$. Thus, the degree of this base graph is $2^d$ compared to the degree of the hypercubic lattice, which is $2d$. Therefore, one is tempted to believe that mean-field behavior follows more easily in lower dimensions. It would be of great interest to verify (either by the classical lace expansion or the NoBLE) whether mean-field behavior for the body-centered cubic lattice can be proved for {\em all} $d\geq 7$. More information on SAW and percolation on the body-centered cubic lattice is given in \cite{HanKamSak17}.

\paragraph{Organization of this paper.} In Section \ref{secExp}, we perform the NoBLE, and thus prove Proposition \ref{prop-LE}.
In Section \ref{secBoundsPerc}, we explain how diagrammatic bounds on the NoBLE coefficients can be obtained.  These diagrammatic bounds are phrased in terms of various building blocks that are informally defined in Section \ref{sec-bounds-Ngeq1}. In Section \ref{secProofBounds}, we explain how such diagrammatic bounds can be obtained, without going in too much detail. In Section \ref{sec-proof-rel-res}, we prove Theorems \ref{thm-x-space}, \ref{thm-IIC-lim} and \ref{thm-one-arm} using results from the literature.

%=================================================Expansion=================================================================
%==========================================================================================================================
\def\figureEventET[#1]{
\begin{figure}
\begin{center}
\begin{tikzpicture}[line width=1pt,auto,scale=#1]
    \color{black!50!white}
    \pgfpathmoveto{\pgfpoint{2cm}{1cm}}
    \pgfpathcurveto{\pgfpoint{3cm}{1.5cm}}{\pgfpoint{8cm}{1cm}}{\pgfpoint{7.5cm}{-0.7cm}}
    \pgfpathcurveto{\pgfpoint{7cm}{-0.5cm}}{\pgfpoint{6.5cm}{0.4cm}}{\pgfpoint{4.5cm}{0.4cm}}
    \pgfpathcurveto{\pgfpoint{4cm}{-0.2cm}}{\pgfpoint{3.3cm}{0.4cm}}{\pgfpoint{2cm}{0.9cm}}
    \pgfusepath{fill,stroke}
    \color{black!80!white}
    \node[right]   at(3.5,2)      {$A$};
    \color{black}

  \node[left]   at(0,0)      {$x$};
  \node[above]   at(8,0)      {$\underline{b}$};
  \node[above]   at(9,0)      {$\bar{b}$};
  \node[right]   at(14,0)      {$y$};
\begin{scope}[shift={(0,0)},rotate=0]
  \draw [-] (0,0) to [out=40,in=180] (1,0.5);
  \draw [-] (0,0) to [out=-40,in=180] (1,-0.5);
  \draw [-] (1,0.5) to [out=0,in=140] (2,0);
  \draw [-] (1,-0.5) to [out=0,in=220] (2,0);
  \draw [-] (2,0) to  (3,0);
 \end{scope}

\begin{scope}[shift={(3,0)},rotate=0]
  \draw [-] (0,0) to [out=40,in=180] (1,0.5);
  \draw [-] (0,0) to [out=-40,in=180] (1,-0.5);
  \draw [-] (1,0.5) to [out=0,in=140] (2,0);
  \draw [-] (1,-0.5) to [out=0,in=220] (2,0);
    \draw [-] (2,0) to  (3,0);
 \end{scope}

\begin{scope}[shift={(6,0)},rotate=0]
  \draw [-] (0,0) to [out=40,in=180] (1,0.5);
  \draw [-] (0,0) to [out=-40,in=180] (1,-0.5);
  \draw [-] (1,0.5) to [out=0,in=140] (2,0);
  \draw [-] (1,-0.5) to [out=0,in=220] (2,0);
    \draw [-] (2,0) to  (3,0);
 \end{scope}

\begin{scope}[shift={(9,0)},rotate=0]
  \draw [-] (0,0) to [out=40,in=180] (1,0.5);
  \draw [-] (0,0) to [out=-40,in=180] (1,-0.5);
  \draw [-] (1,0.5) to [out=0,in=140] (2,0);
  \draw [-] (1,-0.5) to [out=0,in=220] (2,0);
    \draw [-] (2,0) to  (3,0);
 \end{scope}
 \begin{scope}[shift={(12,0)},rotate=0]
  \draw [-] (0,0) to [out=40,in=180] (1,0.5);
  \draw [-] (0,0) to [out=-40,in=180] (1,-0.5);
  \draw [-] (1,0.5) to [out=0,in=140] (2,0);
  \draw [-] (1,-0.5) to [out=0,in=220] (2,0);
 \end{scope}
\end{tikzpicture}
\caption{\label{fig-perclemevent} The event $E(x,b,y;A)$ of
Lemma~\protect\ref{lem-cut1}.  The shaded regions represent the
vertices in $A$.  There is no restriction on intersections between
$A$ and $\tilde{\Ccal}^{b}(y)$.}
\end{center}
\end{figure}}

\def\figureStageTwoT[#1]{
\begin{figure}
\begin{center}
\begin{tikzpicture}[line width=1pt,auto,scale=#1]

  \node[left]   at(0,0)      {$0$};
  \node[above]   at(2.5,0)      {$b$};
  \node[above]   at(11.5,0)      {$b'$};
  \node[right]   at(14,0)      {$x$};

\begin{scope}[shift={(0,0)},rotate=0]
  \draw [line width=2pt,-] (0,0) to [out=40,in=180] (1,0.5);
  \draw [line width=2pt,-] (0,0) to [out=-40,in=180] (1,-0.5);
  \draw [line width=2pt,-] (1,0.5) to [out=0,in=140] (2,0);
  \draw [line width=2pt,-] (1,-0.5) to [out=0,in=220] (2,0);
 \end{scope}
  \draw [line width=2pt,-] (1,0.5) to [out=40,in=180] (6,1);
  \draw [line width=2pt,-] (6,1) to [out=0,in=120] (10.5,-0.6);

  \draw [line width=2pt,-] (1,-0.5) to [out=-40,in=180] (2.3,-0.9);
  \draw [line width=2pt,-] (2.3,-0.9) to [out=0,in=250] (3.8,0.2);

\begin{scope}[shift={(3,0)},rotate=0]
  \draw [-] (0,0) to [out=40,in=180] (1,0.5);
  \draw [-] (0,0) to [out=-40,in=180] (1,-0.5);
  \draw [-] (1,0.5) to [out=0,in=140] (2,0);
  \draw [-] (1,-0.5) to [out=0,in=220] (2,0);
    \draw [-] (2,0) to  (3,0);
 \end{scope}

\begin{scope}[shift={(6,0)},rotate=0]
  \draw [-] (0,0) to [out=40,in=180] (1,0.5);
  \draw [-] (0,0) to [out=-40,in=180] (1,-0.5);
  \draw [-] (1,0.5) to [out=0,in=140] (2,0);
  \draw [-] (1,-0.5) to [out=0,in=220] (2,0);
    \draw [-] (2,0) to  (3,0);
 \end{scope}

\begin{scope}[shift={(9,0)},rotate=0]
  \draw [-] (0,0) to [out=40,in=180] (1,0.5);
  \draw [-] (0,0) to [out=-40,in=180] (1,-0.5);
  \draw [-] (1,0.5) to [out=0,in=140] (2,0);
  \draw [-] (1,-0.5) to [out=0,in=220] (2,0);
    \draw [-] (2,0) to  (3,0);
 \end{scope}
 \begin{scope}[shift={(12,0)},rotate=0]
  \draw [-] (0,0) to [out=40,in=180] (1,0.5);
  \draw [-] (0,0) to [out=-40,in=180] (1,-0.5);
  \draw [-] (1,0.5) to [out=0,in=140] (2,0);
  \draw [-] (1,-0.5) to [out=0,in=220] (2,0);
 \end{scope}
\end{tikzpicture}
\caption{A possible configuration appearing in the second stage of the
expansion.}
\label{fig-2pt2}
\end{center}
\end{figure}}

\def\figureEventE{
\begin{figure}
\begin{center}
\setlength{\unitlength}{0.0080in}%
\begin{picture}(500,100)(100,-50)
\put(120,20){${\sss x}$}
\put(350,25){${\sss \underline{b}}$}
\put(368,25){${\sss \bar {b}}$}
\put(520,20){${\sss y}$}

\shade\path(270,-45)(275,-50)(285,-47)(295,-41)
(305,-32)(315,-20)(325,-5)(335,13)(345,34)(346,36)
(335,31)(325,25)(315,20)
(305,12)(295,5)(285,-8)(275,-24)(270,-37)(269,-41)
(270,-45)

\shade\path(155,-65)(160,-70)(170,-67)(180,-61)
(190,-52)(200,-40)(210,-25)(220,-7)(230,14)(230,16)
(220,11)(210,5)(200,0)
(190,-8)(180,-15)(170,-28)(160,-44)(155,-57)(154,-61)
(155,-65)

\put(245,-20){${\scriptstyle A}$}

%\multiput(110,20)(4,0){6}{\circle*{2}}
%\put(110,20){\line(2,0){10}}
\qbezier(130,20)(160,-10)(190,20)
\qbezier(130,20)(160,50)(190,20)
\put(190,20){\line(1,0){20}}
\qbezier(210,20)(240,-10)(270,20)
\qbezier(210,20)(240,50)(270,20)
\put(270,20){\line(1,0){20}}
\qbezier(290,20)(320,-10)(350,20)
\qbezier(290,20)(320,50)(350,20)
\put(350,20){\line(1,0){20}}
\qbezier(370,20)(400,-10)(430,20)
\qbezier(370,20)(400,50)(430,20)
\put(430,20){\line(1,0){20}}
\qbezier(450,20)(480,-10)(510,20)
\qbezier(450,20)(480,50)(510,20)
%\put(2,16){$0$}
%\put(112,16){$1$}
%\qbezier[13](27,9)(35,20)(27,31)
%\qbezier[15](18,25)(34,21)(40,36)
%\put(31,41){$a_1^{(0)}$}
%\put(14,-5){$a_2^{(0)}$}
%\qbezier[16](90,8)(80,20)(90,32)
%\put(89,34){$a_1^{(1)}$}
%\put(89,-2){$a_2^{(1)}$}
%\qbezier[16](40,4)(54,12)(65,0)
%\qbezier[13](55,0)(65,10)(75,3)
%\qbezier[13](55,39)(64,30)(75,37)
\end{picture}
\end{center}
\caption{\label{fig-perclemevent} The event $E(x,b,y;A)$ of
Lemma~\protect\ref{lem-cut1}.  The shaded regions represent the
bonds in $A$.  There is no restriction on intersections between
$A$ and $\tilde{\Ccal}^{\{u,v\}}(y)$.}
\end{figure}
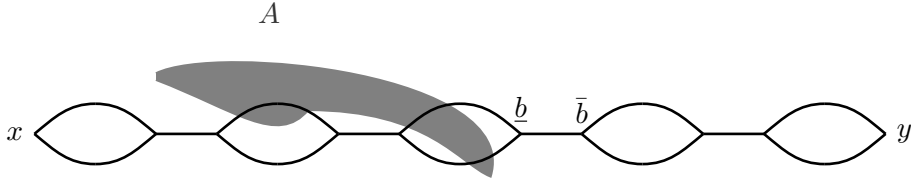
}
\def\figureStage2{
%%%%%FIGFIGFIGFIGFIGFIGFIGFIGFIGFIGFIGFIGFIGFIGFIGFIGFIGFIGFIGFIG
%%%%%FIGFIGFIGFIGFIGFIGFIGFIGFIGFIGFIGFIGFIGFIGFIGFIGFIGFIGFIGFIG
\begin{figure}
\begin{center}
\setlength{\unitlength}{0.0080in}%
\begin{picture}(480,50)
\put(-10,20){${\sss 0}$}
\put(120,25){${\sss b}$}
\put(360,25){${\sss b'}$}
\put(520,20){${\sss x}$}
\thicklines
\qbezier(10,20)(60,-20)(110,20)
\qbezier(10,20)(60,60)(110,20)
\qbezier(60,40)(290,100)(320,-10)
\qbezier(60,0)(120,-40)(160,20)
\thinlines
\multiput(110,20)(4,0){6}{\circle*{2}}
%\put(110,20){\line(2,0){10}}
\qbezier(130,20)(160,-10)(190,20)
\qbezier(130,20)(160,50)(190,20)
\put(190,20){\line(1,0){20}}
\qbezier(210,20)(240,-10)(270,20)
\qbezier(210,20)(240,50)(270,20)
\put(270,20){\line(1,0){20}}
\qbezier(290,20)(320,-10)(350,20)
\qbezier(290,20)(320,50)(350,20)
\put(350,20){\line(1,0){20}}
\qbezier(370,20)(400,-10)(430,20)
\qbezier(370,20)(400,50)(430,20)
\put(430,20){\line(1,0){20}}
\qbezier(450,20)(480,-10)(510,20)
\qbezier(450,20)(480,50)(510,20)
%\put(2,16){$0$}
%\put(112,16){$1$}
%\qbezier[13](27,9)(35,20)(27,31)
%\qbezier[15](18,25)(34,21)(40,36)
%\put(31,41){$a_1^{(0)}$}
%\put(14,-5){$a_2^{(0)}$}
%\qbezier[16](90,8)(80,20)(90,32)
%\put(89,34){$a_1^{(1)}$}
%\put(89,-2){$a_2^{(1)}$}
%\qbezier[16](40,4)(54,12)(65,0)
%\qbezier[13](55,0)(65,10)(75,3)
%\qbezier[13](55,39)(64,30)(75,37)
\end{picture}
\end{center}
\caption{A possible configuration appearing in the second stage of the
expansion.}
\label{fig-2pt2}
\end{figure}
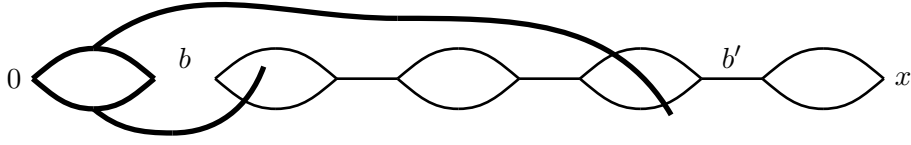
}
\section{The non-backtracking lace expansion}

\label{secExp}
In this section, we derive the NoBLE and thereby prove Proposition \ref{prop-LE}.  We proceed as follows: In Section \ref{secExpPercNotation}, we introduce the necessary notation, including a specific restricted two-point function. In Section \ref{sec-expresttwopoint}, we prove an expansion for this restricted two-point function. In Section \ref{sec-compl}, we use this expansion to obtain Proposition \ref{prop-LE}.

\subsection{Notation}
\label{secExpPercNotation}
Parts of this section are taken almost verbatim from \cite[Section 2]{HofHolSla07b}. Fix $p \in [0,1]$. We write $\tau(x)=\tau_p(x)$ for brevity, and generally drop subscripts indicating dependence on $p$.
\label{sec-factlem}
%%%%%%%%%%%%%% DEFINITION %%%%%%%%%%%%%%%%%%%%%%%%%%%%%%%%%
\begin{definition}[Occurring on and off sets of vertices and bonds]~%\vspace{-5mm}
\label{def-onin2}
\begin{enumerate}[(i)]
\item Given a bond configuration $\omega$ and two points $x,y\in\Zd$, we say that $x$ and $y$ are \emph{connected}, and write $x\conn y$, when there exists a path of occupied edges connecting $x$ and $y$.
      Further, we say that $x$ and $y$ are \emph{doubly connected}, and write $x\dbc y$,
      when there exist two bond-disjoint paths of occupied bonds connecting $x$ and $y$.
      We adopt the convenient convention that $x$ is doubly connected to itself.
\item Given a (deterministic or random)
set of undirected bonds $B$ and a bond configuration $\omega$,
we define $\omega_B$, the restriction of $\omega$ to $B$, to be
    \eqn{
    \omega_B(\{x,y\})
   =
   \left\{
   \begin{array}{lll}
   &\omega(\{x,y\})   &\text{if  }\{x,y\} \in B,\\
   &0           &\text{otherwise},
   \end{array}
   \right.
   }
for every nearest-neighbor pair $x,y$.
In other words, $\omega_B$ is obtained from $\omega$ by making every bond that is not in $B$ vacant.
\item Given a (deterministic or random) set of vertices $A$, we define ${\rm B}(A)$ to be the set of all bonds that have at least one endpoint in $A$.
\item Given a (deterministic or random)
set of bonds $B$ and an event $E$, we say
that \emph{$E$ occurs in $B$}, and write $\{E$ in $B\}$,
if $\omega_B\in E$. In other words, $\{E$ in $B\}$ means that $E$ occurs on the (possibly modified) configuration in which every bond that is not in $B$ is made vacant.  We further say that $E$ occurs \emph{off} $B$ when
$E$ occurs in $B^c$.

\item Given a (deterministic or random)
set of vertices $A$ and an event $E$, we say that \emph{$E$ occurs in $A$}, and write $\{E$ in $A\}$,
when $E$ occurs in ${\rm B}(A)$. We adopt the convenient convention that $\{x\conn x$ in $A\}$ occurs if and only if $x\in A$. We further say that $E$ occurs \emph{off} $A$ when $E$ occurs in ${\rm B}(A)^c$.

\item Given a bond configuration and $x \in \Z^d$, we define
$\Ccal(x)$ to be the set of vertices to which $x$ is connected,
i.e., $\Ccal(x)=\{y \in \Z^d\colon x\conn y\}$.
Given a bond configuration and a bond $b$, we define
$\tilde{\Ccal}^{b}(x)$ to be the set of vertices $y \in \Ccal(x)$
to which $x$ is connected in the (possibly modified)
configuration in which $b$ is made vacant.
\item Given a deterministic set of bonds $B$, we define the probability measure $\prob^{\sss B}$ by
  \eqn{
  \lbeq{prob-b-def}
  \prob^{\sss B}(E)=\prob(E \text{ occurs off } B).
}
\end{enumerate}
\end{definition}
%%%%%%%%%%%%%%%%%%%%%%%%%%%%%%%%%%%%%%%%%%%%%%%%%%%%%%%%%%%%%
\noindent
Regarding this definition we note for all events $E$ and deterministic sets of bonds $B, B'$,
    \eqalign{
    \big\{\{E \text{ off }B\}\text{ off }B'\big\}&=\{E\text { off }B\cup B'\},
    \lbeq{onininters}
    }
and therefore
    \eqalign{
  	\prob^{\sss B}(E \text{ off }B' )=\prob(E \text{ occurs off } B\cup B' ).
    \lbeq{ExpPercoffoff}
    }
Now we introduce the restricted two-point function, that was already stated in \refeq{conditional-Point-Prob-def}.
For any point $y$ we define
\begin{align}
 	\lbeq{tau-b-def-rep-full}
	\tau^y_p(x)&=\prob_p(0\conn x \text{ off } {\rm B}(y))=\prob^{{\rm B}(y)}_p(0\conn x).
	\end{align}
As abbreviation we define for $\iota\in\{\pm1,\dots, \pm d\}$
	\begin{align}
 	\lbeq{tau-b-def-rep}
	\tau^\iota_p(x)&=\tau^{\ve[\iota]}_p(x)=\prob_p(0\conn x \text{ off } \ve[\iota]).
	\end{align}
and note that $\tau^\iota_p(y-x)=\prob_p(x\conn y \text{ off } B(x+\ve[\iota]))$.
%\\We next define the notion of \emph{connections through}:
\begin{definition}[Connections through]
\label{def-through}
Given a bond configuration and a set of bonds $B\subseteq \Zd\times \Zd$, we say that $x$ is {\it connected to} $y$ \emph{through} $B$, and write
$x \ct{B} y$, if every occupied path connecting $x$ to $y$ contains at least one bond in $B$.
Given a bond configuration and a set $A\subseteq \Zd$, we
say that $x$ is {\it connected to} $y$ \emph{through} $A$, and write
$x \ct{A} y$, if $x$ is connected to $y$ through $\rm{B}(A)$.
By convention, $x \ct{A} x$ holds if and only if $x \in A$.
\end{definition}

\noindent
In terms of these events, it is clear that, for every set of vertices $A\subseteq \Z^d$,
\eqn{
    \lbeq{incl-excl-events}
    \{x\ct{A}y\}=\{x\conn y\}\setminus \{x\conn y \text{ off }{\rm B}(A)\}.
}
We can generalize this as follows: For any points $w,x,y\in\Zd$ and set of bonds $B$ such that\\ ${\rm B}(w) \subseteq B$ we know that
\begin{eqnarray}
    \lbeq{incl-excl-events-general}
    \{x\conn y\text { off } B\}=\{x\conn y\text { off }{\rm B}(w)\} \setminus \{x\ct{B} y\text{ off }{\rm B}(w)\},
	\end{eqnarray}
which implies that
\begin{eqnarray}
    \lbeq{ProbOffSplit}
   \prob\big(x \conn y \text{ off }B\big)&=&\tau^{w-x}(y-x)  -\prob^{w}(x \ctx{B}y).
\end{eqnarray}

Using this notation we first prove a general form of the expansion
and use it in Section \ref{sec-compl} to derive the expansion stated in Proposition \ref{prop-LE}:

\begin{lemma}[General NoBLE equation]
 \label{Lemma-tauxMA}
Fix $x,y\in\Zd$. Let $M \in\Nbold$, $A$ be any deterministic set of vertices and $B$ be any deterministic set of bonds satisfying that either $B\subseteq {\rm B}(x)$ or $B={\rm B}(A')$ for some set of vertices $A'\subseteq \Z^d$. Then there exist
$\Xi^{\sss B}_{\sss M}$, $\Psi^{{\sss B},\kappa}_{\sss M}$ and $R^{\sss B}_{\sss M}$ such that
	\eqn{
  	\lbeq{tauxMA}
    	\prob^{\sss B}(x\ct{A}y)=\Xi^{\sss B}_{\sss M}(x,y;A)+\sum_{w,\kappa}p\Psi^{{\sss B},\kappa}_{\sss M}(x,w; A)\tau^\kappa(y-w+\ve[\kappa])+
    	R^{\sss B}_{\sss M}(x,y;A).
	}
The dependence of $\Xi^{\sss B}_{\sss M}$ and $\Psi^{{\sss B},\kappa}_{\sss M}$ on $M$ is given by
	\begin{align}
    	\lbeq{2pt.37}
    	\Xi^{\sss B}_{\sss M}(x,y;A) &= \sum_{N=0}^{M} (-1)^N
    	\Xi^{\sss B,\ssc[N]}(x,y;A),%\\
    	%\lbeq{2pt.37b}
	\ \
    	\Psi^{{\sss B},\kappa}_{\sss M}(x,w;A) =
    	\sum_{N=0}^{M} (-1)^N  \Psi^{{\sss B},\ssc[N],\kappa}(x,w;A),
	\end{align}
with $\Xi^{{\sss B,}\ssc[N]}(x,y;A)$ and $\Psi^{{\sss B,}\ssc[N],\kappa}(x,w;A)$ independent of $M$.
\end{lemma}
The functions $\Xi^{\sss B}_{\sss M}$ and $\Psi^{{\sss B},\kappa}_{\sss M}$ are the key quantities in the NoBLE, and $R_{\sss M}^{\sss B}$ is a remainder term. The alternating signs in \refeq{2pt.37} arise via repeated use of inclusion-exclusion.
We will apply Lemma \ref{Lemma-tauxMA} for three choices of bond sets $B$, namely, $B=\varnothing, B={\rm B}(v)$ for some $v$ incident to $x$, and $B=\{b\}$ for some bond $b$ incident to $x$. The first and the last choices satisfy that $B\subseteq {\rm B}(x)$, the second satisfies the alternative restriction. This restriction arises since we wish to use the Cutting Lemma (see Lemma \ref{lem-cut1} below), which is traditionally stated in terms of vertex sets.

The next section is devoted to the proof of Lemma \ref{Lemma-tauxMA}.

\subsection{Expansions for restricted two-point functions}
\label{sec-expresttwopoint}
We next define what it means for a bond to be pivotal:
    \begin{definition}[Pivotal bonds]
    \label{def-inon}
    Given a bond configuration, a bond $\{u,v\}$ (occupied or not) is called
    \emph{pivotal} for the connection from $x$ to $y$, if (i) either $x \conn u$ and
    $y \conn v$, or $x \conn v$ and $y \conn u$, and (ii)
    $y \, \nin \, \tilde{\Ccal}^{\{u,v\}}(x)$.
    Bonds are not usually regarded as directed.  However, it will be convenient at times
    to regard a bond $\{u,v\}$ as directed from $u$ to $v$, and we will
    emphasize this point of view by writing $(u,v)$ for a directed bond.
    A directed bond $(u,v)$ is pivotal for the connection from $x$ to
    $y$, if $x \conn u$, $v\conn y$ and $y \, \nin \, \tilde{\Ccal}^{\{u,v\}}(x)$.
    For a directed bond $b=(u,v)$, we denote its starting point by $\bb=u$ and its ending point by $\tb=v$.
        \end{definition}
%\bigskip
\noindent

In terms of Definition \ref{def-onin2}, we have the characterization of a pivotal bond for $v \conn y$ as
    \eqalign{
   \{ b &\text{ pivotal    for }
    v \conn y\}
	=\{v \ct{b} y\}\nnb
    &=\big\{\{v\conn \bb,~\tb\nin \tilde{\Ccal}^b(v)\} \text{ in }{\rm B}(\tilde{\Ccal}^b(v))\setminus\{b\}\big\}\cap \big\{\tb\conn y
    \text{ in } {\rm B}(\tilde{\Ccal}^b(v))^c\big\}.
    \lbeq{pivrew}
    }
For a set of vertices $A$, we define the events
    \eqalign{
    \lbeq{317}
    E'(v, y; A) & = \{ v \ct{A} y\} \cap
    \left\{\begin{array}{c}\nexists b' \text{ occupied and pivotal for }\\
    v\conn y \text{ such that }    v \ct{A} \bb'\end{array}\right\}
    }
and
    \eqalign{
    \lbeq{317a}
    E(x, b, y; A) & = E'(x, \bb; A) \cap
    \{\text{$b$ is occupied and pivotal for $x \conn  y$}\}.
    }
    \figureEventET[0.8]

Given a configuration in which $x\ctx{A} y$, the {\em cutting bond}\/ $b$ is defined
to be the first bond that is pivotal for $x\conn y$ such that
$x\ctx{A} \bb$. It is possible that no such bond exists.
%with $y$ in the location currently occupied by $\bb$.
By partitioning $\{x\ct{A} y\}$ according to the location of the cutting bond (or the lack of a cutting bond), we obtain
    \eqn{
    \{x\ct{A} y\}
    = E'(x, y; A)
    \dot{\bigcup}\
    \dot{\bigcup_b}
    E(x, b, y; A),
    }
which implies that
    \eqalign{
    \lbeq{2pt.31a}
    \prob^{\sss B}(x\ct{A} y)
    & =
    \prob^{\sss B}(E'(x,y;A))+
    \sum_{b}\prob^{\sss B}(E(x, b, y;A))\\
    &=\prob^{\sss B}(E'(x,y;A))+
    \sum_{b\nin B}\prob^{\sss B}(E(x, b, y;A))\nn
    ,}
where the last equality follows from the fact that under $\prob^{\sss B}$ the event $E(x, b, y;A)$ is supposed to occur off $B$.
The following lemma is the major tool that we use to derive the expansion:
\begin{lemma}[The cutting lemma]
\label{lem-cut1}
Let $p<p_c(d)$, $x,y\in \ver$, and $A\subseteq \Z^d$. Then, for all bonds $b$,
    \eqarray
    \lbeq{percpivineq}
    \prob\left(E(x, b, y; A)\right)
    = p\expec_{\sss 0}\left(\indicwo{E'(x,\bb;A)}
    \indic{\tb \nin \tilde{\Ccal}^{b}_{\sss 0}(x)}
    \prob_{\sss 1}\big(\tb \conn y \text{ off }\tilde{\Ccal}^b_{\sss 0}(x)\big)\right).
    \enarray
\end{lemma}
We emphasize the fact that we deal with two percolation configurations by adding subscripts $0$ and $1$, so that the law of $\tilde{\Ccal}^b_{\sss 0}(x)$ is described by $\prob_{\sss 0}$ and
$\tilde{\Ccal}^b_{\sss 0}(x)$ can be considered to be {\em deterministic} when it appears in events described by $\prob_{\sss 1}$.

\proof
The lemma is proved e.g.\ in \cite[Lemma 10.1]{Slad06}, with the exception that the indicator $\tb \nin \tilde{\Ccal}^{b}_{\sss 0}(x)$ is absent on the right-hand side there. When $\tb \in \tilde{\Ccal}^{b}_{\sss 0}(x)$, however, we have
$\prob_{\sss 1}(\tb \conn y \text{ off }{\rm B}(\tilde{\Ccal}^b_{\sss 0}(x)) )\equiv 0$, so the statement is also true with the indicator. For the NoBLE, keeping this indicator is crucial.
\qed

\medskip
We remark here that \cite[Lemma 10.1]{Slad06} proves Lemma \ref{lem-cut1} for percolation on \emph{all} graphs. As a result, Lemma \ref{lem-cut1} also applies to the measure $\prob^{\sss B}$ for all deterministic bond sets $B$ and we obtain that for every $p<p_c(d)$, $x,y\in \ver$, set of bonds $B$, set of vertices $A$ and bonds $b$,
    \eqarray
    \lbeq{percpivineq-rep}
    \prob^{\sss B}\left(E(x, b, y; A)\right)
    = p\expec_{\sss 0}^{\sss B}\left(\indicwo{E'(x,\bb;A)}
    \indic{\tb \nin \tilde{\Ccal}^{b}_{\sss 0}(x)}
    \prob_{\sss 1}^{\sss B}\big(\tb \conn y \text{ off }\tilde{\Ccal}^b_{\sss 0}(x)\big)\right).
    \enarray
For the probability in the expectation we use \refeq{onininters} to see that
   	\eqarray
	\lbeq{off-union-B-protoStep}
   	\prob^{\sss B}_{\sss 1}\big(\tb \conn y \text{ off }\tilde{\Ccal}^b_{\sss 0}(x)\big)
  	= \prob_{\sss 1}\big(\tb \conn y \text{ off } {\rm B}(\tilde{\Ccal}^b_{\sss 0}(x))\cup B\big).
   	\enarray
To apply Lemma \ref{lem-cut1} once more, we need to consider connections that are off a set of vertices, while \refeq{off-union-B-protoStep}
instead considers a set of {\em bonds}. It is here that we rely on the two special choices of $B$ that we assumed in Lemma \ref{Lemma-tauxMA}.
Indeed, there we consider either a set of bonds $B$ such that $B={\rm B}(A')$ for some set of vertices $A'\subseteq \Z^d$,
or we consider $B\subseteq {\rm B}(x)$. In the latter case, we have that ${\rm B}(\tilde{\Ccal}^b_{\sss 0}(x))\cup B={\rm B}(\tilde{\Ccal}^b_{\sss 0}(x))$,
since $x\in \tilde{\Ccal}^b_{\sss 0}(x)$. For this choice of $B$, for convenience we write $A'=\varnothing$. Considering only the cases of Lemma \ref{Lemma-tauxMA},
we conclude
   	\eqarray
	\lbeq{off-union-B}
   	\prob^{\sss B}_{\sss 1}\big(\tb \conn y \text{ off }\tilde{\Ccal}^b_{\sss 0}(x)\big)
  	= \prob_{\sss 1}\big(\tb \conn y \text{ off } \tilde{\Ccal}^b_{\sss 0}(x)\cup A' \big),
   	\enarray
and now $\tilde{\Ccal}^b_{\sss 0}(x)\cup A'$ is a collection of vertices, as required in Lemma \ref{Lemma-tauxMA}.

The term in \refeq{off-union-B} denotes the restricted two-point function \emph{given} the cluster $\tilde{\Ccal}^{b}_{\sss 0}(x)$ of the outer expectation $\expec_{\sss 0}^{\sss B}$.
In other words, in  \refeq{percpivineq-rep} the inner expectation that defines $\prob_{\sss 1}$, effectively introduces a second percolation model on a second graph, which depends on the original percolation model via the set $\tilde{\Ccal}^{b}_{\sss 0}(x)$.  We stress this delicate point here, as it is also crucial for the further expansion.
Combining \refeq{off-union-B} with \refeq{2pt.31a} leads to

    \eqalign{
    \prob^{\sss B}(x\ct{A} y)
     =&
    \prob^{\sss B}(E'(x,y;A))\nnb
    \lbeq{2pt.31b}
    &+    \sum_{b\nin B}p\expec^{\sss B}_{\sss 0}\left(\indicwo{E'(x,\bb;A)}
    \indic{\tb \nin \tilde{\Ccal}^{b}_{\sss 0}(x)}
    \prob_{\sss 1}\big(\tb \conn y \text{ off }\tilde{\Ccal}^b_{\sss 0}(x)\cup A'\big)\right).
    }

As in \refeq{2pt.31b} the indicator $\indicwo{E'(x,\bb;A)}$ is present we know that only configurations with $\bb\in\tilde{\Ccal}^b(x)$ contribute and we can apply \refeq{ProbOffSplit} with $B=\tilde{\Ccal}^b(x)\cup A'$ and $w=\bb$ to obtain:
    \eqalign{
    \nonumber
    \prob^{\sss B}(x\ct{A} y)
     =&
    \prob^{\sss B}(E'(x,y;A))+  \sum_{b\nin B}p\expec^{\sss B}_{\sss 0}\left(\indicwo{E'(x,\bb;A)}
    \indic{\tb \nin \tilde{\Ccal}^{b}_{\sss 0}(x)}\right)\tau^{\bb}(\tb,y)\\
        \lbeq{2pt.26Perversion}
    &\quad-\sum_{b\nin B}p
    \expec^{\sss B}_{\sss 0}\left(\indicwo{E'(x,\bb;A)}\indic{\tb \nin \tilde{\Ccal}^{b}_{\sss 0}(x)}
    \prob^{\bb}_{\sss 1}(\tb \ctx{\tilde{\Ccal}^b_{\sss 0}(x)\cup A'}y)\right)\\
%        \lbeq{2pt.26}
    =&\Xi^{\sss B,\ssc[0]}(x,y;A)+\sum_{\kappa,w}\indic{(w,w-\ve[\kappa])\nin B}  p\Psi^{{\sss B,\ssc[0]},\kappa}(x,w;A)\tau^{\kappa}(y-w+\ve[\kappa])\nn\\
    &\qquad +R^{\sss B}_{\sss 0}(x,y;A),\nn
    }
where we define
    \eqan{
    \lbeq{Xi0def}
    \Xi^{\sss B,\ssc[0]}(x,y;A)&=\prob^{\sss B}(E'(x,y;A)),\\
    \lbeq{Psi0def}
    \Psi^{{\sss B,\ssc[0]},\kappa}(x,w;A)&=\indic{(w,w-\ve[\kappa])\nin B}\prob^{\sss B}(E'(x,w;A)\cap\{w-\ve[\kappa]\nin \tilde{\Ccal}^{(w,w-\ve[\kappa])}_{\sss 0}(x)\}),\\
    \lbeq{R0def}
    R^{\sss B}_{\sss 0}(x,y;A)&=-\sum_{b\nin B}p
    \expec^{\sss B}_{\sss 0}\left(\indicwo{E'(x,\bb;A)}\indic{\tb \nin \tilde{\Ccal}^{b}_{\sss 0}(x)}
    \prob^{\bb}_{\sss 1}\big(\tb \ctx{\tilde{\Ccal}^{b}_{\sss 0}(x)\cup A'}y\big)\right)\nnb
	&=-\sum_{b}p
    \expec^{\sss B}_{\sss 0}\left(\indicwo{E'(x,\bb;A)}\indic{\tb \nin \tilde{\Ccal}^{b}_{\sss 0}(x)\cup A'}
    \prob^{\bb}_{\sss 1}\big(\tb \ctx{\tilde{\Ccal}^{b}_{\sss 0}(x)\cup A'}y\big)\right),
    }
where the last equality holds since $b\nin B$, so that also $b\nin A'$ trivially holds. See the text below \refeq{off-union-B-protoStep} for details on the choice of $B$ and $A'$.
Further, after this change, we may remove the restriction $b\nin B$ from the sum, since
the expectation is trivially zero for $b\in B$, both when $B\subseteq {\rm B}(x)$ and when $B={\rm B}(A')$.

\figureStageTwoT[0.8]

\noindent
This proves Lemma \ref{Lemma-tauxMA} for $M=0$. To continue the expansion, we use \refeq{2pt.26Perversion} and ${\rm B}(\bb)\subseteq  {\rm B}(\tilde{\Ccal}^{b}_{\sss 0}(x))$ since $\bb\in \tilde{\Ccal}^{b}_{\sss 0}(x)$, to rewrite the factor $\prob^{\bb}_{\sss 1}(\tb \ctx{\tilde{\Ccal}^{b}_{\sss 0}(x)\cup A'}y)$ appearing in $R^{\sss B}_{\sss 0}(x,y;A)$ as
    \eqalign{
    \nn
    \prob^{\bb}_{\sss 1}(\tb \ctx{\tilde{\Ccal}^{b}_{\sss 0}(x)\cup A'}y)
    =&
    \prob^{\bb}_{\sss 1}(E'(\tb,y;\tilde{\Ccal}^{b}_{\sss 0}(x)\cup A'))\\
    &\quad +
    \sum_{b_{\sss 1}}p\expec^{\bb}_{\sss 1}\left(\indicwo{E'(\tb,\bb_{\sss 1};\tilde{\Ccal}^{b}_{\sss 0}(\tb)\cup A')}
    \indic{\tb_{\sss 1} \nin \tilde{\Ccal}^{b_{\sss 1}}_{\sss 1}(\tb)}\right)\tau^{\bb_{\sss 1}}(\tb_{\sss 1},y)\nn \\
\lbeq{secondstageDelv}
    &\quad-\sum_{b_{\sss 1}}p
    \expec^{\bb}_{\sss 1}\left(\indicwo{E'(\tb,y;\tilde{\Ccal}^{b}_{\sss 0}(\tb)\cup A')}\indic{\tb_{\sss 1} \nin \tilde{\Ccal}^{b_{\sss 1}}_{\sss 1}(\tb)\cup \{\bb\}}
    \prob^{\bb_{\sss 1}}_{\sss 2} \Big(\tb_{\sss 1} \ctx{\tilde{\Ccal}^{b_{\sss 1}}_{\sss 1}(\tb)\cup \{\bb\}}y\Big)\right).
    }
We introduce subscripts for $\tilde{\Ccal}$, the expectations and the bonds to indicate to which expectation they belong.
To derive this rewrite first add the restriction $\tb_{\sss 1}\neq \bb$, after which we can remove the restriction $b_1\nin {\rm B}(\bb)$ since otherwise the summand is trivially zero.
For brevity, we write $\tilde \Ccal_{\sss 0}=\tilde \Ccal^{b_{\sss 0}}_{\sss 0}(x)\cup A'$ and $\tilde \Ccal_i = \tilde \Ccal^{b_i}_i(\tb_{i-1})\cup\{\bb_{i-1}\}$ for $i\geq 1$.
We insert \refeq{secondstageDelv} into $R^{\sss B}_{\sss 0}(x,y;A)$ and obtain \refeq{tauxMA} for $M=1$ with
\eqan{
    \lbeq{Xi1def}
    \Xi^{\sss B,\ssc[1]}(x,y;A) &= \sum_{b_{\sss 0}}p\expec_{\sss 0}^{\sss B} \left(\indicwo{E'(x,\bb_{\sss 0};A)}\indic{\tb_{\sss 0}\nin \tilde{\Ccal}_{\sss 0}}
    \prob_{\sss 1}^{\bb_{\sss 0}} \big(E(\tb_{\sss 0},y;\tilde{\Ccal}_{\sss 0})\big)
    \right),\\
    \Psi^{{\sss B,\ssc[1]},\kappa}(x,w;A)&=\sum_{b_{\sss 0}} p \expec_{\sss 0}^{\sss B}\Big( \indicwo{E'(x,\bb_{\sss 0};A)}\indic{\tb_{\sss 0}\nin \tilde{\Ccal}_{\sss 0}}     \lbeq{Psi1def}\\ \nn
    &\qquad\qquad\qquad\times \prob_{\sss 1}^{\bb_{\sss 0}} \left(E'(\tb_{\sss 0},w;\tilde{\Ccal}_{\sss 0})  \cap \{w-\ve[\kappa]\nin \tilde{\Ccal}^{(w,w-\ve[\kappa])}_{\sss 1}(\tb_{\sss 0})\}\cup\{\bb_0\}\right) \Big),
    }
and
    \eqan{
    R^{\sss B}_{\sss 1}(x,y;A)&=\sum_{b_{\sss 0},b_{\sss 1}}p^2\expec_{\sss 0}^{\sss B}\Big(\indicwo{E'(x,\bb_{\sss 0};A)}\indic{\tb_{\sss 0}\nin \tilde{\Ccal}_{\sss 0}}
    \expec_{\sss 1}^{\bb_{\sss 0}}\Big(    \indicwo{E'(\tb_{\sss 0},\bb_{\sss 1};\tilde{\Ccal}_{\sss 0})}
    \indic{\tb_{\sss 1}\nin \tilde{\Ccal}_1}%\nnb
    \lbeq{R1-def}
    %&\qquad\qquad \times
    \prob^{\bb_{\sss 1}}_{\sss 2}\big(\tb_{\sss 1} \ctx{\tilde{\Ccal}_{\sss 1}}y\big)\Big)
    \Big).
    }
This proves Lemma \ref{Lemma-tauxMA} for $M=1$.
We now repeat using \refeq{secondstageDelv} recursively, for
	\begin{align}
	\prob^{\bb_{\sss M}}_{\sss M+1}
	\Big(\tb_{\sss M} \ctx{\tilde{\Ccal}_{\sss M}}y\Big)
	\end{align}
that appears in the remainder term $R^{\sss B}_{\sss M}(x,y;A)$. This leads to Lemma \ref{Lemma-tauxMA} for all $M\geq 0$ with
$\Xi^{\sss B,\ssc[N]}$, $\Psi^{{\sss B,\ssc[N]},\kappa}$ and $R^{\sss B}_{\sss N}$
given in \refeq{Xi0def}-\refeq{R0def} for $N=0$,
in \refeq{Xi1def}-\refeq{R1-def} for $N=1$ and for $N\geq 2$ given by
   \eqalign{
    \lbeq{XiNdef}
    \Xi^{\sss B,\ssc[N]}(x,y;A)
    & =
    p^N\sum_{b_{\sss 0}, \ldots, b_{N-1}}
    \expec_{\sss 0}^{\sss B} \indicwo{E'(x,\bb_{\sss 0};A)}\indic{\tb_{\sss 0} \nin \tilde{\Ccal}_{\sss 0}} \expec_{\sss 1}^{\bb_{\sss 0}} \indicwo{E'(\tb_{\sss 0}, \bb_{\sss 1}; \tilde{\Ccal}_{\sss 0})}
    \indic{\tb_{\sss 1} \nin \tilde{\Ccal}_{\sss 1}}\, \,
    \\\nn &
    \quad \quad \times \expec_{\sss 2}^{\bb_{\sss 1}}
    \indicwo{E'(\tb_{\sss 1}, \bb_{\sss 2};\tilde{\Ccal}_{\sss 1})}
    \indic{\tb_{\sss 2}\nin\tilde{\Ccal}_{\sss 2}}
    \cdots
    \expec_{\sss N}^{\bb_{\sss N-1}} \indicwo{E'(\tb_{\sss N-1}, y; \tilde{\Ccal}_{\sss N-1})},\\
    \lbeq{PsiNdef}
    \Psi^{{\sss B,\ssc[N]},\kappa}(x,y;A)
    & =p^{N}\sum_{b_{\sss 0}, \ldots, b_{N-1}}
    \expec_{\sss 0}^{\sss B}\indicwo{E'(x,\bb_{\sss 0};A)}\indic{\tb_{\sss 0} \nin \tilde{\Ccal}_{\sss 0}} \expec_{\sss 1}^{\bb_{\sss 0}} \indicwo{E'(\tb_{\sss 0}, \bb_{\sss 1}; \tilde{\Ccal}_{\sss 0})}
    \indic{\tb_{\sss 1} \nin \tilde{\Ccal}_{\sss 1}}\, \,
    \\ \nonumber &
    \quad \quad \times
    \expec_{\sss 2}^{\bb_{\sss 1}}
    \indicwo{E'(\tb_{\sss 1}, \bb_{\sss 2};\tilde{\Ccal}_{\sss 1})}
    \indic{\tb_{\sss 2}\nin\tilde{\Ccal}_{\sss 2}} \cdots
    \expec_{\sss N-1}^{\bb_{N-2}}
    \indicwo{E'(\tb_{\sss N-2}, \bb_{\sss N-1};\tilde{\Ccal}_{\sss N-2})}
    \indic{\tb_{\sss N-1}\nin\tilde{\Ccal}_{\sss N-1}}
    \\ \nonumber &
    \quad \quad \times \expec_{\sss N}^{\bb_{N-1}}
    \big(\indicwo{E'(\tb_{N-1},y;\tilde{\Ccal}_{\sss N-1})}\indic{y-\ve[\kappa]\nin \tilde{\Ccal}^{(y,y-\ve[\kappa])}_{\sss N}(\tb_{\sss N-1})\cup\{\bb_{\sss N-1}\}}\big),
    }
    \eqalign{
    \lbeq{Rdef}
    R^{\sss B}_{\sss N}(x,y;A)
    & = %\!\!\!\!\!\shift\shift
    (-1)^{N+1}p^{N+1}\sum_{b_{\sss 0}, \ldots, b_{\sss N}}
    \expec_{\sss 0}^{\sss B} \indicwo{E'(x,\bb_{\sss 0};A)}\indic{\tb_{\sss 0} \nin \tilde{\Ccal}_{\sss 0}} \, \,
    \\ \nonumber &
    \quad \quad \times
    \expec_{\sss 1}^{\bb_{\sss 0}}
    \indicwo{E'(\tb_{\sss 0}, \bb_{\sss 1}; \tilde{\Ccal}_{\sss 0})}
    \indic{\tb_{\sss 1} \nin \tilde{\Ccal}_{\sss 1}}
    \cdots
    \expec_{\sss N-1}^{\bb_{\sss N-2}}
    \indicwo{E'(\tb_{\sss N-2}, \bb_{\sss N-1}; \tilde{\Ccal}_{\sss N-2})}
    \indic{\tb_{\sss N-1} \nin \tilde{\Ccal}_{\sss N-1}}\nnb
    &\qquad \times
    \expec_{\sss N}^{\bb_{\sss N-1}} \indicwo{E'(\tb_{\sss N-1}, \bb_{\sss N};\tilde{\Ccal}_{\sss N-1})}
    \indic{\tb_{\sss N}\nin \tilde{\Ccal}_{\sss N}}
    \prob_{\sss N+1}^{\bb_{\sss N}}(\tb_{\sss N}\ctx{\tilde{\Ccal}_{\sss N}} y).
    \nn
    }
%for $N\geq 2$.
Since
    \eqn{
    \prob_{\sss N+1}^{\bb_{\sss N}}(\tb_{\sss N}\ctx{\tilde{\Ccal}_{\sss N}} y)\leq \tau^{\bb_{N}}(\tb_{\sss N},x),
    }
it follows from \refeq{PsiNdef}--\refeq{Rdef} that
    \eqn{
    \lbeq{Rxbd}
    |R^{\sss B}_{\sss N}(x,y;A)|
    \leq \sum_{w,\kappa}
    \Psi^{\sss B,\ssc[N],\kappa}(x,w;A)p\tau^\kappa(y-w+\ve[\kappa]).
    }
When we take $M\rightarrow \infty$, and assume that $\lim_{M\rightarrow \infty} |R^{\sss B}_{\sss M}(x,y;A)|=0$, we arrive at
    \eqn{
    \lbeq{taux}
        \prob^{\sss B}(x\ct{A}y)=\Xi^{\sss B}(x,y;A)+\sum_{w,\kappa}p\Psi^{{\sss B},\kappa}(x,w; A)\tau^\kappa(y-w+\ve[\kappa]),
    }
where
	\eqn{
     	\Xi^{\sss B}(x,y;A)=\sum_{N=0}^{\infty} (-1)^N  \Xi^{\sss B, \ssc[N]}(x,y;A),
	\quad
    	\lbeq{PhiPsialt}
    	\Psi^{{\sss B},\kappa}(x,w; A)= \sum_{N=0}^{\infty} (-1)^N \Psi^{{\sss B, \ssc[N]},\kappa}(x,w;A).
 	}
Naturally, the convergence of the expansion needs to be obtained to reach the above conclusion. This convergence follows from \refeq{Rxbd} and the bounds on $ \Psi^{\sss B,\ssc[N],\kappa}$ that we prove in Section \ref{secBoundsPerc}, by showing that the remainder term $R^{\sss B}_{\sss N}$ converges to zero.
The expansion developed here is different from the traditional lace expansion as it expands in terms of $\tau^\iota(x)$ rather than $\tau(x)$. This difference causes that the formulas \refeq{XiNdef},
\refeq{PsiNdef} and \refeq{Rdef} involve $\expec_{\sss j}^{\bb_{j-1}}$ rather than just $\expec_{\sss j}$
as in \cite{HarSla90a}. Further, we explicitly keep the factors $\indic{\tb_{\sss j} \nin \tilde{\Ccal}_{\sss j}}$.
Finally, the set appearing in the $E'$ events is now $\tilde{\Ccal}_j^{b_j}(\tb_{j-1})\cup \{\bb_{j-1}\}$,
while in the classical lace expansion $\tilde{\Ccal}_j^{b_j}(\tb_{j-1})$ appears, see e.g.\ \cite{HarSla90a}.

These differences ensure, as we argue in the following, that each loop in the lace-expansion coefficients now involve paths of at least four steps, whereas in \cite{HarSla90a} they can have length equal to two.
By a loop we denote a closed path of occupied bonds. The involved bonds may be occupied on different percolation configurations enforced by the events $E'$ and $\tb_{j+1}\nin \tilde{\Ccal}_{j}$. By the parity of the hypercubic lattice, a loop consists of an even number of steps. On the lattice there exists only one possibility for a two-step loop, namely, when $\tb_{j-1}=\bb_j$ and $\tb_{j}=\bb_{j-1}$. We now argue by contradiction that $\bb_j=\tb_{j-1}$ does not contribute. Let us assume instead that  $\bb_j=\tb_{j-1}$. Then, since $E'(x, x; A)=\{x\ctx{A}x\}=\{x\in A\}$,
    \eqn{
    \lbeq{bbj-neq-tbj-1a}
    E'(\tb_{\sss j-1}, \bb_{\sss j}; \tilde{\Ccal}_{\sss j-1})
    =\{\tb_{j-1}\in \tilde{\Ccal}_{\sss j-1}\},
    }
which does not contribute to the lace-expansion coefficients, due to the presence of the indicator $\indic{\tb_{\sss j-1} \nin \tilde{\Ccal}_{\sss j-1}}$. Thus, indeed, loops in the lace-expansion coefficients consist of at least four bonds.
Due to these differences the largest contributions to the classical lace-expansion coefficients
are not present for the NoBLE lace-expansion coefficients.

\subsection{Completion of the NoBLE}
\label{sec-compl}
In this section, we complete the NoBLE. Lemma \ref{Lemma-tauxMA} with $B=\varnothing$ (so that trivially $B\subseteq {\rm B}(0)$) and $A=\{0\}$
yields
	\eqn{
  	\tau(x) =\Xi^{\sss \varnothing}_{\sss M}(0,x;\{0\})
	+\sum_{w,\kappa}p\Psi^{{\sss \varnothing},\kappa}_{\sss M}(0,w; \{0\})\tau^\kappa(x-w+\ve[\kappa])+
    	R^{\sss \varnothing}_{\sss M}(0,x;\{0\}).
    	\lbeq{taux-M-in-ex-1-proto}
	}
We extract the dominant contribution of $\Xi^{\sss \varnothing}_{\sss M}(0,x;\{0\})$ and $\Psi^{{\sss \varnothing},\kappa}_{\sss M}(0,w; \{0\})$ from this.
We note that $\Xi^{{\sss \varnothing,}\ssc[0]}(0,0;\{0\})=1$ and
	\eqn{
	\Psi^{{\sss \varnothing},\ssc[0],\kappa}(0,0; \{0\})
	=\prob(-\ve[\kappa]\nin \tilde{\Ccal}^{(0,-\ve[\kappa])}(0))
	=\prob(0\nc -\ve[\kappa] \text{ off the bond }\{0,-\ve[\kappa]\}).
	}
Define, recalling \refeq{Xi0def}--\refeq{Psi0def} and \refeq{317},
	\begin{align}
  	\lbeq{def-xi-zero}
    	\Xi^{\ssc[0]}(x)&=(1-\delta_{0,x})
	\Xi^{\sss
    	{\sss \varnothing,}\ssc[0]}(0,x;\{0\})=(1-\delta_{0,x})\prob(0\dbc x),\\
	\lbeq{def-psi-zero}
    	\Psi^{\ssc[0],\kappa}(x)&=\frac p
    	{\aap}(1-\delta_{0,x})\Psi^{{\sss \varnothing,}\ssc[0],\kappa}(0,x;\{0\})\\
	&=(1-\delta_{0,x})\frac p \aap \prob_p(\{0\dbc x \} \cap \{(x-\ve[\kappa])\nin\tilde{\Ccal}^{\{x,x-\ve[\kappa]\}}(x)\} ),\nn
	\end{align}
with $\aap=p\prob(e_\kappa\nin \tilde{\Ccal}^{(0,e_\kappa)}(0) )$, and where, in \refeq{def-psi-zero}, we use that $\tilde{\Ccal}^{\{x,x-\ve[\kappa]\}}(0)=\tilde{\Ccal}^{\{x,x-\ve[\kappa]\}}(x)$ on the event that $\{0\conn x\}$. For $N\geq 1$, we define
	\begin{align}
    \lbeq{tau-LEC-ident}
    \Xi^{\ssc[N]}(x)&=\Xi^{\sss \varnothing,\ssc[N]}(0,x;\{0\}),\qquad\quad
    \Psi^{\ssc[N],\kappa}(x)=\frac p {\aap} \Psi^{{\sss\varnothing,}\ssc[N],\kappa}(0,x;\{0\}).
	\end{align}
and use these functions to define
	\eqn{
    	\lbeq{tau-LEC-ident-2}
    	\shift\Xi_{\sss M}(x)=\sum_{N=0}^M(-1)^N  \Xi^{\ssc[N]}(x),\quad
	\Psi^\kappa_{\sss M}(x)=\sum_{N=0}^M(-1)^N  \Psi^{\ssc[N],\kappa}(x),\quad
     	R_{\sss M}(x)=R^{\sss\varnothing}_{\sss M}(0,x;\{0\}).
	}
In this notation, \refeq{taux-M-in-ex-1-proto} becomes
	\eqn{
      	\tau(x)=\delta_{0,x}+\Xi_{\sss M}(x)
	+\aap \sum_{w,\kappa}(\delta_{0,w}+\Psi^{\kappa}_{\sss M}(w))\tau^\kappa(x-w+\ve[\kappa])+ R_{\sss M}(x).
  	}
This proves the first relation of the NoBLE in \refeq{taux-M-in-ex-1}.
To obtain the second relation of the NoBLE in \refeq{taux-M-in-ex-2}, we first define $b_\iota=(0,\ve[\iota])$ and see that
	\begin{align}
  	\lbeq{tau-tauiota-split}
    	\tau(x)-\tau^\iota(x) =\prob(0\ct{\ve[\iota]}x)=\prob(0\ct{b_\iota}x)
	+\prob^{b_\iota}(0\ct{\ve[\iota]} x),
	\end{align}
where, for a bond $b$, we abbreviate $\prob^{b}=\prob^{\{b\}}$.
We investigate both terms separately, starting with $\prob(0\ct{b_\iota}x)$, with the aim to extract the NBW-like contribution $p\tau^{-\iota}( x-\ve[\iota])$. We see that $\{0\ct{b_\iota}x\}=E(0,b_\iota,x;\{0\})$ as it is equivalent to $b_\iota$ being occupied and pivotal for $0\conn x$. Thus, we can apply Lemma \ref{lem-cut1}, with $x=0,b=b_\iota, y=x$ and $A=\{0\}$ to obtain
	\begin{eqnarray}
     	\prob(0\ct{b_\iota}x)&=&p\expect_{\sss 0}\left[\indic{\ve[\iota]\nin\tilde \Ccal^{b_\iota}_{\sss 0}(0)}
	\prob_{\sss 1}(\ve[\iota]\conn x \text{ off }\tilde \Ccal^{b_{\iota}}_{\sss 0}(0))\right].
	\end{eqnarray}
Next, we analyze $\prob_{\sss 1}(\ve[\iota]\conn x \text{ off }\tilde \Ccal^{b_{\iota}}_{\sss 0}(0))$ within the outside expectation $\expect_{\sss 0}$. For this we consider $\tilde \Ccal^{b_{\iota}}_{\sss 0}(0)$ to be a fixed deterministic set.
Since $0\in \tilde \Ccal^{b_{\iota}}_{\sss 0}(0)$, we conclude as in \refeq{ExpPercoffoff} that
\begin{align}
	&\prob_{\sss 1}(\ve[\iota]\conn x \text{ off }\tilde \Ccal^{b_{\iota}}_{\sss 0}(0))
	=\prob_{\sss 1}^{\sss 0}(\ve[\iota]\conn x \text{ off }\tilde \Ccal^{b_{\iota}}_{\sss 0}(0)).
	\end{align}
Then, we use an inclusion-exclusion argument to rewrite (recall the definition of $\prob_{\sss 1}^{\sss 0}$ in \refeq{conditional-Point-Prob-def})
\begin{align}
	\prob_{\sss 1}^{\sss 0}(\ve[\iota]\conn x \text{ off }\tilde \Ccal^{b_{\iota}}_{\sss 0}(0))
	&=\prob_{\sss 1}^{\sss 0}(\ve[\iota]\conn x)
	-\prob_{\sss 1}^{\sss 0}(\ve[\iota]\ct{\tilde \Ccal^{b_\iota}_{\sss 0}(0)} x)
	%\nnb&
	=\tau^{-\iota}(x-\ve[\iota])-\prob_{\sss 1}^{\sss 0}(\ve[\iota]\ct{\tilde \Ccal^{b_{\iota}}_{\sss 0}(0)} x),\nn
	\end{align}
and obtain
	\begin{align}
    	\prob(0\ct{b_\iota}x)&
	=p\tau^{-\iota}(x-\ve[\iota])\prob^{b_\iota}\big(\ve[\iota]\nin\tilde \Ccal^{b_\iota}_{\sss 0}(0)\big)
	-p\expect_{\sss 0}^{b_\iota}\Big[\indic{\ve[\iota]\nin\tilde \Ccal^{b_\iota}_{\sss 0}(0)}
	\prob_{\sss 1}^{\sss 0}(\ve[\iota]\ct{\tilde \Ccal^{b_{\iota}}_{\sss 0}(0)} x)\Big].
	\end{align}
As $\tilde \Ccal^{b_{\iota}}_{\sss 0}(0)$ is deterministic in the inner probability $\prob_{\sss 1}^{\sss 0}$, we apply Lemma \ref{Lemma-tauxMA} to $\prob_{\sss 1}^{\sss 0}(\ve[\iota]\ct{\tilde \Ccal^{b_{\iota}}_{\sss 0}(0)} x)$ with $B={\rm B}(0)$ and $A=\tilde \Ccal^{b_{\iota}}_{\sss 0}(0)$ to obtain
\begin{align}
\prob(0\ct{b_\iota}x)
	=&\aap \tau^{-\iota}(x-\ve[\iota])
	-p\expect_{\sss 0}^{b_\iota}\Big[\indic{\ve[\iota]\nin\tilde \Ccal^{b_\iota}_{\sss 0}(0)}\left( \Xi^{\sss{\rm B}(0)}_{\sss M}(\ve[\iota],x;\tilde \Ccal^{b_{\iota}}_{\sss 0}(0))+   R^{\sss {\rm B}(0)}_{\sss M}(\ve[\iota],x;\tilde \Ccal^{b_{\iota}}_{\sss 0}(0))\right)\Big]\nnb
	&-p\expect_{\sss 0}^{b_\iota}\Big[\indic{\ve[\iota]\nin\tilde \Ccal^{b_\iota}_{\sss 0}(0)}\sum_{w,\kappa}p
    \Psi^{{\sss {\rm B}(0)},\kappa}_{\sss M}(\ve[\iota],w; \tilde \Ccal^{b_{\iota}}_{\sss 0}(0))\tau^\kappa(x-w+\ve[\kappa])\Big].
	  	\lbeq{tau-tauiota-split-p1}
\end{align}
%%% NOTE: Here we should bear in mind that e_{\iote} should appear in the first \tilde{C} cluster inside the expansion terms! This is the ONLY case where this happens...
To deal with $\prob^{b_\iota}(0\ct{\ve[\iota]} x)$ in \refeq{tau-tauiota-split}, we directly apply Lemma \ref{Lemma-tauxMA} with $B=\{b_{\iota}\}\subseteq{\rm B}(0)$ and $A=\{e_{\iota}\}$ to arrive at
\begin{align}
    \prob^{b_\iota}(0\ct{\ve[\iota]}x)=\Xi^{b_\iota}_{\sss M}(0,x;\{\ve[\iota]\})+\sum_{w,\kappa}p\Psi^{b_\iota,\kappa}_{\sss M}(0,w; \{\ve[\iota]\})\tau^\kappa(x-w+\ve[\kappa])+   R^{b_\iota}_{\sss M}(0,x;\{\ve[\iota]\}).
    	  	\lbeq{tau-tauiota-split-p2}
\end{align}
Combining \refeq{tau-tauiota-split},  \refeq{tau-tauiota-split-p1} and \refeq{tau-tauiota-split-p2} concludes the derivation of \refeq{taux-M-in-ex-2}, i.e.,
	\begin{align}
	\tau(x)&=\tau^{\iota}(x)+ \aap \tau^{-\iota}(x-\ve[\iota])
	+\sum_{y,\kappa}\Pi^{\iota,\kappa}_{\sss M}(y)\tau^{\kappa}(x-y+\ve[\kappa])
	+\Xi^{\iota}_{\sss M}(x)+R^{\iota}_{\sss M}(x),
	\end{align}
with
	\begin{eqnarray}
    	\lbeq{tau-tauj-LEC-ident-NZero}
        \Xi^{{\sss (0)},\iota}(x)&=&\Xi^{b_\iota,\ssc[0]}(0,x;\{\ve[\iota]\}),\qquad
        \Pi^{{\sss (0)},\iota,\kappa}(x)=p\Psi^{b_\iota,{\sss (0)}, \kappa}(0,x;\{\ve[\iota]\}),\\
         R^{\iota}_{\sss 0}(x)&=&R^{b_\iota}_{\sss 0}(0,x;\{\ve[\iota]\}),
	\end{eqnarray}
and, for $N,M\geq 1$,
	\begin{eqnarray}
    	\lbeq{tau-tauj-LEC-ident1}
        \Xi^{\ssc[N],\iota}(x)&=&\Xi^{b_\iota,\ssc[N]}(0,x;\{\ve[\iota]\})+p\expect_{\sss 0}^{b_\iota}\left[\indic{\ve[\iota]\nin\tilde \Ccal^{b_\iota}_{\sss 0}(0)}\Xi^{{\rm B}(0),\ssc[N-1]}(\ve[\iota],x;\tilde \Ccal^{b_\iota}_{\sss 0}(0))\right],\\
            \lbeq{tau-tauj-LEC-ident2}
        \!\!\Pi^{\ssc[N],\iota,\kappa}(x)&=&p\Psi^{b_\iota,\ssc[N], \kappa}(0,x;\{\ve[\iota]\})+p^2\expect_{\sss 0}^{b_\iota}\left[\indic{\ve[\iota]\nin\tilde \Ccal^{b_\iota}_{\sss 0}(0)}\Psi^{{\rm B}(0),{\ssc[N-1]}, \kappa}(\ve[\iota],x;\tilde \Ccal^{b_\iota}_{\sss 0}(0))\right],\\
            \lbeq{tau-tauj-LEC-ident3}
         R^\iota_{\sss M}(x)&=&R^{b_\iota}_{\sss M}(0,x;\{\ve[\iota]\})+p\expect_{\sss 0}^{b_\iota}\left[\indic{\ve[\iota]\nin\tilde \Ccal^{b_\iota}_{\sss 0}(0)}R^{{\rm B}(0)}_{\sss M-1}(\ve[\iota],x;\tilde \Ccal^{b_\iota}_{\sss 0}(0))\right].
	\end{eqnarray}
Finally, \refeq{Rxbd} together with the above characterization of $R_{\sss M}(x)$ and $R_{\sss M}^\iota(x)$ proves \refeq{RM-bd}--\refeq{RMiota-bd}.

This completes the derivation of the NoBLE for percolation and thereby the proof of Proposition \ref{prop-LE}.
Further, we have obtained a description of the NoBLE coefficients that will be the starting point to obtain  bounds on them in Section \ref{secBoundsPerc}.
\qed

\subsection{Split of the coefficients of the NoBLE analysis}
\label{subsectionSplit}
For the analysis in \cite[Section 3]{FitHof13b}, we extract some explicit contributions from the coefficients, so as to improve the numerical precision of our method.  When choosing the terms to extract we are guided by the intuition that they should be  substantial contributions to the coefficients and that we need to be able to accurately approximate them numerically.  This is usually only possible for contributions created for $x=0$ and $\|x\|=1$, in particular, for $x=\ve[\iota]$.

Terms with subscripts $\alpha$ correspond to the leading order contributions, while terms with subscripts $R$ correspond to errors. Further, the subscripts $I$ refer to the extraction of terms for which $\|x-\ve[\kappa]\|_2\leq 1$ for $\kappa$ fixed, while the subscripts $II$ refer to the extraction of terms with $\|x\|_2\leq 1$. Thus, for example,
	\eqn{
	\sum_{x\in \Zd}[1-\e^{\ii k\cdot  (x-\ve[\kappa])}]\Psi^{\ssc[0],\kappa}_{\alpha,{\sss I},p}(x)
	=\alpha_{\sss I} [1-\hat{D}(k)],
	\quad
	\sum_{x\in \Zd}[1-\e^{\ii k\cdot x}]\Psi^{\ssc[0],\kappa}_{\alpha,{\sss II},p}(x)
	=\alpha_{\sss II} [1-\hat{D}(k)],
	}
for some $\alpha_{\sss I},\alpha_{\sss II}$. These terms can be incorporated in the random walk contributions, while other contributions cannot. The terms labeled with $I$ are numerically larger and contribute to bigger contributions in the analysis of \cite{FitHof13b}. In this document however we often focus on the terms with subscripts $II$ as they tend to be easier to define and bound.\\

In the sequel, we will make these notions precise.  We start by formulating the split for $N=0$. We define
	\begin{align}
  	\lbeq{Def-Xi-Split}
	\Xi^{\ssc[0]}_{\alpha,p}(x)&:=\indic{\|x\|_2= 1}\prob_p(\{0\dbc x\}\cap\{(0,x)\text{ is occ.}\} ),\\
  	\lbeq{Def-Psi-Split-II}
	\Psi^{\ssc[0],\kappa}_{\alpha,{\sss II},p}(x)&:=\frac {p}  {\aap} \indic{\|x\|_2= 1}
	\prob_p\Big(\{0\dbc x\}\cap\{(0,x)\text{ is occ.}\}\cap\{(x-\ve[\kappa])\nin\tilde \Ccal^{(x,x-\ve[\kappa])}(0)\} \Big),\\
    \lbeq{Def-Psi-Split-III}
	\Xi^{\ssc[0]}_{{\sss R},p}(x)&:=\Xi^{\ssc[0]}_{p}(x)-\Xi^{\ssc[0]}_{\alpha,p}(x),\qquad
	\qquad
	\Psi^{\ssc[0],\kappa}_{{\sss R, II},p}(x):=\Psi^{\ssc[0],\kappa}_{p}(x)-\Psi^{\ssc[0],\kappa}_{\alpha,{\sss II},p}(x).
	\end{align}
%Here $\Psi^{\ssc[0],\kappa}_{\alpha,{\sss II},p}(x)$ extracts the main contribution to $\Psi^{\ssc[0],\kappa}_{p}(x)$ for $\|x\|_2= 1$.
In $\Psi^{\ssc[0],\kappa}_{\alpha,{\sss I},p}(x)$, we extract the main contributions to $\Psi^{\ssc[0],\kappa}_{p}(x)$ for $\|x-\ve[\kappa]\|_2\leq 1$. Let
	\begin{align}
  	\lbeq{Def-Psi-Split-I}
	\Psi^{\ssc[0],\kappa}_{\alpha,{\sss I},p}(x):=
	&\indic{\|x-\ve[\kappa]\|_2\leq 1}\frac {p}  {\aap}
	\prob_p(\{0\dbc x\}\cap\{(x-\ve[\kappa])\nin\tilde \Ccal^{(x,x-\ve[\kappa])}(0)\}\\
	&\qquad\qquad\quad \cap\{\exists\text{ path between $0$ and $x$ consisting of one or two occ. bonds}\} ),\nn
	\end{align}
and $\Psi^{\ssc[0],\kappa}_{{\sss R, I},p}(x):=\Psi^{\ssc[0],\kappa}_{p}(x)-\Psi^{\ssc[0],\kappa}_{\alpha,{\sss I},p}(x)$. We split $\Xi^{\ssc[0],\iota}_{p}(x)$ as
    \begin{align}
	\lbeq{Def-XiIota-Split-I}
	\Xi^{\ssc[0],\iota}_{\alpha,{\sss I},p}(x)
	:=&\delta_{x,\ve[\iota]}\Xi^{\ssc[0],\iota}_{p}(\ve[\iota])\\ &
	+\indic{\|x-\ve[\iota]\|=1}
	\prob ( \{0\ct{\{\ve[\iota]\}} x \text{ off } \{(\ve[\iota],x)\} \} \circ \{(\ve[\iota],x) \text{ is occ.}\} \mid (0,\ve[\iota])\text{ is vacant} ),\nnb
 	\lbeq{Def-XiIota-Split-II}
	\Xi^{\ssc[0],\iota}_{\alpha,{\sss II},p}(x):=&\delta_{x,\ve[\iota]}\Xi^{\ssc[0],\iota}_{p}(\ve[\iota]),\\
 	\Xi^{\ssc[0],\iota}_{{\sss R, I},p}(x):=&\Xi^{\ssc[0],\iota}_{p}(x)-\Xi^{\ssc[0],\iota}_{\alpha,{\sss I},p}(x)\qquad
 	\Xi^{\ssc[0],\iota}_{{\sss R, II},p}(x):=\Xi^{\ssc[0],\iota}_{p}(x)-\Xi^{\ssc[0],\iota}_{\alpha,{\sss II},p}(x).
	\end{align}
Finally,
	\begin{align}
 	\lbeq{Def-Pi-Split}
	\Pi^{\ssc[0],\iota,\kappa}_{\alpha,p}(x)&:=\delta_{x,\ve[\iota]} \Pi^{\ssc[0],\iota,\kappa}_{p}(x),\qquad
	\Pi^{\ssc[0],\iota,\kappa}_{{\sss R},p}:=\Pi^{\ssc[0],\iota,\kappa}_{p}(x)-\Pi^{\ssc[0],\iota,\kappa}_{\alpha,p}(x).
	\end{align}
This completes the definition of the relevant splits for $N=0$.

For $N=1$, we recall the definition of $\Xi^{\sss B,\ssc[1]}(x,y;A)$ and $\Psi^{{\sss B,\ssc[1]},\kappa}(x,w;A)$ in \refeq{Xi1def}-\refeq{Psi1def} with $B=\varnothing$, $A=\{0\}$ and of the event $E'$ in \refeq{317}. Due to the way in which we bound the coefficients, the definition of the split is a bit involved as we only want to extract some specific contributions.  In each case, we extract the contribution where the pivotal edge $b_0=(0,e)$ starts at the origin and the cut through occurs directly at $x$.
We define, for $x$ with $\|x\|_1=1$,
	\begin{align}
   	\lbeq{Def-XiOne-Split}
    	\Xi^{\ssc[1]}_{\alpha,p}(x) &= p\sum_{e\colon \|e\|_1=1} \expec_{\sss 0}
	\left(\indic{e\nin \tilde\Ccal^{(0,e)}_{\sss 0}(0)}  \indic{(0,x)\text{ is occ.}}
	\prob_{\sss 1}^{\sss 0} \left(E'(e,x;\tilde \Ccal^{(0,e)}_{\sss 0}(0)) \right)\right),\\
   	\lbeq{Def-PsiOne-SplitTwo}
    	\Psi^{\ssc[1],\kappa}_{\alpha,{\sss II},p}(x) &= \frac {p^2}\aap
    	\sum_{e\colon \|e\|_1=1}
    	\expec_{\sss 0} \Big(\indic{e\nin \tilde\Ccal^{(0,e)}_{\sss 0}(0)}
	\indic{(0,x)\text{ is occupied}} \nnb
    	&\quad\qquad\qquad\qquad	\times\expec_{\sss 1}^{\sss 0} \left(\indicwo{E'(e,x;\tilde \Ccal^{(0,e)}_{\sss 0}(0))}
	\indic{x-\ve[\kappa]\nin \tilde \Ccal^{(x,x-\ve[\kappa])}_{\sss 1}(e)\cup\{0\}}
	\right)\Big),
	\end{align}
and $\Xi^{\ssc[1]}_{\alpha,p}(x)=\Psi^{\ssc[1],\kappa}_{\alpha,{\sss II},p}(x)=0$ for all other $x$.
In $\Psi^{\ssc[1],\kappa}_{\alpha,{\sss I},p}(x)$ we collect contributions in which $\|x-\ve[\kappa]\|\leq 1$ and at least one of the connections $\{e\conn x\}$ is realised in no more than two steps, i.e.,
\begin{align}	
	\Psi^{\ssc[1],\kappa}_{\alpha,{\sss I},p}(x)&= \indic{\|x-\ve[\kappa]\|\leq 1}
	\frac {p^2}  {\aap} 	\sum_{e\colon \|e\|_1=1} \expec_{\sss 0} \big(\indic{e\nin \tilde \Ccal^{(0,e)}_{\sss 0}(0)}\indic{x\in \tilde\Ccal^{(0,e)}_{\sss 0}(0)}
	    	\expec_{\sss 1}^{\sss 0} \big(\indicwo {E'(e,x;\tilde \Ccal^{(0,e)}_{\sss 0}(0))}\\\
    	&\qquad\times\indic{x-\ve[\kappa]\nin \tilde \Ccal^{(x,x-\ve[\kappa])}_{\sss 1}(e)\cup\{0\}} \indic{\exists\text{ path between $e$ and $x$ consisting of one or two occ. bonds}} \big)\big).\nn
    	\lbeq{Def-PsiOne-SplitOne}
	\end{align}	
We define the remainder terms by
	\begin{align}
	\Xi^{\ssc[1]}_{{\sss R},p}(x):=&\ \Xi^{\ssc[1]}_{p}(x)-\Xi^{\ssc[1]}_{\alpha,p}(x),\qquad\quad
	\Psi^{\ssc[1],\kappa}_{{\sss R, I},p}(x):= \Psi^{\ssc[1],\kappa}_{p}(x)
	-\Psi^{\ssc[1],\kappa}_{\alpha,{\sss I},p}(x),\\
	\Psi^{\ssc[1],\kappa}_{{\sss R, II},p}(x):=&\Psi^{\ssc[1],\kappa}_{p}(x)
	-\Psi^{\ssc[1],\kappa}_{\alpha,{\sss II},p}(x).
	\end{align}
It turns out that it is numerically not worthwhile to split $\Xi^{\ssc[1],\iota}_{p}(x),\Pi^{\ssc[1],\iota,\kappa}_{p}(x)$ any further. This completes the definition of the relevant splits for $N=1$.

\subsection{Verification of assumptions on NoBLE coefficients}
\label{subsec-prop-coef}

In \cite{FitHof13b}, we analyze the asymptotic properties of the NoBLE by making a number of assumptions. In this section, we verify the assumption on the NoBLE coefficients formulated in \cite[Assumptions 3.1, 3.2 and 3.4]{FitHof13b}.

\paragraph{\cite[Definition 2.5]{FitHof13b} Symmetry of the model.}
We denote by $\mathcal{P}_d$ the set of all permutations of $\{1,2,\dots, d\}$. For $\nu\in \mathcal{P}_d$, $\delta\in\{-1,1\}^d$ and $x\in\Zd$, we define $p(x;\nu,\delta)\in\Zd$ to be the vector with entries $(p(x;\nu,\delta))_j=\delta_j x_{\nu_j}$. We say that a function $f:\Zd\mapsto \Rbold$ is {\em totally rotationally symmetric} when $f(x)=f(p(x;\nu,\delta))$ for all $\nu\in \mathcal{P}_d$ and $\delta\in\{-1,1\}^d$.\\

Total rotational symmetry is natural on $\Zd$, e.g., the two-point function $\tau_p$ as well as the NBW two-point function have this symmetry. We next argue that the following assumption holds for percolation:\\[2mm]
{\bf \cite[Assumption 4.1]{FitHof13b}.} Let $\iota,\kappa\in\{\pm 1,\pm 2,\dots,\pm d\}$. The following symmetries hold for all $x\in\Zd$, $p\leq p_c$, $N\in\Nbold$ and $\iota,\kappa$:
	\begin{eqnarray*}
	\lbeq{Symmetrie-goal}
	\Xi^\ssc[N]_p(x)&=& \Xi^\ssc[N]_p(-x), \qquad  \qquad \qquad
	\Xi^{\ssc[N],\iota}_p(x)= 	\Xi^{\ssc[N],-\iota}_p (-x),\\
	\Psi^{\ssc[N],\iota}_p(x)&=& \Psi^{\ssc[N],-\iota}_p(-x), \qquad \qquad\
	\Pi^{\ssc[N],\iota,\kappa}_p(x)= \Pi^{\ssc[N],-\iota,-\kappa}_p(-x).
	\end{eqnarray*}
For all $N\geq 0,$ the coefficients
	\begin{align}
	\lbeq{TRS-sums}
	\Xi^\ssc[N]_p(x),\qquad \sum_{\iota}\Psi^{\ssc[N],\iota}_p(x),
	\qquad \sum_{\iota}\Xi^{\ssc[N],\iota}_p(x) \quad \text{and}\quad
	\sum_{\iota,\kappa}\Pi^{\ssc[N],\iota,\kappa}_p(x),
	\end{align}
as well as the remainder terms of the split
	\begin{align}
	\Xi^{\ssc[0]}_{{\sss R},p}(x),\quad \sum_{\iota}\Psi^{\ssc[0],\iota}_{{\sss R, I},p}(x),
	\quad \sum_{\iota}\Psi^{\ssc[0],\iota}_{{\sss R, II},p}(x),
	\quad \sum_{\iota}\Xi^{\ssc[0],\iota}_{{\sss R, I},p}(x),
	\quad \sum_{\iota}\Xi^{\ssc[0],\iota}_{{\sss R, II},p}(x),
	\quad \sum_{\iota,\kappa}\Pi^{\ssc[0],\iota,\kappa}_{{\sss R},p}(x),\nn
	\end{align}
are totally rotationally symmetric functions of $x\in\Zd$.
Finally, the dimensions are exchangeable, in the sense that, for all $\iota,\kappa$,
	\begin{align}
 	\lbeq{assCoefficentsDimInterchange}
	\hat \Psi^{\ssc[N],\iota}_p(0)=\ \hat \Psi^{\ssc[N],\kappa}_p(0),\quad
	\hat \Xi^{\ssc[N],\iota}_p(0)=\ \hat \Xi^{\ssc[N],\kappa}_p(0),\quad
	\sum_{\kappa'}\hat \Pi^{\ssc[N],\iota,\kappa'}_p(0)=\sum_{\iota'}\hat \Pi^{\ssc[N],\iota',\kappa}_p(0).
	\end{align}
\medskip

For $p<p_c$ and $N$ fixed, all the above functions are well defined. We now check the stated symmetry properties, and will return to the case $p=p_c$ at the end. By the definition of the NoBLE-coefficients in Section \ref{sec-compl}, it is easy to see that \refeq{Symmetrie-goal}, \refeq{assCoefficentsDimInterchange} hold. By the definition of $\Xi^\ssc[N]_p$ it is not difficult to see that $x\mapsto \Xi^\ssc[N]_p(x)$ is TRS for all $N\geq 0$. The other three NoBLE coefficients are not TRS as their definition includes constraints on one or two specific directions. For example, the coefficient $\Psi^{\ssc[N],\kappa}(x)$ includes the constraint that $x-\ve[\kappa]$ is not in the last cluster. When we sum over $\kappa$, the directional constraint is averaged out and $\sum_{\kappa}\Psi^{\ssc[N],\kappa}_p(x)$ is TRS. For the same reason, the sums over $\iota$ and $\iota,\kappa$ in \refeq{TRS-sums},
%	\begin{align}
%	\sum_{\iota}\Xi^{\ssc[N],\iota}_p(x) \qquad \text{and}\qquad
%	\sum_{\iota,\kappa}\Pi^{\ssc[N],\iota,\kappa}_p(x)
%	\end{align}
as well as the stated remainder terms, are TRS.

\cite[Assumption 4.1]{FitHof13b} states that the symmetry properties also hold for $p=p_c$. However, it is not even obvious that these objects are well defined at $p=p_c$. We verify the left-continuity in \cite[Assumption 4.4]{FitHof13b} below, from which the symmetries will follow also for $p=p_c$. Further, inspection of the proof in \cite{FitHof13b} shows that the symmetries are only used for $p<p_c$, while properties at $p=p_c$ are concluded by left continuity arguments instead.
\qed

\paragraph{\cite[Assumption 4.2]{FitHof13b} Relation between coefficients.}{\it
%{\bf Assumption 3.2 of \cite{FitHof13b}}
For all $x\in\Zd$, $p\leq p_c$, $N\in\Nbold$ and $\iota,\kappa\in\{\pm 1,\pm 2,\dots,\pm d\}$, the following bounds hold:
\begin{align}
\lbeq{XidominatespsiImproved}
\Psi^{\ssc[N],\kappa}_p(x)\leq&  \frac {p} {\aap} \Xi^\ssc[N]_p(x),
\qquad\qquad\qquad \Pi^{\ssc[N],\iota,\kappa}_p(x)\leq p \Xi^{\ssc[N],\iota}_p(x).
\end{align}}
\medskip

Comparing the definitions of $\Xi^{\ssc[N]},\Psi^{\ssc[N],\kappa}$ and $\Xi^{\ssc[N],\iota}$ and $\Pi^{\ssc[N],\iota,\kappa}$, we see that they differ by the presence of the additional indicator of the event $\{y+\ve[\kappa]\nin \tilde{\Ccal}^{(y,y+\ve[\kappa])}_{\sss N}(\tb_{\sss N-1})\}$ and a factor $p/\aap$ and $p$, respectively. We bound the indicator by 1 and obtain the relations stated in \refeq{XidominatespsiImproved}.\qed

\paragraph{\cite[Assumption 4.4]{FitHof13b} Growth at the critical point.}{ \it
%{\bf Assumption 3.2 of \cite{FitHof13b}}
The functions $p\mapsto \hat \Xi_p(k), p\mapsto \hat  \Xi^{\iota}_p(k), p\mapsto \hat \Psi^{\kappa}_p(k)$ and $p\mapsto \hat  \Pi^{\iota,\kappa}_p(k)$ are continuous for $p\in(0,p_c)$. Further, let $\Gamma_1,\Gamma_2,\Gamma_3\geq 0$ be such that $f_i(p)\leq \Gamma_i$ and that Assumption \cite[Assumption 4.3]{FitHof13b} holds. Then, the functions stated above are left-continuous at $p_c$ with a finite limit $p\nearrow p_c$ for all $x\in\Zd$. Further, for technical reasons, we assume that $p_c<1/2$.}
\medskip

The two-point functions $\tau_p,\tau^\iota_p$ and the coefficients $\Xi^{\ssc[N]}_p,\Xi^{{\ssc[N]},\iota}_p,\Psi^{{\ssc[N]},\iota}_p$ and $\Pi^{{\ssc[N]},\iota,\kappa}_p$ are defined as sums of probabilities and expectations of intertwined events. The percolation measure $\prob_p$, in which each bond is occupied/vacant independently, is a product measure.  Restricted to a finite graph, the above functions are clearly continuous. The continuity for $p<p_c$ can be obtained using a finite-volume approximation that is non-trivial. We omit the proof of this here, and instead refer the reader to \cite[Appendix A.2]{Hara08}  where such a statement is proved for the coefficients of the classical lace expansion. The extension to our setting is straightforward.

Since  $f_i(p)\leq \Gamma_i$ holds for all $p\in(p_I,p_c)$, the coefficients are uniformly bounded in $p\in(p_I,p_c)$. We obtain the left-continuity of the coefficients using a finite-volume approximation, which is a bit more elaborate than the arguments used to obtain continuity for $p<p_c$ and requires that the coefficients are uniformly bounded. We omit the proof of this and again refer the reader to \cite[Appendix A.2]{Hara08}.

%=================================================Bounds=================================================================
%==========================================================================================================================

\def\picPercolationFdefinition[#1]{
\begin{figure}[ht]
\begin{center}
\begin{tikzpicture}[line width=1pt,auto,scale=#1]

 \node[left]   at(-1,0)      {$ F_0(b_0,w_0,z_1)=$};

\begin{scope}[shift={(0,0)},rotate=0]
  \path[draw,fill=black!30!white] (0,0) to [out=40,in=180] (1,0.5) to [out=0,in=140]  (2,0) to [out=220,in=0] (1,-0.5) to [out=180,in=-40] (0,0);
  \node[left]   at(0,0)      {$0$};
  \node[left]   at(1.1,0.7)      {$w_0$};
  \fill (1,0.5) circle (2pt);
  \node[right]   at(2,0)      {$\bb_0$};
  \node[right]   at(6,1)      {$z_1$};
 \end{scope}
  \draw [-] (1,0.5) to [out=20,in=180] (6,1);
  \node[right]   at(6.5,0)      {$z_1\nin b_0$};

\begin{scope}[shift={(0,-2)},rotate=0]
  \node[left]   at(-1,0)      {$ F'(b_{i-1},t_i,z_i,b_i, w_i, z_{i+1})=$};
  \node[left]   at(0,0)      {$\tb_{i-1}$};
  \node[below] at(1,0)      {$w_i$};
  \fill (1,0) circle (2pt);
  \node[right]   at(5,0)      {$\bb_i=t_i=z_i,z_{i+1}\nin b_i$};
  \draw [-] (0,0) to (5,0);
  \node[right]   at(6,1)      {$z_{i+1}$};
  \draw [-] (1,0) to [out=30,in=180] (6,1);
%  \node[right]   at(8,0)      {$$};
 \end{scope}

 \begin{scope}[shift={(0,-4)},rotate=0]
  \node[left]   at(-1,0)      {$F''(b_{i-1},t_i,z_i,b_i, w_i, z_{i+1})=$};
  \node[left]   at(0,0)      {$\tb_{i-1}$};
  \node[below] at(1,0)      {$w_i\neq t_i$};
    \fill (1,0) circle (2pt);
  \node[below] at(4,-0.5)      {$z_i$};
    \fill (4,-0.5) circle (2pt);
  \node[below] at(3-0.1,0)      {$t_i$};
  \node[right]   at(5,0)      {$\bb_i\neq z_i,z_{i+1}\nin b_i$};
  \draw [-] (0,0) to (3,0);

  \begin{scope}[shift={(3,0)},rotate=0]
  \draw [-] (0,0) to [out=40,in=180] (1,0.5);
  \draw [-] (0,0) to [out=-40,in=180] (1,-0.5);
  \draw [-] (1,0.5) to [out=0,in=140] (2,0);
  \draw [-] (1,-0.5) to [out=0,in=220] (2,0);
 \end{scope}

  \node[right]   at(6,1)      {$z_{i+1}$};
  \draw [-] (1,0) to [out=30,in=180] (6,1);
 \end{scope}

 \begin{scope}[shift={(0,-6)},rotate=0]
  \node[left]   at(-1,0)      {$F'''(b_{i-1},t_i,z_i,b_i, w_i, z_{i+1})=$};
  \node[left]   at(0,0)      {$\tb_{i-1}$};
  \node[above] at(4,0.5)      {$w_i$};
      \fill (4,0.5) circle (2pt);
  \node[below] at(4,-0.5)      {$z_i$};
      \fill (4,-0.5) circle (2pt);
  \node[below] at(3-0.1,0)      {$t_i$};
  \node[right]   at(5,0)      {$\bb_i\neq z_i,z_{i+1}\nin b_i$};
  %\node[right]   at(5,-0.5)   {$$};
  \draw [-] (0,0) to (3,0);

  \begin{scope}[shift={(3,0)},rotate=0]
  \draw [-] (0,0) to [out=40,in=180] (1,0.5);
  \draw [-] (0,0) to [out=-40,in=180] (1,-0.5);
  \draw [-] (1,0.5) to [out=0,in=140] (2,0);
  \draw [-] (1,-0.5) to [out=0,in=220] (2,0);
 \end{scope}

  \node[right]   at(6,1)      {$z_{i+1}$};
  \draw [-] (4,0.5) to [out=30,in=180] (6,1);
%s  \node[right]   at(6.5,0)      {$z_i\neq \bb_i,z_i\neq t_i ,z_{i+1}\nin b_i$};
 \end{scope}

   \begin{scope}[shift={(0,-8)},rotate=0]
  \node[left]   at(-1,0)      {$ F_{\sss N}(b_{\sss N-1},t_{\sss N},z_{\sss N},y)=$};
  \node[left]   at(0,0)      {$\tb_{N-1}$};
  \node[below] at(4,-0.5)      {$z_N$};

  \node[below] at(3-0.1,0)      {$t_N$};
  \node[right]   at(5,0)      {$y$};
  \draw [-] (0,0) to (3,0);

  \begin{scope}[shift={(3,0)},rotate=0]
  \path[draw,fill=black!30!white] (0,0) to [out=40,in=180] (1,0.5) to [out=0,in=140]  (2,0) to [out=220,in=0] (1,-0.5) to [out=180,in=-40] (0,0);
 \end{scope}
       \fill (4,-0.5) circle (2pt);
 \end{scope}
\end{tikzpicture}
\caption{Diagrammatic representations of the events $F_0(b_0,w_0,z_1)$,
$F(b_{i-1},t_i,z_i,b_i, w_i, z_{i+1})$, $F_{\sss N}(b_{\sss N-1},t_{\sss N},z_{\sss N},y)$.
Lines indicate disjoint connections. Shaded cycles might be trivial, meaning that they might consist of a single vertex.}
\label{fig-F}
\end{center}
\end{figure}}

\def\picPercolationFIotadefinition[#1]{
\begin{figure}[ht]
\begin{center}
\begin{tikzpicture}[line width=1pt,auto,scale=#1]

  \node[left]   at(-0.75,0)      {$F^{\iota,{\rm \sss I}}_0(b_0,w_0,z_1)=$};
  \node[left]   at(0,0)      {$0$};
  \draw [-] (0,0) to [out=90,in=180] (0.75,0.7);
  \draw [-] (1.5,0) to [out=90,in=0]   (0.75,0.7);

\begin{scope}[shift={(1.5,0)},rotate=0]
  \path[draw,fill=black!30!white] (0,0) to [out=40,in=180] (1,0.5) to [out=0,in=140]  (2,0) to [out=220,in=0] (1,-0.5) to [out=180,in=-40] (0,0);
    \fill (1,0.5) circle (2pt);
  \node[left]   at(1,0.6)      {$w_0$};
  \node[left]   at(0,0)      {$e_\iota$};
  \node[right]   at(2,0)      {$\bb_0$};

\end{scope}
  \draw [-] (2.5,0.5) to [out=20,in=180] (6,1);
  \node[right]   at(6,1)      {$z_1$};
  \node[right]   at(7,0)      {$z_1\nin b_0$};

\begin{scope}[shift={(0,-2)},rotate=0]
  \node[left]   at(-0.75,0)      {$F^{\iota,{\rm \sss II}}_0(b_0,w_0,z_1)=$};
  \node[left]   at(0,0)      {$0$};
  \draw [-] (0,0) to [out=90,in=180] (1.2,0.7);
  \draw [-] (1.5,0) to [out=170,in=290]   (1.2,0.7);
  \fill (1.2,0.7) circle (2pt);

    \begin{scope}[shift={(1.5,0)},rotate=0]
  \path[draw,fill=black!30!white] (0,0) to [out=40,in=180] (1,0.5) to [out=0,in=140]  (2,0) to [out=220,in=0] (1,-0.5) to [out=180,in=-40] (0,0);
  \node[right]   at(2,0)      {$\bb_0$};
  \node[left]   at(0,0)      {$e_\iota$};
    \end{scope}
  \node[above]   at(1.2,0.75)      {$w_0$};
  \draw [-] (1.2,0.7) to [out=20,in=180] (6,1);
  \node[right]   at(6,1)      {$z_1$};
  \node[right]   at(7,0)      {$z_1\nin b_0, w_0\neq \ve[\iota]$};
 \end{scope}

\begin{scope}[shift={(0,-4)},rotate=0]
  \node[left]   at(-0.75,0)      {$F^{\iota,{\rm \sss III}}_0(b_0,w_0,z_1)=$};
    \begin{scope}[shift={(1.5,0)},rotate=0]
  \node[left]   at(0,0)      {$0=w_0$};
  \draw [-] (0,-0.2) to (0,0.2);
  \draw [-] (1.5,-0.2) to  (1.5,0.2);

%  \path[draw,fill=black!30!white] (0,0) to [out=40,in=180] (1,0.5) to [out=0,in=140]  (2,0) to [out=220,in=0] (1,-0.5) to [out=180,in=-40] (0,0);
%  \node[right]   at(2,0)      {$\bb_0$};
  \node[below]   at(1.5,-0.1)      {$e_\iota=\tb_0$};
  \draw [-] (0,0) to [out=20,in=180] (6-1.5,1);
  \node[right]   at(6-1.5,1)      {$z_1$};
  \node[right]   at(7-1.5,0)      {$z_1\nin b_0$};
     \end{scope}
 \end{scope}

\end{tikzpicture}
\caption{Diagrammatic representations of the events $F^{\iota}_0(b_0,w_0,z_1)$.
Lines indicate disjoint connections. Shaded cycles might be trivial.} \label{fig-F-iota}
\end{center}
\end{figure}}

\def\picXiDecomposeOddDiffenceOne[#1]{
}

\def\picXiStructure[#1]{
\begin{tikzpicture}[line width=1pt,auto,scale=#1]
 \fill[gray] (0,0) -- (1,1)--(1,-1);

  \fill (0,0) circle (3pt);
  \node[left]   at(0,0)      {$0$};
  \node[right]   at(10,0)      {$x$};

 \draw [-](0,0) to  (1,1);
 \draw [-](0,0) to  (1,-1);
 \draw [-](1,1) to  (1,-1);
 \draw [-](1,1) to  (3,1);
 \draw [-](1.5,-1) to  (3,-1);
 \draw [-](1+0.03,-1-0.2) to (1+0.03,-1+0.2);
 \draw [-](1.5-0.03,-1-0.2) to   (1.5-0.03,-1+0.2);
 \draw [-](2.5,1) to  (2.5,-1);
 \draw [-](2.5,1) to  (4,1);
 \draw [-](2.5,-1) to  (4,-1);
 \draw [-](4,1) to  (4,-1);
 \draw [-](4+0.03,1-0.2) to (4+0.03,1+0.2);
 \draw [-](4.5-0.03,1-0.2) to   (4.5-0.03,1+0.2);
 \draw [-](4.5,1) to  (6,1);
 \draw [-](4,-1) to  (6-0.5,-1);
 \draw [-](6,1) to  (6,-1+0.5);
 \fill[gray] (6-0.5,-1) -- (6+0.5,-1)--(6,-1+0.5);
 \draw [-](6-0.5,-1) to  (6+0.5,-1);
 \draw [-](6-0.5,-1) to  (6,-1+0.5);
  \draw [-](6+0.5,-1) to  (6,-1+0.5);
 \draw [-](6+0.5+0.03,-1-0.2) to (6+0.5+0.03,-1+0.2);
 \draw [-](6+1-0.03,-1-0.2) to   (6+1-0.03,-1+0.2);
 \draw [-](6,1) to  (7.5-0.5,1);
 \draw [-](7.5,-1) to  (7.5,1-0.5);
 \draw [-](6+1,-1) to  (7.5,-1);
 \fill[gray] (7.5-0.5,1) -- (7.5+0.5,1)--(7.5,1-0.5);
 \draw [-](7.5-0.5,1) to  (7.5+0.5,1);
 \draw [-](7.5-0.5,1) to  (7.5,1-0.5);
 \draw [-](7.5+0.5,1) to  (7.5,1-0.5);
 \draw [-](7.5+0.5+0.03,1-0.2) to (7.5+0.5+0.03,1+0.2);
 \draw [-](7.5+1-0.03,1-0.2) to   (7.5+1-0.03,1+0.2);
 \draw [-](7.5,-1) to  (9,  -1);
 \draw [-](8.5,1) to  (9,  1);
 \fill[gray] (9,1) -- (9,-1)--(10,0);
 \draw [-](9,-1) to  (9,  1);
 \draw [-](9,-1) to  (10,  0);
 \draw [-](9,1) to  (10,  0);
   \fill (10,0) circle (3pt);

  \node[below]   at(1.25,-1)      {$b_0$};
  \node[below]   at(6.75,-1)      {$b_2$};
  \node[above]   at(4.25,1)      {$b_1$};
  \node[above]   at(8.25,1)      {$b_3$};
  \node[above]   at(1,1)       {$w_0$};
  \node[above]   at(2.5,1)     {$z_1$};
  \node[below]   at(2.5,-1)    {$t_1$};
  \node[below]   at(4,-1)      {$w_1$};
  \node[below]   at(5.5,-1)      {$z_2$};
  \node[right]   at(6,-1+0.5)      {$t_2$};
  \node[above]   at(6,1)      {$w_2$};
  \node[above]   at(7,1)      {$z_3$};
  \node[right]   at(7.5,1-0.5)      {$t_3$};
  \node[below]   at(7.5,-1)      {$w_3$};
  \node[below]   at(9,-1)      {$z_4$};
  \node[above]   at(9,1)      {$t_4$};

  \node[below]   at(0.5,-1.5)      {$F_0$};
  \node[below]   at(2.5,-1.5)      {$F'''$};
  \node[below]   at(6,-1.5)      {$F'\cup F''$};
  \node[below]   at(8,-1.5)      {$F'\cup F''$};
  \node[below]   at(9.5,-1.5)      {$F_{\sss N}$};
\end{tikzpicture}}

\def\picXiDecomposeStructure[#1]{
\begin{tikzpicture}[line width=1pt,auto,scale=#1]
{\large
 \node[left]   at(-0.5,0)      {$\sum_{x} \Xi^{\ssc[N]}_p(x)=\sum_{\vec l}$};
}
 \node[right]   at(9,1.5)      {$N-1$};
 \node[left]   at(0,0)      {$0$};
 \fill[gray] (0,0) -- (1,1)--(1,-1);
 \draw [-](0,0) to  (1,1);
 \draw [-](0,0) to  (1,-1);
 \draw [-](1,1) to  (1,-1);
 \fill (0,0) circle (3pt);

 \draw [-](0.5+1.5,1.4) to  (0.5+1.5,-1.4);
 \draw [-](0.5+1.5,1.4) to  (1+1.5,1.4);
  \draw [-](0.5+1.5,-1.4) to  (1+1.5,-1.4);

 \draw [-](1+1.5,1) to  (3+1.5,1);
 \draw [-](1.5+1.5,-1) to  (3+1.5,-1);
 \draw [-](1+0.03+1.5,-1-0.2) to (1+0.03+1.5,-1+0.2);
 \draw [-](1.5-0.03+1.5,-1-0.2) to   (1.5-0.03+1.5,-1+0.2);
 \draw [-](2.5+1.5,1) to  (2.5+1.5,-1);
 \draw [-](2.5+1.5,1) to  (4+1.5,1);
 \draw [-](2.5+1.5,-1) to  (4+1.5,-1);
 \draw [-](4+1.5,1) to  (4+1.5,-1);

 \draw [-](6,-0.2) to  (6,0.2);
 \draw [-](5.8,0) to  (6.2,0);

\draw [-](6.5+0.03,-1-0.2) to (6.5+0.03,-1+0.2);
 \draw [-](6.5+0.5-0.03,-1-0.2) to   (6.5+0.5-0.03,-1+0.2);
 \draw [-](6.5,1) to  (8-0.5,1);
 \draw [-](8,-1) to  (8,1-0.5);
 \draw [-](7,-1) to  (8,-1);
 \fill[gray] (8-0.5,1) -- (8+0.5,1)--(8,1-0.5);
 \draw [-](8-0.5,1) to  (8+0.5,1);
 \draw [-](8-0.5,1) to  (8,1-0.5);
 \draw [-](8+0.5,1) to  (8,1-0.5);

 \draw [-](9,1.4) to  (9-0.5,1.4);
 \draw [-](9,-1.4) to  (9-0.5,-1.4);
 \draw [-](9,1.4) to  (9,-1.4);

 \draw [-](10,1) to  (11.5,1);
 \draw [-](10.5,-1) to  (11.5,-1);
 \draw [-](10+0.03,-1-0.2) to (10+0.03,-1+0.2);
 \draw [-](11-0.5-0.03,-1-0.2) to   (11-0.5-0.03,-1+0.2);

 \fill[gray] (11.5,1) -- (11.5,-1)--(12.5,0);
 \draw [-](11.5,-1) to  (11.5,  1);
 \draw [-](11.5,-1) to  (12.5,  0);
 \draw [-](11.5,1) to  (12.5,  0);

  \node[right]   at(1,0)      {$l_0$};
   \draw [dotted](10,1) to  (10,-1);
  \node[left]   at(10.7,0)      {$l_{N-1}$};

  \draw [dotted](2.5,1) to  (2.5,-1);
  \node[right]   at(2.5,0.5)      {$l_{r-1}$};

  \node[right]   at(5.5,-0.5)      {$l_{r}$};

  \draw [dotted]   (6.5,1) to  (6.5,-1);
  \node[right]   at(6.5,0.5)      {$l_{r-1}$};

  \draw [dotted]   (8.5,1) to  (8,-1);
  \node[right]   at(8.2,-0.5)      {$l_{r}$};
\end{tikzpicture}}

\def\picXiNOne[#1]{
\begin{tikzpicture}[line width=1pt,auto,scale=#1]
 \fill[gray] (0,0) -- (1,1)--(1,-1);

  \fill (0,0) circle (3pt);
  \node[left]   at(0,0)      {$0$};
  \node[right]   at(3.5,0)      {$x$};

 \draw [-](0,0) to  (1,1);
 \draw [-](0,0) to  (1,-1);
 \draw [-](1,1) to  (1,-1);
 \draw [-](1,1) to  (2.5,1);
 \draw [dashed](1.5,-1) to  (2.5,-1);
 \draw [-](1+0.03,-1-0.2) to (1+0.03,-1+0.2);
 \draw [-](1.5-0.03,-1-0.2) to   (1.5-0.03,-1+0.2);

 \fill[gray] (3.5,0) -- (2.5,1)--(2.5,-1);
 \draw [dashed](3.5,0) to  (2.5,-1);
 \draw [dashed](3.5,0) to  (2.5,1);
 \draw [dashed](2.5,-1) to  (2.5,1);
 \fill (3.5,0) circle (3pt);

  \node[below]   at(0.8,-1)      {$b_{\sss 0}=(u,u+\ve[\iota])$};
  \node[above]   at(1,1)       {$w$};
  \node[above]   at(2.5,1)     {$z$};
  \node[below]   at(2.5,-1)    {$t$};
\end{tikzpicture}}

\def\picXiNTwoMain[#1]{
\begin{tikzpicture}[line width=1pt,auto,scale=#1]

  \fill (0,0) circle (3pt);
  \fill (4,0) circle (3pt);

  \node[left]   at(0,0)      {$0$};
  \node[right]   at(4,0)      {$x$};

 \draw [-](0,0) to  (2,1);
 \draw [dashed](0+0.6,0-0.25) to  (2,-1);
 \draw [dashed](2,1) to  (2,-1);
 \draw [dotted](2+0.55,1-0.25) to  (4,0);
 \draw [dashed](2,-1) to  (4,0);

\begin{scope}[shift={(0.05,0-0.025)},rotate=-25]
 \draw [-](0.03,-0.2) to (0.03,+0.2);
 \draw [-](0.6-0.03,-0.2) to   (0.6-0.03,+0.2);
 \end{scope}

\begin{scope}[shift={(2,1)},rotate=-25]
 \draw [-](0.03,-0.2) to (0.03,+0.2);
 \draw [-](0.6-0.03,-0.2) to   (0.6-0.03,+0.2);
 \end{scope}

  \node[below]   at(-0.2,-.5)      {$b_0=(0,\ve[\iota])$};
  \node[above]   at(2.2,1)      {$b_1=(y,y+\ve[\kappa])$};
  \node[below]   at(2,-1)    {$t$};
\end{tikzpicture}}

\def\picTable[#1]{
\begin{tikzpicture}[line width=1pt,auto,scale=#1]
 \fill[gray] (0,0) -- (2,0)--(1,1.5);
  \fill (0,0) circle (3pt);
  \fill (2,0) circle (3pt);
  \fill (1,1.5) circle (3pt);
  \node[left]   at(0,0)      {$0$};
  \node[above]   at(1,1.5)      {$y$};
  \node[right]   at(1.5,1)      {$b$};
  \node[right]   at(2,0)      {$x$};
  \draw [-] (0,0) to  (1,1.5);
  \draw [-] (0,0) to  (2,0);
  \draw [-] (2,0) to  (1,1.5);

\end{tikzpicture}}

\def\picTableCaseOne[#1]{
\begin{tikzpicture}[line width=1pt,auto,scale=#1]

  \fill (0,0) circle (3pt);
  \node[right]   at(0,0)      {$x=y=0$};
\end{tikzpicture}}

\def\picTableCaseThree[#1]{
\begin{tikzpicture}[line width=1pt,auto,scale=#1]

  \fill (0,0) circle (3pt);
  \fill (2,0) circle (3pt);
  \node[left]   at(0,0)      {$0$};
  \node[right]   at(2,0)      {$x=y$};
  \draw [-] (0,0) to  [out=90,in=90]   (2,0);
  \draw [-] (0,0) to  [out=270,in=270]   (2,0);
\end{tikzpicture}}

\def\picTableCaseFour[#1]{
\begin{tikzpicture}[line width=1pt,auto,scale=#1]

  \fill (0,0) circle (3pt);
  \fill (2,0) circle (3pt);
  \node[left]   at(0,0)      {$0=y$};
  \node[right]   at(2,0)      {$x$};
  \draw [-]   (0,0) to  (2,0);
  \draw [-] (0,0) to  [out=270,in=270]   (2,0);
  \node   at(1.8,-0.75)      {$\geq 3$};
  \node   at(1,0.4)      {$=1$};
\end{tikzpicture}}

\def\picTableCaseFive[#1]{
\begin{tikzpicture}[line width=1pt,auto,scale=#1]

  \fill (0,0) circle (3pt);
  \fill (2,0) circle (3pt);
  \node[left]   at(0,0)      {$0$};
    \node[right]   at(2,1.5)      {$y$};
  \node[right]   at(2,0)      {$x$};
  \draw [-]   (0,0) to  (2,0);
  \draw [-]   (2,0) to  (2,1.5);
  \draw [-] (0,0) to  [out=90,in=180]   (2,1.5);
  \node   at(0,1.5)      {$\geq 2$};
  \node   at(1,0.4)      {$=1$};
  \node   at(2.7,0.8)      {$=1$};
\end{tikzpicture}}

\def\picTableCaseSix[#1]{
\begin{tikzpicture}[line width=1pt,auto,scale=#1]

  \fill (0,0) circle (3pt);
  \fill (2,0) circle (3pt);
  \node[left]   at(0,0)      {$0$};
      \node[right]   at(2,1.2)      {$y$};
  \node[right]   at(2,0)      {$x$};
  \draw [-]   (0,0) to  (2,0);
  \draw [-]   (2,0) to  (2,1.2);
  \draw [-] (0,0) to  [out=90,in=180]   (2,1);
  \node   at(0,1.1)      {$\geq 1$};
  \node   at(1,0.4)      {$\geq 1$};
  \node   at(2.7,0.6)      {$=1$};
\end{tikzpicture}}

\def\picTableCaseEight[#1]{
\begin{tikzpicture}[line width=1pt,auto,scale=#1]

  \fill (0,0) circle (3pt);
  \fill (2,0) circle (3pt);
  \node[left]   at(0,0)      {$0=y$};
  \node[right]   at(2,0)      {$x\neq 0$};
  \draw [-] (0,0) to  [out=90,in=90]   (2,0);
  \draw [-] (0,0) to  [out=270,in=270]   (2,0);
\end{tikzpicture}}

\def\picTableCaseNine[#1]{
\begin{tikzpicture}[line width=1pt,auto,scale=#1]

  \fill (0,0) circle (3pt);
  \fill (2,0) circle (3pt);
  \fill (2,1.2) circle (3pt);
  \node[left]   at(0,0)      {$0$};
  \node[right]   at(2,1.2)      {$y$};
  \node[right]   at(2,0)      {$x$};
  \draw [-] (0,0) to  [out=90,in=180]   (2,1.2);
  \draw [-] (0,0) to  (2,0);
  \draw [-] (2,0) to  (2,1.2);
  \node[below]   at(1,0)      {$\geq 1$};
  \node[left]    at(0,1.1)      {$\geq 1$};
  \node[right]   at(2,0.6)      {$\geq 2$};
\end{tikzpicture}}

\def\picAIota[#1]{
\begin{tikzpicture}[line width=1pt,auto,scale=#1]
  \fill (2,0) circle (3pt);
  \fill (0,2) circle (3pt);
  \fill (2,2) circle (3pt);
  \node[left]   at(0,0)      {$0$};
  \node[below]   at(0.5,0)      {$e_\iota$};
  \node[left]   at(0,2)      {$v$};
  \node[left]   at(0,1)      {$a$};
  \node[right]   at(2,2)    {$y$};
  \node[right]   at(2,1)    {$b$};
  \node[right]   at(2,0)      {$x$};
  \draw [-] (0.5,0) to  (2,0);
  \draw [-] (2,0) to  (2,2);
  \draw [-] (0,2) to  (2,2);
 \draw [-](0,-0.2) to (0,+0.2);
 \draw [-](0.5-0.03,-0.2) to   (0.5-0.03,+0.2);
\end{tikzpicture}}

%===================================
\def\picACaseTwo[#1]{
\begin{tikzpicture}[line width=1pt,auto,scale=#1]

 \begin{scope}[shift={(4,0.3)}]
  \node[left]   at(0,0)      {$0$};
  \node[right]   at(2,0)      {$x$};
  \fill (0,0) circle (3pt);
  \fill (2,0) circle (3pt);
  \node [below]  at(1,-0.5)      {$\geq 1$};
  \node [above]  at(1,0.5)      {$\geq 1$};

 \path[draw] (0,0) to [out=40,in=180] (1,0.5) to [out=0,in=140]  (2,0) to [out=220,in=0] (1,-0.5) to [out=180,in=-40] (0,0);
 \end{scope}
 \end{tikzpicture}}

\def\picACaseThree[#1]{
\begin{tikzpicture}[line width=1pt,auto,scale=#1]

 \begin{scope}[shift={(0,0)}]
  \fill (0,0) circle (3pt);
  \fill (2,0) circle (3pt);
  \fill (2,1.5) circle (3pt);
  \node[left]   at(0,0)      {$0$};
  \node[right]   at(2,0)      {$x$};
  \node[right]   at(2,1.5)      {$y$};
  \draw [-] (0,0) to  [out=90,in=180]   (2,1.5);
  \draw [-] (2,0) to  (2,1.5);
  \draw [-] (0,0) to  (2,0);
  \node [below]  at(1,0)      {$\geq 1$};
  \node [right]  at(2,0.7)      {$=1$};
  \node  at(0.3,1.5)      {$\geq 1$};
 \end{scope}

\end{tikzpicture}}

\def\picACaseFour[#1]{
\begin{tikzpicture}[line width=1pt,auto,scale=#1]
 \begin{scope}[shift={(0,0)}]
  \fill (0,0) circle (3pt);
  \fill (2,0) circle (3pt);
  \fill (2,1.5) circle (3pt);
  \node[left]   at(0,0)      {$0$};
  \node[right]   at(2,0)      {$x$};
  \node[right]   at(2,1.5)      {$y$};
  \draw [-] (0,0) to  [out=90,in=180]   (2,1.5);
  \draw [-] (2,0) to  (2,1.5);
  \draw [-] (0,0) to  (2,0);
  \node [below]  at(1,0)      {$\geq 1$};
  \node [right]  at(2,0.7)     {$\geq 2$};
  \node  at(0.3,1.5)          {$\geq 1$};
 \end{scope}

\end{tikzpicture}}

\def\picACaseFive[#1]{
\begin{tikzpicture}[line width=1pt,auto,scale=#1]
 \begin{scope}[shift={(0,0)}]
  \fill (0,0) circle (3pt);
  \fill (2,0) circle (3pt);
  \fill (0,1.5) circle (3pt);
  \node[left]   at(0,0)      {$0$};
  \node[right]   at(1.6,-0.4)      {$x=y$};
  \node[above]   at(0,1.5)      {$v$};
  \draw [-] (0,1.5) to  [out=0,in=90]   (2,0);
  \draw [-] (0,0) to  (2,0);
  \node [below]  at(0.9,0)      {$\geq 1$};
  \node [left]  at(0,1)      {$=1$};
  \node [right]  at(1.5,1.4)      {$\geq 1$};
 \end{scope}
\end{tikzpicture}}

\def\picACaseSix[#1]{
\begin{tikzpicture}[line width=1pt,auto,scale=#1]
 \begin{scope}[shift={(0,0)}]
  \node[left]   at(0,0)      {$0$};
  \node[right]   at(2,0)      {$x$};
  \node[left]   at(0,2)      {$v$};
  \node[right]   at(2,2)      {$y$};

  \fill (0,0) circle (3pt);
  \fill (2,0) circle (3pt);
  \fill (0,2) circle (3pt);
  \fill (2,2) circle (3pt);
  \node [below]  at(1,0)      {$\geq 1$};
  \node [left]  at(0,1)      {$=1$};
  \node [right]  at(2.5,1)      {$=1$};
  \node [above]  at(1,2)      {$\geq 0$};
 \path[draw] (0,0) to (2,0) to (2,2) to (0,2);
 \end{scope}

\end{tikzpicture}}

\def\picACaseSeven[#1]{
\begin{tikzpicture}[line width=1pt,auto,scale=#1]
 \begin{scope}[shift={(0,0)}]
  \node[left]   at(0,0)      {$0$};
  \node[right]   at(2,0)      {$x$};
  \node[left]   at(0,2)      {$v$};
  \node[right]   at(2,2)      {$y$};

  \fill (0,0) circle (3pt);
  \fill (2,0) circle (3pt);
  \fill (0,2) circle (3pt);
  \fill (2,2) circle (3pt);
  \node [below]  at(1,0)      {$\geq 1$};
  \node [left]  at(0,1)      {$=1$};
  \node [right]  at(2,1)      {$\geq 2$};
  \node [above]  at(1,2)      {$\geq 0$};
 \path[draw] (0,0) to (2,0) to (2,2) to (0,2);
 \end{scope}
\end{tikzpicture}}

\def\picACaseEight[#1]{
\begin{tikzpicture}[line width=1pt,auto,scale=#1]
 \begin{scope}[shift={(0,0)}]
  \node[left]   at(0,0)      {$0$};
  \node[right]   at(2,0)      {$x=y$};
  \node[left]   at(0,2)      {$v$};
  \fill (0,0) circle (3pt);
  \fill (2,0) circle (3pt);
  \fill (0,2) circle (3pt);
  \node [below]  at(1,0)      {$\geq 1$};
  \node [right]  at(1.4,1.5)      {$\geq 0$};
  \node [left]  at(0,1)      {$\geq 2$};
  \path[draw] (0,0) to (2,0) to [out=90,in=0] (0,2);
 \end{scope}
\end{tikzpicture}}

\def\picACaseNine[#1]{
\begin{tikzpicture}[line width=1pt,auto,scale=#1]
 \begin{scope}[shift={(0,0)}]
  \node[left]   at(0,0)      {$0$};
  \node[right]   at(2,0)      {$x$};
  \node[left]   at(0,2)      {$v$};
  \node[right]   at(2,2)      {$y$};

  \fill (0,0) circle (3pt);
  \fill (2,0) circle (3pt);
  \fill (0,2) circle (3pt);
  \fill (2,2) circle (3pt);
  \node [below]  at(1,0)      {$\geq 1$};
  \node [left]  at(0,1)      {$\geq 2$};
  \node [right]  at(2,1)      {$=1$};
  \node [above]  at(1,2)      {$\geq 0$};
 \path[draw] (0,0) to (2,0) to (2,2) to (0,2);
 \end{scope}
\end{tikzpicture}}

\def\picACaseTen[#1]{
\begin{tikzpicture}[line width=1pt,auto,scale=#1]
 \begin{scope}[shift={(0,0)}]
  \node[left]   at(0,0)      {$0$};
  \node[right]   at(2,0)      {$x$};
  \node[left]   at(0,2)      {$v$};
  \node[right]   at(2,2)      {$y$};

  \fill (0,0) circle (3pt);
  \fill (2,0) circle (3pt);
  \fill (0,2) circle (3pt);
  \fill (2,2) circle (3pt);
  \node [below]  at(1,0)      {$\geq 1$};
  \node [left]  at(0,1)      {$\geq 2$};
  \node [right]  at(2,1)      {$\geq 2$};
  \node [above]  at(1,2)      {$\geq 0$};
 \path[draw] (0,0) to (2,0) to (2,2) to (0,2);
 \end{scope}
\end{tikzpicture}}
%=================================

\def\picAIotaCaseTwo[#1]{
\begin{tikzpicture}[line width=1pt,auto,scale=#1]
 \begin{scope}[shift={(0,0)}]
  \fill (0,0) circle (3pt);
  \fill (2,0) circle (3pt);
  \fill (2,1.5) circle (3pt);
  \node[left]   at(0,0)      {$0$};
  \node[right]   at(2,0)      {$e_\iota$};
  \node[right]   at(2,1.5)      {$x$};
  \draw [-] (0,0) to  [out=90,in=180]   (2,1.5);
  \draw [-] (2,0) to  (2,1.5);
  \draw [-] (0,0) to  (2,0);
  \node [below]  at(1,0)      {$=1$};
  \node [right]  at(2,0.7)      {$\geq 1$};
  \node [left]  at(0.5,1)      {$\geq 1$};
 \end{scope}

\end{tikzpicture}}

\def\picAIotaCaseThree[#1]{
\begin{tikzpicture}[line width=1pt,auto,scale=#1]
 \begin{scope}[shift={(-5,0)}]
  \fill (5,0) circle (3pt);
  \fill (7,0) circle (3pt);
  \fill (5,1.5) circle (3pt);
  \fill (7,1.5) circle (3pt);
  \node[left]   at(5,0)      {$0$};
  \node[left]   at(5,1.5)      {$y$};
  \node[right]   at(7,0)      {$e_\iota$};
  \node[right]   at(7,1.5)      {$x$};
  \draw [-] (7,0) to  (7,1.5);
  \draw [-] (5,0) to  (5,1.5);
  \draw [-] (5,0) to  (7,0);
   \draw [-] (5,1.5) to  (7,1.5);
  \node [below]  at(6,0)      {$=1$};
  \node [right]  at(7,0.7)      {$\geq 0$};
  \node [left]  at(5,0.7)      {$\geq 1$};
  \node [above]  at(6,1.5)      {$=1$};
\end{scope}
\end{tikzpicture}}

\def\picAIotaCaseFour[#1]{
\begin{tikzpicture}[line width=1pt,auto,scale=#1]
  \fill (5,0) circle (3pt);
  \fill (7,0) circle (3pt);
  \fill (5,1.5) circle (3pt);
  \fill (7,1.5) circle (3pt);
  \node[left]   at(5,0)      {$0$};
  \node[left]   at(5,1.5)      {$y$};
  \node[right]   at(7,0)      {$e_\iota$};
  \node[right]   at(7,1.5)      {$x$};
  \draw [-] (7,0) to  (7,1.5);
  \draw [-] (5,0) to  (5,1.5);
  \draw [-] (5,0) to  (7,0);
   \draw [-] (5,1.5) to  (7,1.5);
  \node [below]  at(6,0)      {$=1$};
  \node [right]  at(7,0.7)      {$\geq 0$};
  \node [left]  at(5,0.7)      {$\geq 1$};
  \node [above]  at(6,1.5)      {$\geq 2$};

\end{tikzpicture}}

\def\picAIotaCaseFive[#1]{
\begin{tikzpicture}[line width=1pt,auto,scale=#1]

 \begin{scope}[shift={(-5,0)}]
  \fill (5,0) circle (3pt);
  \fill (7,0) circle (3pt);
  \fill (5,1.5) circle (3pt);
  \fill (7,1.5) circle (3pt);
  \node[left]   at(5,0)      {$0$};
  \node[left]   at(5,1.5)      {$v$};
  \node[right]   at(7,0)      {$e_\iota$};
  \node[above]   at(7.5,1.5)      {$x=y$};
  \draw [-] (7,0) to  (7,1.5);
  \draw [-] (5,0) to  (7,0);
   \draw [-] (5,1.5) to  (7,1.5);
  \node [below]  at(6,0)      {$=1$};
  \node [right]  at(7,0.7)      {$\geq 1$};
  \node [left]  at(4.8,0.7)      {$=1$};
  \node [above]  at(6,1.5)      {$\geq 0$};
 \end{scope}
\end{tikzpicture}}

\def\picAIotaCaseSix[#1]{
\begin{tikzpicture}[line width=1pt,auto,scale=#1]
  \fill (5,0) circle (3pt);
  \fill (7,0) circle (3pt);
  \fill (5,1.5) circle (3pt);
  \fill (7,1.5) circle (3pt);
  \node[left]   at(5,0)      {$0$};
  \node[left]   at(5,1.5)      {$v$};
  \node[right]   at(7,0)      {$e_\iota$};
  \node[right]   at(7,1.5)      {$x=y$};
  \draw [-] (7,0) to  (7,1.5);
  \draw [-] (5,0) to  (7,0);
   \draw [-] (5,1.5) to  (7,1.5);
  \node [below]  at(6,0)      {$=1$};
  \node [right]  at(7,0.7)     {$\geq 1$};
  \node [left]  at(5,0.7)      {$\geq 2$};
  \node [above]  at(6,1.5)     {$\geq 0$};
\end{tikzpicture}}

\def\picAIotaCaseSeven[#1]{
\begin{tikzpicture}[line width=1pt,auto,scale=#1]
  \fill (0,0) circle (3pt);
  \fill (2,0) circle (3pt);
  \fill (-0.5,1.5) circle (3pt);
  \fill (2.5,1.5) circle (3pt);
  \fill (1,2.7) circle (3pt);
  \node[left]   at(0,0)      {$0$};
  \node[right]   at(2,0)      {$e_\iota$};
  \node[left]   at(-0.5,1.5)      {$v$};
  \node[right]   at(1,2.7)      {$y$};
  \node[right]   at(2.5,1.5)      {$x$};
  \draw [-] (0,0) to  (2,0);
  \draw [-] (-.5,1.5) to  (1,2.7);
  \draw [-] (2.5,1.5) to  (1,2.7);
  \draw [-] (2.5,1.5) to  (2,0);
  \node [below]  at(1,0)      {$=1$};
  \node [right]  at(2.2,0.7)     {$\geq 0$};
  \node [left]  at(-0.3,0.7)      {$=1$};
  \node [right]  at(1+0.5,2.4)     {$= 1$};
  \node [left]  at(1-0.4,2.4)     {$\geq 0$};
\end{tikzpicture}}

\def\picAIotaCaseEight[#1]{
\begin{tikzpicture}[line width=1pt,auto,scale=#1]
  \fill (0,0) circle (3pt);
  \fill (2,0) circle (3pt);
  \fill (-0.5,1.5) circle (3pt);
  \fill (2.5,1.5) circle (3pt);
  \fill (1,2.7) circle (3pt);
  \node[left]   at(0,0)      {$0$};
  \node[right]   at(2,0)      {$e_\iota$};
  \node[left]   at(-0.5,1.5)      {$v$};
  \node[right]   at(1,2.7)      {$y$};
  \node[right]   at(2.5,1.5)      {$x$};
  \draw [-] (0,0) to  (2,0);
  \draw [-] (-.5,1.5) to  (1,2.7);
  \draw [-] (2.5,1.5) to  (1,2.7);
  \draw [-] (2.5,1.5) to  (2,0);
  \node [below]  at(1,0)      {$=1$};
  \node [right]  at(2.2,0.7)     {$\geq 0$};
  \node [left]  at(-0.3,0.7)       {$=1$};
  \node [right]  at(1+0.6,2.4)     {$\geq 2$};
  \node [left]  at(1-0.4,2.4)     {$\geq 0$};
\end{tikzpicture}}

\def\picAIotaCaseNine[#1]{
\begin{tikzpicture}[line width=1pt,auto,scale=#1]
  \fill (0,0) circle (3pt);
  \fill (2,0) circle (3pt);
  \fill (-0.5,1.5) circle (3pt);
  \fill (2.5,1.5) circle (3pt);
  \fill (1,2.7) circle (3pt);
  \node[left]   at(0,0)      {$0$};
  \node[right]   at(2,0)      {$e_\iota$};
  \node[left]   at(-0.5,1.5)      {$v$};
  \node[right]   at(1,2.7)      {$y$};
  \node[right]   at(2.5,1.5)      {$x$};
  \draw [-] (0,0) to  (2,0);
  \draw [-] (-.5,1.5) to  (1,2.7);
  \draw [-] (2.5,1.5) to  (1,2.7);
  \draw [-] (2.5,1.5) to  (2,0);
  \node [below]  at(1,0)      {$=1$};
  \node [right]  at(2.2,0.7)     {$\geq 0$};
  \node [left]  at(-0.3,0.7)       {$\geq 2$};
  \node [right]  at(1+0.5,2.4)     {$\geq 2$};
  \node [left]  at(1-0.4,2.4)     {$\geq 0$};
\end{tikzpicture}}

\def\picAIotaCaseTen[#1]{
\begin{tikzpicture}[line width=1pt,auto,scale=#1]
  \fill (0,0) circle (3pt);
  \fill (2,0) circle (3pt);
  \fill (-0.5,1.5) circle (3pt);
  \fill (2.5,1.5) circle (3pt);
  \fill (1,2.7) circle (3pt);
  \node[left]   at(0,0)      {$0$};
  \node[right]   at(2,0)      {$e_\iota$};
  \node[left]   at(-0.5,1.5)      {$v$};
  \node[right]   at(1,2.7)      {$y$};
  \node[right]   at(2.5,1.5)      {$x$};
  \draw [-] (0,0) to  (2,0);
  \draw [-] (-.5,1.5) to  (1,2.7);
  \draw [-] (2.5,1.5) to  (1,2.7);
  \draw [-] (2.5,1.5) to  (2,0);
  \node [below]  at(1,0)      {$=1$};
  \node [right]  at(2.2,0.7)     {$\geq 0$};
  \node [left]  at(-0.3,0.7)       {$\geq 2$};
  \node [right]  at(1+0.5,2.4)     {$=1$};
  \node [left]  at(1-0.4,2.4)     {$\geq 0$};
\end{tikzpicture}}

\def\picA[#1]{
\begin{tikzpicture}[line width=1pt,auto,scale=#1]
  \fill (2,0) circle (3pt);
  \fill (0,2) circle (3pt);
  \fill (2,2) circle (3pt);
  \fill (0,0) circle (3pt);
  \node[left]   at(0,0)      {$0$};
  \node[left]   at(0,2)      {$v$};
  \node[left]   at(0,1)      {$a$};
  \node[right]   at(2,2)    {$y$};
  \node[right]   at(2,1)    {$b$};
  \node[right]   at(2,0)      {$x$};
  \draw [-] (0,0) to  (2,0);
  \draw [-] (2,0) to  (2,2);
  \draw [-] (0,2) to  (2,2);
\end{tikzpicture}}

\def\picABar[#1]{
\begin{tikzpicture}[line width=1pt,auto,scale=#1]
  \fill (2,0) circle (3pt);
  \fill (0,1.5) circle (3pt);
  \fill (2,1.5) circle (3pt);
  \node[left]   at(0,0)      {$0$};
  \node[below]   at(0.5,0)      {$e_\iota$};
  \node[left]   at(0,1.5)      {$v$};
  \node[left]   at(0,0.75)      {$a$};
  \node[right]   at(2,1.5)    {$y$};
  \node[right]   at(2,0.75)    {$b$};
  \node[right]   at(2,0)      {$x$};
  \draw [-] (0.5,0) to  (2,0);
  \draw [-] (0,1.5) to  (2,1.5);
 \draw [-](0,-0.2) to (0,+0.2);
 \draw [-](0.5-0.03,-0.2) to   (0.5-0.03,+0.2);
\end{tikzpicture}}

\def\picBOne[#1]{
\begin{tikzpicture}[line width=1pt,auto,scale=#1]

 \draw [-](1+1.5,1) to  (3+1.5,1);
 \draw [-](1.5+1.5,-1) to  (3+1.5,-1);
 \draw [-](1+0.03+1.5,-1-0.2) to (1+0.03+1.5,-1+0.2);
 \draw [-](1.5-0.03+1.5,-1-0.2) to   (1.5-0.03+1.5,-1+0.2);
 \draw [-](2.5+1.5,1) to  (2.5+1.5,-1);
 \draw [-](2.5+1.5,1) to  (4+1.5,1);
 \draw [-](2.5+1.5,-1) to  (4+1.5,-1);
 \draw [-](4+1.5,1) to  (4+1.5,-1);

 \fill (2.5,-1) circle (2pt);
 \fill (2.5,1) circle (2pt);
 \fill (4,-1) circle (2pt);
 \fill (4,1) circle (2pt);
 \fill (5.5,1) circle (2pt);
 \fill (5.5,-1) circle (2pt);

  \node[left]   at(2.5,-1)    {$0$};
  \node[left]   at(2.5,1)    {$v$};
  \node[below]   at(3,-1.1)    {$e_\iota$};
  \node[above]   at(4,1)    {$w$};
  \node[below]   at(4,-1)    {$u$};
  \node[right]   at(5.5,1)    {$y$};
  \node[right]   at(5.5,-1)    {$x$};
  \node[right]   at(-2,0)    {$B^{1,\iota}(0,v,x,y):=\sum_{u,w}$};

\end{tikzpicture}}

\def\picBTwo[#1]{
\begin{tikzpicture}[line width=1pt,auto,scale=#1]

 \draw [-](6.5+0.03,-1-0.2) to (6.5+0.03,-1+0.2);
 \draw [-](6.5+0.5-0.03,-1-0.2) to   (6.5+0.5-0.03,-1+0.2);
 \draw [-](6.5,1) to  (8-0.5,1);
 \draw [-](8,-1) to  (8,1-0.5);
 \draw [-](7,-1) to  (8,-1);
 \draw [-](8-0.5,1) to  (8+0.5,1);
 \draw [-](8-0.5,1) to  (8,1-0.5);
 \draw [-](8+0.5,1) to  (8,1-0.5);

  \fill (6.5,-1) circle (2pt);
  \fill (8,-1) circle (2pt);
  \fill (6.5,1) circle (2pt);
   \fill[gray] (8-0.5,1) -- (8+0.5,1)--(8,1-0.5);
  \fill (8,0.5) circle (2pt);
  \fill (7.5,1) circle (2pt);
  \fill (8.5,1) circle (2pt);
  \node[left]   at(6.5,-1)    {$0$};
  \node[left]   at(6.5,1)    {$v$};
  \node[below]   at(7,-1.1)    {$e_\iota$};
  \node[above]   at(7.5,1)    {$w$};
  \node[right]   at(8,0.5)    {$u$};
  \node[right]   at(8.5,1)    {$y$};
   \node[right]   at(8,-1)    {$x$};

  \node[right]   at(3,0)    {$B^{2,\iota}(0,v,x,y):=\sum_{u,w}$};
\end{tikzpicture}}

\def\picBbarOne[#1]{
\begin{tikzpicture}[line width=1pt,auto,scale=#1]
 \draw [-](1+1.5,1) to  (3+1,1);
 \draw [-](1+1.5,-1) to  (3+1.5,-1);
  \draw [-](1+1.5,-1) to  (1+1.5,1);
 \draw [-](2.5+0.03+1.5,1-0.2) to (2.5+0.03+1.5,1+0.2);
 \draw [-](3-0.03+1.5,1-0.2) to   (3-0.03+1.5,1+0.2);
 \draw [-](2.5+1.5,1) to  (2.5+1.5,-1);
 \draw [-](3+1.5,1) to  (4+1.5,1);
 \draw [-](2.5+1.5,-1) to  (4+1.5,-1);

 \fill (2.5,-1) circle (2pt);
 \fill (2.5,1) circle (2pt);
 \fill (4,-1) circle (2pt);
 \fill (4,1) circle (2pt);
 \fill (5.5,1) circle (2pt);
 \fill (5.5,-1) circle (2pt);

  \node[left]   at(2.5,-1)    {$y$};
  \node[left]   at(2.5,1)    {$x$};
  \node[above]   at(3+1.8,1.1)    {$w+e_\iota$};
  \node[above]   at(3.8,1.1)    {$w$};
  \node[below]   at(4,-1)    {$u$};
  \node[right]   at(5.5,1)    {$o$};
  \node[right]   at(5.5,-1)    {$v$};
  \node[right]   at(-2,0)    {$\bar B^{1,\iota}(0,v,x,y):=\sum_{u,w}$};
\end{tikzpicture}}

\def\picBbarTwo[#1]{
\begin{tikzpicture}[line width=1pt,auto,scale=#1]

 \draw [-](8.5+0.03,1-0.2) to (8.5+0.03,1+0.2);
 \draw [-](9-0.03,1-0.2) to   (9-0.03,1+0.2);
 \draw [-](9,1) to  (10,1);
 \draw [-](8,-1) to  (8,1-0.5);
 \draw [-](10,-1) to  (8,-1);
 \draw [-](8-0.5,1) to  (8+0.5,1);
 \draw [-](8-0.5,1) to  (8,1-0.5);
 \draw [-](8+0.5,1) to  (8,1-0.5);

  \fill (10,-1) circle (2pt);
  \fill (8,-1) circle (2pt);
  \fill (10,1) circle (2pt);
   \fill[gray] (8-0.5,1) -- (8+0.5,1)--(8,1-0.5);
  \fill (8,0.5) circle (2pt);
  \fill (7.5,1) circle (2pt);
  \fill (8.5,1) circle (2pt);
  \node[left]   at(8,-1)    {$y$};
  \node[left]   at(7.5,1)    {$x$};
  \node[above]   at(9.2,1.1)    {$w+e_\iota$};
  \node[above]   at(8.3,1)    {$w$};
  \node[right]   at(8,0.5)    {$u$};
  \node[right]   at(10,1)    {$0$};
  \node[right]   at(10,-1)   {$v$};

  \node[right]   at(3,0)    {$\bar B^{2,\iota}(0,v,x,y):=\sum_{u,w}$};
\end{tikzpicture}}

\def\picBTwoPrime[#1]{
\begin{tikzpicture}[line width=1pt,auto,scale=#1]

 \draw [-](6.5+0.03,-1-0.2) to (6.5+0.03,-1+0.2);
 \draw [-](6.5+0.5-0.03,-1-0.2) to   (6.5+0.5-0.03,-1+0.2);
 \draw [-](6.5,1) to  (8-0.5,1);
 \draw [-](8,-1) to  (8,1-0.5);
 \draw [-](7,-1) to  (8,-1);
 \draw [-](8-0.5,1) to  (8+0.5,1);
 \draw [-](8-0.5,1) to  (8,1-0.5);
 \draw [-](8+0.5,1) to  (8,1-0.5);

  \fill (6.5,-1) circle (2pt);
  \fill (8,-1) circle (2pt);
  \fill (6.5,1) circle (2pt);
  \fill (8,0.5) circle (2pt);
  \fill (7.5,1) circle (2pt);
  \fill (8.5,1) circle (2pt);
  \node[left]   at(6.5,-1)    {$0$};
  \node[left]   at(6.5,1)    {$v$};
  \node[below]   at(7,-1)    {$e_\iota$};
  \node[above]   at(7.5,1)    {$w$};
  \node[right]   at(8,0.5)    {$u$};
  \node[right]   at(8.5,1)    {$y$};
   \node[right]   at(8,-1)    {$x$};
%     \node[right]   at(2,0)    {$B^{NT,2,\iota}(0,v,x,y):=\sum_{u,w}$};
  \node[left]   at(7,0)    {$\sum_{u,w}$};
\end{tikzpicture}}

\def\picBTwoPrimeCaseZeroOneOne[#1]{
\begin{tikzpicture}[line width=1pt,auto,scale=#1]
 \draw [-](6.5+0.03,-1-0.2) to (6.5+0.03,-1+0.2);
 \draw [-](6.5+1.5-0.03,-1-0.2) to   (6.5+1.5-0.03,-1+0.2);
 \draw [-] (6.5,-1) to  [out=90,in=180]   (7.5,0.5);
 \draw [-](8-0.5,0.5) to  (9,0.5);
 \draw [-](8-0.5,0.5) to  (8,-1);
 \draw [-](9,0.5) to  (8,-1);

  \fill (6.5,-1) circle (2pt);
  \fill (8,-1) circle (2pt);
  \fill (7.5,0.5) circle (2pt);
  \fill (9,0.5) circle (2pt);
  \node[left]   at(6.5,-1)    {$0=v$};
  \node[above]   at(7.5,0.5)    {$w$};
  \node[right]   at(9,0.5)    {$y$};
  \node[right]   at(8  ,-1)    {$e_\iota=x=u$};
  \node[below]   at(7.2,-1.1)    {$=1$};
  \node[right]   at(8.4,-0.2)    {$=1$};
  \node   at(7.2,-0.2)    {$=1$};
  \node   at(6.2,0.3)    {$\geq 2$};
  \node[above]   at(8.5,0.7)    {$\geq 2$};
\end{tikzpicture}}

\def\picBbarPrimeTwo[#1]{
\begin{tikzpicture}[line width=1pt,auto,scale=#1]

 \draw [-](8.5+0.03,1-0.2) to (8.5+0.03,1+0.2);
 \draw [-](9-0.03,1-0.2) to   (9-0.03,1+0.2);
 \draw [-](9,1) to  (10,1);
 \draw [-](8,-1) to  (8,1-0.5);
 \draw [-](10,-1) to  (8,-1);
 \draw [-](8-0.5,1) to  (8+0.5,1);
 \draw [-](8-0.5,1) to  (8,1-0.5);
 \draw [-](8+0.5,1) to  (8,1-0.5);

  \fill (10,-1) circle (2pt);
  \fill (8,-1) circle (2pt);
  \fill (10,1) circle (2pt);
  \fill (8,0.5) circle (2pt);
  \fill (7.5,1) circle (2pt);
  \fill (8.5,1) circle (2pt);
  \node[left]   at(8,-1)    {$x$};
  \node[left]   at(7.5,1)    {$y$};
  \node[above]   at(9.2,1.1)    {$w+e_\iota$};
  \node[above]   at(8.3,1)    {$w$};
  \node[right]   at(8,0.5)    {$u$};
  \node[right]   at(10,1)    {$v$};
  \node[right]   at(10,-1)   {$0$};

%  \node[right]   at(3,0)    {$\bar B^{NT,2,\iota}(0,v,x,y):=\sum_{u,w}$};
  \node[left]   at(8,0)    {$\sum_{u,w}$};
\end{tikzpicture}}

\def\picBTwoPrimeCaseZeroOneTwo[#1]{
\begin{tikzpicture}[line width=1pt,auto,scale=#1]
 \draw [-](6.5+0.03,-1-0.2) to (6.5+0.03,-1+0.2);
 \draw [-](6.5+0.5-0.03,-1-0.2) to   (6.5+0.5-0.03,-1+0.2);
 \draw [-] (6.5,-1) to  [out=90,in=180]   (7.5,0.5);
 \draw [-](7,-1) to  (8,-1);
 \draw [-](8-0.5,0.5) to  (9,0.5);
 \draw [-](8-0.5,0.5) to  (8,-1);
 \draw [-](9,0.5) to  (8,-1);

  \fill (6.5,-1) circle (2pt);
  \fill (8,-1) circle (2pt);
  \fill (7.5,0.5) circle (2pt);
  \fill (9,0.5) circle (2pt);
  \node[left]   at(6.5,-1)    {$0=v$};
  \node[below]   at(7,-1)    {$e_\iota$};
  \node[above]   at(7.5,0.5)    {$w$};
  \node[right]   at(9,0.5)    {$y$};
  \node[right]   at(8  ,-1)    {$x=u$};
  \node[below]   at(7.8,-1.2)    {$\geq 1$};
  \node[right]   at(8.4,-0.2)    {$=1$};
  \node[above]   at(8.5,0.7)    {$\geq 2$};
  \node   at(7.2,-0.2)    {$=1$};
  \node   at(6.2,0.3)    {$\geq 1$};
\end{tikzpicture}}

\def\picBTwoPrimeCaseZeroOneThree[#1]{
\begin{tikzpicture}[line width=1pt,auto,scale=#1]
 \draw [-](6.5+0.03,-1-0.2) to (6.5+0.03,-1+0.2);
 \draw [-](6.5+0.5-0.03,-1-0.2) to   (6.5+0.5-0.03,-1+0.2);
 \draw [-] (6.5,-1) to  [out=90,in=180]   (7.5,0.5);
 \draw [-](7,-1) to  (8,-1);
 \draw [-](8-0.5,0.5) to  (9,0.5);
 \draw [-](8-0.5,0.5) to  (8,-1);
 \draw [-](9,0.5) to  (8,-1);

  \fill (6.5,-1) circle (2pt);
  \fill (8,-1) circle (2pt);
  \fill (8,0.5) circle (2pt);
  \fill (7.5,0.5) circle (2pt);
  \fill (9,0.5) circle (2pt);
  \node[left]   at(6.5,-1)    {$0=v$};
  \node[below]   at(7,-1)    {$e_\iota$};
  \node[above]   at(7.5,0.5)    {$w$};
  \node[right]   at(9,0.5)    {$y$};
  \node[right]   at(8  ,-1)    {$x=u$};
  \node[below]   at(7.8,-1.2)    {$\geq 0$};
  \node[right]   at(8.4,-0.2)    {$=1$};
  \node[above]   at(8.5,0.7)    {$\geq 1$};
  \node   at(7.2,-0.2)    {$\geq 2$};
  \node   at(6.2,0.3)    {$\geq 1$};
\end{tikzpicture}}

\def\picBTwoPrimeCaseZeroTwoOne[#1]{
\begin{tikzpicture}[line width=1pt,auto,scale=#1]
 \draw [-](6.5+0.03,-1-0.2) to (6.5+0.03,-1+0.2);
 \draw [-](6.5+0.5-0.03,-1-0.2) to   (6.5+0.5-0.03,-1+0.2);
 \draw [-] (6.5,-1) to  [out=90,in=180]   (7.5,0.5);
 \draw [-](7,-1) to  (8,-1);
 \draw [-](8-0.5,0.5) to  (9,0.5);
 \draw [-](8-0.5,0.5) to  (8,-1);
 \draw [-](9,0.5) to  (8,-1);

  \fill (6.5,-1) circle (2pt);
  \fill (8,-1) circle (2pt);
  \fill (8,0.5) circle (2pt);
  \fill (7.5,0.5) circle (2pt);
  \fill (9,0.5) circle (2pt);
  \node[left]   at(6.5,-1)    {$0=v$};
  \node[below]   at(7,-1)    {$e_\iota$};
  \node[above]   at(7.5,0.5)    {$w$};
  \node[right]   at(9,0.5)    {$y$};
  \node[right]   at(8  ,-1)    {$x=u$};
  \node[below]   at(7.8,-1.2)    {$\geq 0$};
  \node[right]   at(8.4,-0.2)    {$\geq 2 $};
  \node[above]   at(8.5,0.7)    {$\geq 1$};
  \node   at(7.2,-0.2)    {$= 1$};
  \node   at(6.2,0.3)    {$\geq 1$};
\end{tikzpicture}}

\def\picBTwoPrimeCaseZeroTwoOnePointFive[#1]{
\begin{tikzpicture}[line width=1pt,auto,scale=#1]
 \draw [-](6.5+0.03,-1-0.2) to (6.5+0.03,-1+0.2);
 \draw [-](6.5+0.5-0.03,-1-0.2) to   (6.5+0.5-0.03,-1+0.2);
 \draw [-] (6.5,-1) to  [out=90,in=180]   (7.5,0.5);
 \draw [-](7,-1) to  (8,-1);
 \draw [-](8-0.5,0.5) to  (9,0.5);
 \draw [-](8-0.5,0.5) to  (8,-1);
 \draw [-](9,0.5) to  (8,-1);

  \fill (6.5,-1) circle (2pt);
  \fill (8,-1) circle (2pt);
  \fill (8,0.5) circle (2pt);
  \fill (7.5,0.5) circle (2pt);
  \fill (9,0.5) circle (2pt);
  \node[left]   at(6.5,-1)    {$0=v$};
  \node[below]   at(7,-1)    {$e_\iota$};
  \node[above]   at(7.5,0.5)    {$w$};
  \node[right]   at(9,0.5)    {$y$};
  \node[right]   at(8  ,-1)    {$x=u$};
  \node[below]   at(7.8,-1.2)    {$\geq 0$};
  \node[right]   at(8.4,-0.2)    {$\geq 2 $};
  \node[above]   at(8.5,0.7)    {$\geq 1$};
  \node   at(7.2,-0.2)    {$\geq 2$};
  \node   at(6.2,0.3)    {$\geq 1$};
\end{tikzpicture}}

\def\picBTwoPrimeCaseZeroTwoTwo[#1]{
\begin{tikzpicture}[line width=1pt,auto,scale=#1]
 \draw [-](6.5+0.03,-1-0.2) to (6.5+0.03,-1+0.2);
 \draw [-](6.5+0.5-0.03,-1-0.2) to   (6.5+0.5-0.03,-1+0.2);
 \draw [-](8,-1) to  (8,1-1);
 \draw [-](7,-1) to  (8,-1);
 \draw [-] (6.5,-1) to  [out=90,in=180]   (7.5,1);
 \draw [-](8-0.5,1) to  (8+1,1);
 \draw [-](8-0.5,1) to  (8,1-1);
 \draw [-](8+1,1) to  (8,1-1);

  \fill (6.5,-1) circle (2pt);
  \fill (8,-1) circle (2pt);
  \fill (8,0) circle (2pt);
  \fill (7.5,1) circle (2pt);
  \fill (9,1) circle (2pt);
  \node[left]   at(6.5,-1)    {$0=v$};
  \node[below]   at(7,-1)    {$e_\iota$};
  \node[above]   at(7.5,1)    {$w$};
  \node[right]   at(8,0)    {$u$};
  \node[right]   at(9,1)    {$y$};
   \node[right]   at(8,-1)    {$x$};

  \node[below]   at(7.8,-1.2)    {$\geq 0$};
  \node[right]   at(8.4,-0.5)    {$\geq 1 $};
  \node[right]   at(9,0.5)    {$\geq 1 $};

  \node[above]   at(8.5,1.2)    {$\geq 1$};
  \node   at(7.3,0.2)    {$= 1$};
  \node   at(6,0)    {$\geq 1$};
\end{tikzpicture}}

\def\picBTwoPrimeCaseZeroTwoThree[#1]{
\begin{tikzpicture}[line width=1pt,auto,scale=#1]
 \draw [-](6.5+0.03,-1-0.2) to (6.5+0.03,-1+0.2);
 \draw [-](6.5+0.5-0.03,-1-0.2) to   (6.5+0.5-0.03,-1+0.2);
 \draw [-](8,-1) to  (8,1-1);
 \draw [-](7,-1) to  (8,-1);
 \draw [-] (6.5,-1) to  [out=90,in=180]   (7.5,1);
 \draw [-](8-0.5,1) to  (8+1,1);
 \draw [-](8-0.5,1) to  (8,1-1);
 \draw [-](8+1,1) to  (8,1-1);

  \fill (6.5,-1) circle (2pt);
  \fill (8,-1) circle (2pt);
  \fill (8,0) circle (2pt);
  \fill (7.5,1) circle (2pt);
  \fill (9,1) circle (2pt);
  \node[left]   at(6.5,-1)    {$0=v$};
  \node[below]   at(7,-1)    {$e_\iota$};
  \node[above]   at(7.5,1)    {$w$};
  \node[right]   at(8,0)    {$u$};
  \node[right]   at(9,1)    {$y$};
   \node[right]   at(8,-1)    {$x$};

  \node[below]   at(7.8,-1.2)    {$\geq 0$};
  \node[right]   at(8.4,-0.5)    {$\geq 1 $};
  \node[right]   at(9,0.5)    {$\geq 1 $};

  \node[above]   at(8.5,1.2)    {$\geq 1$};
  \node   at(7.3,0.2)    {$\geq 2$};
  \node   at(6,0)    {$\geq 1$};
\end{tikzpicture}}

\def\picBTwoPrimeCaseOneOneOne[#1]{
\begin{tikzpicture}[line width=1pt,auto,scale=#1]
 \draw [-](6.5+0.03,-1-0.2) to (6.5+0.03,-1+0.2);
 \draw [-](6.5+1.5-0.03,-1-0.2) to   (6.5+1.5-0.03,-1+0.2);
 \draw [-](6.5,0.5) to  (7.5,0.5);

 \draw [-](8-0.5,0.5) to  (9,0.5);
 \draw [-](8-0.5,0.5) to  (8,-1);
 \draw [-](9,0.5) to  (8,-1);

  \fill (6.5,-1) circle (2pt);
  \fill (6.5,0.5) circle (2pt);
  \fill (8,-1) circle (2pt);
  \fill (7.5,0.5) circle (2pt);
  \fill (9,0.5) circle (2pt);

  \node[left]   at(6.5,-1)    {$0$};
  \node[left]   at(6.5,0.5)    {$v$};
  \node[above]   at(7.5,0.5)    {$w$};
  \node[right]   at(9,0.5)    {$y$};
  \node[right]   at(8  ,-1)    {$e_\iota=x$};
  \node[below]   at(7.2,-1.1)    {$=1$};
  \node[right]   at(8.4,-0.2)    {$=1$};
  \node   at(7.2,-0.2)    {$=1$};
  \node[left]   at(6.2,-0.3)    {$=1$};
  \node[above]   at(8.5,0.7)    {$\geq 2$};
  \node[above]   at(7,0.7)    {$\geq 1$};
\end{tikzpicture}}

\def\picBTwoPrimeCaseOneOneTwo[#1]{
\begin{tikzpicture}[line width=1pt,auto,scale=#1]
 \draw [-](6.5+0.03,-1-0.2) to (6.5+0.03,-1+0.2);
  \draw [-](6.5+0.5-0.03,-1-0.2) to   (6.5+0.5-0.03,-1+0.2);
 \draw [-](6.5,0.5) to  (7.5,0.5);
 \draw [-](7,-1) to  (8,-1);
 \draw [-](8-0.5,0.5) to  (9,0.5);
 \draw [-](8-0.5,0.5) to  (8,-1);
 \draw [-](9,0.5) to  (8,-1);

  \fill (6.5,-1) circle (2pt);
  \fill (6.5,0.5) circle (2pt);
  \fill (8,-1) circle (2pt);
  \fill (7.5,0.5) circle (2pt);
  \fill (9,0.5) circle (2pt);

  \node[left]   at(6.5,-1)    {$0$};
  \node[left]   at(6.5,0.5)    {$v$};
  \node[above]   at(7.5,0.5)    {$w$};
  \node[right]   at(9,0.5)    {$y$};
  \node[right]   at(8  ,-1)    {$x=u$};
   \node[below]   at(7,-1)    {$e_\iota$};
  \node[below]   at(7.7,-1.2)    {$\geq 1$};;
  \node[right]   at(8.4,-0.2)    {$=1$};
  \node   at(7.2,-0.2)    {$=1$};
  \node[left]   at(6.2,-0.3)    {$=1$};
  \node[above]   at(8.5,0.7)    {$\geq 2$};
  \node[above]   at(7,0.7)    {$\geq 0$};
\end{tikzpicture}}

\def\picBTwoPrimeCaseOneOneThree[#1]{
\begin{tikzpicture}[line width=1pt,auto,scale=#1]
 \draw [-](6.5+0.03,-1-0.2) to (6.5+0.03,-1+0.2);
  \draw [-](6.5+0.5-0.03,-1-0.2) to   (6.5+0.5-0.03,-1+0.2);
 \draw [-](6.5,0.5) to  (7.5,0.5);
 \draw [-](7,-1) to  (8,-1);
 \draw [-](8-0.5,0.5) to  (9,0.5);
 \draw [-](8-0.5,0.5) to  (8,-1);
 \draw [-](9,0.5) to  (8,-1);

  \fill (6.5,-1) circle (2pt);
  \fill (6.5,0.5) circle (2pt);
  \fill (8,-1) circle (2pt);
  \fill (7.5,0.5) circle (2pt);
  \fill (9,0.5) circle (2pt);

  \node[left]   at(6.5,-1)    {$0$};
  \node[left]   at(6.5,0.5)    {$v$};
  \node[above]   at(7.5,0.5)    {$w$};
  \node[right]   at(9,0.5)    {$y$};
  \node[right]   at(8  ,-1)    {$x=u$};
   \node[below]   at(7,-1)    {$e_\iota$};
  \node[below]   at(7.7,-1.2)    {$\geq 0$};;
  \node[right]   at(8.4,-0.2)    {$=1$};
  \node   at(7.2,-0.2)    {$\geq 2$};
  \node[left]   at(6.2,-0.3)    {$=1$};
  \node[above]   at(8.5,0.7)    {$\geq 1$};
  \node[above]   at(7,0.7)    {$\geq 0$};
\end{tikzpicture}}

\def\picBTwoPrimeCaseOneTwoOne[#1]{
\begin{tikzpicture}[line width=1pt,auto,scale=#1]
 \draw [-](6.5+0.03,-1-0.2) to (6.5+0.03,-1+0.2);
 \draw [-](6.5+0.7-0.03,-1-0.2) to   (6.5+0.7-0.03,-1+0.2);
 \draw [-](6.5,0.5) to  (7.5,0.5);
  \draw [-](7.2,-1) to  (8,-1);

 \draw [-](8-0.5,0.5) to  (9,0.5);
 \draw [-](8-0.5,0.5) to  (8,-1);
 \draw [-](9,0.5) to  (8,-1);

  \fill (6.5,-1) circle (2pt);
  \fill (6.5,0.5) circle (2pt);
  \fill (8,-1) circle (2pt);
  \fill (7.5,0.5) circle (2pt);
  \fill (9,0.5) circle (2pt);

  \node[left]   at(6.5,-1)    {$0$};
  \node[left]   at(6.5,0.5)    {$v$};
  \node[above]   at(7.5,0.5)    {$w$};
  \node[right]   at(9,0.5)    {$y$};
  \node[right]   at(8  ,-1)    {$x=u$};
  \node[below]   at(6.8,-1.1)    {$=1$};
    \node[below]   at(7.9,-1.2)    {$\geq 0$};
  \node[right]   at(8.4,-0.2)    {$\geq 2$};
  \node   at(7.2,-0.2)    {$=1$};
  \node[left]   at(6.2,-0.3)    {$=1$};
  \node[above]   at(8.5,0.7)    {$\geq 1$};
  \node[above]   at(7,0.7)    {$\geq 0$};
\end{tikzpicture}}

\def\picBTwoPrimeCaseOneTwoTwo[#1]{
\begin{tikzpicture}[line width=1pt,auto,scale=#1]
 \draw [-](6.5+0.03,-1-0.2) to (6.5+0.03,-1+0.2);
 \draw [-](6.5+0.7-0.03,-1-0.2) to   (6.5+0.7-0.03,-1+0.2);
 \draw [-](6.5,0.5) to  (7.5,0.5);
  \draw [-](7.2,-1) to  (8,-1);

 \draw [-](8-0.5,0.5) to  (9,0.5);
 \draw [-](8-0.5,0.5) to  (8,-1);
 \draw [-](9,0.5) to  (8,-1);

  \fill (6.5,-1) circle (2pt);
  \fill (6.5,0.5) circle (2pt);
  \fill (8,-1) circle (2pt);
  \fill (7.5,0.5) circle (2pt);
  \fill (9,0.5) circle (2pt);

  \node[left]   at(6.5,-1)    {$0$};
  \node[left]   at(6.5,0.5)    {$v$};
  \node[above]   at(7.5,0.5)    {$w$};
  \node[right]   at(9,0.5)    {$y$};
  \node[right]   at(8  ,-1)    {$e_\iota=x=u$};
  \node[below]   at(6.8,-1.1)    {$=1$};
    \node[below]   at(7.9,-1.2)    {$\geq 0$};
  \node[right]   at(8.4,-0.2)    {$\geq 2$};
  \node   at(7.2,-0.2)    {$\geq 2$};
  \node[left]   at(6.2,-0.3)    {$=1$};
  \node[above]   at(8.5,0.7)    {$\geq 1$};
  \node[above]   at(7,0.7)    {$\geq 0$};
\end{tikzpicture}}

\def\picBTwoPrimeCaseOneTwoThree[#1]{
\begin{tikzpicture}[line width=1pt,auto,scale=#1]
 \draw [-](6.5+0.03,-1-0.2) to (6.5+0.03,-1+0.2);
 \draw [-](6.5+0.5-0.03,-1-0.2) to   (6.5+0.5-0.03,-1+0.2);
 \draw [-](8,-1) to  (8,1-1);
 \draw [-](7,-1) to  (8,-1);
  \draw [-] (7.5,1) to  (6.5,1);
 \draw [-](8-0.5,1) to  (8+1,1);
 \draw [-](8-0.5,1) to  (8,1-1);
 \draw [-](8+1,1) to  (8,1-1);

  \fill (6.5,-1) circle (2pt);
  \fill (6.5,1) circle (2pt);
  \fill (8,-1) circle (2pt);
  \fill (8,0) circle (2pt);
  \fill (7.5,1) circle (2pt);
  \fill (9,1) circle (2pt);
  \node[left]   at(6.5,-1)    {$0$};
  \node[left]   at(6.5,1)    {$v$};
  \node[below]   at(7,-1)    {$e_\iota$};
  \node[above]   at(7.5,1)    {$w$};
  \node[right]   at(8,0)    {$u$};
  \node[right]   at(9,1)    {$y$};
  \node[right]   at(8,-1)    {$x$};

  \node[below]   at(7.8,-1.2)    {$\geq 0$};
  \node[right]   at(8.4,-0.5)    {$\geq 1 $};
  \node[right]   at(9,0.5)    {$\geq 1 $};
  \node[above]   at(8.5,1.2)    {$\geq 1$};
  \node[above]   at(7,1.2)    {$\geq 0$};
  \node   at(7.3,0.2)    {$=1$};
  \node[left]   at(6,0)    {$=1$};
\end{tikzpicture}}

\def\picBTwoPrimeCaseOneTwoFour[#1]{
\begin{tikzpicture}[line width=1pt,auto,scale=#1]
 \draw [-](6.5+0.03,-1-0.2) to (6.5+0.03,-1+0.2);
 \draw [-](6.5+0.5-0.03,-1-0.2) to   (6.5+0.5-0.03,-1+0.2);
 \draw [-](8,-1) to  (8,1-1);
 \draw [-](7,-1) to  (8,-1);
  \draw [-] (7.5,1) to  (6.5,1);
 \draw [-](8-0.5,1) to  (8+1,1);
 \draw [-](8-0.5,1) to  (8,1-1);
 \draw [-](8+1,1) to  (8,1-1);

  \fill (6.5,-1) circle (2pt);
  \fill (6.5,1) circle (2pt);
  \fill (8,-1) circle (2pt);
  \fill (8,0) circle (2pt);
  \fill (7.5,1) circle (2pt);
  \fill (9,1) circle (2pt);
  \node[left]   at(6.5,-1)    {$0$};
  \node[left]   at(6.5,1)    {$v$};
  \node[below]   at(7,-1)    {$e_\iota$};
  \node[above]   at(7.5,1)    {$w$};
  \node[right]   at(8,0)    {$u$};
  \node[right]   at(9,1)    {$y$};
  \node[right]   at(8,-1)    {$x$};

  \node[below]   at(7.8,-1.2)    {$\geq 0$};
  \node[right]   at(8.4,-0.5)    {$\geq 1 $};
  \node[right]   at(9,0.5)    {$\geq 1 $};
  \node[above]   at(8.5,1.2)    {$\geq 1$};
  \node[above]   at(7,1.2)    {$\geq 0$};
  \node   at(7.3,0.2)    {$\geq 2$};
  \node[left]   at(6,0)    {$=1$};
\end{tikzpicture}}

\def\picBTwoPrimeCaseTwoOneOne[#1]{
\begin{tikzpicture}[line width=1pt,auto,scale=#1]
 \draw [-](6.5+0.03,-1-0.2) to (6.5+0.03,-1+0.2);
  \draw [-](6.5+0.5-0.03,-1-0.2) to   (6.5+0.5-0.03,-1+0.2);
 \draw [-](6.5,0.5) to  (7.5,0.5);
 \draw [-](7,-1) to  (8,-1);
 \draw [-](8-0.5,0.5) to  (9,0.5);
 \draw [-](8-0.5,0.5) to  (8,-1);
 \draw [-](9,0.5) to  (8,-1);

  \fill (6.5,-1) circle (2pt);
  \fill (6.5,0.5) circle (2pt);
  \fill (8,-1) circle (2pt);
  \fill (7.5,0.5) circle (2pt);
  \fill (9,0.5) circle (2pt);

  \node[left]   at(6.5,-1)    {$0$};
  \node[left]   at(6.5,0.5)    {$v$};
  \node[above]   at(7.5,0.5)    {$w$};
  \node[right]   at(9,0.5)    {$y$};
  \node[right]   at(8  ,-1)    {$x=u$};
   \node[below]   at(7,-1)    {$e_\iota$};
  \node[below]   at(7.7,-1.2)    {$\geq 0$};;
  \node[right]   at(8.4,-0.2)    {$=1$};
  \node   at(7.2,-0.2)    {$= 1$};
  \node[left]   at(6.2,-0.3)    {$\geq 2$};
  \node[above]   at(8.5,0.7)    {$\geq 2$};
  \node[above]   at(7,0.7)    {$\geq 0$};
\end{tikzpicture}}

\def\picBTwoPrimeCaseTwoOneTwo[#1]{
\begin{tikzpicture}[line width=1pt,auto,scale=#1]
 \draw [-](6.5+0.03,-1-0.2) to (6.5+0.03,-1+0.2);
  \draw [-](6.5+0.5-0.03,-1-0.2) to   (6.5+0.5-0.03,-1+0.2);
 \draw [-](6.5,0.5) to  (7.5,0.5);
 \draw [-](7,-1) to  (8,-1);
 \draw [-](8-0.5,0.5) to  (9,0.5);
 \draw [-](8-0.5,0.5) to  (8,-1);
 \draw [-](9,0.5) to  (8,-1);

  \fill (6.5,-1) circle (2pt);
  \fill (6.5,0.5) circle (2pt);
  \fill (8,-1) circle (2pt);
  \fill (7.5,0.5) circle (2pt);
  \fill (9,0.5) circle (2pt);

  \node[left]   at(6.5,-1)    {$0$};
  \node[left]   at(6.5,0.5)    {$v$};
  \node[above]   at(7.5,0.5)    {$w$};
  \node[right]   at(9,0.5)    {$y$};
  \node[right]   at(8  ,-1)    {$x=u$};
   \node[below]   at(7,-1)    {$e_\iota$};
  \node[below]   at(7.7,-1.2)    {$\geq 0$};;
  \node[right]   at(8.4,-0.2)    {$=1$};
  \node   at(7.2,-0.2)    {$\geq 2$};
  \node[left]   at(6.2,-0.3)    {$\geq 2$};
  \node[above]   at(8.5,0.7)    {$\geq 1$};
  \node[above]   at(7,0.7)    {$\geq 0$};
\end{tikzpicture}}

\def\picBTwoPrimeCaseTwoTwoOne[#1]{
\begin{tikzpicture}[line width=1pt,auto,scale=#1]
 \draw [-](6.5+0.03,-1-0.2) to (6.5+0.03,-1+0.2);
  \draw [-](6.5+0.5-0.03,-1-0.2) to   (6.5+0.5-0.03,-1+0.2);
 \draw [-](6.5,0.5) to  (7.5,0.5);
 \draw [-](7,-1) to  (8,-1);
 \draw [-](8-0.5,0.5) to  (9,0.5);
 \draw [-](8-0.5,0.5) to  (8,-1);
 \draw [-](9,0.5) to  (8,-1);

  \fill (6.5,-1) circle (2pt);
  \fill (6.5,0.5) circle (2pt);
  \fill (8,-1) circle (2pt);
  \fill (7.5,0.5) circle (2pt);
  \fill (9,0.5) circle (2pt);

  \node[left]   at(6.5,-1)    {$0$};
  \node[left]   at(6.5,0.5)    {$v$};
  \node[above]   at(7.5,0.5)    {$w$};
  \node[right]   at(9,0.5)    {$y$};
  \node[right]   at(8  ,-1)    {$x=u$};
   \node[below]   at(7,-1)    {$e_\iota$};
  \node[below]   at(7.7,-1.2)    {$\geq 0$};;
  \node[right]   at(8.4,-0.2)    {$\geq 2$};
  \node   at(7.2,-0.2)    {$\geq 1$};
  \node[left]   at(6.2,-0.3)    {$\geq 2$};
  \node[above]   at(8.5,0.7)    {$\geq 1$};
  \node[above]   at(7,0.7)    {$\geq 0$};
\end{tikzpicture}}

\def\picBTwoPrimeCaseTwoTwoTwo[#1]{
\begin{tikzpicture}[line width=1pt,auto,scale=#1]
 \draw [-](6.5+0.03,-1-0.2) to (6.5+0.03,-1+0.2);
 \draw [-](6.5+0.5-0.03,-1-0.2) to   (6.5+0.5-0.03,-1+0.2);
 \draw [-](8,-1) to  (8,1-1);
 \draw [-](7,-1) to  (8,-1);
  \draw [-] (7.5,1) to  (6.5,1);
 \draw [-](8-0.5,1) to  (8+1,1);
 \draw [-](8-0.5,1) to  (8,1-1);
 \draw [-](8+1,1) to  (8,1-1);

  \fill (6.5,-1) circle (2pt);
  \fill (6.5,1) circle (2pt);
  \fill (8,-1) circle (2pt);
  \fill (8,0) circle (2pt);
  \fill (7.5,1) circle (2pt);
  \fill (9,1) circle (2pt);
  \node[left]   at(6.5,-1)    {$0$};
  \node[left]   at(6.5,1)    {$v$};
  \node[below]   at(7,-1)    {$e_\iota$};
  \node[above]   at(7.5,1)    {$w$};
  \node[right]   at(8,0)    {$u$};
  \node[right]   at(9,1)    {$y$};
  \node[right]   at(8,-1)    {$x$};

  \node[below]   at(7.8,-1.2)    {$\geq 0$};
  \node[right]   at(8.4,-0.5)    {$\geq 1 $};
  \node[right]   at(9,0.5)    {$\geq 1 $};
  \node[above]   at(8.5,1.2)    {$\geq 1$};
  \node[above]   at(7,1.2)    {$\geq 0$};
  \node   at(7.3,0.2)    {$=1$};
  \node[left]   at(6,0)    {$\geq 2$};
\end{tikzpicture}}

\def\picBTwoPrimeCaseTwoTwoThree[#1]{
\begin{tikzpicture}[line width=1pt,auto,scale=#1]
 \draw [-](6.5+0.03,-1-0.2) to (6.5+0.03,-1+0.2);
 \draw [-](6.5+0.5-0.03,-1-0.2) to   (6.5+0.5-0.03,-1+0.2);
 \draw [-](8,-1) to  (8,1-1);
 \draw [-](7,-1) to  (8,-1);
  \draw [-] (7.5,1) to  (6.5,1);
 \draw [-](8-0.5,1) to  (8+1,1);
 \draw [-](8-0.5,1) to  (8,1-1);
 \draw [-](8+1,1) to  (8,1-1);

  \fill (6.5,-1) circle (2pt);
  \fill (6.5,1) circle (2pt);
  \fill (8,-1) circle (2pt);
  \fill (8,0) circle (2pt);
  \fill (7.5,1) circle (2pt);
  \fill (9,1) circle (2pt);
  \node[left]   at(6.5,-1)    {$0$};
  \node[left]   at(6.5,1)    {$v$};
  \node[below]   at(7,-1)    {$e_\iota$};
  \node[above]   at(7.5,1)    {$w$};
  \node[right]   at(8,0)    {$u$};
  \node[right]   at(9,1)    {$y$};
  \node[right]   at(8,-1)    {$x$};

  \node[below]   at(7.8,-1.2)    {$\geq 0$};
  \node[right]   at(8.4,-0.5)    {$\geq 1 $};
  \node[right]   at(9,0.5)    {$\geq 1 $};
  \node[above]   at(8.5,1.2)    {$\geq 1$};
  \node[above]   at(7,1.2)    {$\geq 0$};
  \node   at(7.3,0.2)    {$\geq 2$};
  \node[left]   at(6,0)    {$\geq 2$};
\end{tikzpicture}}

\def\picXiOneDiag[#1]{
\begin{tikzpicture}[line width=1pt,auto,scale=#1]
 \fill[gray] (0,0) -- (1,1)--(1,-1);
 \fill[gray] (4,0) -- (3,1)--(3,-1);

  \fill (0,0) circle (3pt);
  \fill (1,1) circle (3pt);
  \fill (1,-1) circle (2pt);
  \fill (3,-1) circle (3pt);
  \fill (3,1) circle (3pt);
  \fill (4,0) circle (3pt);
  \node[left]   at(0,0)      {$0$};

 \draw [-](0,0) to  (1,1);
 \draw [-](0,0) to  (1,-1);
 \draw [-](1,1) to  (1,-1);
 \draw [-](1,1) to  (3,1);
 \draw [-](1.5,-1) to  (3,-1);
 \draw [-](3,1) to  (3,-1);
 \draw [-](4,0) to  (3,-1);
 \draw [-](4,0) to  (3,1);
 \draw [-](1+0.03,-1-0.2) to (1+0.02,-1+0.2);
 \draw [-](1.5-0.03,-1-0.2) to   (1.5-0.02,-1+0.2);
\end{tikzpicture}}

\def\picBbarPrimeZeroZeroOne[#1]{
\begin{tikzpicture}[line width=1pt,auto,scale=#1]

 \draw [-] (0,0) to  [out=270,in=180]   (1,-1);
 \draw [-] (2,0) to  [out=270,in=0]   (1,-1);
 \draw [-] (0,0) to  [out=90,in=180]   (0.5,0.5);
 \draw [-] (1,0) to  [out=90,in=0]   (0.5,0.5);

 \fill (0,0) circle (2pt);
 \fill (1,0) circle (2pt);
 \fill (2,0) circle (2pt);

 \draw [-](0,-0.25) to (0,+0.25);
 \draw [-](1,-0.25) to   (1,+0.25);
 \draw [-](2,-0.25) to   (2,+0.25);

 \node[below]   at(1,0)    {$w$};
 \node[left]   at(0,0)    {$y$};
 \node[right]   at(2,0)    {$0$};
 \node[above]   at(0.5,0.5)    {$\geq 3$};
 \node[below]   at(1,-1)    {$\geq 2$};
\end{tikzpicture}}

\def\picBbarPrimeZeroZeroTwo[#1]{
\begin{tikzpicture}[line width=1pt,auto,scale=#1]

 \draw [-] (0,0) to  [out=270,in=180]   (1.5,-1);
 \draw [-] (3,0) to  [out=270,in=0]   (1.5,-1);
 \draw [-] (0,0) to  [out=90,in=180]   (0.5,0.5);
 \draw [-] (1,0) to  [out=90,in=0]   (0.5,0.5);

 \fill (0,0) circle (2pt);
 \fill (1,0) circle (2pt);
 \fill (2,0) circle (2pt);
 \fill (3,0) circle (2pt);

 \draw [-](0,-0.25) to (0,+0.25);
 \draw [-](1,-0.25) to   (1,+0.25);
 \draw [-](2,-0.25) to   (2,+0.25);
\draw [-](3,-0.25) to   (3,+0.25);

 \node[below]   at(1,0)    {$w$};
 \node[left]   at(0,0)    {$y$};
 \node[right]   at(3,0)    {$0$};
 \node[above]   at(0.5,0.5)    {$\geq 3$};
 \node[below]   at(1.5,-1)    {$\geq 1$};
\end{tikzpicture}}

\def\picBbarPrimeZeroZeroThree[#1]{
\begin{tikzpicture}[line width=1pt,auto,scale=#1]

 \draw [-] (0,0) to  [out=270,in=180]   (1.5,-1);
 \draw [-] (3,0) to  [out=270,in=0]   (1.5,-1);
 \draw [-] (0,0) to  [out=90,in=180]   (0.5,0.5);
 \draw [-] (1,0) to  [out=90,in=0]   (0.5,0.5);
 \draw [-] (2,0) to (3,0);

 \fill (0,0) circle (2pt);
 \fill (1,0) circle (2pt);
 \fill (2,0) circle (2pt);
 \fill (3,0) circle (2pt);

 \draw [-](0,-0.25) to (0,+0.25);
 \draw [-](1,-0.25) to   (1,+0.25);
 \draw [-](2,-0.25) to   (2,+0.25);

 \node[below]   at(1,0)    {$w$};
 \node[left]   at(0,0)    {$y$};
 \node[right]   at(3,0)    {$0$};
 \node[above]   at(0.5,0.5)    {$\geq 3$};
  \node[above]   at(2.5,0)    {$\geq 0$};
 \node[below]   at(1.5,-1)    {$\geq 0$};
\end{tikzpicture}}

\def\picBbarPrimeZeroZeroFour[#1]{
\begin{tikzpicture}[line width=1pt,auto,scale=#1]

 \draw [-] (0,0) to  [out=270,in=180]   (1.5,-1);
 \draw [-] (3,0) to  [out=270,in=0]   (1.5,-1);
 \draw [-] (0,0) to  [out=90,in=180]   (0.5,0.4);
 \draw [-] (1,0) to  [out=90,in=0]   (0.5,0.4);
  \draw [-] (0,0) to  [out=270,in=180]   (0.5,-0.4);
 \draw [-] (1,0) to  [out=270,in=0]   (0.5,-0.4);

 \draw [-] (2,0) to (3,0);

 \fill (0,0) circle (2pt);
 \fill (1,0) circle (2pt);
 \fill (2,0) circle (2pt);
 \fill (3,0) circle (2pt);

 \draw [-](1,-0.25) to   (1,+0.25);
 \draw [-](2,-0.25) to   (2,+0.25);

 \node[below]   at(1,0)    {$w$};
 \node[left]   at(0,0)    {$y$};
 \node[right]   at(3,0)    {$0$};
 \node[above]   at(0.5,0.4)    {$\geq 2$};
  \node[above]   at(2.5,0)    {$\geq 0$};
 \node[below]   at(1.5,-1)    {$\geq 0$};
\end{tikzpicture}}

\def\picBbarPrimeZeroOneOne[#1]{
\begin{tikzpicture}[line width=1pt,auto,scale=#1]
 \draw [-] (0,0) to  [out=270,in=180]   (2,-1);
 \draw [-] (0,0) to  [out=90,in=180]   (0.5,0.5);
 \draw [-] (1,0) to  [out=90,in=0]   (0.5,0.5);

 \fill (0,0) circle (2pt);
 \fill (1,0) circle (2pt);
 \fill (2,0) circle (2pt);

 \draw [-](0,-0.25) to (0,+0.25);
 \draw [-](1,-0.25) to   (1,+0.25);
 \draw [-](2,-0.25) to   (2,+0.25);
 \draw [-](2-0.25,0) to   (2+0.25,0);
 \draw [-](2-0.25,-1) to   (2+0.25,-1);

 \node[below]   at(1,0)    {$w$};
 \node[left]   at(0,0)    {$y$};
 \node[right]   at(2,0)    {$0$};
  \node[right]   at(2,-1)    {$v$};
 \node[above]   at(0.5,0.5)    {$\geq 3$};
 \node[below]   at(1,-1)    {$\geq 1$};

 \end{tikzpicture}}

\def\picBbarPrimeZeroOneTwo[#1]{
\begin{tikzpicture}[line width=1pt,auto,scale=#1]

 \draw [-] (0,0) to  [out=270,in=180]   (3,-1);
 \draw [-] (0,0) to  [out=90,in=180]   (0.5,0.5);
 \draw [-] (1,0) to  [out=90,in=0]   (0.5,0.5);
 \draw [-] (2,0) to (3,0);

 \fill (0,0) circle (2pt);
 \fill (1,0) circle (2pt);
 \fill (2,0) circle (2pt);
 \fill (3,0) circle (2pt);
  \fill (3,-1) circle (2pt);

 \draw [-](0,-0.25) to (0,+0.25);
 \draw [-](1,-0.25) to   (1,+0.25);
 \draw [-](2,-0.25) to   (2,+0.25);
 \draw [-](3-0.25,-0.05) to   (3+0.25,-0.05);
 \draw [-](3-0.25,-1+0.05) to   (3+0.25,-1+0.05);

 \node[below]   at(1,0)    {$w$};
 \node[left]   at(0,0)    {$y$};
 \node[right]   at(3,0)    {$0$};
  \node[right]   at(3,-1)    {$v$};
 \node[above]   at(0.5,0.5)    {$\geq 3$};
 \node[below]   at(1,-1)    {$\geq 0$};
 \node[above]   at(2.5,0)    {$\geq 0$};
\end{tikzpicture}}

\def\picBbarPrimeZeroOneThree[#1]{
\begin{tikzpicture}[line width=1pt,auto,scale=#1]

 \draw [-] (0,0) to  [out=270,in=180]   (3,-1);
 \draw [-] (0,0) to  [out=90,in=180]   (0.5,0.4);
 \draw [-] (1,0) to  [out=90,in=0]   (0.5,0.4);
  \draw [-] (0,0) to  [out=270,in=180]   (0.5,-0.4);
 \draw [-] (1,0) to  [out=270,in=0]   (0.5,-0.4);

 \draw [-] (2,0) to (3,0);

 \draw [-](3-0.25,-0.05) to   (3+0.25,-0.05);
 \draw [-](3-0.25,-1+0.05) to   (3+0.25,-1+0.05);

 \fill (0,0) circle (2pt);
 \fill (1,0) circle (2pt);
 \fill (2,0) circle (2pt);
 \fill (3,0) circle (2pt);
   \fill (3,-1) circle (2pt);

 \draw [-](1,-0.25) to   (1,+0.25);
 \draw [-](2,-0.25) to   (2,+0.25);

 \node[below]   at(1,0)    {$w$};
 \node[left]   at(0,0)    {$y$};
 \node[right]   at(3,0)    {$0$};
  \node[right]   at(3,-1)    {$v$};
 \node[above]   at(0.5,0.4)    {$\geq 2$};
  \node[above]   at(2.5,0)    {$\geq 0$};
 \node[below]   at(1.5,-1)    {$\geq 0$};
\end{tikzpicture}}

\def\picBbarPrimeZeroTwoOne[#1]{
\begin{tikzpicture}[line width=1pt,auto,scale=#1]

 \draw [-] (0,0) to  [out=270,in=180]   (3,-1);
 \draw [-] (0,0) to  [out=90,in=180]   (0.5,0.5);
 \draw [-] (1,0) to  [out=90,in=0]   (0.5,0.5);
 \draw [-] (2,0) to (3,0);

 \fill (0,0) circle (2pt);
 \fill (1,0) circle (2pt);
 \fill (2,0) circle (2pt);
 \fill (3,0) circle (2pt);
  \fill (3,-1) circle (2pt);

 \draw [-](0,-0.25) to (0,+0.25);
 \draw [-](1,-0.25) to   (1,+0.25);
 \draw [-](2,-0.25) to   (2,+0.25);

 \node[below]   at(1,0)    {$w$};
 \node[left]   at(0,0)    {$y$};
 \node[right]   at(3,0)    {$0$};
  \node[right]   at(3,-1)    {$v$};
    \node[right]   at(3,-0.5)    {$\geq 2$};

 \node[above]   at(0.5,0.5)    {$\geq 3$};
 \node[below]   at(1,-1)    {$\geq 0$};
 \node[above]   at(2.5,0)    {$\geq 0$};
\end{tikzpicture}}

\def\picBbarPrimeZeroTwoTwo[#1]{
\begin{tikzpicture}[line width=1pt,auto,scale=#1]

 \draw [-] (0,0) to  [out=270,in=180]   (3,-1);
 \draw [-] (0,0) to  [out=90,in=180]   (0.5,0.4);
 \draw [-] (1,0) to  [out=90,in=0]   (0.5,0.4);
  \draw [-] (0,0) to  [out=270,in=180]   (0.5,-0.4);
 \draw [-] (1,0) to  [out=270,in=0]   (0.5,-0.4);

 \draw [-] (2,0) to (3,0);

 \fill (0,0) circle (2pt);
 \fill (1,0) circle (2pt);
 \fill (2,0) circle (2pt);
 \fill (3,0) circle (2pt);
  \fill (3,-1) circle (2pt);

 \draw [-](1,-0.25) to   (1,+0.25);
 \draw [-](2,-0.25) to   (2,+0.25);

 \node[below]   at(1,0)    {$w$};
 \node[left]   at(0,0)    {$y$};
 \node[right]   at(3,0)    {$0$};
  \node[right]   at(3,-1)    {$v$};
    \node[right]   at(3,-0.5)    {$\geq 2$};
 \node[above]   at(0.5,0.4)    {$\geq 2$};
  \node[above]   at(2.5,0)    {$\geq 0$};
 \node[below]   at(1.5,-1)    {$\geq 0$};
\end{tikzpicture}}

\def\picBbarPrimeOneZeroTwo[#1]{
\begin{tikzpicture}[line width=1pt,auto,scale=#1]

 \draw [-] (0,-1) to  [out=0,in=270]   (3,0);
 \draw [-] (0,0) to  [out=90,in=180]   (0.5,0.4);
 \draw [-] (1,0) to  [out=90,in=0]   (0.5,0.4);
  \draw [-] (0,0) to  [out=270,in=180]   (0.5,-0.4);
 \draw [-] (1,0) to  [out=270,in=0]   (0.5,-0.4);
 \draw [-] (0,0) to  (0,-1);

 \draw [-](2,0) to   (3,0);
 \fill (0,0) circle (2pt);
 \fill (1,0) circle (2pt);
 \fill (2,0) circle (2pt);
 \fill (3,0) circle (2pt);
 \fill (0,-1) circle (2pt);

 \draw [-](1,-0.25) to   (1,+0.25);
 \draw [-](2,-0.25) to   (2,+0.25);

 \node[below]   at(1,-0.2)    {$w$};
 \node[below]   at(2,-0.1)    {$w+\ve[\iota]$};
 \node[left]   at(0,0)    {$y$};
  \node[left]   at(0,-1)    {$x$};
  \node[left]   at(0,-0.5)    {$=1$};
 \node[right]   at(3,0)    {$0$};
  \node[above]   at(2.5,0)    {$\geq 1$};
 \node[below]   at(1.5,-0.9)    {$\geq 0$};
\end{tikzpicture}}

\def\picBbarPrimeOneOneTwo[#1]{
\begin{tikzpicture}[line width=1pt,auto,scale=#1]

 \draw [-] (0,-1) to   (3,-1);
 \draw [-] (0,0) to  [out=90,in=180]   (0.5,0.4);
 \draw [-] (1,0) to  [out=90,in=0]   (0.5,0.4);
  \draw [-] (0,0) to  [out=270,in=180]   (0.5,-0.4);
 \draw [-] (1,0) to  [out=270,in=0]   (0.5,-0.4);

 \draw [-] (2,0) to (3,0);
 \draw [-](0,0) to   (0,-1);

 \fill (0,0) circle (2pt);
 \fill (1,0) circle (2pt);
 \fill (2,0) circle (2pt);
 \fill (3,0) circle (2pt);
 \fill (3,-1) circle (2pt);
 \fill (0,-1) circle (2pt);

 \draw [-](1,-0.25) to   (1,+0.25);
 \draw [-](2,-0.25) to   (2,+0.25);

 \node[below]   at(1,-0.2)    {$w$};
 \node[left]   at(0,0)    {$y$};
  \node[left]   at(0,-1)    {$x$};
 \node[right]   at(3,0)    {$v$};
 \node[right]   at(3,-1)    {$0$};
  \node[below]   at(2,-0.1)    {$w+\ve[\iota]$};
 \node[above]   at(0.5,0.4)    {$\geq 1$};
  \node[above]   at(2.5,0)    {$\geq 0$};
 \node[below]   at(1.5,-1)    {$\geq 0$};
  \node[above]   at(1.5,0)    {$=1$};
   \node[left]   at(0,-0.5)    {$=1$};
  \node[right]   at(3,-0.5)    {$=1$};
\end{tikzpicture}}

\def\picBbarPrimeOneTwoTwo[#1]{
\begin{tikzpicture}[line width=1pt,auto,scale=#1]

 \draw [-] (0,-1) to   (3,-1);
 \draw [-] (0,-1) to   (0,0);
 \draw [-] (0,0) to  [out=90,in=180]   (0.5,0.4);
 \draw [-] (1,0) to  [out=90,in=0]   (0.5,0.4);
  \draw [-] (0,0) to  [out=270,in=180]   (0.5,-0.4);
 \draw [-] (1,0) to  [out=270,in=0]   (0.5,-0.4);

 \draw [-] (2,0) to (3,0);

 \fill (0,0) circle (2pt);
 \fill (1,0) circle (2pt);
 \fill (2,0) circle (2pt);
 \fill (3,0) circle (2pt);
 \fill (3,-1) circle (2pt);
 \fill (0,-1) circle (2pt);

 \draw [-](1,-0.25) to   (1,+0.25);
 \draw [-](2,-0.25) to   (2,+0.25);

 \node[below]   at(1,-0.2)    {$w$};
 \node[below]   at(2,-0.1)    {$w+\ve[\iota]$};
 \node[left]   at(0,0)    {$y$};
 \node[left]   at(0,-1)    {$x$};
 \node[right]   at(3,0)    {$v$};
 \node[right]   at(3,-1)    {$0$};
 \node[above]   at(2.5,0)    {$\geq 0$};
 \node[below]   at(1.5,-1)    {$\geq 0$};
 \node[left]   at(0,-0.5)    {$=1$};
 \node[right]   at(3,-0.5)    {$\geq 2$};
 \node[above]   at(1.5,0)    {$=1$};
\end{tikzpicture}}

\def\picBbarPrimeTwoZeroOne[#1]{
\begin{tikzpicture}[line width=1pt,auto,scale=#1]

 \draw [-] (0,-1) to  [out=0,in=270]   (3,0);
 \draw [-] (0,0) to  [out=90,in=180]   (0.5,0.4);
 \draw [-] (1,0) to  [out=90,in=0]   (0.5,0.4);
  \draw [-] (0,0) to  [out=270,in=180]   (0.5,-0.4);
 \draw [-] (1,0) to  [out=270,in=0]   (0.5,-0.4);
  \draw [-] (0,0) to  (0,-1);

 \draw [-] (2,0) to (3,0);

 \fill (0,0) circle (2pt);
 \fill (1,0) circle (2pt);
 \fill (2,0) circle (2pt);
 \fill (3,0) circle (2pt);
 \fill (0,-1) circle (2pt);

 \draw [-](1,-0.25) to   (1,+0.25);
 \draw [-](2,-0.25) to   (2,+0.25);

 \node[above]   at(1.2,0)    {$w$};
 \node[below]   at(2,0)    {$w+\ve[\iota]$};

 \node[left]   at(0,0)    {$y$};
  \node[left]   at(0,-1)    {$x$};
    \node[left]   at(0,-0.5)    {$\geq 2$};

 \node[right]   at(3,0)    {$0$};
  \node[above]   at(2.4,0)    {$\geq 1$};
 \node[below]   at(1.5,-1)    {$\geq 0$};
\end{tikzpicture}}

\def\picBbarPrimeTwoZeroTwo[#1]{
\begin{tikzpicture}[line width=1pt,auto,scale=#1]

 \draw [-] (0.5,-2) to [out=0,in=270] (3,0);
 \draw [-] (0,0) to (1,0);
 \draw [-](2,0) to   (3,0);
  \draw [-](0,0) to   (0.5,-1);
   \draw [-](1,0) to   (0.5,-1);
% \draw [-](-0.25*0.7+1,0.25*0.7) to   (1+0.25*0.7,-0.25*0.7);
% \draw [-](+0.25*0.7+0.5,-1-0.25*0.7) to   (-0.25*0.7+0.5,-1+0.25*0.7);
   \draw [-](0.5,-1) to   (0.5,-2);

 \fill (0,0) circle (2pt);
 \fill (1,0) circle (2pt);
 \fill (2,0) circle (2pt);
 \fill (3,0) circle (2pt);

 \fill (0.5,-1) circle (2pt);

 \draw [-](1,-0.25) to   (1,+0.25);
 \draw [-](2,-0.25) to   (2,+0.25);

 \node[above]   at(1.2,0)    {$w$};
 \node[below]   at(2.2,0)    {$w+\ve[\iota]$};
 \node[left]   at(0,0)    {$y$};
  \node[left]   at(0.5,-1)    {$u$};
    \node[left]   at(0.25,-0.5)    {$\geq 1$};
    \node[left]   at(0.5,-2)    {$x$};
        \node[left]   at(0.5,-1.5)    {$\geq 1$};
  \node[right]   at(0.6,-0.7)    {$=1$};

 \node[right]   at(3,0)    {$0$};
 \node[above]   at(0.5,0)    {$\geq 1$};
  \node[above]   at(2.5,0)    {$\geq 1$};
 \node[below]   at(3,-1.2)    {$\geq 0$};
\end{tikzpicture}}

\def\picBbarPrimeTwoZeroThree[#1]{
\begin{tikzpicture}[line width=1pt,auto,scale=#1]

 \draw [-] (0.5,-2) to [out=0,in=270] (3,0);
 \draw [-] (0,0) to (1,0);
 \draw [-](2,0) to   (3,0);
  \draw [-](0,0) to   (0.5,-1);
   \draw [-](1,0) to   (0.5,-1);
    \draw [-](0.5,-1) to   (0.5,-2);

 \fill (0,0) circle (2pt);
 \fill (1,0) circle (2pt);
 \fill (2,0) circle (2pt);
 \fill (3,0) circle (2pt);

 \fill (0.5,-1) circle (2pt);

 \draw [-](1,-0.25) to   (1,+0.25);
 \draw [-](2,-0.25) to   (2,+0.25);

 \node[above]   at(1.2,0)    {$w$};
 \node[below]   at(2.2,0)    {$w+\ve[\iota]$};
 \node[left]   at(0,0)    {$y$};
  \node[left]   at(0.5,-1)    {$u$};
    \node[left]   at(0.25,-0.5)    {$\geq 1$};
    \node[left]   at(0.5,-2)    {$x$};
        \node[left]   at(0.5,-1.5)    {$\geq 1$};
 \node[right]   at(3,0)    {$0$};
  \node[right]   at(0.6,-0.7)    {$\geq 2$};

 \node[above]   at(0.5,0)    {$\geq 1$};
  \node[above]   at(2.5,0)    {$\geq 1$};
 \node[below]   at(3,-1.2)    {$\geq 0$};
\end{tikzpicture}}

\def\picBbarPrimeTwoOneOne[#1]{
\begin{tikzpicture}[line width=1pt,auto,scale=#1]

 \draw [-] (0,-1) to (3,-1);
 \draw [-] (0,0) to  [out=90,in=180]   (0.5,0.4);
 \draw [-] (1,0) to  [out=90,in=0]   (0.5,0.4);
  \draw [-] (0,0) to  [out=270,in=180]   (0.5,-0.4);
 \draw [-] (1,0) to  [out=270,in=0]   (0.5,-0.4);
  \draw [-] (0,0) to  (0,-1);

 \draw [-] (2,0) to (3,0);

 \fill (0,0) circle (2pt);
 \fill (1,0) circle (2pt);
 \fill (2,0) circle (2pt);
 \fill (3,0) circle (2pt);
  \fill (3,-1) circle (2pt);
 \fill (0,-1) circle (2pt);

 \draw [-](1,-0.25) to   (1,+0.25);
 \draw [-](2,-0.25) to   (2,+0.25);

 \node[below]   at(1.2,0)    {$w$};
 \node[below]   at(2.2,0)    {$w+\ve[\iota]$};
 \node[left]   at(0,0)    {$y$};
  \node[left]   at(0,-1)    {$x$};
   \node[right]   at(3,-1)    {$0$};
    \node[left]   at(0,-0.5)    {$\geq 2$};

 \node[right]   at(3,0)    {$v$};
 \node[right]   at(3,-.5)    {$=1$};
  \node[above]   at(2.5,0)    {$\geq 0$};
 \node[below]   at(1.5,-1)    {$\geq 0$};
\end{tikzpicture}}

\def\picBbarPrimeTwoOneThree[#1]{
\begin{tikzpicture}[line width=1pt,auto,scale=#1]

 \draw [-] (0.5,-2) to [out=0,in=270] (3,-1);
 \draw [-] (0,0) to (1,0);
 \draw [-](2,0) to   (3,0);
  \draw [-](0,0) to   (0.5,-1);
   \draw [-](0.5,-1) to   (0.5,-2);
   \draw [-](1,0) to   (0.5,-1);
 \fill (0,0) circle (2pt);
 \fill (1,0) circle (2pt);
 \fill (2,0) circle (2pt);
 \fill (3,0) circle (2pt);
  \fill (3,-1) circle (2pt);
 \fill (0.5,-1) circle (2pt);

 \draw [-](1,-0.25) to   (1,+0.25);
 \draw [-](2,-0.25) to   (2,+0.25);

 \node[above]   at(1.2,0)    {$w$};
 \node[below]   at(2.2,0)    {$w+\ve[\iota]$};
 \node[left]   at(0,0)    {$y$};
  \node[left]   at(0.5,-1)    {$u$};
   \node[right]   at(3,-1)    {$0$};
  \node[right]   at(0.6,-.7)    {$=1$};

    \node[left]   at(0.25,-0.5)    {$\geq 1$};
    \node[left]   at(0.5,-2)    {$x$};
        \node[left]   at(0.5,-1.5)    {$\geq 1$};
 \node[right]   at(3,0)    {$v$};
 \node[right]   at(3.2,-.5)    {$=1$};
 \node[above]   at(0.5,0)    {$\geq 1$};
  \node[above]   at(2.5,0)    {$\geq 0$};
 \node[below]   at(3,-1.5)    {$\geq 0$};
\end{tikzpicture}}

\def\picBbarPrimeTwoOneFour[#1]{
\begin{tikzpicture}[line width=1pt,auto,scale=#1]

 \draw [-] (0.5,-2) to [out=0,in=270] (3,-1);
 \draw [-] (0,0) to (1,0);
 \draw [-](2,0) to   (3,0);
  \draw [-](0,0) to   (0.5,-1);
   \draw [-](1,0) to   (0.5,-1);
    \draw [-](0.5,-1) to   (0.5,-2);

 \fill (0,0) circle (2pt);
 \fill (1,0) circle (2pt);
 \fill (2,0) circle (2pt);
 \fill (3,0) circle (2pt);
  \fill (3,-1) circle (2pt);
 \fill (0.5,-1) circle (2pt);

 \draw [-](1,-0.25) to   (1,+0.25);
 \draw [-](2,-0.25) to   (2,+0.25);

 \node[above]   at(1.2,0)    {$w$};
 \node[below]   at(2.2,0)    {$w+\ve[\iota]$};
 \node[left]   at(0,0)    {$y$};
  \node[left]   at(0.5,-1)    {$u$};
   \node[right]   at(3,-1)    {$0$};
    \node[left]   at(0.25,-0.5)    {$\geq 1$};
    \node[left]   at(0.5,-2)    {$x$};
        \node[left]   at(0.5,-1.5)    {$\geq 1$};
 \node[right]   at(3,0)    {$v$};
 \node[right]   at(3.2,-.5)    {$=1$};

  \node[right]   at(0.6,-0.75)    {$\geq 2$};

 \node[above]   at(0.5,0)    {$\geq 1$};
  \node[above]   at(2.5,0)    {$\geq 0$};
 \node[below]   at(3,-1.5)    {$\geq 0$};
\end{tikzpicture}}

\def\picBbarPrimeTwoTwoOne[#1]{
\begin{tikzpicture}[line width=1pt,auto,scale=#1]

 \draw [-] (0,-1) to   (3,-1);
 \draw [-] (0,0) to  [out=90,in=180]   (0.5,0.4);
 \draw [-] (1,0) to  [out=90,in=0]   (0.5,0.4);
  \draw [-] (0,0) to  [out=270,in=180]   (0.5,-0.4);
 \draw [-] (1,0) to  [out=270,in=0]   (0.5,-0.4);
  \draw [-] (0,0) to  (0,-1);

 \draw [-] (2,0) to (3,0);

 \fill (0,0) circle (2pt);
 \fill (1,0) circle (2pt);
 \fill (2,0) circle (2pt);
 \fill (3,0) circle (2pt);
 \fill (0,-1) circle (2pt);
  \fill (3,-1) circle (2pt);
 \draw [-](1,-0.25) to   (1,+0.25);
 \draw [-](2,-0.25) to   (2,+0.25);

 \node[below]   at(1.2,0)    {$w$};
 \node[below]   at(2.2,0)    {$w+\ve[\iota]$};

 \node[left]   at(0,0)    {$y$};
 \node[left]   at(0,-1)    {$x$};
 \node[right]   at(3,-1)    {$0$};
 \node[left]   at(0,-0.5)    {$\geq 2$};
 \node[right]   at(3,-0.5)    {$\geq 2$};
 \node[right]   at(3,0)    {$v$};
  \node[above]   at(2.5,0)    {$\geq 0$};
 \node[below]   at(1.5,-1)    {$\geq 0$};
\end{tikzpicture}}

\def\picBbarPrimeTwoTwoThree[#1]{
\begin{tikzpicture}[line width=1pt,auto,scale=#1]

 \draw [-] (0.5,-2) to [out=0,in=270] (3,-1.2);
 \draw [-] (0,0) to (1,0);
 \draw [-](2,0) to   (3,0);
  \draw [-](0,0) to   (0.5,-1);
  \draw [-](1,0) to   (0.5,-1);
   \draw [-](0.5,-1) to   (0.5,-2);

 \fill (0,0) circle (2pt);
 \fill (1,0) circle (2pt);
 \fill (2,0) circle (2pt);
 \fill (3,0) circle (2pt);
  \fill (3,-1.2) circle (2pt);
 \fill (0.5,-1) circle (2pt);

 \draw [-](1,-0.25) to   (1,+0.25);
 \draw [-](2,-0.25) to   (2,+0.25);

 \node[above]   at(1.2,0)    {$w$};
 \node[below]   at(2.2,0)    {$w+\ve[\iota]$};
 \node[left]   at(0,0)    {$y$};
 \node[left]   at(0.5,-1)    {$u$};
  \node[right]   at(3,-1.2)    {$0$};
 \node[left]   at(0.25,-0.5)    {$\geq 1$};
 \node[left]   at(0.5,-2)    {$x$};
 \node[left]   at(0.5,-1.5)    {$\geq 1$};
 \node[right]   at(3,-0.6)    {$\geq 2$};
 \node[right]   at(3,0)    {$v$};
 \node[above]   at(0.5,0)    {$\geq 1$};
  \node[above]   at(2.5,0)    {$\geq 0$};
 \node[below]   at(2.1,-1.2)    {$\geq 0$};
  \node[right]   at(0.5,-0.8)    {$=1$};
\end{tikzpicture}}

\def\picBbarPrimeTwoTwoFour[#1]{
\begin{tikzpicture}[line width=1pt,auto,scale=#1]

 \draw [-] (0.5,-2) to [out=0,in=270] (3,-1.2);
 \draw [-] (0,0) to (1,0);
 \draw [-](2,0) to   (3,0);
  \draw [-](0,0) to   (0.5,-1);
   \draw [-](1,0) to   (0.5,-1);
    \draw [-](0.5,-1) to   (0.5,-2);

 \fill (0,0) circle (2pt);
 \fill (1,0) circle (2pt);
 \fill (2,0) circle (2pt);
 \fill (3,0) circle (2pt);
  \fill (3,-1.2) circle (2pt);
 \fill (0.5,-1) circle (2pt);

 \draw [-](1,-0.25) to   (1,+0.25);
 \draw [-](2,-0.25) to   (2,+0.25);

 \node[above]   at(1.2,0)    {$w$};
 \node[below]   at(2.2,0)    {$w+\ve[\iota]$};

 \node[left]   at(0,0)    {$y$};
 \node[left]   at(0.5,-1)    {$u$};
 \node[right]   at(3,-1.2)    {$0$};
 \node[left]   at(0.25,-0.5)    {$\geq 1$};
 \node[left]   at(0.5,-2)    {$x$};
 \node[left]   at(0.5,-1.5)    {$\geq 1$};
 \node[right]   at(3,0)    {$v$};
 \node[right]   at(0.5,-0.8)    {$\geq 2$};
 \node[right]   at(3,-0.6)    {$\geq 2$};

 \node[above]   at(0.5,0)    {$\geq 1$};
 \node[above]   at(2.5,0)    {$\geq 0$};
 \node[below]   at(2.1,-1.2)    {$\geq 0$};
\end{tikzpicture}}

\def\picDelta[#1]{
\begin{tikzpicture}[line width=1pt,auto,scale=#1]
  \fill (2,0) circle (3pt);
  \fill (0,1.5) circle (3pt);
  \fill (2,1.5) circle (3pt);
  \fill (0,0) circle (3pt);
  \node[left]   at(0,0)      {$0$};
  \node[left]   at(0,1.5)      {$v$};
  \node[left]   at(0,0.75)      {$a$};
  \node[right]   at(2,1.5)    {$y$};
  \node[right]   at(2,0.75)    {$b$};
  \node[right]   at(2,0)      {$x$};
  \node[below] at (1,-0.2) {$\|x\|_2^2$};
  \draw [line width=2pt,-] (0,0) to  (2,0);
  \draw [-] (0,1.5) to  (2,1.5);
\end{tikzpicture}}

\def\picDeltaIotaTwo[#1]{
\begin{tikzpicture}[line width=1pt,auto,scale=#1]
  \fill (2,0) circle (3pt);
  \fill (0,1.5) circle (3pt);
  \fill (2,1.5) circle (3pt);
  \fill (0.5,0) circle (3pt);
  \fill (0,0) circle (3pt);
  \node[left]   at(0,0)      {$0$};
  \node[above]   at(0.5,0)  {$e_\iota$};
  \node[left]   at(0,1.5)      {$v$};
  \node[left]   at(0,0.75)      {$a$};
  \node[right]   at(2,1.5)    {$y$};
  \node[right]   at(2,0.75)    {$b$};
  \node[right]   at(2,0)      {$x$};
  \node[below] at (1,-0.2) {$\|x\|_2^2$};
  \draw [-] (0.5,0-0.2) to  (0.5,0.2);
  \draw [-](0,0-0.2) to (0,0+0.2);
  \draw [line width=2pt,-] (0.5,0) to  (2,0);
  \draw [-] (0,1.5) to  (2,1.5);
\end{tikzpicture}}

\def\picDeltaIotaThree[#1]{
\begin{tikzpicture}[line width=1pt,auto,scale=#1]
  \fill (2,0) circle (3pt);
  \fill (0,1.5) circle (3pt);
  \fill (2,1.5) circle (3pt);
  \fill (0,0) circle (3pt);
  \node[left]   at(0,0)      {$0$};
  \node[left]   at(0,1.5)      {$v$};
  \node[left]   at(0,0.75)      {$a$};
  \node[above]   at(0.5,1.5)      {$e_\iota$};
  \node[right]   at(2,1.5)    {$y$};
  \node[right]   at(2,0.75)    {$b$};
  \node[right]   at(2,0)      {$x$};
    \node[below] at (1,-0.2) {$\|x\|_2^2$};
  \draw [line width=2pt,-] (0,0) to  (2,0);
  \draw [-] (0.5,1.5) to  (2,1.5);
 \draw [-](0,1.5-0.2) to (0,1.5+0.2);
 \draw [-](0.5-0.03,1.5-0.2) to   (0.5-0.03,1.5+0.2);
\end{tikzpicture}}

\def\picWeightConstructOne[#1]{
\begin{tikzpicture}[line width=1pt,auto,scale=#1]
\begin{scope}[shift={(0,0)},rotate=0]
  \node[left] at (-2,0) {$\sum_{t,w,z,u}$};
  \node[left] at (-1,0) {$a$};
 \draw [-](-1,1) to  (2,1);
 \draw [line width=3pt,-](-0.5,-1) to (2,-1);
 \draw [-](-1+0.03,-1-0.2) to (-1+0.03,-1+0.2);
 \draw [-](-1+0.5-0.03,-1-0.2) to   (-1+0.5-0.03,-1+0.2);
 \draw [-](1-0.5,1) to  (1-0.5,-1);
 \draw [-](2,1) to  (1,1);
 \draw [-](2,1) to  (2,-1);
 \draw [-](3,-1) to  (2,-1);
  \draw [-](3,1) to  (2.5,1);
 \draw [-](2.03,1-0.2) to (2.03,1+0.2);
 \draw [-](2.5-0.03,1-0.2) to   (2.5-0.03,1+0.2);
 \node[right] at (3,1) {$y$};
 \node[right] at (3,-1) {$x$};
 \node[right] at (3,0) {$b$};
 \node[left] at (-1,1) {$v$};
 \node[left] at (-1,-1) {$0$};

 \node[below] at (1-0.5,-1) {$t$};
 \node[below] at (2,-1) {$w$};
 \node[below] at (4,-1.5) {$\|w\|_2^2$};
 \node[above] at (2,1) {$u$};
 \node[above] at (1-0.5,1) {$z$};
\end{scope}

\begin{scope}[shift={(4.5,0)},rotate=0]
 \draw [-](0,-0.2) to  (0,0.2);
 \draw [-](-0.2,0) to  (0.2,0);
 \end{scope}

\begin{scope}[shift={(6.5,0)},rotate=0]

 \draw [-](0.03-0.5,-1-0.2) to (0.03-0.5,-1+0.2);
 \draw [-](0.5-0.03-0.5,-1-0.2) to   (0.5-0.03-0.5,-1+0.2);
 \draw [-](0-0.5,1) to  (1.5-0.5,1);
  \node[left] at (0-0.5,1) {$v$};
  \node[left] at (0-0.5,0) {$a$};
  \node[left] at (0-0.5,-1) {$0$};
 \draw [-](1,-1) to  (1,1-0.5);
 \draw [line width=3pt,-](0,-1) to  (1,-1);
 \fill[gray] (0.5,1) -- (1.5,1)--(1,1-0.5);
 \draw [-](0.5,1) to  (1.5,1);
 \draw [-](0.5,1) to  (1,1-0.5);
 \draw [-](1.5,1) to  (1,1-0.5);
  \draw [-](2.03,1-0.2) to (2.03,1+0.2);
 \draw [-](1.5-0.03,1-0.2) to   (1.5-0.03,1+0.2);
 \draw [-](2,1) to  (2.5,1);
  \draw [-](1,-1) to  (2.5,-1);
  \node[right] at (3,0) {$b$};
  \node[right] at (3,1) {$y$};
  \node[right] at (3,-1) {$x$};
  \node[below] at (1,-1) {$w$};
  \node[above] at (0.5,1) {$z$};
  \node[left] at (1,0.4) {$t$};
  \node[above] at (1.5,1) {$u$};
  \node[below] at (2.2,1) {$u+e_\kappa$};
\end{scope}

\end{tikzpicture}}

\def\picWeightConstructTwo[#1]{
\begin{tikzpicture}[line width=1pt,auto,scale=#1]
  \node[left] at (-2,0) {$\sum_{t,w,z,u}$};

\begin{scope}[shift={(0,0)},rotate=0]
  \draw [-](2,-1-0.2) to (2,-1+0.2);
  \draw [-](2.5-0.03,-1-0.2) to   (2.5-0.03,-1+0.2);
  \draw [-](-1,1-0.2) to (-1,1+0.2);
  \draw [-](-1+0.5,1-0.2) to   (-1+0.5,1+0.2);
  \draw [line width=3pt,-](-1,-1) to  (2,-1);
  \draw [-](2.5,-1) to  (3,-1);
  \draw [-](-1+0.5,1) to  (3,1);
  \draw [-](2,1) to  (2,-1);
  \draw [-](1-0.5,1) to  (1-0.5,-1);
  \node[left] at (-1,0) {$a$};
  \node[left] at (-1,-1) {$0$};
  \node[left] at (-1,1) {$v$};
  \node[right] at (3,0) {$b$};
  \node[right] at (3,1) {$y$};
  \node[right] at (3,-1) {$x$};
  \node[above] at (1-0.5,1) {$t$};
  \node[above] at (2,1) {$w$};
  \node[below] at (1-0.5,-1) {$z$};
  \node[below] at (2,-1) {$u$};
  \node[below] at (4,-1.5) {$\|u\|_2^2$};
\end{scope}

\begin{scope}[shift={(4.5,0)},rotate=0]
 \draw [-](0,-0.2) to  (0,0.2);
 \draw [-](-0.2,0) to  (0.2,0);
 \end{scope}
\begin{scope}[shift={(6,0)},rotate=0]

 \draw [-](2,-1-0.2) to (2,-1+0.2);
 \draw [-](2.5,-1-0.2) to   (2.5,-1+0.2);
 \draw [-](0,1-0.2) to (0,1+0.2);
 \draw [-](0.5,1-0.2) to   (0.5,1+0.2);
 \draw [-](1.5,1) to  (1.5,-1+0.5);
 \draw [-](0.5,1) to  (3,1);
 \draw [-](2.5,-1) to  (3,-1);
 \fill[gray] (1,-1) -- (2,-1)--(1.5,-1+0.5);
 \draw [line width=3pt,-](0,-1) to  (2,-1);
 \draw [-](1,-1) to  (1.5,-0.5);
 \draw [-](2,-1) to  (1.5,-0.5);
  \node[left] at (0,0) {$a$};
  \node[left] at (0,-1) {$0$};
  \node[left] at (0,1) {$v$};

  \node[right] at (3,0) {$b$};
  \node[right] at (3,1) {$y$};
  \node[right] at (3,-1) {$x$};
  \node[left] at (1.5,-0.3) {$t$};
  \node[above] at (1.5,1) {$w$};
  \node[below] at (1,-1) {$z$};
  \node[below] at (2,-1) {$u$};
\end{scope}
\end{tikzpicture}}

\def\picPiIota[#1]{
\begin{tikzpicture}[line width=1pt,auto,scale=#1]
\begin{scope}[shift={(0,0)},rotate=0]
  \node[left]   at(0,0)      {$0$};
  \node[below] at(1-0.1,0)      {$e_\iota$};
  \node[right]   at(2,1)      {$y$};
  \node[right]   at(2,-1)      {$x$};
  \node[right]   at(2,0)      {$b$};
  \path[draw] (0,0) to [out=90,in=180] (0.5,0.8) to [out=0,in=90]  (1,0);
  \begin{scope}[shift={(1,0)}]
      \path[draw,fill=black!30!white] (1,-1) to [out=180,in=-40] (0,0) to [out=40,in=180] (1,1) to (1,-1);
      \draw [line width=1pt,-](1,-1) to [out=180,in=-40] (0,0);
  \end{scope}
\end{scope}

\begin{scope}[shift={(3,0)},rotate=0]
 \draw [-](0,-0.2) to  (0,0.2);
 \draw [-](-0.2,0) to  (0.2,0);
 \end{scope}

\begin{scope}[shift={(4,-0.5)},rotate=0]
  \node[left]   at(0,0)      {$0$};
  \node[below] at(1-0.1,0)      {$e_\iota$};
%  \node[right]   at(0.5,0.7)      {$\geq 1$};
    \node[right]   at(2,1.5)      {$y$};
  \node[right]   at(2,0)      {$x$};
%  \path[draw] (0,0) to [out=90,in=180] (0.5,0.8) to [out=0,in=90]  (1,0);

  \node[right]   at(2,0.75)      {$b$};
  \path[draw] (0,0) to [out=90,in=180]  (2,1.5) to [out=210,in=90]  (1,0);
%  \path[draw,dashed] (2,0) to [out=90,in=270]  (2,1.5);
  \begin{scope}[shift={(1,0)}]
  \path[draw,fill=black!30!white] (0,0) to [out=60,in=180] (0.5,0.5) to [out=0,in=120]  (1,0) to [out=240,in=0] (0.5,-0.5) to [out=180,in=-60](0,0);
  \path[line width=1pt,draw] (1,0) to [out=240,in=0] (0.5,-0.5) to [out=180,in=-60](0,0);
  \end{scope}
 \end{scope}

\end{tikzpicture}}

\def\picPiIotaOne[#1]{
\begin{tikzpicture}[line width=1pt,auto,scale=#1]
\begin{scope}[shift={(4.5,0)},rotate=0]
  \node[left]   at(0,0)      {$0$};
  \node[below] at(1-0.1,0)      {$e_\iota$};
  \node[above]   at(3.5,1)      {$z$};
  \node[below]   at(3.5,-1)      {$t$};
   \node[below]   at(2,-1.1)      {$u$};
   \node[above]   at(2.65,-1)      {$u+\ve[\kappa]$};
   \node[above]   at(2,1)      {$w$};
  \node[right]   at(4.5,0)      {$x$};
  \path[draw] (0,0) to [out=90,in=180] (0.5,0.8) to [out=0,in=90]  (1,0);
  \begin{scope}[shift={(1,0)}]
      \path[draw,fill=black!30!white] (1,-1) to [out=180,in=-40] (0,0) to [out=40,in=180] (1,1) to (1,-1);
      \draw [line width=1pt,-](1,-1) to [out=180,in=-40] (0,0);
  \end{scope}
  \begin{scope}[shift={(2,0)}]
         \draw [-](0.03,-1-0.2) to (0.03,-1+0.2);
         \draw [-](0.5-0.03,-1-0.2) to   (0.5-0.03,-1+0.2);
   \end{scope}
    \draw [-](2,1) to  (3.5,1);
    \draw [dashed](2.5,-1) to (3.5,-1);
     \path[draw,dashed,fill=black!30!white] (3.5,-1) to (3.5,1) to (4.5,0) to (3.5,-1);
 \end{scope}

 \begin{scope}[shift={(9.75,0)},rotate=0]
 \draw [-](0,-0.2) to  (0,0.2);
 \draw [-](-0.2,0) to  (0.2,0);
 \end{scope}

\begin{scope}[shift={(10.5,0)},rotate=0]
  \node[left]   at(0,0)      {$0$};
  \node[below] at(0.25,-1)      {$e_\iota$};
  \node[right]   at(0.5,0)      {$\geq 1$};
  \node[above]   at(2.5,1)      {$z$};
  \node[above]   at(0.5,1)      {$w$};
  \node[below]   at(2.5,-1)      {$t$};
  \node[right]   at(3.5,0)      {$x$};
  \node[below]   at(1.55,-1.1)      {$u$};
  \node[above]   at(1.9,-1+0.05)      {$u+\ve[\kappa]$};
  \path[draw] (0,0) to [out=90,in=180] (0.5,1) to  (0.5,-1);
  \begin{scope}[shift={(0.5,-1)}]
  \path[draw,fill=black!30!white] (0,0) to [out=60,in=180] (0.5,0.5) to [out=0,in=120]  (1,0) to [out=240,in=0] (0.5,-0.5) to [out=180,in=-60](0,0);
  \path[line width=1pt,draw] (1,0) to [out=240,in=0] (0.5,-0.5) to [out=180,in=-60](0,0);
  \end{scope}
  \begin{scope}[shift={(1.5,0)}]
         \draw [-](0.03,-1-0.2) to (0.03,-1+0.2);
         \draw [-](0.5-0.03,-1-0.2) to   (0.5-0.03,-1+0.2);
   \end{scope}
    \draw [-](0.5,1) to  (2.5,1);
    \draw [dashed] (2,-1) to (2.5,-1);
     \path[draw,dashed,fill=black!30!white] (2.5,-1) to (2.5,1) to (3.5,0) to (2.5,-1);

 \end{scope}

\begin{scope}[shift={(15,0)},rotate=0]
 \draw [-](0,-0.2) to  (0,0.2);
 \draw [-](-0.2,0) to  (0.2,0);
 \end{scope}

\begin{scope}[shift={(15.5,0)},rotate=0]
  \node[left] at (0+0.1,-1) {$0=w=u$};
  \node[below] at (2,-1) {$t$};
  \node[above] at (2,1) {$z$};
  \node[right] at (3,0) {$x$};
 \draw [dashed](0.5,-1) to (2,-1);
 \draw [-](0.03,-1-0.2) to (0.03,-1+0.2);
 \draw [-](0.5-0.03,-1-0.2) to   (0.5-0.03,-1+0.2);
  \path[draw,dashed,fill=black!30!white] (2,-1) to (2,1) to (3,0) to (2,-1);

 \node[below] at(0.5,-1.1)      {$e_\iota$};
 \draw [-](0,-1) to [out=90,in=180] (2,1);
\end{scope}

\end{tikzpicture}}

\def\picWeightPiIota[#1]{
\begin{tikzpicture}[line width=1pt,auto,scale=#1]
\begin{scope}[shift={(-1,0)},rotate=0]
%  \node[left] at (0,0) {$\sum_{w}$};
  \node[right] at (2,-1) {$x$};
  \node[right] at (2,1) {$y$};
  \node[right] at (2,0) {$b$};
 \draw [line width=3pt,-](0.5,-1) to (2,-1);
 \draw [-](0.03,-1-0.2) to (0.03,-1+0.2);
 \draw [-](0.5-0.03,-1-0.2) to   (0.5-0.03,-1+0.2);
 \node[below] at(0.5,-1.2)      {$e_\iota$};
 \draw [-](0,-1) to [out=90,in=180] (2,1);
\end{scope}

 \node[below] at (6,-1.5) {$\|x-\ve[\iota]\|_2^2$};

\begin{scope}[shift={(2.5,0)},rotate=0]
 \draw [-](0,-0.2) to  (0,0.2);
 \draw [-](-0.2,0) to  (0.2,0);
 \end{scope}

\begin{scope}[shift={(4,0)},rotate=0]
  \node[left]   at(0,0)      {$0$};
  \node[below] at(1-0.1,0)      {$e_\iota$};
  \node[right]   at(3.5,1)      {$y$};
  \node[right]   at(3.5,-1)      {$x$};
  \node[right]   at(3.5,0)      {$b$};
  \path[draw] (0,0) to [out=90,in=180] (0.5,0.8) to [out=0,in=90]  (1,0);
  \begin{scope}[shift={(1,0)}]
      \path[draw,fill=black!30!white] (1,-1) to [out=180,in=-40] (0,0) to [out=40,in=180] (1,1) to (1,-1);
      \draw [line width=3pt,-](1,-1) to [out=180,in=-40] (0,0);
  \end{scope}
  \begin{scope}[shift={(2,0)}]
         \draw [-](0.03,-1-0.2) to (0.03,-1+0.2);
         \draw [-](0.5-0.03,-1-0.2) to   (0.5-0.03,-1+0.2);
   \end{scope}
    \draw [-](2,1) to  (3.5,1);
    \draw [line width=3pt,-](2.5,-1) to (3.5,-1);
 \end{scope}

 \begin{scope}[shift={(9,0)},rotate=0]
 \draw [-](0,-0.2) to  (0,0.2);
 \draw [-](-0.2,0) to  (0.2,0);
 \end{scope}

\begin{scope}[shift={(10.5,0)},rotate=0]
  \node[left]   at(0,0)      {$0$};
  \node[below] at(0.25,-1)      {$e_\iota$};
  \node[right]   at(0.5,0)      {$\geq 1$};
    \node[right]   at(3,1)      {$y$};
  \node[right]   at(3,-1)      {$x$};
  \node[right]   at(3,0)      {$b$};
  \path[draw] (0,0) to [out=90,in=180] (0.5,1) to  (0.5,-1);
  \begin{scope}[shift={(0.5,-1)}]
  \path[draw,fill=black!30!white] (0,0) to [out=60,in=180] (0.5,0.5) to [out=0,in=120]  (1,0) to [out=240,in=0] (0.5,-0.5) to [out=180,in=-60](0,0);
  \path[line width=3pt,draw] (1,0) to [out=240,in=0] (0.5,-0.5) to [out=180,in=-60](0,0);
  \end{scope}
  \begin{scope}[shift={(1.5,0)}]
         \draw [-](0.03,-1-0.2) to (0.03,-1+0.2);
         \draw [-](0.5-0.03,-1-0.2) to   (0.5-0.03,-1+0.2);
   \end{scope}
    \draw [-](0.5,1) to  (3,1);
    \draw [line width=3pt,-](2,-1) to (3,-1);
 \end{scope}

\end{tikzpicture}}

\def\picWeightConstructOneBound[#1]{
\begin{tikzpicture}[line width=1pt,auto,scale=#1]

  \node[below] at (-2,0.5) {$(C^{1})_{a,b}\geq $};

\begin{scope}[shift={(-1,0)},rotate=0]
  \node[below] at (-0.2,0) {$a$};
 \draw [-](0,1) to  (2,1);
  \draw [dotted](0,1) to  (0,-1);
 \draw [line width=3pt,-](0.5,-1) to  (2,-1);
 \draw [-](0.03,-1-0.2) to (0.03,-1+0.2);
 \draw [-](0.5-0.03,-1-0.2) to   (0.5-0.03,-1+0.2);
 \draw [-](1,1) to  (1,-1);
 \draw [-](2,1) to  (1,1);
 \draw [-](2,1) to  (2,-1);
 \draw [-](3,-1) to  (2,-1);
  \draw [-](3,1) to  (2.5,1);
  \draw [dotted](3,1) to  (3,-1);
 \draw [-](2.03,1-0.2) to (2.03,1+0.2);
 \draw [-](2.5-0.03,1-0.2) to   (2.5-0.03,1+0.2);
   \node[below] at (2.8,0.7) {$b$};
\end{scope}

\begin{scope}[shift={(3,0)},rotate=0]
 \draw [-](0,-0.2) to  (0,0.2);
 \draw [-](-0.2,0) to  (0.2,0);
 \end{scope}

\begin{scope}[shift={(4,0)},rotate=0]
   \node[below] at (0.3,-0.2) {$a$};
       \draw [dotted](0,1) to  (0,-1);
\draw [-](0.03,-1-0.2) to (0.03,-1+0.2);
 \draw [-](0.5-0.03,-1-0.2) to   (0.5-0.03,-1+0.2);
 \draw [-](0,1) to  (1,1);
 \draw [-](1,-1) to  (1,1-0.5);
 \draw [line width=3pt,-](0.5,-1) to  (1,-1);
 \fill[gray] (0.5,1) -- (1.5,1)--(1,1-0.5);
 \draw [-](0.5,1) to  (1.5,1);
 \draw [-](0.5,1) to  (1,1-0.5);
 \draw [-](1.5,1) to  (1,1-0.5);
  \draw [-](2.03,1-0.2) to (2.03,1+0.2);
 \draw [-](1.5-0.03,1-0.2) to   (1.5-0.03,1+0.2);
 \draw [-](2,1) to  (2.5,1);
  \draw [-](1,-1) to  (2.5,-1);
    \draw [dotted](2.5,1) to  (2.5,-1);
   \node[below] at (2.7,0.7) {$b$};
\end{scope}

\end{tikzpicture}}

\def\picWeightConstructTwoBound[#1]{
\begin{tikzpicture}[line width=1pt,auto,scale=#1]
  \node[below] at (-6.5,0.5) {$(C^{2})_{a,b}\geq $};
  \begin{scope}[shift={(-4,0)},rotate=0]
\draw [dotted](-1,-1) to (-1,1);
 \node[below] at (-1-0.3,0) {$a$};
\draw [-](1,-1-0.2) to (1,-1+0.2);
 \draw [-](1.5-0.03,-1-0.2) to   (1.5-0.03,-1+0.2);
 \draw [-](-1+.03,1-0.2) to (-1+.03,1+0.2);
 \draw [-](-0.5-0.03,1-0.2) to   (-0.5-0.03,1+0.2);
 \draw [line width=2pt,-](-1,-1) to  (1,-1);
 \draw [-](2,-1) to  (1.5,-1);
 \draw [-](-0.5,1) to  (2,1);
  \draw [-](0,1) to  (0,-1);
   \draw [-](1,1) to  (1,-1);

  \draw [dotted](2,1) to  (2,-1);
   \node[below] at (1.8,0.7) {$b$};

\end{scope}
\begin{scope}[shift={(-1.5,0)},rotate=0]
 \draw [-](0,-0.2) to  (0,0.2);
 \draw [-](-0.2,0) to  (0.2,0);
 \end{scope}
\begin{scope}[shift={(0,0)},rotate=0]
   \node[below] at (-1+0.3,0.1) {$a$};
       \draw [dotted](-1+0,1) to  (-1+0,-1);
\draw [-](1.03,-1-0.2) to (1.03,-1+0.2);
 \draw [-](1.5-0.03,-1-0.2) to   (1.5-0.03,-1+0.2);
 \draw [-](-1+.03,1-0.2) to (-1+.03,1+0.2);
 \draw [-](-0.5-0.03,1-0.2) to   (-0.5-0.03,1+0.2);
 \draw [-](0.5,1) to  (0.5,-1+0.5);
 \draw [-](-0.5,1) to  (0.5,1);
 \draw [-](0.5,1) to  (2,1);
 \draw [-](1.5,-1) to  (2,-1);
 \fill[gray] (0,-1) -- (1,-1)--(0.5,-1+0.5);
  \draw [line width=2pt,-](-1,-1) to  (1,-1);
 \draw [-](0,-1) to  (0.5,-0.5);
 \draw [-](1,-1) to  (0.5,-0.5);
   \draw [dotted](2,1) to  (2,-1);
   \node[below] at (1.8,0.7) {$b$};
\end{scope}
\end{tikzpicture}}

\def\picPIotaZero[#1]{
\begin{tikzpicture}[line width=1pt,auto,scale=#1]
  \node[left]   at(0,0)      {$0$};
  \node[below] at(1.5,0)      {$e_\iota$};
  \node[right]   at(3.5,0)      {$x$};
  \path[draw] (0,0) to [out=90,in=180] (0.75,0.8) to [out=0,in=90]  (1.5,0);
  \begin{scope}[shift={(1.5,0)}]
  \path[draw,fill=black!30!white] (0,0) to [out=40,in=180] (1,0.5) to [out=0,in=140]  (2,0) to [out=220,in=0] (1,-0.5) to [out=180,in=-40](0,0);
  \end{scope}
  \fill (0,0) circle (3pt);
  \fill (2+1.5,0) circle (3pt);
  \fill (2+1.5,0) circle (3pt);
\end{tikzpicture}}

\def\picPIotaOne[#1]{
\begin{tikzpicture}[line width=1pt,auto,scale=#1]
  \node[left]   at(0,0)      {$0$};
  \node[below] at(1.5,0)      {$e_\iota$};
  \node[right]   at(2+1.5,0.6)      {$y$};
  \node[right]   at(2+1.5,0)      {$=1$};
  \node[right]   at(2+1.5,-0.6)      {$x$};
  \path[draw] (0,0) to [out=90,in=180] (0.75,0.8) to [out=0,in=90]  (1.5,0);
  \begin{scope}[shift={(1.5,0)}]
      \path[draw] (2,-0.6) to [out=180,in=-40] (0,0) to [out=40,in=180] (2,0.6) to (2,-0.6);
  \end{scope}
  \fill (0,0) circle (3pt);
  \fill (1.5,0) circle (3pt);
  \fill (2+1.5,0.6) circle (3pt);
  \fill (2+1.5,-0.6) circle (3pt);
\end{tikzpicture}}

\def\picPIotaOneTwo[#1]{
\begin{tikzpicture}[line width=1pt,auto,scale=#1]
  \node[left]   at(0,0)      {$0$};
  \node[below] at(2,0)      {$e_\iota$};
  \node[right]   at(2,1)      {$y$};
  \node[right]   at(2,0.5)      {$=1$};
  \path[draw] (0,0) to [out=90,in=180] (2,1) to (2,0);
  \fill (0,0) circle (3pt);
  \fill (2,0) circle (3pt);
  \fill (2,1) circle (3pt);
\end{tikzpicture}}

\def\picPIotaTwo[#1]{
\begin{tikzpicture}[line width=1pt,auto,scale=#1]
  \node[left]   at(0,0)      {$0$};
  \node[below] at(1.5,0)      {$e_\iota$};
  \node[right]   at(2+1.5,0.6)      {$y$};
  \node[right]   at(2+1.5,0)      {$\geq 2$};
  \node[right]   at(2+1.5,-0.6)      {$x$};
  \path[draw] (0,0) to [out=90,in=180] (0.75,0.8) to [out=0,in=90]  (1.5,0);
  \begin{scope}[shift={(1.5,0)}]
      \path[draw] (2,-0.6) to [out=180,in=-40] (0,0) to [out=40,in=180] (2,0.6) to (2,-0.6);
  \end{scope}
  \fill (0,0) circle (3pt);
  \fill (1.5,0) circle (3pt);
  \fill (2+1.5,0.6) circle (3pt);
  \fill (2+1.5,-0.6) circle (3pt);
\end{tikzpicture}}

\def\picPIotaTwoTwo[#1]{
\begin{tikzpicture}[line width=1pt,auto,scale=#1]
  \node[left]   at(0,0)      {$0$};
  \node[below] at(2,0)      {$e_\iota$};
  \node[right]   at(2,1)      {$y$};
  \node[right]   at(2,0.5)      {$\geq 2$};
  \path[draw] (0,0) to [out=90,in=180] (2,1) to (2,0);
  \fill (0,0) circle (3pt);
  \fill (2,0) circle (3pt);
  \fill (2,1) circle (3pt);
\end{tikzpicture}}

\def\picPIotaTwoThree[#1]{
\begin{tikzpicture}[line width=1pt,auto,scale=#1]
  \node[left]   at(0,0)      {$0$};
  \node[right] at(2,1)      {$y$};
  \node[below] at(2,0)      {$e_\iota$};
  \node[right]   at(4,0)      {$x$};
  \path[draw] (0,0) to [out=90,in=180] (2,1) to  (2,0);
  \begin{scope}[shift={(2,0)}]
  \path[draw] (0,0) to [out=40,in=180] (1,0.5) to [out=0,in=140]  (2,0) to [out=220,in=0] (1,-0.5) to [out=180,in=-40](0,0);
  \end{scope}
  \fill (0,0) circle (3pt);
  \fill (2,1) circle (3pt);
  \fill (2,0) circle (3pt);
  \fill (4,0) circle (3pt);
\end{tikzpicture}}

\def\picIntersectionSausage[#1]{
\begin{tikzpicture}[auto,scale=#1]

\begin{scope}[shift={(0,0)},rotate=0]

\draw[black!40,dotted,line width=2.5 pt] (0,0) to (1,0);
\draw[black!40,dotted,line width=2.5 pt] (3,0) to (4,0);

\draw[line width=1 pt] (1,0) to (1,-1);
\draw[line width=1 pt] (2,0) to (2,-1);
\draw[line width=1 pt] (3,0) to (3,-1);
\draw[line width=1 pt] (1,0) to (3,0);
\draw[line width=1 pt] (1,-1) to (3,-1);

\draw[line width=1 pt] (2,0) to (2,1);
\draw[line width=1 pt] (2,1) to (5,1);
\draw[line width=1 pt] (5,1) to (5,0);

\draw[dashed,line width=1 pt] (4,0) to (5,0);

\draw[dashed,line width=1 pt] (0,0) to (0,1);
\draw[dashed,line width=1 pt] (0,1) to (2-.03,1);
\draw[dashed,line width=1 pt] (2-.03,1) to (2-.03,-0);
\draw[dashed,line width=1.5 pt] (2-.03,0) to (2-.03,-1);
\fill (2,0) circle (3pt);
\fill (2,-1) circle (3pt);
\fill (5,0) circle (3pt);
\fill[gray] (1,0) circle (3pt);
\fill[gray] (0,0) circle (3pt);
\node[right] at(2,1.2)   {$u$};
\node[right] at(2,0.2)   {$w_i$};
\node[below] at(2,-1-0.1)   {$z_i$};
\node[below] at(5,-0.1)   {$z_{i+1}$};
\node[below] at (0.2,0)   {$\underline b_{i-1}$};
\node[above] at (1,0)   {$\bar b_{i-1}$};
\node[above] at(0.7,1)   {\Large $\tilde \Ccal_{i-1}$};
\node[below] at(3,-1)   {\Large $\tilde \Ccal_{i}$};
\node[below] at(4.5,-0.3)   {\Large $\tilde \Ccal_{i+1}$};
\node[below] at (3.5,0)   {$b_{i}$};
\end{scope}
\end{tikzpicture}
}

\def\picXiNtwoStructure[#1]{
\begin{tikzpicture}[line width=1pt,auto,scale=#1]

  \node[above]   at(-1,1)      {$F_0(b_0,w_0,z_1)$};
  \node[above]   at(3.5,1.5)      {$F_1(b_0,t_1,z_1,b_1,w_2,z_2)$};
  \node[above]   at(7.5,0.7)      {$F_2(b_1,t_2,z_2,x)$};

\begin{scope}[shift={(0,0)},rotate=0]
  \path[draw,fill=black!30!white] (0,0) to [out=40,in=180] (0.75,0.4) to [out=0,in=140]  (1.5,0) to [out=220,in=0] (0.75,-0.4) to [out=180,in=-40] (0,0);
  \node[left]   at(0,0)      {$0$};
  \node[left]   at(0.6,0.6)      {$w_0$};
  \fill (0.75,0.4) circle (2pt);
  \node[below]   at(1.75,0)      {$b_0$};
  \draw [-] (0.75,0.4) to [out=30,in=150] (4.25,0.4);

  \begin{scope}[shift={(1.5,0)}]
         \draw [-](0.03,-0.2) to (0.03,+0.2);
         \draw [-](0.5-0.03,-0.2) to   (0.5-0.03,+0.2);
   \end{scope}

 \end{scope}

\begin{scope}[shift={(1,0)},rotate=0]
  \node[above] at(1.8,0)      {$w_1$};
    \fill (1.8,0) circle (2pt);
  \node[above] at(3.25,0.4)      {$z_1$};
    \fill (3.25,0.4) circle (2pt);
  \node[above] at(2.5,+0.00)      {$t_1$};
%  \node[right]   at(5,0)      {$\bb_i\neq z_i,z_{i+1}\nin b_i$};
  \draw [-] (1,0) to (2.5,0);

  \begin{scope}[shift={(2.5,0)},rotate=0]
  \draw [-] (0,0) to [out=40,in=180] (0.75,0.4);
  \draw [-] (0,0) to [out=-40,in=180] (0.75,-0.4);
  \draw [-] (0.75,0.4) to [out=0,in=140] (1.5,0);
  \draw [-] (0.75,-0.4) to [out=0,in=220] (1.5,0);
 \end{scope}

  \node[right]   at(6,-1)      {$z_{2}$};
  \draw [-] (1.8,0) to [out=-50,in=210] (6.25,-0.4);
 \end{scope}

   \begin{scope}[shift={(4,0)},rotate=0]
  \node[above]   at(1.3,0)      {$b_{1}$};

  \begin{scope}[shift={(1,0)}]
         \draw [-](0.03,-0.2) to (0.03,+0.2);
         \draw [-](0.5-0.03,-0.2) to   (0.5-0.03,+0.2);
   \end{scope}

  \node[above] at(2.4,0.1)      {$t_2$};
  \node[right]   at(4,0)      {$x$};
  \draw [-] (1.5,0) to (3,0);

  \begin{scope}[shift={(2.5,0)},rotate=0]
  \path[draw,fill=black!30!white] (0,0) to [out=40,in=180] (0.75,0.4) to [out=0,in=140]  (1.5,0) to [out=220,in=0] (0.75,-0.4) to [out=180,in=-40] (0,0);
 \end{scope}
       \fill (3.25,-0.4) circle (2pt);
 \end{scope}

 \begin{scope}[shift={(-5,-3)},rotate=0]
 \fill[gray] (5,0) -- (6,1)--(6,-1);
 \draw [-](5,0) to  (6,1);
 \draw [-](5,0) to  (6,-1);
 \draw [-](6,1) to  (6,-1);
  \fill (5,0) circle (3pt);
  \node[left]   at(5,0)      {$0$};
  \node[right]   at(10,0)      {$x$};
 \draw [-](6+0.03,-1-0.2) to (6+0.03,-1+0.2);
 \draw [-](6+0.5-0.03,-1-0.2) to   (6+0.5-0.03,-1+0.2);
 \draw [-](6,1) to  (7.5-0.5,1);
 \draw [-](7.5,-1) to  (7.5,1-0.5);
 \draw [-](6.5,-1) to  (7.5,-1);
 \draw [-](7.5-0.5,1) to  (7.5+0.5,1);
 \draw [-](7.5-0.5,1) to  (7.5,1-0.5);
 \draw [-](7.5+0.5,1) to  (7.5,1-0.5);
 \draw [-](7.5+0.5+0.03,1-0.2) to (7.5+0.5+0.03,1+0.2);
 \draw [-](7.5+1-0.03,1-0.2) to   (7.5+1-0.03,1+0.2);
 \draw [-](7.5,-1) to  (9,  -1);
 \draw [-](8.5,1) to  (9,  1);
 \fill[gray] (9,1) -- (9,-1)--(10,0);
 \draw [-](9,-1) to  (9,  1);
 \draw [-](9,-1) to  (10,  0);
 \draw [-](9,1) to  (10,  0);
   \fill (10,0) circle (3pt);

  \node[above]   at(6,1)      {$w_0$};
  \node[above]   at(7,1)      {$z_1$};
  \node[above]   at(8.25,1)      {$b_1$};
  \node[right]   at(7.5,1-0.5)      {$t_1$};
  \node[below]   at(7.5,-1)      {$w_1$};
  \node[above]   at(6.25,-1)      {$b_0$};
  \node[below]   at(9,-1)      {$z_2$};
  \node[above]   at(9,1)      {$t_2$};

  \node[below]   at(6.5,-1.5)      {$B^{2,\iota}$};
  \node[below]   at(8.25,-1.5)      {$\bar A^{\iota} $};
  \node[below]   at(10,-1)      {$P^{\ssc[0]}$};
  \node[below]   at(5,-1)      {$P^{\ssc[0]}$};

  \end{scope}

\end{tikzpicture}}

\def\picXiOneExample[#1]{
\begin{tikzpicture}[line width=1pt,auto,scale=#1]

\begin{scope}[shift={(0,0)},rotate=0]
 \path[draw] (0,0) to [out=90,in=180] (2,1.5) to [out=0,in=90]  (4,0);
 \draw [-](0.03,-0.2) to (0.03,+0.2);
 \draw [-](0.5-0.03,-0.2) to   (0.5-0.03,+0.2);
 \draw [-](0,2) to   (1.5,2);
 \draw [-](0.5,0) to   (4,0);
 \fill (0,0) circle (3pt);
\fill (2,0) circle (3pt);
\fill (4,0) circle (3pt);
\fill (2,1.5) circle (3pt);
  \node[below]   at(0,0)      {$0$};
  \node[below]   at(2,0)      {$t$};
  \node[below]   at(4,0)      {$x$};
  \node[above]   at(2,1.5)      {$z$};
\end{scope}

\begin{scope}[shift={(5,0)},rotate=0]
 \path[draw] (0,0) to [out=90,in=180] (2,1.5) to [out=0,in=90]  (4,0);
 \draw [-](0.03,-0.2) to (0.03,+0.2);
 \draw [-](0.5-0.03,-0.2) to   (0.5-0.03,+0.2);
 \draw [-](0,2) to   (1.5,2);
 \draw [-](0.5,0) to   (4,0);
 \fill (0,0) circle (3pt);
\fill (2,0) circle (3pt);
\fill (4,0) circle (3pt);
\fill (2,1.5) circle (3pt);
  \node[below]   at(0,0)      {$0$};
  \node[below]   at(2,0)      {$t$};
  \node[below]   at(4,0)      {$x$};
  \node[above]   at(2,1.5)      {$z$};
\end{scope}

\end{tikzpicture}}

\def\picHIotaDecompositionOne[#1]{
\begin{tikzpicture}[line width=1pt,auto,scale=#1]

\begin{scope}[shift={(0,0)},rotate=0]
  \node[left]   at(0,0)      {$0$};
  \node[below] at(1-0.1,0)      {$e_\iota$};
  \node[right]   at(2.5,1)      {$y$};
  \node[right]   at(2.5,-1)      {$x$};
  \node[right]   at(2.5,0)      {$b$};
  \path[draw] (0,0) to [out=90,in=180] (0.5,0.8) to [out=0,in=90]  (1,0);
  \begin{scope}[shift={(1,1)}]
         \draw [-](0.03,-1-0.2) to (0.03,-1+0.2);
         \draw [-](0.5-0.03,-1-0.2) to   (0.5-0.03,-1+0.2);
   \end{scope}
    \draw [-](1,0) to  [out=90,in=180] (2.5,1);
    \draw [line width=3pt,-](1.5,0) to [out=300,in=180] (2.5,-1);
 \end{scope}

 \begin{scope}[shift={(3.5,0)},rotate=0]
 \draw [-](0,-0.2) to  (0,0.2);
 \draw [-](-0.2,0) to  (0.2,0);
 \node[right]   at(0.3,-0.1)      {{\Large $2\sum_{a=0}^2$}};
 \end{scope}

 \begin{scope}[shift={(6.5,0)},rotate=0]
  \node[left]   at(0,0)      {$0$};
  \node[below] at(1-0.1,0)      {$e_\iota$};
  \node[right]   at(3.5,1)      {$y$};
  \node[right]   at(3.5,-1)      {$x$};
  \node[right]   at(3.5,0)      {$b$};
  \path[draw] (0,0) to [out=90,in=180] (0.5,0.8) to [out=0,in=90]  (1,0);
  \begin{scope}[shift={(1,0)}]
      \path[draw] (1,-1) to [out=180,in=-40] (0,0);
      \path[draw] (0,0) to [out=40,in=180] (1,1) to (1,-1);
      \node[right]   at(1,0)      {$a$};
  \end{scope}
  \begin{scope}[shift={(2,0)}]
         \draw [-](0.03,-1-0.2) to (0.03,-1+0.2);
         \draw [-](0.5-0.03,-1-0.2) to   (0.5-0.03,-1+0.2);
   \end{scope}
    \draw [-](2,1) to  (3.5,1);
    \draw [line width=3pt,-](2.5,-1) to (3.5,-1);
 \end{scope}
 \begin{scope}[shift={(11,0)},rotate=0]
 \draw [-](0,-0.2) to  (0,0.2);
 \draw [-](-0.2,0) to  (0.2,0);
 \end{scope}

 \begin{scope}[shift={(11.8,0)},rotate=0]
  \node[left]   at(0,0)      {$0$};
  \node[below] at(1-0.1,0)      {$e_\iota$};
  \node[right]   at(3.5,1)      {$y$};
  \node[right]   at(3.5,-1)      {$x$};
  \node[right]   at(3.5,0)      {$b$};
  \path[draw] (0,0) to [out=90,in=180] (0.5,0.8) to [out=0,in=90]  (1,0);
  \begin{scope}[shift={(1,0)}]
      \path[draw] (1,-1) to (1,1) to [out=180,in=40] (0,0);
      \path[line width=3pt,draw] (0,0) to [out=-40,in=180] (1,-1);
      \node[right]   at(1,0)      {$a$};
  \end{scope}
  \begin{scope}[shift={(2,0)}]
         \draw [-](0.03,-1-0.2) to (0.03,-1+0.2);
         \draw [-](0.5-0.03,-1-0.2) to   (0.5-0.03,-1+0.2);
   \end{scope}
    \draw [-](2,1) to  (3.5,1);
    \draw [-](2.5,-1) to (3.5,-1);
 \end{scope}

\end{tikzpicture}}

\def\picHIotaDecompositionTwo[#1]{
\begin{tikzpicture}[line width=1pt,auto,scale=#1]

\begin{scope}[shift={(0,0)},rotate=0]
  \node[left]   at(0,0)      {$0$};
  \node[below] at(0.25,-1)      {$e_\iota$};
    \node[right]   at(2.5,1)      {$y$};
  \node[right]   at(2.5,-1)      {$x$};
  \node[above]   at(0.5,1)      {$w$};
  \node[right]   at(0.5,0)      {$\geq 1$};
  \node[right]   at(2.5,0)      {$b$};
  \path[draw] (0,0) to [out=90,in=180] (0.5,1) to  (0.5,-1);
  \begin{scope}[shift={(0.5,0)}]
         \draw [-](0.03,-1-0.2) to (0.03,-1+0.2);
         \draw [-](0.5-0.03,-1-0.2) to   (0.5-0.03,-1+0.2);
   \end{scope}
    \draw [-](0.5,1) to  (2.5,1);
    \draw [line width=3pt,-](1,-1) to (2.5,-1);
 \end{scope}

 \begin{scope}[shift={(4,0)},rotate=0]
 \draw [-](0,-0.2) to  (0,0.2);
 \draw [-](-0.2,0) to  (0.2,0);
 \node[right]   at(0.2,0)      {{\Large $2$}};
 \end{scope}

 \begin{scope}[shift={(5.5,0)},rotate=0]
  \node[left]   at(0,0)      {$0$};
  \node[below] at(0.25,-1)      {$e_\iota$};
  \node[above]   at(0.5,1)      {$w$};
  \node[right]   at(0.5,0)      {$\geq 1$};
  \node[right]   at(3,1)      {$y$};
  \node[right]   at(3,-1)      {$x$};
  \node[right]   at(3,0)      {$b$};
  \path[draw] (0,0) to [out=90,in=180] (0.5,1) to  (0.5,-1);
  \begin{scope}[shift={(0.5,-1)}]
  \path[draw] (0,0) to [out=60,in=180] (0.5,0.5) to [out=0,in=120]  (1,0) to [out=240,in=0] (0.5,-0.5) to [out=180,in=-60](0,0);
  \end{scope}
  \begin{scope}[shift={(1.5,0)}]
         \draw [-](0.03,-1-0.2) to (0.03,-1+0.2);
         \draw [-](0.5-0.03,-1-0.2) to   (0.5-0.03,-1+0.2);
   \end{scope}
    \draw [-](0.5,1) to  (3,1);
    \draw [line width=3pt,-](2,-1) to (3,-1);
 \end{scope}

 \begin{scope}[shift={(9.5,0)},rotate=0]
 \draw [-](0,-0.2) to  (0,0.2);
 \draw [-](-0.2,0) to  (0.2,0);
  \node[right]   at(0.2,0)      {{\Large $2$}};
 \end{scope}

 \begin{scope}[shift={(11,0)},rotate=0]
  \node[left]   at(0,0)      {$0$};
  \node[below] at(0.25,-1)      {$e_\iota$};
  \node[above]   at(0.5,1)      {$w$};
    \node[right]   at(3,1)      {$y$};
  \node[right]   at(3,-1)      {$x$};
  \node[right]   at(3,0)      {$b$};
  \path[draw] (0,0) to [out=90,in=180] (0.5,1) to  (0.5,-1);
  \begin{scope}[shift={(0.5,-1)}]
  \path[draw] (0,0) to [out=60,in=180] (0.5,0.5) to [out=0,in=120]  (1,0);
  \path[line width=3pt,draw] (1,0) to [out=240,in=0] (0.5,-0.5) to [out=180,in=-60](0,0);

  \end{scope}
  \begin{scope}[shift={(1.5,0)}]
         \draw [-](0.03,-1-0.2) to (0.03,-1+0.2);
         \draw [-](0.5-0.03,-1-0.2) to   (0.5-0.03,-1+0.2);
   \end{scope}
    \draw [-](0.5,1) to  (3,1);
    \draw [-](2,-1) to (3,-1);
 \end{scope}

\end{tikzpicture}}

\section{Diagrammatic bounds}
\label{secBoundsPerc}

In Sections \ref{secBoundsPerc}-\ref{secProofBounds} we bound the NoBLE-coefficients, derived in the last section, and define the split of the coefficients as used in \cite[Section 4.1]{FitHof13b}. The bounds are stated in terms of simple diagrams, which can in turn be bounded by combinations of two-point functions. Thus, we prove Proposition \ref{prop-bds-LEC}, and provide the bounds, in term of diagrams, stated in \cite[Assumption 4.3]{FitHof13b}. We start by giving an overview of the bounds on the NoBLE coefficients.

\subsection{Overview of the bounds on the coefficients in Sections \ref{secBoundsPerc}-\ref{secProofBounds}}
\label{sec-overview-NoBLEbds}
In Section \ref{secBoundsSimpleDiagrams}, we first introduce simple diagrams that can be obtained by various generalizations of the two-point function, as well as so-called {\em repulsive diagrams}.
Then, we state and prove the bounds on the coefficients for $N=0$ in Section \ref{sec-bounds-N0}. For $N\geq 1$, the coefficients are defined as combinations of increasingly intertwined events that we first bound in terms of simpler events in Section \ref{secBoundsPercEvents}.

In Sections \ref{sec-bounds-Ngeq1} and \ref{secProofBounds}, we bound these events by so-called building blocks, which are combinations of simple diagrams. We define these building blocks informally in Section \ref{secBoundsPercDefinitionBlocks}. In Sections \ref{sec-bounds-N=1} and \ref{secPercostatingtheBounds}, we state the bounds for $N\geq 1$. In Section \ref{secSummaryBounds}, we give a brief overview of how they prove Proposition \ref{prop-bds-LEC} and explain how the diagrammatic bounds are used to prove \cite[Assumption 4.3]{FitHof13b}.

In Section \ref{secProofBounds}, we indicate how to prove the diagrammatic bounds for $N\geq 1$. We give the full proof of the bounds on $\Xi_p^{\ssc[1]}(x)$ in Section \ref{secProofBoundsNOne}, and explain how to use similar arguments to prove the bounds on $\Xi_p^{\ssc[1],\iota}(x)$. The proof for $N\geq 2$ relies on ideas already used in the classical lace expansion and a distinction of cases for the length of several distinct connections within the diagrams. We use this distinction of cases to make optimal use of the additional avoidance constraints in the events defining the NoBLE coefficients.  The ideas underlying the distinction of cases is discussed within the proof of the bounds of $\Xi^{\ssc[1]}_p$. We give an outline of the proof for $N\geq 2$ in Section \ref{secProofBoundsNBig}.

As the proof of these bounds is quite elaborate, we do not give the full proof.
\iflongversion
In Section \ref{secProofBoundsNBig}, we sketch the proof and discuss some of the more involved steps. These bounds are stated using the thirteen building blocks that we informally define in Section \ref{secBoundsPercDefinitionBlocks}. As we consider several cases for each block, we have to define some 100 different blocks. We give the formal definition in Appendix \ref{app-bounds}.

\else
In the extended version of this document we give a sketch of the proof. A detailed explanation of these bounds and their proof can be found in the thesis of the first author \cite{Fit13}, which can be downloaded from \cite{FitNoblePage}. From \cite{FitNoblePage} you can also download the extended version of this article. In that extended version we give the formal definition of all building blocks, that we only define informally in Section \ref{secBoundsPercDefinitionBlocks}.
\fi

\subsection{Simple diagrams}
\label{secBoundsSimpleDiagrams}
In this section, we define simple diagrams that we use to bound the NoBLE-coefficients. Then, we review how we bound these diagrams using the bootstrap functions given in \refeq{defFunc1}--\refeq{defFunc3}.  Moreover, we derive sharp bounds for the probability of a double connection.

\paragraph{Modified two-point functions.}
For $m\geq 0$, we denote by $\{0\connLe{m}x\}$ the event that $0$ and $x$ are connected and that there exists a path of occupied, disjoint bonds between $0$ and $x$ that consists of at least $m$ bonds. Further, we define $\{0\connLe{\underline m}x\}$ to be the event that $0$ and $x$ are connected by a path of exactly $m$ occupied bonds. For $m\geq 0$, we define
	\begin{align}
	\tau_{m,p}(x)&=\prob_p ( 0\connLe{m} x), \qquad\qquad &\tau^\iota_{m,p}(x)
	=\prob^{\ve[\iota]}_p (0\connLe{m} x),\\
	\tau_{\underline m,p}(x)&=\prob_p (0\connLe{\underline m}  x), \qquad\qquad
	&\tau^\iota_{\underline m,p}(x)=\prob^{\ve[\iota]}_p ( 0\connLe{\underline m} x).
	\end{align}
For $m\geq 1$ and $x\neq 0$, we note that
	\begin{align}
	\lbeq{tau-j-bounds-steps}
	\tau^{\iota}_{m,p}(x)&\leq \tau_{m,p}(x)
	\leq 2d p (D\star \tau_{m-1,p})(x)\leq (2d p)^m (D^{\star m}\star \tau_{p})(x), \\
	\tau^{\iota}_{\underline m,p}(x)&\leq \tau_{\underline m,p}(x)
	\leq 2d p (D\star \tau_{\underline {m-1},p})(x)\leq (2d p)^m D^{\star m}(x).
	\end{align}

\paragraph{Non-repulsive diagrams.}
For $x_i\in\Zd$ and indices $j_i\in\{\underline 0,\underline 1,\dots\}\cup\{0,1,\dots,\}$, where $i=1,\dots,5$, we define the \emph{non-repulsive diagrams} by
	\begin{align}
	\lbeq{defpercbubble}
	\diagRepulsiveLetter{B}^*_{j_1,j_2}(x_1,x_2)=&\tau_{j_1,p}(x_1)\tau_{j_2,p}(x_2-x_1),\\
	\nn
	\diagRepulsiveLetter{T}^*_{j_1,j_2,j_3}(x_1,x_2,x_3)=&\tau_{j_1,p}(x_1)\tau_{j_2,p}(x_2-x_1)\tau_{j_3,p}(x_3-x_2)\\
	\lbeq{defpercTriangle}
	=& \diagRepulsiveLetter{B}^*_{j_1,j_2}(x_1,x_2)\tau_{j_3,p}(x_3-x_2),\\
	\lbeq{defpercSquare}
	\diagRepulsiveLetter{S}^*_{j_1,j_2,j_3,j_4}(x_1,x_2,x_3,x_4)
	=&\diagRepulsiveLetter{T}^*_{j_1,j_2,j_3}(x_1,x_2,x_3)\tau_{j_4,p}(x_4-x_3),\\
	\lbeq{defpercPentagon}
	\diagRepulsiveLetter{P}^*_{j_1,j_2,j_3,j_4,j_5}(x_1,x_2,x_3,x_4,x_5)
	=&\diagRepulsiveLetter{S}^*_{j_1,j_2,j_3,j_4}(x_1,x_2,x_3,x_4)\tau_{j_4,p}(x_5-x_4).
	\end{align}
In the analysis of \cite{FitHof13b} we assume that the bootstrap functions, see \refeq{defFunc1}-\refeq{defFunc3}, are bounded. These bounds in particular imply bounds on $p<p_c$ and $\sup_{k\in(-\pi,\pi)^d }[1-\hat D(k)]\hat \tau_p(k)$. In this discussion, we assume that
\begin{eqnarray}
	2dp\leq \bar \Gamma_1,\qquad \qquad
	\sup_{k\in(-\pi,\pi)^d} [1-\hat D(k)]\hat{\tau}_p(k)\leq \bar \Gamma_2,
	\end{eqnarray}
where $\bar \Gamma_1$ and $\bar \Gamma_2$ are functions of $\Gamma_1,\Gamma_2$ whose precise values are irrelevant for the discussion at hand.
Using this and \refeq{tau-j-bounds-steps} allows us to bound these non-repulsive diagrams,
for $l_1,l_2,l_3\in\Nbold$ with $l_1+l_2+l_3=l$ even, as
	\begin{align}
	\sum_{x_1,x_2}\diagRepulsiveLetter{T}^*_{l_1,l_2,l_3}(x_1,x_2,x_3)
	&\leq(2dp)^{l}(D^{\star l_1}\star \tau_p\star D^{\star l_2}\star
	\tau_p\star D^{\star l_3}\star \tau_p)(x_3)\nnb
	&\leq (2dp)^{l}\int_{(-\pi,\pi)^d} \hat D^{l}(k)\hat \tau_p(k)^3 \frac {d^dk}{(2\pi)^d}
	\leq \bar \Gamma^l_1 \bar \Gamma^3_2 I_{3,l}(0),
	\lbeq{Simple-Bound-triangle}
	\end{align}
with
	\begin{align}
	I_{n,l}(x)=\int_{(-\pi,\pi)^d} \e^{\ii k\cdot x}\frac {\hat D^{l}(k)}{[1-\hat D(k)]^n}\frac {d^dk}{(2\pi)^d}
	\end{align}
being a SRW-integral that we can compute numerically. In \cite[Section 5]{FitHof13b}, we explain how we compute this integral, as well as how to improve the bounds on such non-repulsive diagrams.

\paragraph{Repulsive diagrams.}
Using only bounds as simple as \refeq{Simple-Bound-triangle} we could show mean-field behavior only for $d\geq 50$. In the following, we define \emph{repulsive} diagrams that allow us to prove sharper bounds on the coefficients. In repulsive diagrams, the connections between the points $x_i$ are bond-disjoint. As first example we define the repulsive double connection as
\begin{align}
 \lbeq{defpercRepDouble}
\diagRepulsiveLetter{D}_{j_1,j_2}(x)=\prob_p(\{0\connLe{j_1} x\} \circ \{0\connLe{j_2} x\} ),
\end{align}
where $x\in\Zd$ and $j_1,j_2\in\{\underline 0,\underline 1,\dots\}\cup\{0,1,\dots,\}$, and where the symbol $\circ$ indicates that two events occur disjointly. For events involving the existence of paths, this means that these paths consist of disjoint bonds. A formal definition can be found e.g.\ in \cite[Section~2.3]{Grim99}.\\
The connections in $\diagRepulsiveLetter{D}_{j_1,j_2}(x)$ are realised on the same percolation configuration. In our bounds, we also consider paths on {\em different} percolation configurations. For this reason we need to generalize the notion of disjoint occurrence:

\begin{definition}[Generalized disjoint occurrence]
 \label{def-gen-disjoint}
For a percolation realisation $\omega$, we denote by $\Bcal(\omega;x)$ the set of all bonds that are occupied in $\omega$ and for which one of its endpoints is connected to the point $x$ in $\omega$. By $\{x\conn y\}_i$ we denote that $x,y\in\Zd$ are connected in $\omega_i$, for $i\geq 1$. For $n\geq 2$ and $i=1,2,\dots n$, let $\omega_i$ be a percolation configuration and let $x_i,y_i\in\Zd$. Then, we say that the connections $\{x_i\conn y_i\}_i$ occur generalized-disjointly, and write
	\begin{align}
 	 \{x_1\conn y_1\}_{1}\circledast\{x_2\conn y_2\}_{2}\circledast
	\cdots \circledast\{x_n\conn y_n\}_{n},
	\end{align}
when, for each $i=1,2,\dots,n$, we can choose a path of bonds $p_i\subset \Bcal(\omega_i;x_i)$, such that the path $p_i$ connects $x_i$ to $y_i$ and the paths $(p_i)_{i=1}^n$ are pairwise vertex disjoint. Similar definitions apply to connections of the form $\{x_i\connLe{j_i} y_i\}_{i}$ for indices $j_i\in\{\underline 0,\underline 1,\dots\}\cup\{0,1,\dots,\}$.
\end{definition}
\medskip

Note that if we choose $i=1$ for all $i=2,\dots,n$, then this notion corresponds to the usual disjoint occurrence $\circ$. Further, when $\prob_p$ denotes the distribution of independent percolation configurations $(\omega_1,\ldots,\omega_n)$,
	\begin{align}
	\prob_p(\{x_1\conn y_1\}_{1}\circledast\{x_2\conn y_2\}_{2})
	&\leq \prob_p(x_1\conn y_1) \prob_p(x_2\conn y_2),\\
	\prob_p(\{x_1\conn y_1\}_{1}\circledast\{x_2\conn y_2\}_{2}
	&\circledast\{x_3\conn y_3\}_{2})\nnb
	&\leq \prob_p(x_1\conn y_1)\prob_p(\{x_2\conn y_2\}\circ\{x_3\conn y_3\}).
	\end{align}
For $x_i\in\Zd$ and indices $j_i\in\{\underline 0,\underline 1,\dots\}\cup\{0,1,\dots,\}$, we define the repulsive bubble and triangle to be
	\begin{align}
	\lbeq{defpercRepbubble}
	\diagRepulsiveLetter{B}_{j_1,j_2}(x_1,x_2)
	&= \max_{i=1,2} \prob_p(\{0\connLe{j_1} x_1\}_{1}\circledast
	\{x_1\connLe{j_2} x_2\}_{i}),\\
	\lbeq{defpercRepTriangle}
	\shift\diagRepulsiveLetter{T}_{j_1,j_2,j_3}(x_1,x_2,x_3)&=\max_{(i,j)=\{1,2,3\}^2}
	\prob_p(\{0\connLe{j_1} x_1\}_{1}\circledast
	\{x_1\connLe{j_2} x_2\}_{i}\circledast\{x_2\connLe{j_3} x_3\}_{j}),
	\end{align}
where $\omega_1,\omega_2,\omega_3$ are three i.i.d.\ percolation configurations under $\prob_p$.
The repulsive square $\diagRepulsiveLetter{S}_{j_1,j_2,j_3,j_4}(x_1,x_2,x_3,x_4)$ and pentagon $\diagRepulsiveLetter{P}_{j_1,j_2,j_3,j_4,j_5}(x_1,x_2,x_3,x_4,x_5)$ are defined in the same manner. We omit the formal definitions of these diagrams.

\paragraph{Bounds on repulsive diagrams.}
We bound repulsive diagrams in an efficient manner by extracting explicit contributions.
This is easily seen for the two-point function itself, by noting that
	\begin{align}
	\tau_{n,p}(x)\leq \sum_{r=n}^{M-1} p^r a_r(x)+p^M(a_M\star \tau)(x),
	\lbeq{improved-bound-taun}
	\end{align}
where $a_n(x)$ is the number of $n$-step \emph{bond-avoiding walks} ending at $x$. Bond-avoiding walks are simple random walks that never use a bond twice, i.e., for $i,j$ with $i\neq j$, we have that
$\{\omega_i,\omega_{i+1} \}\neq \{\omega_j,\omega_{j+1} \}$. The number $M$ is some number larger than $n$. For the implementation, we choose $M=10$.
We use the same idea for the repulsive bubbles, triangles and squares. For example consider $l_1, l_2\in{\mathbb N}$, with $l_1+l_2<M$, we conclude that
	\begin{align}
	\sum_y\diagRepulsiveLetter{B}_{l_1,l_2}(y,x)\leq&\ \sum_y \diagRepulsiveLetter{B}_{M-l_2,l_2}(y,x)+\sum_y\sum_{r=l_1}^{M-l_2-1}	\diagRepulsiveLetter{B}_{\underline r,l_2}(y,x)\\
	\leq&\ p^M (a_M\star \tau^{\star 2})(x)+(M-l_1-l_2)p^M(a_M\star \tau)(x)\nnb
    	&\quad +\sum_{s=l_1}^{M-l_2-1} \sum_{r=l_2}^{M-s-1} p^{r+s} a_{r+s}(x)
	.\nn
	\end{align}
More details can be found in \cite[Section 5.3.2]{FitHof13b}.
\paragraph{Bounds on double connections.}
The probability of a double connection  $\diagRepulsiveLetter{D}_{n,n}(x)$ deserves our special attention.
As each double connection uses at least two neighbors, we know that for $x\neq 0$,
	\begin{align}
	\lbeq{N0-bd}
&	\prob_p(0\dbc x )\\
& \leq \sum_{\iota}\sum_{\kappa\neq \iota}
	\prob_p \left( \{\ve[\iota] \conn x\text{ off 0}\}\circ\{\ve[\kappa] \conn x\text{ off 0}\}
	\cap\{(0,\ve[\iota])\text{ occ.}\}\cap\{(0,\ve[\kappa])\text{ occ.}\}\right).\nn
	\end{align}
Reviewing this bound, we see that each percolation configuration in which $\{0\dbc x\}$ occurs for $x\neq 0$, contributes twice to the right-hand side, e.g., once for $\iota=1,\kappa=2$ and once for $\iota=2,\kappa=1$. Thus, this bound overcounts by a factor two, and actually \refeq{N0-bd} holds with an extra factor $1/2$ on the right-hand side.
Another way to view this factor $1/2$ is that the two connections in $\diagRepulsiveLetter{D}_{n,n}(x)$ are interchangeable, while in the bubble $\diagRepulsiveLetter{B}_{n,n}(x,0)$ the two connections are not.
We conclude for $n\geq 1$ that
\begin{align}
\lbeq{double-bound}
\diagRepulsiveLetter{D}_{n,n}(x)\leq
\frac 1 2 \Big(p^M(a_M\star \tau^{\star 2})(x)+(M-2)p^M(a_M\star \tau)(x)
+\sum_{r=n}^{M-1} \sum_{s=n}^{M-r-1} p^{r+s} a_{r+s}(x)\Big).
%&\sum_{x\in \Zd} \prob \left(0 \dbc x\right)\leq 1+\frac 1 2 \sum_{x\in \Zd}\diagRepulsiveLetter{B}_{1,1}(x,0),
%\qquad \sum_{x\neq 0} \prob \left(0 \dbc x\right)\leq \frac 1 2 \sum_{x\in \Zd}\diagRepulsiveLetter{B}_{1,1}(x,0),\\
%\lbeq{doubleN-bound}
%&\sum_{x\in \Zd} \prob \left(\{0 \connLe{n} x\} \circ \{0 \connLe{n} x\} \right)\leq \frac 1 2 \sum_{x\in \Zd}\diagRepulsiveLetter{B}_{n,n}(x,0),\\
%\lbeq{doubleWeighted-bound}
%&\sum_{x\in \Zd} \|x\|_2^2 \prob \left(\{0 \connLe{n} x\} \circ \{0 \connLe{n} x\} \right)
%\leq
%\frac 1 2 \sum_{x\in \Zd} \|x\|_2^2 \diagRepulsiveLetter{B}_{n,n}(x,0),
\end{align}

\subsection{Diagrammatic bounds for $N=0$}
\label{sec-bounds-N0}

In this section, we bound the NoBLE coefficient for $N=0$ and prove a part of the bounds assumed in \cite[Assumption 4.3]{FitHof13b}.
These bounds on the coefficients defined in Sections \ref{sec-compl} and \ref{subsectionSplit} are given in the following two lemmas:
\begin{lemma}[Bounds on $\Xi^{\ssc[0]}_p$ and $\Psi^{\ssc[0],\kappa}_p$]
\label{lemmapercboundXi0}
Let $p<p_c$. Then,
	\begin{align}
	\lbeq{lemmapercboundXi0-bound-1}
	\sum_{x\in\Zd}\Xi^{\ssc[0]}_p(x)\leq& \sum_{x\in\Zd} \diagRepulsiveLetter{D}_{1,1}(x),\qquad \
	\sum_{x\in\Zd} \|x\|_2^2 \Xi^{\ssc[0]}_p(x)\leq  \sum_{x\in\Zd} \|x\|_2^2 \diagRepulsiveLetter{D}_{1,1}(x),\\
	\lbeq{lemmapercboundXi0-bound-4}
	\sum_{x\in\Zd} \Xi^{\ssc[0]}_{{\sss R},p}(x)\leq&  \sum_{x\in\Zd} \diagRepulsiveLetter{D}_{2,2}(x),\qquad
	\sum_{x\in\Zd}\|x\|_2^2 \Xi^{\ssc[0]}_{{\sss R},p}(x)\leq  \sum_{x\in\Zd}\|x\|_2^2\diagRepulsiveLetter{D}_{2,2}(x).
\end{align}
Further, for all $\kappa$,
\begin{align}
	\lbeq{lemmapercboundXi0-bound-5}
	\sum_{x\in\Zd} \Psi^{\ssc[0],\kappa}_{{\sss R, I},p}(x)\leq&(2d-2)p \tau_{3,p}(\ve[1]) + \min\left\{1, \frac {p} {\aap} \frac {d-1} {d } \right\} \sum_{x\in\Zd} \diagRepulsiveLetter{D}_{2,2}(x),\\
	\lbeq{lemmapercboundXi0-bound-6}
	\sum_{x\in\Zd} \|x-\ve[\kappa]\|_2^2 \Psi^{\ssc[0],\kappa}_{{\sss R, I},p}(x)\leq&
	\frac p {\aap} \sum_{x\in\Zd} \ \diagRepulsiveLetter (1+\|x\|_2^2)\diagRepulsiveLetter{D}_{1,1}(x),
\end{align}
and
\begin{align}
	\lbeq{lemmapercboundXi0-bound-7}
	\sum_{x\in\Zd} \Psi^{\ssc[0],\kappa}_{{\sss R, II},p}(x)&\leq \min\left\{1, \frac {p} {\aap} \frac {d-1} {d } \right\}  \sum_{x\in\Zd}  \diagRepulsiveLetter{D}_{2,2}(x),\\
	\lbeq{lemmapercboundXi0-bound-8}
	\sum_{x\in\Zd} \|x\|_2^2 \Psi^{\ssc[0],\kappa}_{{\sss R, II},p}(x)&\leq \min\left\{1, \frac {p} {\aap} \frac {d-1} {d } \right\}   \sum_{x\in\Zd}\|x\|_2^2\diagRepulsiveLetter{D}_{2,2}(x).
\end{align}
\end{lemma}

\begin{lemma}[Bounds on $\Xi^{\ssc[0],\iota}_p$]
\label{lemmapercboundXiiota0}
Let $p<p_c$. Then,
\begin{align}
	\lbeq{lemmapercboundXiiota0-bound-1}
	\sum_{x\in\Zd}\Xi^{\ssc[0],\iota}_p(x)\leq& \tau_{3,p} ( \ve[1]) \Big(1+\sum_{x\in\Zd} \diagRepulsiveLetter{D}_{1,1}(x)\Big),\\
	\lbeq{lemmapercboundXiiota0-bound-2}
	\sum_{x\in\Zd}\|x-\ve[\iota]\|_2^2\Xi^{\ssc[0],\iota}_p(x)\leq &\tau_{3,p}(\ve[1])\sum_{x\in\Zd} \|x\|_2^2 \diagRepulsiveLetter{D}_{1,1}(x),\\
	\lbeq{lemmapercboundXiiota0-bound-3}
	\sum_{x\in\Zd}\|x\|_2^2\Xi^{\ssc[0],\iota}_p(x)\leq&\tau_{3,p} ( \ve[1]) +\tau_{3,p} ( \ve[1])
	\sum_{x\in\Zd} \big(1+\|x\|_2^2\big) \diagRepulsiveLetter{D}_{1,1}(x).
\end{align}
Further,
\begin{align}
	\lbeq{lemmapercboundXiiota0-bound-4}
	\Xi^{\ssc[0],\iota}_{\alpha,{\sss I},p}(\ve[\iota])\leq & \tau_{3,p}(\ve[1]),\qquad
	\Xi^{\ssc[0],\iota}_{\alpha,{\sss II},p}(0)=0,\qquad \sum_\iota \Xi^{\ssc[0],\iota}_{\alpha,{\sss II},p}(\ve[\iota])\leq \tau_{3,p}(\ve[1]),\\
	\lbeq{lemmapercboundXiiota0-bound-5}
	&\sum_\iota \Xi^{\ssc[0],\iota}_{\alpha,{\sss I},p}(\ve[1]+\ve[\iota])\leq \tau_{3,p}(\ve[1]) (2d-1)p\tau_{3,p}(\ve[1]),
\end{align}
and
\begin{align}
	\lbeq{lemmapercboundXiiota0-bound-7}
	\sum_{x\in\Zd}\Xi_{{\sss R,I},p}^{\ssc[0],\iota} (x)\leq& \tau_{3,p}(\ve[1])\sum_{x\in\Zd}\diagRepulsiveLetter{D}_{2,2}(x),\\
	\lbeq{lemmapercboundXiiota0-bound-8}
	\sum_{x\in\Zd}\|x-\ve[\iota]\|_2^2\Xi_{{\sss R,I},p}^{\ssc[0],\iota} (x)\leq&
	\tau_{3,p}(\ve[1])\sum_{x\in\Zd} \|x\|_2^2 \diagRepulsiveLetter{D}_{2,2}(x),\\
	\lbeq{lemmapercboundXiiota0-bound-9}
	\sum_{x\in\Zd}\Xi_{{\sss R,II},p}^{\ssc[0],\iota} (x)\leq& \tau_{3,p} (\ve[1]) \sum_{x\in\Zd}\diagRepulsiveLetter{D}_{1,1}(x),\\
	\lbeq{lemmapercboundXiiota0-bound-10}
	\sum_{x\in\Zd}\|x\|_2^2\Xi_{{\sss R,II},p}^{\ssc[0],\iota} (x)\leq&  \tau_{3,p} ( \ve[1]) \sum_{x\in\Zd} \big(1+\|x\|_2^2\big)\diagRepulsiveLetter{D}_{1,1}(x).
\end{align}
The coefficient $\Pi^{\ssc[0],\iota,\kappa}_p$ can be bounded by
\begin{align}
	\lbeq{lemmapercboundXiiota0-bound-11}
	\sum_{\kappa}\Pi^{\ssc[0],\iota,\kappa}_{\alpha,p}(\ve[1])\leq&  2(d-1)\aap \tau_{3,p}(\ve[1]),\\
		%\lbeq{lemmapercboundXiiota0-bound-12}
		%\sum_{\kappa}\Pi^{\ssc[0],1,\kappa}_{\alpha,p}(\ve[1])\geq& (2d-2)^2p^4(1-\tau_{3,p}(\ve[1])-2\tau_{2,p}^{\sss 1}(\ve[1]+\ve[2])+\tau_{3,p}^{\sss 1}(\ve[1]+\ve[2]+\ve[3])),\\
	\lbeq{lemmapercboundXiiota0-bound-13}
	\sum_{x,\kappa}\Pi^{\ssc[0],\iota,\kappa}_{{\sss R},p} (x)\leq& (2d)^2 p \tau_{3,p}(\ve[1])\sum_{x\in\Zd}\diagRepulsiveLetter{D}_{1,1}(x),\\
	\lbeq{lemmapercboundXiiota0-bound-14}
	\sum_{x,\iota,\kappa}\|x-\ve[\iota]-\ve[\kappa]\|_2^2\Pi^{\ssc[0],\iota,\kappa}_{{\sss R,I},p} (x)\leq&
	(2d)^2 p \tau_{3,p}(\ve[1])\sum_{x\in\Zd} \left(1+ \|x\|_2^2 \right)\diagRepulsiveLetter{D}_{1,1}(x).
\end{align}
\end{lemma}

\begin{proof}[Proof of Lemma \ref{lemmapercboundXi0}.]
We begin by simplifying the coefficients. Recall the definition of $\Xi^{\ssc[0]}_p(x)$ in \refeq{def-xi-zero}. Using the split \refeq{Def-Xi-Split}, \refeq{Def-Psi-Split-III} we extract from this the dominante nearest-neighbor contribution.
Thus, all contributions to the remainder term involve paths of length at least two. We conclude that
\begin{align}
	 \lbeq{XiZeroR-in-clear}
	\Xi^{\ssc[0]}_{{\sss R},p}(x)&=(1-\delta_{0,x})\prob_p \left( \{0\connLe{2} x\} \circ \{ x \connLe{2} 0\}\right).
\end{align}
Thus, the bounds on $\Xi^{\ssc[0]}_p$ and $\Xi^{\ssc[0]}_{{\sss R},p}$, stated in \refeq{lemmapercboundXi0-bound-1},  \refeq{lemmapercboundXi0-bound-4}, follow directly from the definition of the double connections \refeq{defpercRepDouble}. \\
Recall the definition of $\Psi^{\ssc[0],\kappa}_p$ in \refeq{Psi0def}. To bound $\Psi^{\ssc[N],\kappa}_p(x)$, we can use \refeq{XidominatespsiImproved} to bound it by $\tfrac {p}{\aap}\Xi^{\ssc[N]}_p(x)$. For $N=0$, however, we can improve upon this in two different ways that we now present.

First, since $\{0\dbc x\}$ is an increasing event, while $ \{x-\ve[\kappa]\nin \tilde{\Ccal}^{\{x,x-\ve[\kappa]\}}(x)\}$ is decreasing, we conclude from the Harris inequality, see e.g.\ \cite[Section 2.2]{Grim99}, that
	\eqn{
    \lbeq{Harris-bdeQ}
	\prob_p(\{0\dbc x \} \cap \{(x-\ve[\kappa])\nin\tilde{\Ccal}^{\{x,x-\ve[\kappa]\}}(x)\} )\leq \prob_p(0\dbc x )\prob_p((x-\ve[\kappa])\nin\tilde{\Ccal}^{\{x,x-\ve[\kappa]\}}(x)),
	}
so that
	\begin{align}
	\lbeq{Harris-bd}
	\Psi^{\ssc[0],\kappa}_p(x)\leq \tfrac {p}{\aap} \prob_p((x-\ve[\kappa])\nin
	\tilde{\Ccal}^{\{x,x-\ve[\kappa]\}}(x)) \Xi^{\ssc[0]}_p(x)= \Xi^{\ssc[0]}_p(x).
	\end{align}
We cannot use the same argument for $N\geq 1$ as the event $E'(\tb_{\sss N-1}, y; \tilde{\Ccal}_{\sss N-1})$ in the innermost expectation $\expec_{\sss N}^{\bb_{\sss N-1}}$, see \refeq{XiNdef}-\refeq{PsiNdef}, is not increasing.

A second way to improve upon \refeq{XidominatespsiImproved} for $N=0$ is to remove {\em overcounting:}
\label{overcounting} In every fixed configuration in which $\{0\dbc x \}$ occurs, there exist at least two bond-disjoint paths leading from $0$ to $x$. Thus, the event  $x-\ve[\kappa]\in\tilde{\Ccal}^{\{x,x-\ve[\kappa]\}}(x)$ occurs for at least two $\kappa$. As a result, when we sum over $\kappa$, each configuration can contribute at most $(2d-2)$ times to the sum, and we obtain the bound
	\begin{align}
	\sum_{\kappa} \Psi^{\ssc[0],\kappa}_p(x)\leq\frac p {\aap} (2d-2)\Xi^{\ssc[0]}_p(x).
	\end{align}
By symmetry also
	\begin{align}
	\sum_{x\in \Zd} \Psi^{\ssc[0],\kappa}_p(x)=\frac 1 {2d} \sum_{x,\kappa} \Psi^{\ssc[0],\kappa}_p(x)
	\leq\frac p {\aap} \frac {2d-2} {2d} \sum_x \Xi^{\ssc[0]}_p(x).
	\end{align}
The split using $\Psi^{\ssc[0],\kappa}_{\alpha,{\sss II},p}$ removed all contributions in which $x$ is connected to the origin via the direct bond $(0,x)$, which can only occur when $|x|=1$.
Therefore, any connection in the remainder $\Psi^{\ssc[0],\kappa}_{{\sss R, II},p}(x)$ of the split will use at least two bonds, see \refeq{Def-Psi-Split-II}-\refeq{Def-Psi-Split-III}.
Recalling \refeq{def-psi-zero} we conclude that
	\begin{align}
	\lbeq{PsiZeroR-in-clear}
	\Psi^{\ssc[0],\kappa}_{{\sss R, II},p}(x)&=(1-\delta_{0,x})\frac p \aap \prob_p(\{0\connLe{2} x\} \circ \{ x \connLe{2} 0\}
	\cap \{(x-\ve[\kappa])\nin\tilde{\Ccal}^{\{x,x-\ve[\kappa]\}}(x)\} ).
	\end{align}
We either use the Harris inequality as in \refeq{Harris-bdeQ} or remove the overcounting to obtain the bounds stated in  \refeq{lemmapercboundXi0-bound-7}-\refeq{lemmapercboundXi0-bound-8}.

In $\Psi^{\ssc[0],\kappa}_{\alpha,{\sss I},p}$, defined in \refeq{Def-Psi-Split-I}, we have extracted contributions for which $\|x-\ve[\kappa]\|_2\leq 1$ and $x$ is connected to the origin via one or two steps. Further, by definition, the contribution where $x=0$ does not contribute (see \refeq{def-xi-zero}). The remainder term is thus given by
	\begin{align}
	\frac \aap p  \Psi^{\ssc[0],\kappa}_{{\sss R, I},p}(x)=
	&\indic{\|x-\ve[\kappa]\|_1>1} \prob_p(\{0\dbc x\}\cap \{(x-\ve[\kappa])\nin \tilde{\Ccal}^{\{x,x-\ve[\kappa]\}}(x) \})\lbeq{bound-psi0-step1}\\
 	&+\delta_{x,\ve[\kappa]}\prob_p(\{0\connLe{3}\ve[\kappa]\}\circ\{0\connLe{3}\ve[\kappa]\}
	\cap \{(x-\ve[\kappa])\nin \tilde{\Ccal}^{\{x,x-\ve[\kappa]\}}(x) \})\nnb
	&+\indic{\|x-\ve[\kappa]\|_1=1}\prob_p(\{0\connLe{4}x\}\circ\{0\connLe{4}x\}
	\cap \{(x-\ve[\kappa])\nin
	\tilde{\Ccal}^{\{x,x-\ve[\kappa]\}}(x) \}).\nn
	\end{align}
Here, in the last term, the connection from the origin to $x$ requires at least $4$ steps (see $0\connLe{4}x$), as $|x|=2$, the contribution via two steps has already been removed and by the parity of the lattice.
To bound this we conditioning on the length of the connection between $0$ and $x$:
	\begin{align}
	\frac \aap p \Psi^{\ssc[0],\kappa}_{{\sss R, I},p}(x)\leq &
	\indic{x\nin\{-\ve[\kappa],\ve[\kappa]\}}\prob_p(\{0\connLe{\underline 1}x\}\circ\{0\connLe{3}x\}
	\cap \{(x-\ve[\kappa])\nin \tilde{\Ccal}^{\{x,x-\ve[\kappa]\}}(x)\})\nnb
	&+ \prob_p(\{0\connLe{2}x\}\circ\{0\connLe{2}x\}\cap \{(x-\ve[\kappa])\nin\tilde{\Ccal}^{\{x,x-\ve[\kappa]\}}(x)\}).
	\lbeq{bound-psi0-step2}
	\end{align}
For the first term, we use the Harris inequality as in \refeq{Harris-bdeQ} and bound $\prob_p(\{0\connLe{\underline 1}x\}\circ\{0\connLe{3}x\})$ explicitly. For the second term, we either remove the overcounting or use the Harris inequality as on page \pageref{overcounting} and obtain the bound stated in \refeq{lemmapercboundXi0-bound-5}.

To show the bound on the weighted diagram stated in \refeq{lemmapercboundXi0-bound-6} we first use
 \refeq{XidominatespsiImproved}, then $\|x-\ve[\kappa]\|_2^2=\|x\|_2^2-2x_\kappa+1$ and finally spatial symmetry
	\begin{align}
	\sum_{x\in \Zd} \|x-\ve[\kappa]\|_2^2& \Psi^{\ssc[0],\kappa}_{{\sss R, I},p}(x)
	=   \frac 1 {2d}
	\sum_{x,\kappa} \|x-\ve[\kappa]\|_2^2 \Psi^{\ssc[0],\kappa}_{{\sss R, I},p}(x)\nnb
	\leq&\frac 1 {2d}
	\frac {p} {\aap}\sum_{x\neq 0} \sum_{\kappa}(\|x\|_2^2+1-2 x_\kappa) \prob_p(0\dbc x) = \frac {p} {\aap}\sum_{x\neq 0} (\|x\|_2^2+1) \diagRepulsiveLetter{D}_{1,1}(x)\nn
	\end{align}
This completes the proof of all bounds stated in Lemma \ref{lemmapercboundXi0}.
\end{proof}

\begin{proof}[Proof of Lemma \ref{lemmapercboundXiiota0}.]
By the definition in  \refeq{tau-tauj-LEC-ident-NZero}, see also \refeq{317},
	\begin{align}
    	\lbeq{XiIotaZero-in-clear}
	\Xi^{\ssc[0],\iota}_{p}(x)&=\prob (\{0\ct{\{\ve[\iota]\}} x\} \circ \{\ve[\iota]\dbc x\}\mid (0,\ve[\iota])\text{ vacant} ).
	\end{align}
From \refeq{XiIotaZero-in-clear}, it follows immediately that
	\begin{align}
   	\lbeq{XiIotaZero-simple-comment}
	\Xi^{\ssc[0],\iota}_{p}(\ve[\iota])&= \tau_{3,p }(\ve[1]),\qquad \qquad\quad\Xi^{\ssc[0],\iota}_{p}(0)=0,\\
 	\lbeq{XiIotaZero-simplebound}
	\Xi^{\ssc[0],\iota}_{p}(x)&\leq \tau_{3,p }(\ve[1])\diagRepulsiveLetter{D}_{1,1}(x-\ve[\iota])
	\qquad \text { for }x\neq \ve[\iota],0.
      \end{align}
We conclude the bounds stated in \refeq{lemmapercboundXiiota0-bound-1}-\refeq{lemmapercboundXiiota0-bound-2} from this. To bound the weight $\|x\|_2^2$ we apply \refeq{XiIotaZero-simplebound}, average over $\iota$ and shift the sum over $x$:
\begin{align}
\sum_{x\in\Zd} \Xi^{\ssc[0],\iota}_{p}(x)\|x\|_2^2&\leq
\tau_{3,p }(\ve[1])+ \tau_{3,p }(\ve[1])\sum_{x\neq \ve[\iota]} \prob_p(\ve[\iota]\dbc x)\|x\|_2^2\nnb
&\leq \tau_{3,p }(\ve[1])+ \frac {\tau_{3,p }(\ve[1])}{2d} \sum_{x\neq 0}\sum_{\iota} \prob_p(0\dbc x)\|x+\ve[\iota]\|_2^2
\end{align}
Then, we use $\|x+\ve[\iota]\|_2^2= \|x\|^2_2+2\ve[\iota]+1$ and see that in the sum over $\iota$ the term with
$2\ve[\iota]$ cancels. In a final step we apply the bounds in \refeq{lemmapercboundXiiota0-bound-1}-\refeq{lemmapercboundXiiota0-bound-2} to obtain \refeq{lemmapercboundXiiota0-bound-3}.

For the bounds on $\Xi^{\ssc[0],\iota}_{\alpha,{\sss II},p}$ and $\Xi^{\ssc[0],\iota}_{{\sss R,II},p}$ stated in \refeq{lemmapercboundXiiota0-bound-4},   \refeq{lemmapercboundXiiota0-bound-9}, \refeq{lemmapercboundXiiota0-bound-10},
we recall that we have extracted the major contribution $\tau_{3,p}(\ve[1])$ from the coefficients.
Knowing this, these bounds are shown in the same way as \refeq{lemmapercboundXiiota0-bound-1}-\refeq{lemmapercboundXiiota0-bound-3}.

Using the term $\Xi^{\ssc[0],\iota}_{\alpha,{\sss I},p}$ defined in \refeq{Def-XiIota-Split-I} we extract $\Xi^{\ssc[0],\iota}_{p}(\ve[\iota])= \tau_{3,p }(\ve[1])$ and all contributions in which one connection of the double connection $\{\ve[\iota]\dbc x\}$ is realised by the direct bond $(\ve[\iota],x)$.  The bound on $\Xi^{\ssc[0],\iota}_{\alpha,{\sss I},p}$ stated in \refeq{lemmapercboundXiiota0-bound-4} follows from \refeq{XiIotaZero-simple-comment}.
For the bound in \refeq{lemmapercboundXiiota0-bound-5} we remark that $\iota=-1$ does not contribute to the sum.

In the remainder term $\Xi_{{\sss R,I},p}^{\ssc[0],\iota} (x)$, the connection $\ve[\iota]\conn x$ has length at least two. We use this information together with \refeq{XiIotaZero-simplebound} to obtain the bounds \refeq{lemmapercboundXiiota0-bound-7}-\refeq{lemmapercboundXiiota0-bound-8}.

To complete the proof, we still need to prove the bounds on $\Pi_p^{\ssc[0],\iota,\kappa}$, which is defined as
	\begin{align}
  	\lbeq{PiZero-in-clear}
	\shift\Pi^{\ssc[0],\iota,\kappa}_{p}(x)&=p\prob(\{0\conn \ve[\iota]\} \circ\{\ve[\iota]\dbc x\}
	\cap \{(x+\ve[\kappa])\nin\Ccal^{x,x+\ve[\kappa]}(x)\}\mid (0,\ve[\iota])\text{ is vacant} ),
	\end{align}
see \refeq{tau-tauj-LEC-ident-NZero} and \refeq{317}. It is easy to see that $\Pi^{\ssc[0],\iota,\kappa}_{p}(\ve[\iota])\leq \aap \tau_{3,p}(\ve[1])$ using the Harris inequality.
When summing over $\kappa,$ we note that for a given configuration at least two $\kappa$ do not contribute, namely $\kappa=-\iota$ and the direction of the last step of $\{0\conn\ve[\iota]\}$. So using an  overcounting argument, as on page \pageref{overcounting}, we obtain \refeq{lemmapercboundXiiota0-bound-11}.
In the remainder term $\Pi^{\ssc[0],\iota,\kappa}_{{\sss R},p}$ we know that $x\neq \ve[\iota]$, so that
	\begin{align}
  	\lbeq{Bound-PiRemainder-Simple}
	\sum_{x,\iota}\Pi^{\ssc[0],\iota,\kappa}_{{\sss R},p} (x)\leq& p\sum_{\iota,\kappa} \tau_{3,p}(\ve[\iota])
	\sum_{x\neq \ve[\iota]} \prob(\ve[\iota]\dbc  x)=(2d)^2 p\tau_{3,p}(\ve[1])\sum_{x\neq 0}\diagRepulsiveLetter{D}_{1,1}(x).
	\end{align}
and \refeq{lemmapercboundXiiota0-bound-13} holds. %We improve this bound, to obtain \refeq{lemmapercboundXiiota0-bound-13}, by using the overcounting argument twice: To show that \refeq{lemmapercboundXiiota0-bound-14} holds, we compute
	\begin{align}
	\sum_{x,\iota}\|x-\ve[\iota]-\ve[\kappa]\|_2^2\Pi^{\ssc[0],\iota,\kappa}_{{\sss R},p} (x)
    	&\leq\sum_{\iota,\kappa} p\tau_{3,p}(\ve[\iota])\sum_{x\neq \ve[\iota]}\|x-\ve[\iota]-\ve[\kappa]\|_2^2 \prob(\ve[\iota]\dbc x)\nnb
	&=\sum_{\iota,\kappa}p \tau_{3,p}(\ve[\iota])\sum_{x\neq 0}\|x-\ve[\kappa]\|_2^2  \diagRepulsiveLetter{D}_{1,1}(x)\nnb
	&=2d \sum_{\kappa} p\tau_{3,p}(\ve[1])\sum_{x\neq 0}(\|x\|_2^2+1-2x_\kappa)\diagRepulsiveLetter{D}_{1,1}(x)\nnb
	&=(2d)^2p\tau_{3,p}(\ve[1]) \sum_{x\neq 0}\Big(1+\|x\|_2^2\Big)\diagRepulsiveLetter{D}_{1,1}(x).
\end{align}
This proves the last bound stated in Lemma \ref{lemmapercboundXiiota0}, and thereby completes the proof.
\end{proof}

\subsection{Bounding events}
\label{secBoundsPercEvents}
For $N\geq 1$, the NoBLE-coefficients are defined in terms the probability of $E'(x,y;A)$ events.
In this section, we show that these events are bounded by simpler events, whose probabilities we bound in the following sections. We adapt arguments that can be found in either \cite[Proof of Lemma~2.5]{HarSla90a} or \cite[Proof of Lemma~5.5.8]{MadSla93}.

Let ${\prob}^{{\ssc[N]}}$ denote the product measure on $N+1$ copies of percolation on $\Z^d$, where in the $i$th copy, all bonds emanating from $\bb_{i-1}$ are made vacant, i.e.,
    \eqn{
    \lbeq{PN-def}
    {\prob}^{{\ssc[N]}}
    =\prob^{\sss B}_{\sss 0}\times \prob_{\sss 1}^{\bb_0}\times \cdots\times \prob_{\sss N}^{\bb_{\sss N-1}}.
    }
Using Fubini's Theorem and \refeq{XiNdef}, we conclude that
    \eqan{
    &\Xi^{{\sss B},\ssc[N]}(x,y;A)     \lbeq{XiBAPlainWritten}\\
    &=
    \sum_{b_0,\ldots,b_{\sss N-1}}
    p^N
    {\prob}^{\ssc[N]}\left(\begin{array}{c}E'(x,\bb_0;A)_0\cap\{\tb_0\nin \tilde{\Ccal}_0\}\cap E'(\tb_{0}, \bb_{1}; \tilde{\Ccal}_{0})_1
    \cap\{\tb_1\nin \tilde{\Ccal}_1\}\\
    \cap \big( \bigcap _{i=2}^{N-1}E'(\tb_{i-1}, \bb_{i}; \tilde{\Ccal}_{i-1})_i
    \cap\{\tb_i\nin \tilde{\Ccal}_i\}\big)
    \cap E'(\tb_{\sss N-1}, y; \tilde{\Ccal}_{\sss N-1})_{\sss N}\end{array}\right),\nn
    }
where, for an event $F$, we write $F_i$ to denote that $F$ occurs on the $i$th percolation copy.

We next define events to bound $E'(\tb_{i-1}, \bb_{i}; \tilde{\Ccal}_{i-1})_i$. For increasing events $E,F$, we recall that $E\circ F$ denotes the event that $E$ and $F$ occur disjointly and focus first on the bounding events used to bound $\Xi^{\ssc[N]}_p(x)$. Note that $B=\varnothing$ for $\Xi^{\ssc[N]}_p(x)$. We define the events
    \eqan{
    F_0(b_0,w_0,z_1) =& \{0\conn \bb_0\}\circ
    \{0\conn w_0\}\circ \{w_0\conn \bb_0\}\nnb
    & \circ \{w_0\conn z_1\}\cap\{z_1\nin b_0\},
     \lbeq{Fdefa}\\
    F_{\sss N}(b_{\sss N-1},t_{\sss N},z_{\sss N},y) =& \{\tb_{\sss N-1}\conn t_{\sss N}\}\circ \{t_{\sss N}\conn z_{\sss N}\}
    \circ \{t_{\sss N}\conn y\}\nnb
    & \circ \{z_{\sss N}\conn y\}\cap\{\bb_{\sss N-1}\nin\{t_{\sss N},z_{\sss N},y\}\},
    \lbeq{FNdefa}
    }
for $N\geq 1$, and
    \eqan{
    F'(b_{i-1},t_i,z_i,b_i, w_i, z_{i+1})
    &= \{\tb_{i-1}\conn w_i\}\circ \{w_i\conn \bb_i\}\circ \{w_i\conn z_{i+1}\}\nn\\
    &\qquad \cap\{z_i=\bb_i= t_i\}\cap\{z_{i+1}\nin b_i\}
%    \nnb    &\qquad
    \cap\{\bb_{i-1}\nin\{t_i,w_i,z_i,\bb_i\}\},
    \lbeq{F1defa}\\
    F''(b_{i-1},t_i,z_i,b_i, w_i, z_{i+1})
    &= \{\tb_{i-1}\conn w_i\}\circ \{w_i\conn t_i\}\circ \{t_i\conn z_i\}\circ \{t_i\conn \bb_i\} \nonumber\\
    &\qquad \circ \{z_i\conn \bb_i\}\circ \{w_i\conn z_{i+1}\}\cap\{z_i\neq \bb_i\}\cap\{w_{i}\neq t_i\}\nn\\
    &\qquad \cap\{z_{i+1}\nin b_i\} \cap\{\bb_{i-1}\nin\{t_i,w_i,z_i,\bb_i\}\},
    \lbeq{F2defa}\\
    F'''(b_{i-1},t_i,z_i,b_i, w_i, z_{i+1})
    &= \{\tb_{i-1}\conn t_i\}\circ \{t_i\conn z_i\}\circ \{t_i\conn w_i\}\circ \{z_i\conn \bb_i\}\nonumber\\
    &\qquad \circ \{w_i\conn \bb_i\} \circ \{w_i\conn z_{i+1}\}\cap\{z_i\neq \bb_i\}\cap\{z_{i+1}\nin b_i\}\nnb
    &\qquad \cap\{\bb_{i-1}\nin\{t_i,w_i,z_i,\bb_i\}\},
    \lbeq{F3defa}\\
    F(b_{i-1},t_i,z_i,b_i, w_i, z_{i+1})& =F'(b_{i-1},t_i,z_i,b_i, w_i, z_{i+1})\cup F''(b_{i-1},t_i,z_i,b_i, w_i, z_{i+1})\nnb
    &\qquad\cup F'''(b_{i-1},t_i,z_i,b_i, w_i, z_{i+1}),
    \lbeq{Fdefsa}
}
for $i\in\{1,2,\dots,N-1\}$. The events $F_0$, $F$, $F_{\sss N}$ are depicted in Figure~\ref{fig-F}.

\picPercolationFdefinition[0.8]

These events are similar to those used for the classical lace expansion, see e.g.\ \cite[Section~2.2]{HarSla90a} or \cite[Section 4]{BorChaHofSlaSpe05b}. The difference is that the NoBLE creates additional self-avoidance constraints, which we incorporate into the definition of the $F$-events. These conditions ensure that certain loops in the diagrams have length at least four. In order to bound $E'(\tb_{\sss N-1}, y; \tilde{\Ccal}_{\sss N-1})_{\sss N}$ in terms of these events, we define $E_{\text{vac}}(x)$ to be the event that all bonds that contain $x$  are vacant. Note that $\prob_p^{x}$\, defined in \refeq{tau-b-def-rep-full}, is the same as the percolation measure conditioned on $E_{\text{vac}}(x)$. We will now argue that
	\begin{align}
  	E'(\tb_{\sss N-1}, &y; \tilde{\Ccal}_{\sss N-1})_{\sss N}\cap E_{\text{vac}}(\bb_{\sss N-1})_{\sss N}
	\lbeq{E'bd}\\\nn
    	&\subset  \bigcup_{z_{\sss N}\in \tilde{\Ccal}_{\sss N-1}}
    	\bigcup_{t_{\sss N}\in \ver} F_{\sss N}(b_{\sss N-1},t_{\sss N},z_{\sss N},y)_{\sss N}
	\cap E_{\text{vac}}(\bb_{\sss N-1})_{\sss N},
	\end{align}
where we have defined $E'$ in \refeq{317}. For the event $E'$ to occur there must exist at least one vertex $z_{\sss N}\in\tilde {\Ccal}_{\sss N-1}$ that lies on the last sausage. We denote by $t_{\sss N}$ the first point of the last sausage. As we restrict to the configurations in which $E_{\text{vac}}(\bb_{\sss N-1})_{\sss N}$ holds, we know that only $\bb_{\sss N-1}\nin\{t_{\sss N},z_{\sss N},y\}$ contribute. Since $t_{\sss N},w_{\sss N},z_{\sss N},y\in {\tilde \Ccal}_{\sss N}$, we know that there exists a path of open edges connecting them as shown in Figure~\ref{fig-F}.  In the right-hand side of \refeq{E'bd}, we simply sum over all possible $t_{\sss N}$ and $z_{\sss N}$ and conclude that \refeq{E'bd} holds. Next we argue that, for $N \geq 1$ and $i \in \{1,2,\dots,N-1\}$,
	\begin{align}
    	E'(\tb_{i-1}, \bb_{i}; \tilde{\Ccal}_{i-1})_{i}&
    	\cap \{z_{i+1}\in  \tilde{\Ccal}_i^{b_i}(\tb_{i-1})\} \cap \{ \tb_{i}\nin \tilde{\Ccal}_{i}\}
	\cap E_{\text{vac}}(\bb_{\sss i-1})_i\nn\\
    	\subset \bigcup_{z_{i}\in \tilde{\Ccal}_{i-1}}&  \bigcup_{t_i,w_i\in \ver}
    	F(b_{i-1},t_{i},z_{i},b_{i}, w_{i}, z_{i+1})_{i}\cap E_{\text{vac}}(\bb_{\sss i-1})_i.
    	\lbeq{Eind}
	\end{align}
If $E'(\tb_{i-1}, \bb_{i}; \tilde{\Ccal}_{i-1})_{i}$ occurs, then a string of sausages connects $\tb_{i-1}$ and $\bb_{i}$ and the last sausage of the string is cut through by $\tilde{\Ccal}_{i-1}$.
We denote the ``first'' point  of the last sausage by $t_i$. We identify the first point that every path from $\tb_{i-1}$ to $\bb_{i}$ and from $\tb_{i-1}$ to $z_{i+1}$ share by $w_i$.
By $z_i\in \tilde{\Ccal}_{i-1}\cap \tilde{\Ccal}_i^{b_i}(\tb_{i-1})$ we identify one point in the last sausage, where $\tb_{i-1}\conn \bb_{i}$ is cut through. As $z_i\in \tilde{\Ccal}_{i-1}\cap \tilde{\Ccal}_i^{b_i}(\tb_{i-1})$, with $\tilde{\Ccal}_{j} = \tilde{\Ccal}_j^{b_j}(\tb_{j-1})\cup\{\bb_{j-1}\}$, while $\tb_i\not\in \tilde{\Ccal}_{i}$ and $\bb_i\not\in \tilde{\Ccal}_{i+1}^{b_{i+1}}(\tb_{i})$ since $E_{\text{vac}}(\bb_{i})_{i+1}$ occurs, we know that $z_{i+1}\not \in b_i$.
The restriction on configurations for which $E_{\text{vac}}(\bb_{\sss i-1})_{i}$ occurs, further guarantees that $\bb_{i-1}\nin\{t_i,w_i,z_i,\bb_i\}$.\\
Now we distuinguish between three different cases, characterized by $F'$,$F''$,$F'''$, depending on the relative position of $w_i,t_i$.
The event $F'$ represents the cases for which $z_i=\bb_i$, in which case we define $t_i=\bb_i$.
The event $F''$ represents cases for which $z_i\neq \bb_i$ and that $w_i$ is before the last sausage.
As $t_i$ is on the last sausage we know for this event that $w_i\neq t_i$.
The event $F'''$ captures the configurations in which $w_i$ is in the last sausage.
If $w_{i}$ is in the last sausage, then we choose $z_{i}$ such that it is on the opposite side of the sausage, i.e., we choose $z_i$ such that there exist two bond-disjoint paths from $t_{i} $ to $\bb_i$ such that $z_i$ lies on the path and $w_i$ on the other path. This is always possible as a sausage is formed by a double connection and, since this sausage is cut though by $\tilde{\Ccal}_{i-1}$, all connections contain an element of $\tilde{\Ccal}_{i-1}$. \\
The event at level $0$ is bounded in the same way using $F_0$. As this is completely analogous, we omit further discussions here. We conclude from \refeq{E'bd} and \refeq{Eind} that
    	\eqan{
    	\lbeq{XiFs}
    	\Xi_p^{\ssc[N]}(x)
    	& \leq \sum_{\vec{t},\vec{w},\vec{z},\vec b}
    	p^N \prob_p (F_{\sss 0}(b_0,w_0,z_1)\cap\{\tb_0\nin \tilde{\Ccal}_0\})\\
        & \quad \times
    	\prod_{i=1}^{N-1} \prob_p^{\bb_{i-1}}(F_{\sss i}(b_{i-1},t_i,z_i,b_i, w_i, z_{i+1})
	\cap\{\tb_i\nin \tilde{\Ccal}_i\})
    	\prob_p^{\bb_{\sss N-1}}(F_{\sss N}(b_{\sss N-1},t_{\sss N},z_{\sss N},x)),\nn
    	}
where the summation is over $\vec{t}= (t_0,\ldots,t_{\sss N})$, $\vec{w}= (w_0,\ldots,w_{\sss N-1}), \vec{z}= (z_1,\ldots,z_{\sss N})$ and $\vec b=(b_0,\dots b_{\sss N-1})$. The probabilities in \refeq{XiFs} factor as the events $F_0,\ldots, F_{\sss N}$ occur on different percolation configurations and are thus independent.
If we would at this point follow the classical lace expansion, then we would apply the BK-inequality on \refeq{XiFs} and obtain a bound on $\Xi^\ssc[N]_p$ in terms of combinations of two-point functions $\tau_p$. In doing so, we would lose the information that all loops have length at least four and that the intersection at $z_{i+1}$ cannot occur at the bond $b_i$.

For our bounds, we use one additional property of the diagram that we now explain. Recall the definition of repulsive diagrams in Section \ref{secBoundsSimpleDiagrams}. In most cases, we can choose $z_i$ such that there exists a path from $w_{i-1}$ to $z_i$ in $\tilde{\Ccal}_{i-1}$ that intersects $\tilde{\Ccal}_i$ only at its endpoint $z_i$. We can bound such events using repulsive diagrams. Indeed, if $w_{i}$ is not in the last sausage, then we simply define $z_i$ to be the point in $\tilde{\Ccal}_{i-1}\cap \tilde{\Ccal}_{i}$ with the smallest intrinsic distance to $\bb_{i-1}$. In this case, all paths involved in the above connections are bond disjoint, even when they occur in different levels.

There is one exception, in which we have to resort to non-repulsive diagrams. Indeed, if $w_i$ is in the last sausage in $\tilde{\Ccal}_{i-1}$ containing $\bb_i$, then it can occur that, for every choice of $z_i$, every path from $z_{i}$ to $\bb_{i-1}$ in $\tilde{\Ccal}_{i-1}$ contains at least one bond of any path in $\tilde{\Ccal}_{i}$ connecting $w_{i}$ to $z_{i+1}$, see Figure \ref{fig-intersectSausage} for an example. As we cannot exclude this case, we resort to non-repulsive diagrams to bound the $F'''$ events in which $w_i$ and $z_i$ are both in the sausage containing $\bb_i$.

\begin{figure}[h]
\begin{center}
\picIntersectionSausage[1.2]
\caption{A configuration in which we can not define $z_{i}$ and $w_{i}$ such that the path $\bb_{i-1}\conn z_i$ in $\tilde{\Ccal}_{i-1}$ and the path $z_{i+1}\conn w_i$ in $\tilde{\Ccal}_{i}$ are bond-disjoint. The reason is that all paths connecting $\tb_{i-1}$ to $z_{i+1}$ in $\tilde{\Ccal}_{i}$ (indicated in dark dashed style), as well as all paths in  $\tilde{\Ccal}_{i-1}$ cutting through the double connection from $\tb_{i-1}$ to $\bb_i$ (indicated in solid lines), use the bond $\{u,w_i\}$.} \label{fig-intersectSausage}
\end{center}
\end{figure}
\noindent
Next, we derive a bound as in \refeq{XiFs} for $\Xi^{\ssc[N]}_p(x)$,  but now for $\Xi^{\ssc[N],\iota}_p(x)$ instead, see \refeq{tau-LEC-ident} and \refeq{tau-tauj-LEC-ident1} for their definitions. For $N\geq 1$, $\Xi^{\ssc[N],\iota}_p(x)$ is the sum of two terms. The only difference between the first term $\Xi_p^{b_\iota,\ssc[N]}(0,x;\{\ve[\iota]\})$ and $\Xi_p^{\ssc[N]}(x)=\Xi_p^{\sss \varnothing,\ssc[N]}(0,x;\{0\})$ is at the level of the graphs $0$ and $1$ that describe the configurations $\omega_0$ and $\omega_1$.
Thus, we can use \refeq{E'bd} and \refeq{Eind} to estimate the event defining $\Xi_p^{b_\iota,\ssc[N]}(0,x;\{\ve[\iota]\})$ on levels $i=2,\dots,N$. Since $0\in b_{\iota}$, $\mathrm{B}(\tilde{\Ccal}_{0})\cup \{b_\iota\}=\mathrm{B}(\tilde{\Ccal}_{0})$ and we can use \refeq{Eind} also to bound the event on level $i=1$. To bound the event on level $0$, we define
 \eqan{
    F^{\iota,{\sss \rm I}}_0(b_0,w_0,z_1) =& \{0\conn \ve[\iota]\}\circ
    \{\ve[\iota]\conn w_0\}\circ \{w_0\conn \bb_0\} \circ \{w_0\conn z_1\}\circ \{\ve[\iota]\conn \bb_0\}\nnb
    &\cap\{z_1\nin b_0\}\cap \{b_\iota \text{ is vacant}\},  \lbeq{Fiota1defa}\\
    F^{\iota,{\sss \rm II}}_0(b_0,w_0,z_1) =& \{0\conn w_0\}\circ \{w_0\conn \ve[\iota]\} \circ \{\ve[\iota]\dbc\bb_0\} \circ \{w_0\conn z_1\}\nnb
    &\cap\{w_0\neq \ve[\iota]\}\cap\{z_1\nin b_0\}\cap \{b_\iota \text{ is vacant}\},
     \lbeq{Fiota2defa}\\
    F^{\iota,{\sss \rm III}}_0(b_0,w_0,z_1) =& \{0\conn z_1\}\circ  \{b_\iota \text{ is occupied}\} \cap\{z_1\nin b_0\}
    %\nnb    &
    \cap\{w_0=0\}\cap\{b_0=(0,\ve[\iota])\}.
     \lbeq{Fiota3defa}
}
See Figure \ref{fig-F-iota} for diagrammatic representations of these events.\\
\picPercolationFIotadefinition[0.8]

\noindent
Now, we first argue that
	\begin{align}
	\lbeq{E'-FI-II-inclusion}
	\shift E'(0, \bb_{0}; \{\ve[\iota]\})_0\cap \{ z_{1}\in \tilde{\Ccal}_{0}\}\cap \{b_\iota\text{ is vacant}\}
    	&\subset \bigcup_{w_0\in \Zd} \left(F^{\iota,{\sss \rm I}}_0(b_0,w_0,z_1)
	\cup F^{\iota,{\sss \rm II}}_0(b_0,w_0,z_1)\right).
	\end{align}
Indeed, if $E'(0, \bb_{0}; \{\ve[\iota]\})_0$ occurs, then there exists a path of occupied bonds from $0$ to $\ve[\iota]$. As $\ve[\iota]$ cuts the connection $0\conn\bb_0$, either $\bb_0=\ve[\iota]$ or $\ve[\iota]$ and $\bb_0$ are connected by a sausage. We denote by $w_0$ the last point that the connections $0\conn\bb_0$ and $0\conn z_1$ share. If $w_0$ is on the last sausage, then the event is part of $F^{\iota,{\sss \rm I}}_0$ and otherwise part of $ F^{\iota,{\sss\rm II}}_0$. This proves \refeq{E'-FI-II-inclusion} and bounds the events on the first level of $\Xi_p^{b_\iota,\ssc[N]}(0,x;\{\ve[\iota]\})$.
We conclude that $\Xi_p^{b_\iota,\ssc[N]}(0,x;\{\ve[\iota]\})$ is bounded as in \refeq{XiFs} when $F(b_0,w_0,z_1)$ is replaced by $F^{\iota,{\sss\rm I}}_0(b_0,w_0,z_1)\cup F^{\iota,{\sss\rm II}}_0(b_0,w_0,z_1)$.\\[2mm]
The second part of $\Xi_p^{\ssc[N],\iota}(x)$, defined in \refeq{tau-tauj-LEC-ident1}, is given by
	\begin{eqnarray}
	\lbeq{XiIotaDefPart2}
	p\expect_0^{b_\iota}\left[\indic{\ve[\iota]\nin\tilde \Ccal^{b_\iota}(0)}
	\Xi_p^{b_\iota,\ssc[N-1]}(\ve[\iota],x;\Ccal^{b_\iota}(0))\right].
	\end{eqnarray}
We apply \refeq{XiBAPlainWritten} with $A=\tilde\Ccal^{b_\iota}(0)$ and $B=\{b_\iota\}$ and use \refeq{E'bd} and \refeq{Eind} to bound the event at levels $i=2,\dots,N$. For $i=1$,
\begin{align}
    E'(\ve[\iota], \bb_{1};& \tilde\Ccal^{b_\iota}(0))_{1}
    \cap
    \{ z_{2}\in \tilde\Ccal_1 \}\cap \{ \tb_{1}\nin \tilde{\Ccal}_{1}\} \cap \{b_\iota \text{ is vacant}\}\nnb
    \lbeq{Eind2}
    &\subset \bigcup_{z_{2}\in \tilde\Ccal_1}  \bigcup_{t_1,w_1\in \Zd}
    F(\ve[\iota],t_{1},z_{1},\bb_{1}, w_{1}, z_{2})_{1}\cap \{b_\iota \text{ is vacant}\}.
	\end{align}
For level $0$ in \refeq{XiIotaDefPart2} we see that
	\begin{eqnarray}
	\{ z_{1}\in \tilde{\Ccal}_{0}\}\cap \{\ve[\iota]\nin\tilde \Ccal^{b_\iota}_0(0)\}
	\cap \{b_\iota\text{ is vacant}\} &\subset& F^{\iota,{\sss \rm III}}_0( (0,\ve[\iota]),0,z_1).
	\end{eqnarray}
Thus, the term in \refeq{XiIotaDefPart2}	is bounded as in \refeq{XiFs} when replacing $F(b_0,w_0,z_1)$ by
$F^{\iota,{\sss\rm III}}_0(b_0,w_0,z_1)$.

%$F^{\iota}_0(b_0,w_0,z_1)=F^{\iota,{\sss\rm I}}_0(b_0,w_0,z_1)\cup F^{\iota,{\sss\rm II}}_0(b_0,w_0,z_1)\cup F^{\iota,{\sss\rm III}}_0(b_0,w_0,z_1)$.

In this section we have obtained bounds on $\Xi_p^{\ssc[N]}(x)$ and $\Xi_p^{\ssc[N],\iota}(x)$ in terms of sums over $\vec{t},\vec{w},\vec{z},\vec b$ of products of probabilities of bounding events as in \refeq{XiFs}. The bounding events, in turn, can be bounded using products of two-point functions, or by repulsive diagrams. This leads to enormous sums of complicated diagrams. To structure such sums more effectively, we reformulate them in terms of building blocks in the next section.

\section{Bounding diagrams for $N\geq 1$}
\label{sec-bounds-Ngeq1}
In Section \ref{secBoundsPercDefinitionBlocks}, we define the building blocks that we use to bound the NoBLE coefficients. These building block are defined as combinations of simple diagrams that we can bound numerically, see \cite[Section 5]{FitHof13b}. In Section \ref{sec-bounds-N=1}, we then provide the bounds for $N=1$. In Section \ref{secPercostatingtheBounds}, we extend the bounds to $N\geq 2$.
For $N\geq 2$, we only give bounds on $\Xi^{\ssc[N]}_p$ and $\Xi^{\ssc[N],\iota}_p$. Bounds on $\Psi^{\ssc[N],\kappa}_p$ and $\Pi^{\ssc[N],\iota,\kappa}_p$ follow from \refeq{XidominatespsiImproved}. The proof of these bounds is  discussed in the next section.

\subsection{Building blocks}
\label{secBoundsPercDefinitionBlocks}
The coefficients of the lace expansion are usually displayed as diagrams.
Reviewing the bounding events, which we have created in the preceding section,
the coefficient $\Xi_p^{\ssc[4]}$ is shown in Figure \ref{fig-Form-Xi4}.
\begin{figure}[!htb]
\begin{center}
\picXiStructure[1]
\caption{Diagrammatic representations of the bound on $\Xi_p^{\ssc[4]}(x)$. Lines indicate disjoint connections. Shaded triangles might be trivial.
On the bottom we mark the $F$-event corresponding to the part of the diagram.}
\label{fig-Form-Xi4}
\end{center}
\end{figure}

\noindent
In this section we informally define the simple diagrams that serve as building blocks for our diagrammatic bounds.
In the following sections, we combine these building blocks to construct the bounds on the coefficients, as shown informally in Figure \ref{fig-Form-Xi4-Decomposed}.
\begin{figure}[!htb]
\begin{center}
\picXiDecomposeStructure[0.75]
\caption{Diagrammatic decomposition of $\Xi_p^{\ssc[N]}(x)$. The numbers $l_i$ denote the lengths of shared connections. } \label{fig-Form-Xi4-Decomposed}
\end{center}
\end{figure}

\noindent
The diagrams of the NoBLE-coefficients have stronger repulsive properties than the coefficients of the classical lace expansion. For example, we know that $z_{i+1}\nin b_i$ and that all non-trivial closed triangles/squares consist of at least four occupied bonds.
We use this to obtain sharper bounds. We incorporate information on the lengths of connections shared by two blocks into the definition of the building blocks. This way we can combine the blocks such that all non-trivial loops have length at least four.
We decompose a diagram as shown in Figure \ref{fig-Form-Xi4-Decomposed}, where we denote the length of a line that two squares share by $l_i$.
The length of a connection corresponds to the number of bonds used by the shortest connected path of occupied bonds.
To obtain the infrared bound in $d\geq \dmin$ we distinguish between the  three cases $l_i=0,1$ and $l_i\geq 2$.

\newcolumntype{L}[1]{>{\raggedright\arraybackslash}p{#1}}
\newcolumntype{M}[1]{>{\centering\arraybackslash}p{#1}}
\begin{table}[!htbp]
\centering
We define $a=d_{\Ccal}(0,v)$ and $b=d_{\Ccal}(x,y)$ as the intrinsic distance in the percolation cluster between the points $0$ and $v$ and $x$ and $y$, respectively. We consider $a,b\in\{0,1,\geq 2\}$.
\begin{tabular}[]{|M{3cm}|L{9cm}|}
\hline
\multicolumn{2}{|c|}{Repulsive triangles $P^{{\sss \rm S},b}(x,y)$ and $P^{{\sss \rm E},b}(x,y)$}\\
\hline
\picTable[0.6]& \vspace{-1.5cm} When $x=0$, the triangle shrinks to a point. If $x\neq 0$, then the triangle consists of at least four bonds. For $P^{{\sss \rm E},b}$ we know that $y\neq 0$ and $b\geq 1$.\\
\hline
\multicolumn{2}{|c|}{Open triangle $A^{a,b}(0,v,x,y)$}\\
\hline
\picA[0.6]&  \vspace{-1.6cm} The complete square consists of at least four bonds and $x,y\neq 0$ and $x\neq v$.
The missing connection $0\conn v$ contributes to the neighboring block.\\
\hline
\multicolumn{2}{|c|}{\rule{0pt}{8pt}  Open repulsive triangle with one pivotal edge $A^{\iota,a,b}(0,v,x,y)$, }\\
\hline
\picAIota[0.5]&  \vspace{-1.8cm}  Alike an open triangle with the additional property that $y,v\neq\ve[\iota]$.\\
\hline
\multicolumn{2}{|c|}{\rule{0pt}{10pt}  Double-open bubble $\bar A^{\iota,a,b}(0,v,x,y)$}\\[1mm]
\hline
\picABar[0.5]&  \vspace{-1.5cm} Alike $A^{\iota,a,b}$. Both connections $0\conn v$ and $x\conn y$ contribute to the neighboring blocks. \\
\hline
\multicolumn{2}{|c|}{ \rule{0pt}{9pt}  Open non-repulsive diagram $A^{\iota,a,b,*},\ A^{a,b,*}$ and $\bar A^{\iota,a,b,*}$}\\[1mm]
\hline
&  \vspace{-0.2cm}
 Alike $A^{\iota,a,b}(0,v,x,y),\  A^{a,b}(0,v,x,y)$ and $\bar A^{\iota,a,b}(0,v,x,y)$, with the difference that if $b\neq 0$ the connections are not repulsive.
   \vspace{0.2cm} \\
\hline
\multicolumn{2}{|c|}{\rule{0pt}{9pt}  Double non-trivial triangle $B^{{\sss (2)},\iota,a,b}(0,v,x,y)$, right}\\[1mm]
\hline
\picBTwoPrime[0.8]& \vspace{-2.4cm} A combination of a closed triangle and an open square.
All points $(u,w,y)$ of the small triangle are distinct and $u,w\nin \{0,x\}$.\\
\hline
\multicolumn{2}{|c|}{\rule{0pt}{10pt}Double non-trivial triangle $\bar B^{{\sss (2)},\iota,a,b}(0,v,x,y)$, left}\\[1mm]
\hline
\hspace{-0.5cm}\picBbarPrimeTwo[0.75]& \vspace{-2.2cm} A combination of a closed triangle and an open square.
The small triangle is non-trivial, i.e., $w\neq u,y$ and $x\neq u$, and also $0\nin \{w,w+\ve[\iota]\}$.\\
\hline
\multicolumn{2}{|c|}{The initial piece $P^{\iota,a}$ for $\Xi^{\iota}_p$}\\
\hline
\multicolumn{2}{|c|}{
\begin{tabular}{M{5cm}|L{7cm}}
  \picPiIota[0.7] & \vspace{-1.8cm} When the shaded diagrams are non-trivial, their loop consist of at least four bonds.
The length of $0\conn \ve[\iota]$ is at least three, except when $y=0,x=\ve[\iota]$ and $b=0$.  In the second diagram, $y\neq\ve[\iota]$.
\end{tabular}}\\
\hline
\end{tabular}
\caption{The unweighted building blocks: $P^{{\sss \rm S}}$, $P^{{\sss \rm E}}$, $P^{\iota}$, $A^{\iota}$, $A$, $\bar A$, $B^{{\sss (2)},\iota}$, $\bar B^{{\sss (2)},\iota}$.}
\label{InformalDefinitonOfBlock}
\end{table}

\begin{table}[!ht]
\centering
\begin{tabular}[]{|M{4cm}|M{4cm}|L{4cm}|}
\hline
\multicolumn{3}{|c|}{Weighted double open bubble}\\
\hline
\picDelta[0.8]& \picDeltaIotaTwo[0.8] & \picDeltaIotaThree[0.8]\\
 $H^{{\sss (1)},a,b}(0,v,x,y)$ &
 $H^{{\sss (2)},\iota,a,b}(0,v,x,y)$&
 $H^{{\sss (3)},\iota,a,b}(0,v,x,y)$ \\
 $v=y$ is possible &
 $v=y$ is possible &
 $v\neq y$ \\
 \multicolumn{3}{|c|}{Each complete square consists of at least four bonds.}\\
\hline
\multicolumn{3}{|c|}{Weighted intermediate piece}\\
\hline
\multicolumn{2}{|L{8cm}|}{\picWeightConstructOne[0.5]} & \vspace{-1.8cm} with triangle on top $C^{{\sss (1)},\iota,\kappa,a,b}(0,v,x,y)$\\
\hline
\multicolumn{2}{|L{8cm}|}{\picWeightConstructTwo[0.5]}& \vspace{-1.8cm} with triangle on bottom $C^{{\sss (2)},\iota,\kappa,a,b}(0,v,x,y)$\\
\hline
\multicolumn{3}{|c|}{The weighted initial piece for $\Xi^{\iota}_p$}\\
\hline
\multicolumn{2}{|L{8cm}|}{\picWeightPiIota[0.5]}& \vspace{-1.8cm} $h^{\iota,\kappa,b}(0,v,x,y)$\\
\hline
\end{tabular}
\caption{The weighted building blocks: $H^{\sss (1)}, H^{{\sss (2)},\iota}, H^{{\sss (3)},\iota},C^{\sss (1)}, C^{\sss (2)}$ and $h^{\iota,\kappa}$.}
\label{InformalDefinitonOfBlockDelta}
\end{table}

%\clearpage
The formal definition is quite lengthly, as we need $12$ different building blocks, where each depends on two parameters $a,b\in\{0,1,\geq 2\}$. Thus, we introduce them in this section only informally in Tables \ref{InformalDefinitonOfBlock} and \ref{InformalDefinitonOfBlockDelta}.
\iflongversion
The precise definition of these diagrams is given in Appendix \ref{app-bounds}.
\else
The formal definition of the 100 different cases for the blocks can be found in the appendix of the extended version of this paper \cite{FitHof13d-ext}.
 \fi
To give an idea, we define the block that we use to bound the initial and final triangle:
\begin{eqnarray}
  \lbeq{Pa-exampledefinition}
P^{{\sss \rm S},0}(x,y)&=& \delta_{x,y}\prob(0\dbc x),\\
P^{{\sss \rm S},1}(x,y)&=& \delta_{0,y}\diagRepulsiveLetter{B}_{\underline 1,3}(x,0)+\diagRepulsiveLetter{T}_{ 1,\underline 1,1}(x,y,0),\\
P^{{\sss \rm S},2}(x,y)&=& \delta_{0,y}\diagRepulsiveLetter{D}_{2,2}(x)+\diagRepulsiveLetter{T}_{1,2,1}(x,y,0).
\end{eqnarray}
We combine the diagrams to create larger diagrams. We define $B^{\iota,a,b}$ and $\bar B^{\iota,a,b}$ by
\begin{align}
\lbeq{def-Biota-perc}
B^{\iota,a,b}(0,v,x,y)=&\sum_{u,w}\sum_{c=0}^2 A^{\iota,a,c,*}(0,v,u,w)A^{c,b}(u,w,x,y)\nnb
&+\sum_u A^{\iota,a,b}(0,v,x,u)P^{0}(y-u,y-u) +B^{{\sss (2)},\iota}(0,v,x,y),\\
\bar B^{\iota,a,b}(0,v,x,y)=&\sum_{u,w}\sum_{c=0}^2 A^{\iota,a,c}(0,v,w,u)A^{c,b,*}(w,u,x,y)\nnb
&+A^{\iota,a,b}(0,v,x,y)+\bar B^{{\sss (2)},\iota,a,b}(0,v,x,y).
\end{align}
For example, the block $B^{\iota,a,b}(0,v,x,y)$ corresponds to the middle pieces in Figure \ref{fig-Form-Xi4-Decomposed}.
The non-repulsive diagrams are used to bound the combination of squares, corresponding to the event $F'''$, see Section
\ref{secBoundsPercEvents}. Most weighted diagrams are defined as combinations of unweighted diagrams, e.g.,
\begin{align}
h^{\iota,\kappa,b}(x,y)=&\sum_{a=0}^2\sum_{u,v} \left(\delta_{0,a}\delta_{\kappa,\iota}\delta_{0,u}\delta_{0,v} +P^{\iota,a}(u,v)\right)\bar A^{\kappa,a,b}(u,v,x,y)\|x-\ve[\iota]\|_2^2.
 \end{align}
We also use the following adaptation of $h^{\iota,\kappa}$:
\begin{align}
\lbeq{defHiotaII}
h^{\iota,\kappa,{\sss\rm II},b}(x,y)=& h^{\iota,\kappa,b}(x,y)+
\sum_{c=0}^2\sum_{w,t} A^{\iota,0,c}(0,0,w,t) \sum_{\kappa_2} \bar A^{\kappa_2,c,b,*}(t,w,y,x)\|x-\ve[\iota]\|_2^2\nnb
&+ \sum_{a,c=0}^2\sum_{u,v,w,t}  P^{\iota,a}(u,v) A^{\kappa,a,c}(u,v,w,t)
 \sum_{\kappa_2} \bar A^{\kappa_2,c,b,*}(t,w,y,x)\|x-\ve[\iota]\|_2^2.
\end{align}
\subsection*{Elements of the bounds.}
Here we define the objects which we use to state the bounds on the coefficients.
We define the vectors $\vec P^{\sss \rm S},\vec P^{\sss \rm E}\in\Rbold^{3}$ and the matrices ${ \bf {\bar A}}^\iota, {\bf A}^\iota, {\bf A}$,
${ \bf {\bar A}}^{\iota,*}, {\bf A}^{\iota,*}, {\bf A}^*$, ${\bf B},{\bf \bar B}\in\Rbold^{3\times 3}$ by
\begin{eqnarray*}
(\vec P^{\sss \rm S})_b= \sum_{x,y} P^{{\sss \rm S},b}(x,y), && (\vec P^{\sss \rm E})_b= \sum_{x,y} P^{{\sss \rm E},b}(x,y),
\\
({\bf A}^\iota)_{a,b}= \sup_{v\in\Zd}\sum_{\iota,x,y} A^{\iota,a,b}(0,v,x,y),&&
({\bf A}^{\iota,*})_{a,b}= \sup_{v\in\Zd}\sum_{\iota,x,y} A^{\iota,a,b,*}(0,v,x,y),\\
({\bf A})_{a,b}= \sup_{v\in\Zd}\sum_{x,y} A^{a,b}(0,v,x,y),&&
({\bf A}^{*})_{a,b}= \sup_{v\in\Zd}\sum_{x,y} A^{a,b,*}(0,v,x,y),\\
(\bar {\bf A}^\iota)_{a,b}= \sup_{v,y\in\Zd}\sum_{\iota,x,y} \bar A^{\iota,a,b}(0,v,x,x+y),&&
(\bar {\bf A}^{\iota,*})_{a,b}= \sup_{v,y\in\Zd}\sum_{\iota,x,y} \bar A^{\iota,a,b,*}(0,v,x,x+y),\\
({\bf B})_{a,b}= \sup_{v\in\Zd}\sum_{\iota,x,y} B^{\iota,a,b}(0,v,x,y),&&({\bf \bar  B})_{a,b}= \sup_{v\in\Zd}\sum_{\iota,x,y} \bar B^{\iota,a,b}(0,v,x,y).
\end{eqnarray*}
These are sufficient to state the bounds on $\hat \Xi_p^{\ssc[N]}(0)$. For bounds on weighted diagrams we define the vectors $\vec h^{\sss \rm S}$, $\vec h^{\sss \rm E}$ and the matrices ${\bf H}^{\sss (1)},$ ${\bf H}^{\sss (2)},$ ${\bf H}^{\sss (3)},$ ${\bf C}^{\sss (1)}$ and ${\bf C}^{\sss (2)}$ with entries
\begin{align*}
  ({\bf C}^{\sss (1)})_{a,b}=&\sup_{v,y\in\Zd} \sum_{\iota,\kappa,x}  C^{{\sss (1)},\iota,\kappa,a,b}(0,v,x,x+y),\quad
({\bf C}^{\sss (2)})_{a,b}=\sup_{v,y\in\Zd} \sum_{\iota,\kappa,x}  C^{{\sss (2)},\iota,\kappa,a,b}(0,v,x,x+y),\\
({\bf H}^{\sss (1)})_{a,b}=&\sup_{v,y\in\Zd} \sum_{x} H^{{\sss (1)},a,b}(0,v,x,x+y), \qquad
({\bf H}^{\sss (2)})_{a,b}=\sup_{v,y\in\Zd} \sum_{\iota,x} H^{{\sss (2)},\iota,a,b}(0,v,x,x+y), \\
({\bf H}^{\sss (3)})_{a,b}=&\sup_{v,y\in\Zd} \sum_{\iota,x} H^{{\sss (3)},\iota,a,b}(0,v,x+y,y)\qquad
  (\vec h^{\sss \rm S})_b=({\bf H}^{\sss (1)})_{0,b},\qquad (\vec h^{\sss \rm E})_b=({\bf H}^{\sss (3)})_{0,b}.
  \end{align*}
For bounds on $\Xi_p^{\ssc[N],\iota}$, we additionally require the vectors $\vec P^\iota, \vec h^\iota, \vec h^{\iota,{\sss \rm II}}$ with entries
\begin{align*}
  (\vec P^\iota)_b&= \frac 1 {2d} \Big[ \delta_{0,b}+ \sum_{\iota,x,y} P^{\iota,b}(x,y)\Big]\\
(\vec h^\iota)_b&=\frac 1 {2d} \sum_{\iota,\kappa,x,y} h^{\iota,\kappa,b}(x,y),\qquad
(\vec h^{\iota,{\sss \rm II}})_b=\frac 1 {2d} \sum_{\iota,\kappa,x,y} h^{\iota,\kappa,{\sss \rm II},b}(x,y).
\end{align*}
{\bf Remark}:
For convenience, we will interpret starting vectors, such as $\vec P^{\sss \rm S}$, $\vec P^\iota$, $\vec h^{\sss \rm S}$ and $\vec h^\iota$, as row vectors, while ending vectors such as $\vec P^{\sss \rm E}$ and $\vec h^{\sss \rm E}$ are considered to be column vectors.
\iflongversion
In Appendix \ref{secBoundNBigSpecial}, we explain how we can obtain bounds on ${\bf C}^{\sss (1)}$, ${\bf C}^{\sss (2)}$, $\vec h^\iota$ and $\vec h^{\iota,{\sss \rm II}}$ using the bootstrap functions $f_i(p)$ for $i=1,2,3$.
\fi
With these definitions in hand, we are ready to state the bounds on the NoBLE coefficients, first for $N=1$ and then for $N\geq 2$.

\subsection{Diagrammatic bounds for $N=1$}
\label{sec-bounds-N=1}
In this section we state bounds on the NoBLE coefficients for $N=1$ and provide a part of the bounds assumed in \cite[Assumption 4.3]{FitHof13b}.
We abbreviate $\vec u=(1,0,0)$  and
	\begin{align}
  	\diagRepulsiveLetter{H}_{n}(x):=
	&\max\Big\{
	\sum_y \|y\|_2^2 \diagRepulsiveLetter{B}_{0,n}(y,x),
	\sum_{e,y} \|y\|_2^2 \diagRepulsiveLetter{T}_{1,n-1,\underline{1}}(y,x-e,x),
	\label{Hi-defs}
	\\	&\qquad \qquad
	%\sum_{e,y} \|y\|_2^2 \diagRepulsiveLetter{T}_{\underline{1},1,n}(e,y,x),
	\sum_{e,y} \|y\|_2^2  \diagRepulsiveLetter{T}_{1,\underline 1,n-1}(y,e+y,x)
	\Big\},\nnb
	\label{HD-def}
	\diagRepulsiveLetter{H}^{\sss D}_{n}:=&\sum_x \|x\|_2^2 \diagRepulsiveLetter{D}_{n,n}(x),\\
	\beta_{{\sss \Delta \Xi}}^\ssc[1]:=
	&\vec u^T  {\bf H}^{\sss (3)}\vec u
	+2\diagRepulsiveLetter{H}_{2}(0)\Big(\sum_{x}\diagRepulsiveLetter{D}_{1,1}(x)\Big)
	+2\diagRepulsiveLetter{H}^{\sss D}_{1}(0)\Big(\sum_{x}\diagRepulsiveLetter{B}_{1,2}(x,0)\Big)\nnb
      &+8d p \diagRepulsiveLetter{H}_{2}(\ve[1]) \sum_{x}\diagRepulsiveLetter{B}_{1,1}(x,\ve[1]) +8d p \diagRepulsiveLetter{H}_{1}(\ve[1]) \sum_{\iota,x}\diagRepulsiveLetter{T}_{1,1,\underline 1}(x,\ve[1]+\ve[\iota],\ve[1])\nnb
	&+\Big(\sup_{x\neq 0} \diagRepulsiveLetter{H}_{2}(x)\Big)
	\Big(\sum_{x}\diagRepulsiveLetter{D}_{1,1}(x)
	+4\sum_{x,y}\diagRepulsiveLetter{T}_{1,2,1}(x,y,0)\Big)\nnb
	&+4\Big(\sup_{x\neq 0} \diagRepulsiveLetter{H}_{1}(x)\Big) \Big(\sum_{\iota,x,y}\diagRepulsiveLetter{S}_{\underline 1, 1,1,2}(\ve[\iota],x,y,0)\Big)\nnb
	&+3 \vec h^{{\sss \rm S}} ({\bf A}^\iota)^T (\vec P^{\sss \rm E}-\vec u)
	+3 (\vec P^{\sss \rm S}-\vec u^T) {\bf H}^{\sss (3)} (\vec P^{\sss \rm E}-\vec u)
	+3 (\vec P^{\sss \rm S}-\vec u^T) {\bf A}^\iota\vec h^{\sss \rm E}.
	\end{align}

%(*Note that we use for this an repulisive version of the matrix (A^\[Iota]). For other diagram we could not choose A^\[Iota] to consist of repulsive diagrams as there exists (unlikely) combination in which we can not choose the point that characterizes the cut through (z) such that the connections  w\[LongLeftRightArrow]z do not intersect  Subscript[Overscript[b, _], 0]\[LongLeftRightArrow] t, see Figure 8 of (II). For N=1 this event can however not occur, so we use the better, repulsive bounds.*)

\begin{lemma}[Bounds on $\Xi^{\ssc[1]}_p$ and $\Psi^{\ssc[1],\kappa}_p$]
\label{lemmapercboundXi1}
Let $p<p_c$. Then,
\begin{align}
  \lbeq{lemmapercboundXi1-1}
\sum_{x\in\Zd}\Xi^{\ssc[1]}_{p}(x) \leq&\ \vec P^{\sss \rm S} {\bf \bar A}^\iota \vec P^{\sss \rm E},\\
    \lbeq{lemmapercboundXi1-3}
\sum_{x}\Xi_{{\sss R},p}^\ssc[1](x) \leq &\ \vec P^{\sss \rm S} {\bf \bar A}^\iota \vec P^{\sss \rm E}-({\bf \bar A}^\iota)_{0,0}+\sum_{\iota,x}\diagRepulsiveLetter{T}_{\underline 1,1,2}(\ve[\iota],x,0),\\
\sum_{x}\Psi_{{\sss R,I},p}^{\ssc[1],\kappa} (x) \leq&  \frac {2d-1}{2d}
\frac p {\aap} \Big(\vec P^{\sss \rm S} {\bf \bar A}^\iota \vec P^{\sss \rm E}-({\bf \bar A}^\iota)_{0,0}+ (2d-2)\sum_{\iota}\diagRepulsiveLetter{T}_{\underline 1,2,\underline 1}(\ve[\iota],\ve[2],0)\Big)\nnb
    \lbeq{lemmapercboundXi1-5}
&+\frac {2d-1}{2d}\frac p {\aap}\Big( \sum_{x} \sum_{\iota} \diagRepulsiveLetter{T}_{\underline 1,1,2}(\ve[\iota],x,0)\Big),\\
  \lbeq{lemmapercboundXi1-6}
\sum_{x}\Psi_{{\sss R,II},p}^{\ssc[1],\kappa} (x)\leq& \frac {2d-1}{2d}
\frac p {\aap} \Big(\vec P^{\sss \rm S} {\bf \bar A}^\iota \vec P^{\sss \rm E}-({\bf \bar A}^\iota)_{0,0}+\sum_{\iota,x}\diagRepulsiveLetter{T}_{\underline 1,1,2}(\ve[\iota],x,0)\Big).
\end{align}
The weighted diagrams are bounded by
\begin{align}
  \lbeq{lemmapercboundXi1-2}
  \sum_{x\in\Zd} \|x\|_2^2 \Xi^{\ssc[1]}_p(x)\leq&\ \beta_{{\sss \Delta \Xi}}^\ssc[1],\\
  \lbeq{lemmapercboundXi1-4}
  \sum_{x}\|x\|_2^2\Xi_{{\sss R},p}^\ssc[1](x) \leq&\ \beta_{{\sss \Delta \Xi}}^\ssc[1]-({\bf H}^{\sss (2)})_{0,0}+\sum_{\iota,x}\|x\|_2^2\diagRepulsiveLetter{T}_{\underline 1,1,2}(\ve[\iota],x,0),\\
  \lbeq{lemmapercboundXi1-7}
\sum_{x}\|x-\ve[\kappa]\|_2^2\Psi_{{\sss R,I},p}^{\ssc[1],\kappa} (x)\leq&\ \frac {p}{\aap}\left(\beta_{{\sss \Delta \Xi}}^\ssc[1]+\vec P^{\sss \rm S} {\bf \bar A}^\iota \vec P^{\sss \rm E}\right)\\
  \lbeq{lemmapercboundXi1-8}
\sum_{x}\|x\|_2^2\Psi_{{\sss R,II},p}^{\ssc[1],\kappa} (x)\leq&\frac p {\aap}
\Big(\beta_{{\sss \Delta \Xi}}^\ssc[1]-({\bf H}^{\sss (2)})_{0,0}+\sum_{\iota,x}\|x\|_2^2\diagRepulsiveLetter{T}_{\underline 1,1,2}(\ve[\iota],x,0)\Big).
\end{align}
\end{lemma}

\begin{lemma}[Bounds on $\Xi^{\ssc[1],\iota}_p$]
\label{lemmapercboundXiiota1}
Let $p<p_c$. Then,
\begin{align}
\lbeq{BoundXiIotaOne-1}
\sum_{x}  \Xi^{\ssc[1],\iota}_{p}(x) \leq& \vec P^\iota {\bf \bar A}^\iota \vec P^{\sss \rm E},\\
\lbeq{BoundXiIotaOne-2}
\sum_{x} \|x-\ve[\iota]\|_2^2 \Xi^{\ssc[1],\iota}_{p}(x)\leq& (\vec h^{\iota})_{0}+ 2 \vec h^\iota(\vec P^{\sss \rm E}-(1,0,0)) +2 \vec P^\iota{\bf A}^\iota\vec h^{\sss \rm E},\\
\lbeq{BoundXiIotaOne-3}
\sum_{x}\|x\|_2^2 \Xi^{\ssc[1],\iota}_{p}(x) \leq &2\sum_{\iota,x} (\|x-\ve[\iota]\|_2^2+1) \Xi^{\ssc[1],\iota}_{p}(x).
\end{align}
%Further, the following lower bounds hold:
\end{lemma}

\begin{lemma}[Lower bounds]
\label{lemmapercboundLowerBounds}
Let
\begin{align}
\theta_2&=\max\{\tau_{2,p}^{\sss 1}(2\ve[1]),\tau_{2,p}^{\sss 1}(\ve[1]+\ve[2])\},\quad  \theta_4=\max\{\tau_{4,p}^{\sss 1}(2\ve[1]),\tau_{4,p}^{\sss 1}(\ve[1]+\ve[2])\},\\
\vartheta&=\frac {d^2}{(d-1)(d-2)}(D^{\star 3}\star \tau_{5,p})(0).
\end{align}
The following lower bounds hold for all $p<p_c$:
\begin{align}
\sum_{\kappa}\Pi^{\ssc[0],\iota,\kappa}_{\alpha,p}(\ve[\iota])\geq&
(2d-1)(2d-2)p^4 (1-p^3)^{2d-3}
-(2d-2)^2p^4\tau_{3,p}(\ve[1])^2\nnb
&-(2d-2)p^4\left( \tau_{4,p}^{1}(2\ve[1])+(4d-5) \tau_{2,p}^{1}(\ve[1]+\ve[2])+(4d-4)p^3+2d \vartheta\right)\nnb
&+16(d-1)(d-2)(2d-3)p^6(1-p^3)^{2d-2}(1-p^5)^{16(d-1)(d-2)(2d-3)-1}\nnb
&\qquad \times\Big(1-\tau_{3,p}(\ve[1])-\theta_2-3\sup_{x\neq 0}\tau_{1,p}(\ve[1])\Big).
\lbeq{lemmaLowerbound-1}
\end{align}
Further,
\begin{align}
  \lbeq{lemmaLowerbound-2}
\hat \Psi^{\ssc[0],\kappa}_p(0)\geq&\
 (2d-1)(2d-2)p^4 (1-p^3)^{2d-3}
-(2d-2)^2p^4\tau_{3,p}(\ve[1])^2\\
&-(2d-2)p^4\left( \tau_{4,p}^{1}(2\ve[1])+(4d-5) \tau_{2,p}^{1}(\ve[1]+\ve[2])+(4d-4)p^3+2d \vartheta\right)\nnb
&+(2d-2)^2 p^4 (1-\tau_{3,p}(\ve[1])-2\theta_2-2p^3 )-2d (2d-2) p^4\vartheta\nnb
&+ 64d(d-1)(d-2)(1-p^4)^{2d-2}(1-p^6)^{16(d-1)(d-2)}\nnb
&\quad \times\left(1-\tau_{3,p}(\ve[1])-2p^2-2\theta_2- 2\vartheta-\tau^{\sss 1}_{5,p}(2\ve[1]+\ve[2])\right),\nn
\end{align}
and
\begin{align}
\lbeq{lemmaLowerbound-3}
\sum_{\kappa}\hat \Pi^{\ssc[1],\iota,\kappa}_p(0)\geq
&(2d-1)(2d-2) p^5(1-p-3\tau_{3,p}(\ve[1])-\theta_2-\theta_4)-(2d-2)p^5(\theta_4 +\vartheta)\\
&+(2d-2)^2(2d-3)p^7(1-p^3)^{2d-3} \left( 1-p-2p^2-2p^3-2\tau_{3,p}(\ve[1]) -4\theta_4-2\vartheta\right)\nnb
&\ \times \left( 1-\tau_{3,p}(\ve[1]) -\theta_2-\vartheta\right)\nnb
&+(2d-2)^2(2d-3)p^7(1-p^3)^{2d-3}(1-p-p^2-2\tau_{3,p}(\ve[1])-2\theta_4-\vartheta)\nnb
&\ \times  (1-\tau_{3,p}(\ve[1])-\theta_2)\nnb
&+(2d-2)^3(2d-3)p^8(1-p^3)^{2d-3}\left( 1-2\tau_{3,p}(\ve[1]) -\theta_2-\vartheta\right)\nnb
&\ \times \left( 1-p-2p^2-2\tau_{3,p}(\ve[1]) -4\theta_4-3\vartheta-(2dp)^4 (D^{\star 4}\star \tau_{4,p})(0)\right).\nn
\end{align}
\end{lemma}

We expect that the coefficients for $N=0$ and $N=1$ are of comparable size.
Thus, we use the following bounds to cancel out a part of the NoBLE coefficient in our analysis:
\begin{lemma}[Bounds on differences]
\label{lemmapercboundXi0minus1}
Let $p<p_c$. Then,
\begin{align}
\lbeq{Differencebound-1}
\Xi^\ssc[0]_{\alpha,p}(0)-\Xi^\ssc[1]_{\alpha,p}(0)=&0,\\
\lbeq{Differencebound-2}
\Xi^\ssc[1]_{\alpha,p}(\ve[1])-\Xi^\ssc[0]_{\alpha,p}(\ve[1])\leq&
(2d-2)p^4\left(1-(1-p^3)^{2d-3}\right)+p^2((2d-2)\tau_{4,p}^1(\ve[1]+\ve[2])+\tau_{4,p}^1(2\ve[1])),\\
\lbeq{Differencebound-3}
\Xi^\ssc[0]_{\alpha,p}(\ve[1])-\Xi^\ssc[1]_{\alpha,p}(\ve[1])\leq&
p\tau_{5,p}(\ve[1])+(2d-2)p^5+2(2d-2)p^4 (\tau_{3,p}(\ve[1])+\tau_{4,p}^1(\ve[1]+\ve[2])),
\end{align}
and
\begin{align}
\lbeq{Differencebound-5}
\sum_\iota \big(\Psi^{\ssc[0],1}_{\alpha,{\sss I},p}&(\ve[1]+\ve[\iota])  -\Psi^{\ssc[1],1}_{\alpha,{\sss I},p}(\ve[1]+\ve[\iota]) \big)\\
\leq&
(2d-2)\frac {p^5}{\aap}+ \frac {p^3}{\aap} \left( 2(2d-2)\tau_{4,p}^{\sss 1 }(\ve[1]+\ve[2])+\tau_{4,p}^{\sss 1 }(2\ve[1])+(4d-3)\tau_{3,p}(\ve[1])^2\right)\nnb
&+(2d-2) \frac {p^5}  {\aap}
\left[1-2(1-2\tau_{3,p}(\ve[1])-\tau_{2,p}^{\sss 1}(\ve[1]+\ve[2]))(1-\tau_{3,p}(\ve[1])-\theta_2)\right],\nnb
\lbeq{Differencebound-4}
\sum_\iota \big(\Psi^{\ssc[1],1}_{\alpha,{\sss I},p}& (\ve[1]+\ve[\iota])  -\Psi^{\ssc[0],1}_{\alpha,{\sss I},p}(\ve[1]+\ve[\iota])\big)\\
&\leq\frac {p^3}{\aap}\left((2d-2)p^2+ 2(2d-2) \tau_{4,p}^{\sss 1}(\ve[1]+\ve[2])+\tau_{2,p}^{\sss 1}(2\ve[1])\right)\nnb
&\quad+(2d-2)\frac {p^5}{\aap}
\left[ \tau_{3,p}(\ve[1])+2\theta_2+\tau_{5,p}^{\sss 1}(2\ve[1]+\ve[2]))\right],\nn
\end{align}
\begin{align}
\sum_\iota &\big(\Psi^{\ssc[0],1}_{\alpha,{\sss II},p}(\ve[\iota])  -\Psi^{\ssc[1],1}_{\alpha,{\sss II},p}(\ve[\iota])\big)
\lbeq{Differencebound-7}\leq (2d-1)p\tau_{5,p}(\ve[1]) \\
&+(2d-1)(2d-2)\frac {p^5}\aap
\left[1-\left(1-p-2\tau_{3,p}(\ve[1])-2\theta_4\right)
\left(1-\tau_{3,p}(\ve[1])-\theta_2-\vartheta\right)\right]\nnb
\lbeq{Differencebound-6}
 \sum_\iota \big(\Psi^{\ssc[1],1}_{\alpha,{\sss II},p}&(\ve[\iota])  -\Psi^{\ssc[0],1}_{\alpha,{\sss II},p}(\ve[\iota])\big)\\
\leq &(2d-1)(2d-2)\frac {p^5}{\aap} \left[ 1-(1-p^3)^{2d-3}\right]\nnb
&\quad+\frac {(2d-1)p^3}{\aap} \left((2d-2)\tau^{\sss 1}_{4,p}(\ve[1]+\ve[2])+\tau^{\sss 1}_{4,p}(2\ve[1])\right)+(2d-2)^2p^4\tau_{3,p}(\ve[1])^2\nnb
&\quad+(2d-2)p^4\left( \tau_{4,p}^{1}(2\ve[1])+(4d-5) \tau_{2,p}^{1}(\ve[1]+\ve[2])+(4d-4)p^3+2d \vartheta\right).\nn
\end{align}
\end{lemma}

Lemmas \ref{lemmapercboundXi1}--\ref{lemmapercboundXiiota1} and the lower bounds in Lemma \ref{lemmapercboundLowerBounds} for $N=0$ are proved in Section \ref{secProofBoundsNOne}.
\iflongversion
The proof of the missing lower bound in Lemma \ref{lemmapercboundLowerBounds} for $N=1$, as well as the proof of Lemma \ref{lemmapercboundXi0minus1}, can be found in
 Appendix \ref{Appendix-ToTheProofs-Difference}.
\else
The proof of the missing lower bound in Lemma \ref{lemmapercboundLowerBounds} for $N=1$, as well as the proof of Lemma \ref{lemmapercboundXi0minus1}, is given in \cite[Appendix B]{FitHof13d-ext}. It is based on similar ideas as used in the proof of Lemma \ref{lemmapercboundXi1} and \ref{lemmapercboundLowerBounds}.
\fi

\subsection{Diagrammatic bounds for $N\geq 2$}
\label{secPercostatingtheBounds}
We state the bounds on the NoBLE-coefficients for $N\geq 2$ in Propositions \ref{PropBoundXiBig} and \ref{PropBoundXiIotaBig}. We discuss the proof of these bounds in Section \ref{secProofBoundsNBig}.
\begin{prop}[Bounds on $\Xi^{\ssc[N]}_p$ for $N\geq 2$]
\label{PropBoundXiBig}
Let $p<p_c$. Then,
\begin{align}
\lbeq{BoundXiPercNTwo}
\hat \Xi^{\ssc[N]}_{p}(0)\leq& \vec P^{\sss \rm S} ({\bf B}^\iota)^{N-1} {\bf \bar  A}^\iota \vec P^{\sss \rm E}.
\end{align}
For $N\geq 2$ even,
\begin{align}
\lbeq{BoundXiDeltaTreeNTwoEven}
\sum_{x}\|x\|_2^2\Xi^{\ssc[N]}_{p}(x)\leq& (N+2) \left[ \vec h^{\sss \rm S} \left( {\bf A}^\iota\right)^T ({\bf \bar B}^\iota)^{N-1}  \vec P^{\sss \rm E}+  \vec P^{\sss \rm S} ({\bf B}^\iota)^{N-1} ({\bf H}^{\sss (3)} \vec P^{\sss \rm E}+{\bf A}^\iota \vec h^{\sss \rm E})\right]\\
&+(N+2)\sum_{M=0}^{N/2-1}
\vec P^{\sss \rm S} ({\bf B}^\iota)^{2M} \left(
{\bf C}^{\sss (1)}{\bf \bar  B}^\iota+\indic{M\neq \tfrac N 2-1}{\bf B}^\iota{\bf C}^{\sss (2)}\right) ({\bf \bar  B}^\iota)^{N-3-2M}\vec P^{\sss \rm E}, \nn
\end{align}
and, for $N\geq 2$ odd,
\begin{align}
\lbeq{BoundXiDeltaTreeNTwoOdd}
\sum_{x}\|x\|_2^2\Xi^{\ssc[N]}_{p}(x)\leq& (N+2)
\left[ \vec h^{\sss \rm S} \left({\bf A}^\iota\right)^T ({\bf  \bar  B}^\iota)^{N-1} \vec P^{\sss \rm E}+  \vec P^{\sss \rm S} ({\bf B}^\iota)^{N-1} ({\bf H}^{\sss (2)} \vec P^{\sss \rm E}+{\bf A}^\iota \vec h^{\sss \rm E})\right]\\
&+(N+2)\sum_{M=0}^{(N-3)/2} \vec P^{\sss \rm S} ({\bf B}^\iota)^{2M} ({\bf C}^{\sss (1)}{\bf \bar  B}^\iota+{\bf B}^\iota {\bf C}^{\sss (2)})( {\bf \bar B}^\iota)^{N-3-2M}\vec P^{\sss \rm E}. \nn
\end{align}
\end{prop}

\begin{prop}[Bounds on $\Xi^{\ssc[N],\iota}_p$ for $N\geq 2$]
\label{PropBoundXiIotaBig}
Let $p<p_c$. Then,
\begin{align}
\lbeq{BoundXiPercNTwoIota}
\hat \Xi^{\ssc[N],\iota}_{p}(0)\leq& \vec P^\iota ({\bf B}^\iota)^{N-1} {\bf \bar  A}^{\iota}\vec P^{\sss \rm E}.
\end{align}
For $N\geq 2$ even,
\begin{align}
\lbeq{BoundXiIotaDeltapercNTwoEvenIota}
\sum_{x} \|x-\ve[\iota]\|_2^2 \Xi^{\ssc[N],\iota}_{p}(x)
\leq& (N+1) \left[ \vec h^{\iota,{\sss\rm II}} ({\bf \bar  B}^\iota)^{N-1} \vec P^{\sss \rm E}+  \vec P^\iota ({\bf B}^\iota)^{N-1} ({\bf H}^{\sss (3)} \vec P^{\sss \rm E}+{\bf A}^\iota \vec h^{\sss \rm E})\right]\\
&+(N+1)\indic{N\geq 4}\sum_{M=0}^{(N-4)/2} \vec P^\iota
({\bf B}^\iota)^{2M+1} ({\bf C}^{\sss (2)}{\bf \bar B}^\iota+{\bf B}^\iota {\bf C}^{\sss (1)})
({\bf \bar  B}^\iota)^{N-4-2M}\vec P^{\sss \rm E},\nnb
\lbeq{BoundXiIotaDeltapercNTwoEven}
\sum_{x} \|x\|_2^2 \Xi^{\ssc[N],\iota}_{p}(x)
\leq& (N+2) \left[ \vec h^{\iota,{\sss\rm II}} ({\bf \bar  B}^\iota)^{N-1} \vec P+  \vec P^\iota ({\bf B}^\iota)^{N-1} ({\bf H}^{\sss (3)} \vec P^{\sss \rm E}+{\bf A}^\iota \vec h^{\sss \rm E})\right]\\
&+(N+2)\indic{N\geq 4}\sum_{M=0}^{(N-4)/2} \vec P^\iota
({\bf B}^\iota)^{2M+1} ({\bf C}^{\sss (2)}{\bf \bar B}^\iota+{\bf B}^\iota {\bf C}^{\sss (1)})
({\bf \bar  B}^\iota)^{N-4-2M}\vec P^{\sss \rm E}\nnb
&+ (N+2) (\vec P^\iota ({\bf B}^\iota)^{N-1} {\bf \bar A}^{\iota} \vec P^{\sss \rm E}). \nn
\end{align}
For $N\geq 2$ odd,
\begin{align}
\lbeq{BoundXiIotaDeltapercNTwoOddIota}
\sum_{x} \|x-\ve[\iota]\|_2^2 \Xi^{\ssc[N],\iota}_{p}(x)
\leq&  (N+1) \left[ \vec h^{\iota,{\sss \rm II}} ({\bf \bar  B}^\iota)^{N-1} \vec P^{\sss \rm E}+  \vec P^\iota ({\bf B}^\iota)^{N-1} ({\bf H}^{\sss (2)}\vec P^{\sss \rm E}+{\bf A}^\iota \vec h^{\sss \rm E})\right]\\
&+ (N+1)\sum_{M=0}^{(N-3)/2} \vec P^\iota ({\bf B}^\iota)^{2M+1} {\bf C}^{\sss (2)}( {\bf \bar B}^\iota)^{N-3-2M}\vec P^{\sss \rm E}\nnb
&+ (N+1)\sum_{M=0}^{(N-5)/2} \vec P^\iota ({\bf B}^\iota)^{2M+2} {\bf C}^{\sss (1)}({\bf \bar B}^\iota)^{N-4-2M}\vec P^{\sss \rm E},\nnb
\lbeq{BoundXiIotaDeltapercNTwoOdd}
\sum_{x} \|x\|_2^2 \Xi^{\ssc[N],\iota}_{p}(x)
\leq& (N+2)
\left[ \vec h^{\iota,{\sss \rm II}} ({\bf \bar  B}^\iota)^{N-1} \vec P^{\sss \rm E}+  \vec P^\iota ({\bf B}^\iota)^{N-1} ({\bf H}^{\sss (2)} \vec P^{\sss \rm E}+{\bf A}^\iota \vec h^{\sss \rm E})\right]\\
&+ (N+2)\sum_{M=0}^{(N-3)/2} \vec P^\iota ({\bf B}^\iota)^{2M} ({\bf C}^{\sss (1)}{\bf \bar  B}^\iota+{\bf B}^\iota {\bf C}^{\sss (2)})( {\bf \bar B}^\iota)^{N-3-2M}\vec P^{\sss \rm E}\nnb
&+ (N+2) \vec P^\iota ({\bf B}^\iota)^{N-1} {\bf \bar A}^{\iota} \vec P^{\sss \rm E}.\nonumber
\end{align}
\end{prop}
In Section \ref{secProofBounds} we explain how these bounds are proved, see in particular Section \ref{secProofBoundsNBig}. The objects appearing in the bounds in Proposition \ref{PropBoundXiBig}--\ref{PropBoundXiIotaBig} can be evaluated numerically using the assumed bounds on $f_1(p)$, $f_2(p)$ and $f_3(p)$ and methods proved in the Mathematica notebook \verb|Percolation|. To obtain the mean-field result in dimension $d=11$ and $d=12$ we further improved the bounds stated in Proposition \ref{PropBoundXiBig}--\ref{PropBoundXiIotaBig} for $N=2,3$, by considering the special cases that the left- and right-most triangles are trivial, see Figure \ref{fig-Form-Xi4}.\\
In Section \ref{secProofBoundsNBig}, we sketch the proof of the proposition and comment on the improvement for $N=2,3$. A detailed explanation of these bounds and their proof can be found in Chapter 4 of the thesis of the first author \cite{Fit13}, which can be downloaded from \cite{FitNoblePage}.

\subsection{Summary of the bounds}
\label{secSummaryBounds}
We have now stated all the bounds on the NoBLE coefficients required for the NoBLE analysis, that is explained in \cite{FitHof13b} on the level of diagrams. In doing so we have proven Proposition \ref{prop-bds-LEC}. \\
In this section we review where to find the bounds stated in \cite[Assumption 4.3]{FitHof13b}.
We want to emphasize that, next to the diagrammatic bounds proven in this document, the proof of \cite[Assumption 4.3]{FitHof13b} also requires an analysis that enables us to bound them numerically and a computer program that computes the stated bounds numerically. Only once these numerical bounds are computed we can apply the analysis of \cite{FitHof13b} to obtain the mean-field results.
In the previous sections, we have proven the diagrammatic bounds that allow us to prove the following assumption. We first state it, and then check all required bounds one by one:
\begin{ass}[{\cite[Assumption 4.3]{FitHof13b}}: Diagrammatic bounds]
\label{assDiagBoundsCoeff}
Let $\Gamma_1,\Gamma_2,\Gamma_3\geq 0$. Assume that $p\in (p_I,p_c)$ is such that $f_i(p)\leq \Gamma_i$ holds.
Then $\hat \tau_p(k)\geq 0$ for all $k\in(-\pi,\pi)^d$. There exists $\betaaa\geq 1,\betaaalow>0$ such that
	\begin{align}
	\lbeq{analys-assumed-rho-Bound-two}
	\frac {\aabp}{ \aap}\leq \betaaa,\qquad \aap \geq \betaaalow.
	\end{align}
Further, there exist $\beta_{\sss \Xi}^\ssc[N],\beta_{\sss \Xi^\iota}^\ssc[N],\beta_{{\sss \Delta \Xi}}^\ssc[N],\beta_{{\sss \Delta \Xi^{\iota}},0}^\ssc[N],\beta_{{\sss \Delta \Xi^{\iota}},\iota}^\ssc[N] \geq 0$, such that
	\begin{align}
 	\lbeq{analys-assumed-BoundXi}
	\hat \Xi^\ssc[N]_p(0)&\leq \beta_{\sss \Xi}^\ssc[N],\qquad\qquad \hat \Xi^{\ssc[N],\iota}_p(0)
	\leq \beta_{{\sss \Xi}^{\iota}}^\ssc[N],\\
	\lbeq{analys-assumed-displacement-one-InXDiff}
	\sum_{x}\|x\|_2^2\Xi^{\ssc[N]}_p(x)&\leq \beta_{{\sss \Delta \Xi}}^\ssc[N],
	\qquad\quad\sum_{x} \|x\|_2^2 \Xi^{\ssc[N],\iota}_p(x)\leq \beta_{{\sss \Delta \Xi^{\iota},0}}^\ssc[N],\\
	\lbeq{analys-assumed-displacement-three-InXDiff}
	\sum_{x} \|x-\ve[\iota]\|_2^2 \Xi^{\ssc[N],\iota}_p(x)&\leq \beta_{{\sss \Delta \Xi^{\iota},\iota}}^\ssc[N],
	\end{align}
for all $N\geq 0$ and $k\in(-\pi,\pi)^d$. Moreover, we assume that $\sum_{N=0}^\infty \beta_{\bullet}^\ssc[N] <\infty$ for $\bullet \in\{ \Xi, \Xi^{\iota}, \Delta\Xi, \{\Delta \Xi^\iota,0\},\{\Delta
 \Xi^\iota,1\}\}$ and that
	\begin{align}
	\lbeq{analys-assumed-invertablecondition}
	\frac {(2d-1)\aabp}{1-\aap}\sum_{N=0}^\infty\beta_{{\sss \Xi}^{\iota}}^\ssc[N]<1.
	\end{align}
Further, there exist $\underline{\beta}_{\sss \Psi}^\ssc[0]$, $\underline{\beta}_{\sss \sum \Pi}^\ssc[1]$  such that
	\begin{align}
	\hat \Psi^{\ssc[0],\iota}_p(0)\geq&\  \underline{\beta}_{\sss \Psi}^\ssc[0],
	&\sum_\kappa \hat \Pi^{\ssc[1],\iota,\kappa}_p(0)\geq\  \underline{\beta}_{\sss
	\sum\Pi}^\ssc[1].
	\end{align}
Additionally, there exist $\beta_{{\sss\Xi_\alpha(0)}}^\ssc[1-0],\beta_{{\sss\Xi_\alpha(0)}}^\ssc[0-1],\beta_{{\sss\Xi_\alpha(\ve[1])}}^\ssc[1-0],\beta_{{\sss\Xi_\alpha(\ve[1])}}^\ssc[0-1]$ with
	\begin{align}
	\lbeq{boundOnXizero}
	-\beta_{{\sss\Xi_\alpha(0)}}^\ssc[1-0]\leq
	&~\Xi^\ssc[0]_{\alpha,p}(0)-\Xi^\ssc[1]_{\alpha,p}(0)\leq \beta_{{\sss\Xi_\alpha(0)}}^\ssc[0-1],\\
	-\beta_{{\sss\Xi_\alpha(\ve[1])}}^\ssc[1-0]
	\leq&~\Xi^\ssc[0]_{\alpha,p}(\ve[1])-\Xi^\ssc[1]_{\alpha,p}(\ve[1])
	\leq \beta_{{\sss\Xi_\alpha(\ve[1])}}^\ssc[0-1] ,
	\end{align}
and $\beta_{{\sss\Xi^{\iota}_\alpha,I}}^\ssc[0],\beta_{{\sss \sum}{\sss\Xi^{\iota}_\alpha,I}}^\ssc[0],\beta_{{\sss\Xi^{\iota}_\alpha,II}}^\ssc[0], \beta_{ {\sss\sum \Xi^{\iota}_\alpha,II}}^\ssc[0],
\geq 0$ such that
	\begin{align}
	\Xi^{\ssc[0],\iota}_{\alpha,{\sss I},p}(\ve[\iota])
	\leq& \beta_{{\sss\Xi^{\iota}_\alpha,I}}^\ssc[0],\qquad\qquad
	\sum_\kappa\Xi^{\ssc[0],\iota}_{\alpha,{\sss I},p}(\ve[\iota]+\ve[\kappa])
	\leq \beta_{{\sss\sum\Xi^{\iota}_\alpha,I}}^\ssc[0],\\
	\Xi^{\ssc[0],\iota}_{\alpha,{\sss II},p}(0)\leq& \beta_{{\sss\Xi^{\iota}_\alpha,II}}^\ssc[0],
	\qquad\qquad
	\sum_\kappa\Xi^{\ssc[0],\iota}_{\alpha,{\sss II},p}(\ve[\kappa])\leq
	\beta_{{\sss\sum \Xi^{\iota}_\alpha,II}}^\ssc[0].
	\end{align}
Also, there exist
$\beta_{{\sss\sum \Psi^{\iota}_\alpha,I}}^\ssc[0-1],\beta_{{\sss\sum \Psi^{\iota}_\alpha,II}}^\ssc[0-1],$
$\beta_{{\sss\sum \Psi^{\iota}_\alpha,I}}^\ssc[1-0],$
$\beta_{{\sss\sum \Psi^{\iota}_\alpha,II}}^\ssc[1-0],$
$\underline{\beta}_{ \sss \sum \Pi_\alpha}^\ssc[0],$
$\beta_{\sss\sum \Pi_\alpha}^\ssc[0]$, such that
	\begin{align}
	-\beta_{{\sss\sum \Psi^{\iota}_\alpha,I}}^\ssc[1-0]\leq&
	\sum_\kappa \left(\Psi^{\ssc[0],\iota}_{\alpha,{\sss I},p}(\ve[\iota]+\ve[\kappa])
	-\Psi^{\ssc[1],\iota}_{\alpha,{\sss I},p}(\ve[\iota]+\ve[\kappa])  \right)
	\leq \beta_{{\sss\sum \Psi^{\iota}_\alpha,I}}^\ssc[0-1],\\
	-\beta_{{\sss\sum \Psi^{\iota}_\alpha,II}}^\ssc[1-0]
	\leq&
	\sum_\kappa \left(\Psi^{\ssc[0],\iota}_{\alpha,{\sss II},p}(\ve[\kappa])  -\Psi^{\ssc[1],\iota}_{\alpha,{\sss II},p}(\ve[\kappa])\right)
	\leq \beta_{{\sss\sum \Psi^{\iota}_\alpha,I}}^\ssc[0-1],\\
	\underline{\beta}_{ \sss \sum \Pi_{\alpha}}^\ssc[0]
	\leq &\sum_{\kappa}\Pi^{\ssc[0],\iota,\kappa}_{\alpha,p}(\ve[\iota])\leq \beta_{\sss \sum \Pi_\alpha}^\ssc[0].
	\end{align}
For $N=0,1$, there exist
$\beta_{\sss\Xi,R}^\ssc[N]$,
$\beta_{\Delta \sss\Xi,R}^\ssc[N]$,
$\beta_{\sss\Psi,R,I}^\ssc[N]$,
$\beta_{\Delta \sss\Psi,R,I}^\ssc[N]$,
$\beta_{\sss\Psi,R,II}^\ssc[N]$,
$\beta_{\Delta \sss\Psi,R,II}^\ssc[N]\geq 0$,
such that
	\begin{align}
	\sum_{x}\Xi_{{\sss R},p}^\ssc[N](x) \leq &\beta_{\sss\Xi,R}^\ssc[N],  \qquad
	\sum_{x}\|x\|_2^2\Xi_{{\sss R},p}^\ssc[N](x) \leq \beta_{\Delta\sss\Xi,R}^\ssc[N],  \\
	\sum_{x}\Psi_{{\sss R,I},p}^{\ssc[N],\iota} (x) \leq& \beta_{\sss\Psi,R,I}^\ssc[N],  \qquad
	\sum_{x}\|x-\ve[\iota]\|_2^2\Psi_{{\sss R,I},p}^{\ssc[N],\iota} (x)\leq \beta_{\Delta\sss\Psi,R,I}^\ssc[N],
	\\
	\sum_{x}\Psi_{{\sss R,II},p}^{\ssc[N],\iota} (x)\leq& \beta_{\sss\Psi,R,II}^\ssc[N],  \qquad
	\sum_{x}\|x\|_2^2\Psi_{{\sss R,II},p}^{\ssc[N],\iota} (x)\leq \beta_{\Delta\sss\Psi,R,II}^\ssc[N].
	\end{align}
Further, there exist
$\beta_{\sss\Xi^\iota,R,I}^\ssc[0]$,
$\beta_{\Delta \sss\Xi^\iota,R,I}^\ssc[0]$,
$\beta_{\sss\Xi^\iota,R,II}^\ssc[0]$,
$\beta_{\Delta \sss\Xi^\iota,R,II}^\ssc[0]$,
$\beta_{\sss\Pi,R}^\ssc[0]$,
$\beta_{\Delta \sss\Pi,R}^\ssc[0]\geq 0$,
such that
	\begin{align}
	\sum_{x}\Xi_{{\sss R,I},p}^{\ssc[0],1} (x)\leq& \beta_{\sss\Xi^\iota,R,I}^\ssc[0],  \qquad
	\sum_{x}\|x-\ve[\iota]\|_2^2\Xi_{{\sss R,I},p}^{\ssc[0],\iota} (x+\ve[\iota])\leq \beta_{\Delta\sss\Xi^\iota,R,I}^\ssc[0],\\
	\sum_{x}\Xi_{{\sss R,II},p}^{\ssc[0],1} (x)\leq& \beta_{\sss\Xi^\iota,R,II}^\ssc[0],  \qquad
	\sum_{x}\|x\|_2^2\Xi_{{\sss R,II},p}^{\ssc[0],\iota} (x)\leq \beta_{\Delta\sss\Xi^\iota,R,II}^\ssc[0],  \\
	\sum_{x,\iota}\Pi^{\ssc[0],\iota,\kappa}_{{\sss R},z} (x)\leq& \beta_{\sss\Pi,R}^\ssc[0],  \qquad
	\sum_{x,\iota,\kappa}\|x\|_2^2\Pi^{\ssc[0],\iota,\kappa}_{{\sss R},p} (x+\ve[\iota]+\ve[\kappa])\leq
	\beta_{\Delta\sss\Pi,R}^\ssc[0].
	\lbeq{analys-assumed-displacement-xiikD}
	\end{align}
For all $\bullet \in\{ \Xi, \Xi^{\iota}, \Delta\Xi, \{\Delta \Xi^\iota,0\},\{\Delta \Xi^\iota,1\}\}$ and $N\in \Nbold$, $\beta_{\bullet}^{\ssc[N]}$ depends only on $\Gamma_1,\Gamma_2,\Gamma_3,d$ and on the model. The bounds stated above also holds for $p_I=(2d-1)^{-1}$ with the constants $\beta_{\bullet}$ only depending on the dimension $d$ and the model.
\end{ass}

In Table \ref{TableLinktoBounds}, we give the line numbers in which a given bound $\beta_{\bullet}$ is stated.
Some of the assumed bounds were not discussed yet. We derive these missing bounds now.
For percolation it is well known that $\hat \tau_p(k)\geq 0$, see \cite{AizNew84}. For the bounds stated in \refeq{analys-assumed-rho-Bound-two} we recall $\aabp=p,\ \aap=p\prob(e_\kappa\nin \tilde{\Ccal}^{(0,e_\kappa)}(0))$ and $p_I=(2d-1)^{-1}$. Thus,
\begin{align}
\frac {\aabp}{\aap}&=\frac 1 {1-\prob^{\{(0,e_\kappa)\}}(0\conn \ve[\kappa] )}=\frac 1 {1-\tau_{3,p}(\ve[1])}
\leq \frac 1 {1-\beta(\tau_{3,p}(\ve[1]))}
:=\betaaa,\\
\aap &\geq p_I \prob(e_\kappa\nin \tilde{\Ccal}^{(0,e_\kappa)}(0))=\frac {1-\tau_{3,p}(\ve[1])}{2d-1}
\geq \frac {1-\beta(\tau_{3,p}(\ve[1]))  }{2d-1}:=\betaaalow,
\end{align}
where $\beta(\tau_{3,p}(\ve[1]))$ is a numerical upper bound on $\tau_{3,p}(\ve[1])$.\\
The condition in \refeq{analys-assumed-invertablecondition} is numerical condition that is verified explicitly in the Mathematica notebooks. We remark that this condition is a relatively weak, in the sense that the bootstrap analysis, which in particular includes an  improvement of bounds, fails before \refeq{analys-assumed-invertablecondition}.

\begin{table}
\centering
\begin{tabular}{c|c || c|c || c|c }
  Bound             & defined in & Bound &  defined in & Bound &  defined in \\
  \hline
$\beta_{\sss \Xi}^\ssc[0]$ & \refeq{lemmapercboundXi0-bound-1} &
$\beta_{\sss \Xi}^\ssc[1]$ &  \refeq{lemmapercboundXi1-1} &
$\beta_{\sss \Xi}^\ssc[N],N\geq 2$ & \refeq{BoundXiPercNTwo}\\
\hline
$\beta_{\sss \Xi^\iota}^\ssc[0]$ & \refeq{lemmapercboundXiiota0-bound-1}&
$\beta_{\sss \Xi^\iota}^\ssc[1]$ & \refeq{BoundXiIotaOne-1}&
$\beta_{\sss \Xi^\iota}^\ssc[N],N\geq 2$ & \refeq{BoundXiPercNTwoIota}\\
\hline
$\beta_{{\sss \Delta \Xi}}^\ssc[0]$ & \refeq{lemmapercboundXi0-bound-1}&
$\beta_{{\sss \Delta \Xi}}^\ssc[1]$ & \refeq{lemmapercboundXi1-2}&
$\beta_{{\sss \Delta \Xi}}^\ssc[N],N\geq 2$ &  \refeq{BoundXiDeltaTreeNTwoEven}, \refeq{BoundXiDeltaTreeNTwoOdd}\\
\hline
$\beta_{{\sss \Delta \Xi^{\iota}},0}^\ssc[0]$ & \refeq{lemmapercboundXiiota0-bound-3}&
$\beta_{{\sss \Delta \Xi^{\iota}},0}^\ssc[1]$ & \refeq{BoundXiIotaOne-3}&
$\beta_{{\sss \Delta \Xi^{\iota}},0}^\ssc[N],N\geq 2$ &  \refeq{BoundXiIotaDeltapercNTwoEven}, \refeq{BoundXiIotaDeltapercNTwoOdd}\\
\hline
$\beta_{{\sss \Delta \Xi^{\iota}},\iota}^\ssc[0]$ & \refeq{lemmapercboundXiiota0-bound-2}&
$\beta_{{\sss \Delta \Xi^{\iota}},\iota}^\ssc[1]$ & \refeq{BoundXiIotaOne-2}&
$\beta_{{\sss \Delta \Xi^{\iota}},\iota}^\ssc[N],N\geq 2$ &  \refeq{BoundXiIotaDeltapercNTwoEvenIota}, \refeq{BoundXiIotaDeltapercNTwoOddIota}\\
\hline
$\underline{\beta}_{\sss \Psi}^\ssc[0]$ &   \refeq{lemmaLowerbound-2}&
 $\underline{\beta}_{\sss \sum \Pi}^\ssc[1]$ & \refeq{lemmaLowerbound-3}&
$ \beta_{{\sss\Xi_\alpha(0)}}^\ssc[0-1]$ &\refeq{Differencebound-1}\\
\hline
 $ \beta_{{\sss\Xi_\alpha(0)}}^\ssc[1-0]$ &\refeq{Differencebound-1}&
 $\beta_{{\sss\Xi_\alpha(\ve[1])}}^\ssc[1-0]$&\refeq{Differencebound-2}&
$\beta_{{\sss\Xi_\alpha(\ve[1])}}^\ssc[0-1]$ &\refeq{Differencebound-3}\\
\hline
$\beta_{{\sss\Xi^{\iota}_\alpha,I}}^\ssc[0]$ & 	\refeq{lemmapercboundXiiota0-bound-4} &
$\beta_{{\sss \sum}{\sss\Xi^{\iota}_\alpha,I}}^\ssc[0]$& \refeq{lemmapercboundXiiota0-bound-4} &
$\beta_{{\sss\Xi^{\iota}_\alpha,II}}^\ssc[0]$ & 	\refeq{lemmapercboundXiiota0-bound-4} \\
\hline
$\beta_{ {\sss\sum \Xi^{\iota}_\alpha,II}}^\ssc[0]$ & 	\refeq{lemmapercboundXiiota0-bound-5}  &&&\\
\hline
$\beta_{{\sss\sum \Psi^{\iota}_\alpha,I}}^\ssc[0-1]$  & \refeq{Differencebound-5} &
$\beta_{{\sss\sum \Psi^{\iota}_\alpha,II}}^\ssc[0-1]$ & \refeq{Differencebound-7} &
$\beta_{{\sss\sum \Psi^{\iota}_\alpha,I}}^\ssc[1-0]$  & \refeq{Differencebound-4} \\
\hline
$\beta_{{\sss\sum \Psi^{\iota}_\alpha,II}}^\ssc[1-0]$  & \refeq{Differencebound-6} &
$\underline{\beta}_{ \sss \sum \Pi_\alpha}^\ssc[0]$ & \refeq{lemmaLowerbound-1} &
$\beta_{\sss\sum \Pi_\alpha}^\ssc[0]$& \refeq{lemmapercboundXiiota0-bound-11} \\
\hline
$\beta_{\sss\Xi,R}^\ssc[0]$ & \refeq{lemmapercboundXi0-bound-4} &
$\beta_{\Delta \sss\Xi,R}^\ssc[0]$ &  \refeq{lemmapercboundXi0-bound-4}&
$\beta_{\sss\Psi,R,I}^\ssc[0]$ & \refeq{lemmapercboundXi0-bound-5}\\
\hline
$\beta_{\Delta \sss\Psi,R,I}^\ssc[0]$ & \refeq{lemmapercboundXi0-bound-6}&
$\beta_{\sss\Psi,R,II}^\ssc[0]$ &  \refeq{lemmapercboundXi0-bound-7}&
$\beta_{\Delta \sss\Psi,R,II}^\ssc[0]$ & \refeq{lemmapercboundXi0-bound-8}\\
\hline
$\beta_{\sss\Xi,R}^\ssc[1]$ &     \refeq{lemmapercboundXi1-3} &
$\beta_{\Delta \sss\Xi,R}^\ssc[1]$ &  \refeq{lemmapercboundXi1-4}&
$\beta_{\sss\Psi,R,I}^\ssc[1]$ & \refeq{lemmapercboundXi1-5} \\
\hline
$\beta_{\Delta \sss\Psi,R,I}^\ssc[1]$ & \refeq{lemmapercboundXi1-7} &
$\beta_{\sss\Psi,R,II}^\ssc[1]$ & \refeq{lemmapercboundXi1-6} &
$\beta_{\Delta \sss\Psi,R,II}^\ssc[1]$ & \refeq{lemmapercboundXi1-8} \\
\hline
$\beta_{\sss\Xi^\iota,R,I}^\ssc[0]$ & \refeq{lemmapercboundXiiota0-bound-7}&
$\beta_{\Delta \sss\Xi^\iota,R,I}^\ssc[0]$ & \refeq{lemmapercboundXiiota0-bound-8}&
$\beta_{\sss\Xi^\iota,R,II}^\ssc[0]$ & \refeq{lemmapercboundXiiota0-bound-9}\\
\hline
$\beta_{\Delta \sss\Xi^\iota,R,II}^\ssc[0]$ & \refeq{lemmapercboundXiiota0-bound-10}&
$\beta_{\sss\Pi,R}^\ssc[0]$ &\refeq{lemmapercboundXiiota0-bound-13}&
$\beta_{\Delta \sss\Pi,R}^\ssc[0]$ & \refeq{lemmapercboundXiiota0-bound-14}
\end{tabular}
\caption{An overview where to find the bounds stated in  \cite[Assumption 4.3]{FitHof13b}
in terms of diagrams.}
\label{TableLinktoBounds}
\end{table}
.his completes the summary of the bounds that we have proved. Using that $p=p_I$, see \cite[Assumptions 2.2]{FitHof13b}, or that the bootstrap function are bounded, we can compute numerical bounds on these diagrammatic bounds, see \refeq{Simple-Bound-triangle} for the idea of these bounds, or \cite[Section 5]{FitHof13b} for a complete description.

\section{Proof of the bounds}
\label{secProofBounds}
The bounds stated in the previous section are proved using ideas that are quite standard in lace expansion analyses, in combination with a consideration of cases for the number of edges involved in shared lines. This consideration is needed to use the additional avoidance constraints.
The proof of the bound for the classical lace expansion is already elaborate, adding the consideration of cases makes the proof even more lengthy.
In the proof of Lemma \ref{lemmapercboundXi1}, we discuss in detail how we use these different cases for our bounds. We will omit details in the explanation of the proof for $N\geq 2$.
For a detailed description of such bounds we refer the reader to \cite{Slad06} or \cite{Fit13}.
\subsection{Proof of the bounds for $N= 1$}
\label{secProofBoundsNOne}
\begin{proof}[Proof of Lemma \ref{lemmapercboundXi1}.]
We first prove the bounds on $\Xi_p^\ssc[1]$ and  $\Xi^\ssc[1]_{{\sss R},p}$ and then explain how to modify the arguments used to obtain the bounds on $\Psi_{{\sss R,I},p}^{\ssc[1],\kappa} $ and $\Psi_{{\sss R,II},p}^{\ssc[1],\kappa} $.\\
For $N=1$ we simplify the definition \refeq{tau-LEC-ident} to see that
\begin{align}
  	\lbeq{def-XiOne-Simple}
   	\Xi^{\ssc[1]}_p(x)
	&= \sum_{b_0}p\expec_{\sss 0} \left(\indic{0\dbc \bb_{\sss 0}}
	\indic{\tb_0\nin \tilde{\Ccal}_{\sss 0}}
    	\prob_{\sss 1}^{\bb_0} \big(E(\tb_{\sss 0},x;\tilde{\Ccal}_{\sss 0})\big)\right).%\\
%      	\lbeq{def-PsiOne-Simple}
 %   	\Psi^{\ssc[1],\kappa}_p(x) &=\frac {p} {\aap} \sum_{b_0}\expec_{\sss 0}
%	\left(\indic{0\dbc \bb_0}\indic{\tb_0\nin \tilde{\Ccal}_0}
 %   	\expec_{\sss 1}^{\bb_0} \big(\indic{E(\tb_0,x;\tilde{\Ccal}_0)}
%	\indic{x+\ve[\kappa]\nin \Ccal^{(x,x+\ve[\kappa])}_1(x)}\big)
 %   	\right).
	\end{align}
In Section \ref{secBoundsPercEvents} we have proven with \refeq{XiFs} that
 	\eqan{
 	\lbeq{def-XiOne-Simple-Bound-Event}
    	\Xi_p^{\ssc[1]}(x)    & \leq \sum_{b_{\sss 0},t,w,z}
    	p\prob_p (F_{\sss 0}(b_{\sss 0},w,z)\cap\{\tb_{\sss 0}\nin \tilde{\Ccal}_0\})
	\prob_p^{\bb_{\sss 0}}(F_{\sss 1}(b_{\sss 0},t,z,x)).
	}
This can be displayed as in Figure \ref{fig-Form-Xi1}.

\begin{figure}[!htb]
\begin{center}
\picXiNOne[1]
\caption{Diagrammatic representations of the bound on $\Xi_p^{\ssc[1]}(x)$ in  \refeq{def-XiOne-Simple-Bound-Event}. The solid lines are connections in $\tilde{\Ccal}_{\sss 0}=\tilde \Ccal^{b_{\sss 0}}_{\sss 0}(0)$, while the dashed lines represent connections in $ \Ccal_1\subset \Zd\setminus\{\bb\}$.  Shaded triangles might be trivial. As explained in Section \ref{secBoundsPercEvents} we choose $z$ such that the connection $\{w\conn z\}$ intersects with $\tilde\Ccal_{\sss 1}$ only at $z$, so that all connections are bond-disjoint.}
\label{fig-Form-Xi1}
\end{center}
\end{figure}
\noindent
Using Figure \ref{fig-Form-Xi1}, it is straightforward to obtain that
 	\eqan{
    	\Xi_p^{\ssc[1]}(x)    & \leq \sum_{b_{\sss 0},t,w,z} \diagRepulsiveLetter{T}^*_{0,0,0}(\bb_{\sss 0},w,0)
	2d p \tau_p(t-\tb_{\sss 0})\tau_p(w-z)\diagRepulsiveLetter{T}^*_{0,0,0}(t-x,z-x,0).
	}
We use the repulsiveness properties to obtain a better bound. Namely, we prove that
	\begin{align}
   	\lbeq{lemmapercboundXi1-1-step0}
	\Xi_p^{\ssc[1]}(x) \leq \sum_{a,b=0}^2\sum_{u,\iota,w,z,t} P^{{\sss \rm S},a}(u,w)\bar A^{\iota,a,b}(u,w,t,z)P^{{\sss \rm E},b}(t-x,z-x),
	\end{align}
where, to avoid confusion between the bond $b_{\sss 0}$ and the number of edges in $A^{\iota,a,b}$, we replace the bond $b_{\sss 0}$ by $(u,u+\ve[\iota])$.
Once this is established, the bound \refeq{lemmapercboundXi1-1} follows as
\begin{align}
\sum_x \Xi_p^{\ssc[1]}(x) \leq& \sum_{a,b=0}^2\sum_{x,u,\iota,w,z,t} P^{{\sss \rm S},a}(u,w)\bar A^{\iota,a,b}(u,w,t,z)P^{{\sss \rm E},b}(t-x,z-x)\nnb
=& \sum_{a,b=0}^2\sum_{u,w} P^{{\sss \rm S},a}(u,u+w)\sum_{y,x} P^{{\sss \rm E},b}(x,x+y) \sum_{\iota,t}\bar A^{\iota,a,b}(0,w,t,t+y)\nnb
\leq& \sum_{a,b=0}^2\sum_{u,w} P^{{\sss \rm S},a}(u,u+w)\sum_{y,x} P^{{\sss \rm E},b}(x,x+y) \sup_{w,y}\sum_{\iota,t}\bar A^{\iota,a,b}(0,w,t,t+y)\nnb
   \lbeq{lemmapercboundXi1-1-summation}
=& \sum_{a,b=0}^2 (\vec P^{{\sss \rm S}})_a({\bf \bar A}^\iota)_{a,b} (\vec P^{{\sss \rm E}})_b=\vec P^{{\sss \rm S}} {\bf \bar A}^\iota \vec P^{{\sss \rm E}}.
\end{align}

Let us now prove \refeq{lemmapercboundXi1-1-step0}.
We denote by $d_{\Ccal}(x,y)$ the intrinsic distance between $x$ and $y$ in $\Ccal$, so the length of the shortest path of bonds that are occupied in $\Ccal$ and connect $x$ and $y$. We define $a=d_{\tilde \Ccal^{b_{\sss 0}}_0(0)}(\bb_{\sss 0},w)$ and $b=d_{\tilde \Ccal_1(x)}(t,z)$. We first discuss the left most triangle $0,\bb_{\sss 0},w$ and show that it is bounded by $P^{{\sss \rm S},a}(u,w)$, see \refeq{Pa-exampledefinition}. We split between several cases depending on the value of $a$:\\
{\bf Case $a=0$.}  In this case $w=u$. If $0=w=u$, then the left triangle shrinks to a point, otherwise $0$ and $w$ are doubly connected:
	\begin{align}
	\delta_{u,w}\prob(0\dbc w)=P^{{\sss \rm S},0}(u,w).
	\end{align}
{\bf Case $a=1$.} We conclude from $a=1$ that $u$ and $w$ are neighbors, $2dD(u-w)=1$, the bond $\{u,w\}$ is occupied and $u\neq 0$. We split between $w=0$ and $w\neq 0$ to obtain the desired bound:
	\begin{align}
	\delta_{w,0} \diagRepulsiveLetter{B}_{\underline 1,3}(u,0) + \diagRepulsiveLetter{T}_{1,\underline 1,1}(u,w,0)=P^{{\sss \rm S},1}(u,w).
	\end{align}
{\bf Case $a\geq 2$.} We consider the cases $w=0$ and $w\neq 0$ to obtain the bound
	\begin{align}
	\delta_{w,0} \prob(\{0\connLe{2} u\}\circ \{0\connLe{2} u\} )+\diagRepulsiveLetter{T}_{1,2,1}(u,w,0)\leq P^{{\sss \rm S},2}(u,w).
	\end{align}
Further, the right triangle $z,t,x$ is bounded by $P^{{\sss \rm E},b}(t-x,z-x)$ for the three different cases of $b$.
The difference to the left triangle is that when $x=z$, we have the freedom to choose $t=x$, so that we can exclude the case $x=z$ for $b\geq 1$.\\

Let us now discuss the middle piece of Figure \ref{fig-Form-Xi1} consisting of the square $(b_{\sss 0}=(u,u+\ve[\iota]),t,z,w)$. By definition of $F_{\sss 0}$ and $F_{\sss N}$ in \refeq{Fdefa}-\refeq{FNdefa},
we know that $z\nin b=(u,u+\ve[\iota])$, $u\neq x,t$.
Further, we note that the connections $\{u+\ve[\iota]\conn t\}$ and $\{w\conn z\}$ are realised on different percolation configurations.
For this reason we have introduced the concept of generalized-disjoint occurrence, see Definition \ref{def-gen-disjoint}.
We have to consider the nine combinations of $(a,b)$:\\
{\bf Case $a=0,b=0$.} We begin with the simplest case. We use $u+\ve[\iota]\neq z=t$ and $u=w=\bb\neq z$ to conclude
	\begin{align}
  	\lbeq{Bound-Xi-case-abZero}
	\diagRepulsiveLetter{T}_{\underline 1,1,1}(\ve[\iota],t-u,0)=\bar A^{\iota,0,0}(u,u,t,t).
	\end{align}
{\bf Case $a=1,b=0$.} We note that $t=z\nin\{ u,u+\ve[\iota]\}$ and $2dD(u-w)=1$, which implies that
	\begin{align}
  	2d D(u-w) \diagRepulsiveLetter{T}_{\underline 1,1,0}(\ve[\iota],t-u,w-u)\leq \bar A^{\iota,1,0}(u,w,t,t).
	\end{align}
{\bf Case $a\geq 2,b=0$.} We note that $t=z\nin\{ u,u+\ve[\iota]\}$ and obtain the bound
	\begin{align}
	\diagRepulsiveLetter{T}_{\underline 1,1,0}(\ve[\iota],t-u,w-u) \leq \bar A^{\iota,2,0}(u,w,t,t).
	\end{align}
{\bf Case $a=0,b=1$.} We note that $u=w\neq t,z$ and $2dD(z-t)=1$ and conclude
	\begin{align}
  	2d D(t-z)  \diagRepulsiveLetter{T}_{1,\underline 1,0}(u-z,u+\ve[\iota]-z,t-z)\leq \bar A^{\iota,0,1}(u,u,z,t).
	\end{align}
{\bf Case $a=0,b\geq 2$.} We note that $u=w\neq t,z$ and obtain
	\begin{align}
 	\diagRepulsiveLetter{T}_{1,\underline 1,0}(u-z,u+\ve[\iota]-z,t-z)
	\leq \bar A^{\iota,0,2}(u,u,z,t).
	\end{align}
{\bf Cases $a\geq 1$ and $b\geq 1$.} For $a\geq 1$ and $b\geq 1$, the two paths realising the connections $\{(u,u+\ve[\iota]) \text{ occ.}, u+\ve[\iota]\conn t\}$ and $\{w\conn z\}$ have no common vertices, leading to a repulsive diagram. Thus, we obtain
	\begin{align}
	\diagRepulsiveLetter{B}_{\underline 1,0}(\ve[\iota]-u,t-u)\tau_p(z-w)\leq \bar A^{\iota,a,b}(u,w,z,t).
	\end{align}
When $a=1$ and/or $b=1$, we include the information that either $u,w$ and/or $z,t$ are neighbors into the definition of $\bar A^{\iota,a,b}$. Using the information and the parity of the lattice allows us to obtain improved numerical bounds on $\bar A^{\iota,a,b}$.
This completes the proof of the bound \refeq{lemmapercboundXi1-1}.\\

To prove the bound \refeq{lemmapercboundXi1-3}, we review what contributions of $\Xi_p^{\ssc[1]}(x)$ have been extracted using
$\Xi^{\ssc[1]}_{\alpha,p}(x)$, see \refeq{Def-XiOne-Split}.
Indeed, we extract the contributions in which $\bb_0=0$, $\{\tb_0\conn x\}$ is cut through at $x$ and the connection to the cutting point is established in $\tilde\Ccal^{(0,e)}_0$ directly, so via the bond $(0,x)$. This corresponds to a contribution of $a=b=0$ in which $t$ and $u$ are directly connected. We split the bound in \refeq{Bound-Xi-case-abZero} into
	\begin{align}
	\lbeq{Bound-Xi-case-abZero-split}
	\diagRepulsiveLetter{T}_{\underline 1,1,\underline 1}(\ve[\iota],t-u,0)
	+\diagRepulsiveLetter{T}_{\underline 1,1,2}(\ve[\iota],t-u,0),
	\end{align}
and see that the first term corresponds to the event that we removed with $\Xi^{\ssc[1]}_{\alpha,p}(x)$. In \refeq{lemmapercboundXi1-3} we simply remove the bound in \refeq{Bound-Xi-case-abZero}  and replace it with a bound on the second term in \refeq{Bound-Xi-case-abZero-split}. \\[2mm]

Next, we explain how to obtain the bound on the weighted diagram \refeq{lemmapercboundXi1-2}.
First, we define an open bubble that will replace the left and right triangle:
\begin{align}
\lbeq{definition-Q1}
Q^{{\sss \rm S},0}(x,y)&= Q^{{\sss \rm E},0}(x,y)=P^{{\sss \rm S},0}(x,y),\\
Q^{{\sss \rm S},1}(x,y)&= 2d D(x-y) \left(\delta_{0,y}\tau_{3,p}(x)+\diagRepulsiveLetter{B}_{1,1}(-y,x-y)\right),\\
Q^{{\sss \rm S},2}(x,y)&= \delta_{0,y}\tau_{2,p}(x)+\diagRepulsiveLetter{B}_{1,1}(-y,x-y),
\lbeq{definition-Q3}
\end{align}
and $Q^{{\sss \rm E},b}(x,y)=(1-\delta_{0,y})Q^{{\sss \rm S},b}(x,y)$ for $b=1,2$.\\
Then, we show that, next to the bound in \refeq{lemmapercboundXi1-1-step0}, also the bounds
\begin{align}
  \lbeq{lemmapercboundXi1-1-step0.1}
\Xi_p^{\ssc[1]}(x) \leq& \sum_{a,b=0}^2\sum_{u,\iota,w,z,t}
P^{{\sss \rm S},a}(u,w) A^{\iota,a,b}(u,w,t,z)Q^{{\sss \rm E},b}(t-x,z-x),\\
  \lbeq{lemmapercboundXi1-1-step0.2}
\Xi_p^{\ssc[1]}(x) \leq& \sum_{a,b=0}^2\sum_{u,\iota,w,z,t}
Q^{{\sss \rm S},a}(u,w) A^{\iota,b,a}(t,z,u,w)P^{{\sss \rm E},b}(t-x,z-x),
\end{align}
hold. As the proof of these bounds is very similar to the proof of \refeq{lemmapercboundXi1-1-step0}, we omit it here.
For our bound we split the weight $\|x\|_2^2$ using the inequality:
\begin{align}
\|x\|_2^2\leq 3(\|w\|_2^2+\|z-w\|_2^2+\|x-z\|_2^2).
\end{align}
More precisely, we first use this inequality for each given configuration and then apply the bounds
   \refeq{lemmapercboundXi1-1-step0},
   \refeq{lemmapercboundXi1-1-step0.1},
   \refeq{lemmapercboundXi1-1-step0.2}
  to obtain
\begin{align}
\|x\|_2^2\Xi_p^{\ssc[1]}(x) \leq&3\sum_{a,b=0}^2\sum_{u,\iota,w,z,t} Q^{{\sss \rm S},a}(u,w) A^{\iota,b,a}(t,z,u,w)P^{{\sss \rm E},b}(t-x,z-x)\|w\|_2^2\nnb
&+3\sum_{a,b=0}^2\sum_{u,\iota,w,z,t} P^{{\sss \rm S},a}(u,w)\bar A^{\iota,a,b}(u,w,t,z)P^{{\sss \rm E},b}(t-x,z-x)\|w-z\|_2^2\nnb
&+3 \sum_{a,b=0}^2\sum_{u,\iota,w,z,t} P^{{\sss \rm S},a}(u,w) A^{\iota,a,b}(u,w,t,z) Q^{{\sss \rm E},b}(t-x,z-x)\|x-z\|_2^2.
  \lbeq{lemmapercboundXi1-1-step0.3}
\end{align}
We have defined the diagrams in $\vec h^{\sss \rm S},\vec h^{\sss \rm E}$ and ${\bf H}^{\sss (3)}$ as the bound on the weighted version of $Q^{\sss \rm S}$, $Q^{\sss \rm E}$ and $\bar A^\iota$, so that
\refeq{lemmapercboundXi1-1-step0.3} implies
\begin{align}
\sum_x\|x\|_2^2 \Xi_p^{\ssc[1]}(x) \leq& 3\vec h^{{\sss \rm S}}{\bf A}^\iota \vec P^{\sss \rm E}+3\vec P^{{\sss \rm S}} {\bf H}^{\sss (3)}\vec P^{\sss \rm E}+3\vec P^{{\sss \rm S}}{\bf A}^\iota \vec h^{\sss \rm E}.
\end{align}
We obtain the bound \refeq{lemmapercboundXi1-2} by extracting the special case that one or both of the triangles on the left and right are trivial, characterized by $\vec u^T=(1,0,0)$. In this case we simply use the weight $\|x\|_2^2$ or apply the inequalities
\begin{align}
\|x\|_2^2\leq 2(\|w\|_2^2+\|x-w\|_2^2)\text{ if $z=x$, \ \  or }\quad \|x\|_2^2\leq 2(\|z\|_2^2+\|x-z\|_2^2) \text{ if }w=0.
\end{align}
In this way we obtain the bound
\begin{align}
\sum_x\|x\|_2^2 \Xi_p^{\ssc[1]}(x) \leq&
\vec u^T  {\bf H}^{\sss (3)}\vec u
+2\vec u^T  {\bf H}^{\sss (3)}\big(\vec P^{\sss \rm E}-\vec u\big)+2\vec u^T {\bf A}^\iota\vec h^{\sss \rm E}
+2\vec h^{{\sss \rm S}}{\bf A}^\iota \vec u+2\big(\vec P^{{\sss \rm S}}-\vec u^T\big)  {\bf H}^{\sss (3)}\vec u\nnb
&+3\vec h^{{\sss \rm S}}{\bf A}^\iota \big(\vec P^{\sss \rm E}-\vec u\big)
+3\big(\vec P^{{\sss \rm S}}-\vec u^T\big)  {\bf H}^{\sss (3)}\big(\vec P^{\sss \rm E}-\vec u\big)
+3\big(\vec P^{{\sss \rm S}}-\vec u^T\big) {\bf A}^\iota \vec h^{\sss \rm E}.
\lbeq{XiIota-Deltabounds-versions1}
\end{align}
As this is a central quantity, we improve this bounds once more, by improving the bound for diagrams which involve only two triangles. These are precisely the terms carrying the factor $2$ in \refeq{XiIota-Deltabounds-versions1}.
\iflongversion
See Appendix \ref{Appendix-ToTheProofs-DoubleTriangle} for the details.
\else
As this is a simple analysis of special cases we omit it here. The details are given in the appendix of the extended version \cite{FitHof13d-ext}.
\fi
This creates the  term $\beta_{{\sss \Delta \Xi}}^\ssc[1]$ and proves \refeq{lemmapercboundXi1-2}. We obtain \refeq{lemmapercboundXi1-4} by reviewing the contribution that we remove in $\Xi^{\ssc[1]}_{\alpha,p}(x)$ and subtract the contribution it creates from the bound, see also \refeq{Bound-Xi-case-abZero-split}.\\[2mm]

Now we prove the bounds on $\sum_{\kappa}\Psi_{{\sss R,II},p}^{\ssc[1],\kappa}(x)$.
The coefficients $\Xi^{\ssc[1]}_{{\sss R},p}$ and $\Psi^{\ssc[1],\kappa}_{{\sss R,II},p}$ only differ by the factor $p/\aap$ and the constraint that $x+\ve[\kappa]\nin\tilde\Ccal_1^{(x,x+\ve[\kappa])}$. The constraint is created by the next pivotal bond $b_1=(x,x+\ve[\kappa])$ in the expansion, see Section \ref{secExp}. For each realisation at most $2d-1$ values of $\kappa$ can contribute, so that
\begin{align}
  \lbeq{lemmapercboundXi1-6-step1}
\sum_{\kappa}\Psi_{{\sss R,II},p}^{\ssc[1],\kappa} (x) \leq (2d-1)\frac{p}{\aap} \Xi^{\ssc[1]}_{{\sss R},p}(x),
\end{align}
for all $x$. Combining this with the bound on $\Xi^{\ssc[1]}_{{\sss R},p}$ in \refeq{lemmapercboundXi1-3} and \refeq{lemmapercboundXi1-4}, we obtain the stated upper bounds on $\Psi_{{\sss R,II},p}^{\ssc[1],\kappa}$ in \refeq{lemmapercboundXi1-6} and \refeq{lemmapercboundXi1-8}.\\[2mm]
The argument in \refeq{lemmapercboundXi1-6-step1} also implies that
\begin{align}
\sum_{\kappa}\Psi^{\ssc[N],\kappa}_p(x)\leq& (2d-1)\frac p {\aap}\Xi^\ssc[N]_p(x).
\end{align}
Combining this with the bound \refeq{lemmapercboundXi1-1} gives
\begin{align}
\lbeq{lemmapercboundXi1-5-step1}
\sum_{x}\Psi_{{\sss R,I},p}^{\ssc[1],\kappa} (x) \leq& \frac p {\aap}\frac {2d-1}{2d} \vec P^{{\sss \rm S}} {\bf \bar A}^\iota \vec P^{\sss \rm E}.
\end{align}
By the definition in \refeq{Def-PsiOne-SplitOne}, in $\Psi_{\alpha,p}^{\ssc[1],\kappa} $  we extract contributions in which $\bb_0=u=w=0$, $t=z=x$, $\|x-\ve[\kappa]\|\leq 1$ and $\tb$ and $x$ are connected by a short path.  In the bound on $\Xi^{\ssc[1]}$ these contribute to the case bounded in \refeq{Bound-Xi-case-abZero}.
Inspecting the proof in \refeq{Bound-Xi-case-abZero}, we see that we can bound this case for $\Psi_{{\sss R,I},p}^{\ssc[1],\kappa} $ by
\begin{align}
 \lbeq{lemmapercboundXi1-5-step2}
   \sum_{x}   \sum_{\iota}  \left( \indic{\|x-\ve[\kappa]\|_2> 1} \diagRepulsiveLetter{T}_{\underline 1,1,1}(\ve[\iota],x,0)
+   \indic{\|x-\ve[\kappa]\|_2\leq 1} \diagRepulsiveLetter{T}_{\underline 1,3,1}(\ve[\iota],x,0)\right).
\end{align}
We can remove $x=-\ve[\kappa]$ from the sum as $\Psi_{{\sss R,I},p}^{\ssc[1],\kappa}(-\ve[\kappa])=0$.
For our bound we extract all contributions in which $0$ and $x$ are connected via the direct edge, and note that the direct connection does not contribute for $x=\ve[\kappa]$. In this way we obtain the bound for this case
\begin{align}
 \lbeq{lemmapercboundXi1-5-step3}
     (2d-2)\sum_{\iota}\diagRepulsiveLetter{T}_{\underline 1,2,
     \underline 1}(\ve[\iota],\ve[2],0)
     + \sum_{x} \sum_{\iota} \diagRepulsiveLetter{T}_{\underline 1,1,2}(\ve[\iota],x,0).
\end{align}
We replace the original bound $({\bf \bar A}^\iota)_{0,0}$ in \refeq{lemmapercboundXi1-5-step1}  by this term and obtain \refeq{lemmapercboundXi1-5}.

In the bounds on weighted version of $\Psi_{{\sss R,I},p}^{\ssc[1],\kappa}$, see \refeq{lemmapercboundXi1-7}, we can unfortunately not benefit from the extracted contribution.
We explain the reason for this after proving the bound. We first use \refeq{XidominatespsiImproved} to bound, for every $x$,
\begin{align}
	\lbeq{Psi-R-x}
\Psi_{{\sss R,I},p}^{\ssc[1],\kappa}(x)\leq \Psi^{\ssc[N],\kappa}_p(x)\leq& \frac {p}{\aap} \Xi^\ssc[N]_p(x),
\end{align}
In the following, we first use symmetry to perform the sum over $\kappa$, then apply $\|x-\ve[\kappa]\|_2^2=\|x\|_2^2-2x_\kappa+1$ and \refeq{Psi-R-x}, to obtain
\begin{align}
  \lbeq{lemmapercboundXi1-7-step2}
\shift\sum_{x}\|x-\ve[\kappa]\|_2^2\Psi_{{\sss R,I},p}^{\ssc[1],\kappa} (x)\leq&
\frac 1 {2d} \frac {p}{\aap} \sum_{x,\kappa}(\|x\|_2^2-2x_\kappa+1)\Xi_{p}^{\ssc[1]} (x)=\frac {p}{\aap}\sum_{x}(\|x\|_2^2+1)\Xi_{p}^{\ssc[1]} (x).
\end{align}
In the second step we have used that $\Xi_{p}^{\ssc[1]} (x)$ is symmetric to conclude that $x_\kappa\Xi_{p}^{\ssc[1]} (x)$ vanishes when we sum over $x$. Using the already proven bounds \refeq{lemmapercboundXi1-1}, \refeq{lemmapercboundXi1-2} we obtain the bound claimed in \refeq{lemmapercboundXi1-7}. \\
For this bound it is not beneficial to extract contributions that contribute to
$\Psi_{{\sss R,I},p}^{\ssc[1],\kappa}$, but not to $\Psi_{p}^{\ssc[1],\kappa} (x)$, as this would create terms that are not symmetric. Without the symmetry in $x$, we would need to use the inequality $\|x-\ve[\kappa]\|_2^2\leq 2\|x\|_2^2+2$ to split the weight. The factor $2$ in this split is numerically worse than any gain we can possibly expect from the extraction of explicit contributions.
\end{proof}

\begin{proof}[Proof of Lemma \ref{lemmapercboundXiiota1}.]
This proof is similar to the proof of Lemma \ref{lemmapercboundXi1}.
First, we recall that, in Section \ref{secBoundsPercEvents}, we have proven that
 \eqan{
    \lbeq{lemmapercboundXiiota1-event-bound}
    \Xi_p^{\ssc[1],\iota}(x)    & \leq \sum_{b_{\sss 0},t,w,z} p
    \expec_{\sss 0} \left(\indicwo{F^\iota_{\sss 0}(b_{\sss 0},w,z)}\indic{\tb_{\sss 0}\nin \tilde{\Ccal}_0} \expec_{\sss 1}^{\bb_{\sss 0}}\left(\indicwo{F_{\sss 1}(\tb_{\sss 0},t,z,x)}\indic{0\conn z\text{ off }{\tilde\Ccal}_1\setminus\{z\}} \right)\right).
}
The event $F^\iota_{\sss 0}$ is given as a union of three events, so the diagram representing  \refeq{lemmapercboundXiiota1-event-bound} consists of three parts that are shown in Figure \ref{fig-Form-Xi1Iota}.

\begin{figure}[!htb]
\begin{center}
\picPiIotaOne[0.9]
\caption{Diagrammatic representations of $\Xi_p^{\ssc[1],\iota}(x)$. The solid lines are connections in $\tilde \Ccal_{\sss 0}=\tilde \Ccal^{b_{\sss 0}}_{\sss 0}(0)$, while the dashed lines represent connections in $\tilde\Ccal_{\sss 1}$.  Shaded triangles can be trivial. All connections are bond-disjoint. The first two contributions give rise to the $P^{\iota,a}(u,w)$ term in \refeq{lemmapercboundXiiota1-1-step0}, the last gives rise to the $\delta_{a,0}\delta_{\iota,\kappa}\indic{u=w=0}$ term.}
\label{fig-Form-Xi1Iota}
\end{center}
\end{figure}

\noindent
We define $a=d_{\tilde \Ccal_0}(u,w)$ (where we recall that $u=\bb_{\sss 0}$), and $b=d_{\tilde\Ccal_1}(t,z)$ and show that
\begin{align}
   \lbeq{lemmapercboundXiiota1-1-step0}
\shift \Xi_p^{\ssc[1],\iota}(x) \leq \sum_{a,b=0}^2\sum_{u,\kappa,w,z,t} (\delta_{a,0}\delta_{\iota,\kappa}\indic{u=w=0} +P^{\iota,a}(u,w))\bar A^{\kappa,a,b}(u,w,t,z)P^{{\sss\rm E},b}(t-x,z-x).
\end{align}

Once this is established, the bound \refeq{BoundXiIotaOne-1} follows by repeating the steps leading to \refeq{lemmapercboundXi1-1-summation}.
The right triangle $z,t,x$ and the square $w,z,t,\bb$ are bounded in the same way as the left triangle and the middle piece of $\Xi_p^{\ssc[1]}$. Thus, we will only discuss the bound on the left parts of the three diagrams. Again we consider the different cases $a=0,1,\geq 2$.\\
{\bf Case $a=0$.} In this case we know that $u=w$, which is possible for the events $F^{\iota,{\sss \rm I}}_0$ and $F^{\iota,{\sss \rm III}}_0$, and this is the only contribution due to $F^{\iota,{\sss \rm III}}_0$.
If $F^{\iota,{\sss \rm III}}_0$ occurs, then we have $w=u=0$ and $\kappa=\iota$.
For $F^{\iota,{\sss \rm I}}_0$, we first have a connection $0\conn \ve[\iota]$ that does not use
the bond $(0,\ve[\iota])$ and then $\ve[\iota]\dbc w=u$. We bound the sum of the probabilities of the contributions due to $F^{\iota,{\sss \rm I}}_0$ and $F^{\iota,{\sss \rm III}}_0$ by
\begin{align}
 \delta_{u,w}\big(\delta_{\iota,\kappa}\delta_{w,0} + (1-\delta_{0,w})\tau_{3,p}(\ve[\iota])\prob(\ve[\iota]\dbc w )\big)
  = \delta_{u,w} \left(\delta_{w,0}\delta_{\iota,\kappa}+  P^{\iota,0}(u,w)\right).
\end{align}
{\bf Case $a=1$.} The events $F^{\iota,{\sss \rm I}}_0$ and $F^{\iota,{\sss\rm II}}_0$ contribute.
The event $F^{\iota,{\sss\rm II}}_0$ can occur for $a=1$ only when  $u=\ve[\iota]$ and when $w$ is directly connected to $\ve[\iota]$. For $F^{\iota,{\sss \rm I}}_0$ we distinguish between the cases $w=\ve[\iota]$ and $w\neq \ve[\iota]$. We bound the sums of the probabilities of the discussed events by
\begin{align}
 \delta_{u,\ve[\iota]} \diagRepulsiveLetter{B}_{2,\underline 1}(w,\ve[\iota])+
\tau_{3,p}(\ve[\iota]) (\delta_{w,\ve[\iota]}\diagRepulsiveLetter{B}_{\underline 1,3}(u-\ve[\iota],0)+ \diagRepulsiveLetter{T}_{1,\underline 1,1}(w-\ve[\iota],u-\ve[\iota],0))=P^{\iota,1}(u,w).
\end{align}
{\bf Case $a\geq 2$.} The events $F^{\iota,{\sss \rm I}}_0$ and $F^{\iota,{\sss \rm II}}_0$ contribute.
For $F^{\iota,{\sss \rm I}}_0$, we note that $u\neq \ve[\iota]$ as the bubble would shrink to a point otherwise. For $F^{\iota,{\sss \rm II}}_0$, we distinguish between whether $u=\ve[\iota]$ or not, and and whether $w=0$ or not. As $a=d_{\tilde \Ccal_0}(u,w)\geq 2$, we conclude the bound to be
\begin{align}
\tau_{3,p}(\ve[\iota])P^{2}(u-\ve[\iota],w-\ve[\iota])+\delta_{u,\ve[\iota]}(\delta_{0,w} \tau_{3,p}(\ve[\iota])+\diagRepulsiveLetter{B}_{1,2}(w,\ve[\iota])) \nnb
+(\delta_{0,w}\tau_{3,p}(\ve[\iota])+\diagRepulsiveLetter{B}_{1,1}(w,\ve[\iota]))(1-\delta_{\ve[\iota],u})\prob(\ve[\iota]\dbc u)=P^{\iota,1}(u,w).
\end{align}
This completes the proof of  \refeq{lemmapercboundXiiota1-1-step0} and thus also of \refeq{BoundXiIotaOne-1}.\\

The bound on the weighted sums are obtained in the same way as the bound on $\sum_{x} \|x\|_2^2 \Xi^{\ssc[1]}_{p}(x)$.
We first prove a decompositions similar to \refeq{lemmapercboundXiiota1-1-step0}.
Thereby, we use $\|x-\ve[\iota]\|_2^2\leq 2\|t-\ve[\iota]\|_2^2+2\|t-x\|_2^2$ if the right triangle is non-trivial. As this follows the same ideas as demonstrated above, we omit the proof.
\end{proof}

\begin{proof}[Proof of Lemma \ref{lemmapercboundLowerBounds}.]
In this proof and the proof of Lemma \ref{lemmapercboundXi0minus1} we prove lower bounds on the coefficients.
We create most of our bounds using the FKG and Harris inequalities, which are standard tools in percolation (see \cite{Grim99}).
The coefficients are defined as the probability of combinations of increasing and decreasing events.
For the lower bounds we have the problem that we can not rearrange them such that these inequalities can be applied to our advantage.\\
As we explain in the following, we create these lower bounds by counting explicit contributions which use at most four steps and bound these by hand. We denote by
\begin{align}
\gamma_{\rho}=\{(0,\ve[\rho]),(\ve[\rho],\ve[1]+\ve[\rho]),(\ve[1]+\ve[\rho],\ve[1])\}
\end{align}
the three-step path from $0$ to $\ve[1]$ that passes through $\ve[\rho]\neq \ve[1],-\ve[1]$.
We say that $\gamma_\rho$ is occupied if all three bonds of $\gamma_\rho$ are occupied and otherwise we call it vacant.

We start by deriving a lower bound on $\tau_{3}(\ve[1])$, for which we note that
\begin{align}
    \lbeq{Lower-Bounds-tau3}
  \tau_{3}(\ve[1])\geq& \prob(\bigcup_{\rho:|\rho|\neq 1} \{ \gamma_\rho\text{ is occ.}\})  \geq
\prob(\bigcup_{\rho:|\rho|\neq 1} \{ \gamma_\rho\text{ is occ.}\}\cap \bigcap_{\iota\neq -1,1,\rho}
  \{\gamma_\iota \text{ is vac.}\})\\
=&\sum_{\rho}\prob( \gamma_\rho\text{ is occ.}) \prod_{\iota\neq -1,1,\rho}\prob(\gamma_\iota \text{ is vac.})= (2d-2)p^3(1-p^3)^{2d-3},
\nn
\end{align}
where the independence is due to the fact that the edges on the different paths $(\gamma_\rho)_\rho$ are bond disjoint. %This idea can easily be extended to paths with more than three steps.
Further, we define the event
	\begin{align}
  	\lbeq{Abbreviation-for-extra-bond}
	\mathcal{T}_\kappa
	:=\{\ve[1]+\ve[\kappa]\nin \tilde \Ccal^{(\ve[1],\ve[1]+\ve[\kappa])}(\ve[1])\},
	%\qquad 		
	%\quad
	%\mathcal{T}^C_\kappa :=\{\ve[1]+\ve[\kappa]\in \tilde \Ccal^{\ve[1],\ve[1]+\ve[\kappa]}(\ve[1])\}
	\end{align}
and say that a vertex $v\in \Z^d$ is contained in a path $\gamma$ and write $v\in \gamma$ if it is the starting or endpoint of one of the bonds in $\gamma$. For the lower bound, we remark that
\begin{align}
\prob(\{0\connLe{3} \ve[1]\}\cap\ \mathcal{T}_\kappa )
\geq &\ \prob(\{0\connLe{\underline{3}} \ve[1]\}\cap\ \mathcal{T}_\kappa )\nnb
=&\ \prob(0\connLe{\underline{3}} \ve[1] )-\prob(\{0\connLe{\underline{3}} \ve[1]\}\cap\ \mathcal{T}^c_\kappa )\nnb
\stackrel{   \refeq{Lower-Bounds-tau3}} \geq&
 (2d-2)p^3(1-p^3)^{2d-3}-\prob(\{0\connLe{\underline{3}} \ve[1]\}\cap\ \mathcal{T}^c_\kappa ).
 \lbeq{Lower-Bounds-tau3-creation}
\end{align}
We bound the second term by
\begin{align}
\prob(\{0\connLe{\underline{3}} \ve[1]\}\cap\ \mathcal{T}^c_\kappa)
&\leq \sum_{\rho\colon |\rho|\neq 1} \prob(\{ \gamma_\rho\text{ is occ.}\}\cap\ \mathcal{T}^c_\kappa)\\
&\leq p^3\sum_{\rho:|\rho|\neq 1} \sum_{v\in \gamma_\rho}
\prob^{(\ve[1],\ve[1]+\ve[\kappa])}\big( v\conn \ve[1]+\ve[\kappa] \text{ off }
\{0,\ve[\rho],\ve[\rho]+\ve[1],\ve[1]\}\setminus \{v\}\big),\nn
\end{align}
so that
\begin{align}
\prob(\{0\connLe{\underline{3}} \ve[1]\}\cap\ \mathcal{T}^c_1)\leq& (2d-2)p^3
\left(\tau_{3,p}(\ve[1])+\tau_{2,p}^{1}(\ve[1]+\ve[2])+\tau_{5,p}^{1}(2\ve[1]+\ve[2])+\tau_{4,p}^{1}(2\ve[1])\right),\nnb
\prob(\{0\connLe{\underline{3}} \ve[1]\}\cap\ \mathcal{T}^c_\kappa)\leq& (2d-3)p^3
\left(\tau_{3,p}(\ve[1])+2\tau_{2,p}^{1}(2\ve[1]+\ve[2])+2p^3 + \tau_{5,p}(\ve[1]+\ve[2]+\ve[3])\right),\nn
\end{align}
for $|\kappa|\neq 1$. To summarize this, we state the bound when summing over $\kappa$
and note that $\kappa= -1$ does not contribute to the original object, to obtain
\begin{align}
 \lbeq{Lower-Bounds-tau3-sumed}
\sum_{\kappa}& \prob( \{0\connLe{3} \ve[1]\} \cap \mathcal{T}_\kappa )\nnb
\geq & (2d-1)(2d-2)p^3 (1-p^3)^{2d-3}
-(2d-2)^2p^3\tau_{3,p}(\ve[1])^2-(2d-2)p^3\tau_{4,p}^{1}(2\ve[1])\nnb
&-(2d-2)p^3\Big( (4d-5) \tau_{2,p}^{1}(\ve[1]+\ve[2])+(4d-4)p^3+\sum_{\kappa}
\tau_{5,p}(\ve[1]+\ve[2]+\ve[\kappa])
\Big).
%  \lbeq{Lower-Bounds-tau3-plusStep}
\end{align}
Now we start to prove the stated lower bounds. We recall \refeq{Def-Pi-Split}, \refeq{PiZero-in-clear} to see that
\begin{align}
\Pi^{\ssc[0],{\sss 1},\kappa}_{\alpha,p}(x)&=
\delta_{x,\ve[1]} p\prob (\{0\conn \ve[1]\}\cap \mathcal{T}_\kappa \mid (0,\ve[1])\text{ is vacant} )
%\nnb&
=\delta_{x,\ve[1]} p\prob (\{0\connLe{3} \ve[1]\}\cap \mathcal{T}_\kappa  ),
\lbeq{PiAlphaZero-in-clear}
\end{align}
so that the lower bound in \refeq{lemmaLowerbound-1} follows from \refeq{Lower-Bounds-tau3-sumed}.\\
Next, we create a lower bound for $\Psi^{\ssc[0],\kappa}$, see \refeq{def-psi-zero}.
For the $2d$ direct neighbours of the origin, we compute
\begin{align}
\sum_{\iota} \Psi^{\ssc[0],\kappa}(\ve[\iota])
&\geq \frac {p^2} {\aap} \sum_{\kappa} \prob^{(0,\ve[1])}_p(\{0\conn \ve[1] \} \cap \mathcal{T}_\kappa),%\\
% \sum_{\iota} \Psi^{\ssc[0],\kappa}(\ve[\iota])
%&\geq \frac {p} {\aap} \sum_{\kappa} \prob_p(\{(0,\ve[1])\text{ occ.}\}\cap \{0\connLe{3} \ve[1] \} \cap \mathcal{T}_\kappa),
\end{align}
we bound this using \refeq{Lower-Bounds-tau3-sumed} (noting that the event is independent from the occupation status of $(0,\ve[1])$) and $p/\aap>1$
to obtain the first part of \refeq{lemmaLowerbound-2}.
For a better bound, we also consider the vertices at distance 2 from the origin that can be part of a four step loop:
\begin{align}
\sum_{\iota,\rho} \Psi^{\ssc[0],\kappa}(\ve[\iota]+\ve[\rho]) \geq&  2d(2d-2) \sum_{\kappa} \prob_p(\{0\connLe{\underline 2} \ve[1]+\ve[2] \} \circ \{0\connLe{\underline 2} \ve[1]+\ve[2] \} \cap \mathcal{T}_\kappa)\nnb
=&  (2d-2) p^4 \sum_{\kappa} \prob_p(\mathcal{T}_\kappa \mid 0,\ve[1],\ve[2],\ve[1]+\ve[2]\text { are occ.})\nnb
\geq &  (2d-2)^2 p^4 (1-\tau_{3,p}(\ve[1])-2\max_\kappa\{\tau_{2,p}^{1}(\ve[1]+\ve[\kappa])\}-2p^3 )\nnb
&- (2d-2) p^4\sum_{\kappa}\tau_{5,p}^{1}(\ve[1]+\ve[2]+\ve[\kappa])
\end{align}
This creates the bound on the second term in \refeq{lemmaLowerbound-2}. We improve this bound by also considering the $32 d(d-1)(d-2)$ paths that return to the origin in $6$ steps using three different dimensions. For these paths we have to exclude that a double connection is present in four steps and that the path passes the point $x+\ve[\kappa]$. We exclude these events by using a bound of of the following type: Let $\gamma$ and $\gamma'$ be two paths whose bonds do not touch (i.e., there is no $v$ such that $v\in \gamma, v\in \gamma'$), with $x\in \gamma$ and $y\in \gamma'$. Then,
	\begin{align}
	\lbeq{Lower-Bounds-not-connected-paths}
  	\prob( x \nc y \mid \gamma\text{ and }\gamma' \text{ are occ.})=&1-\prob( x \conn y \mid \gamma\text{ and }\gamma' \text{ are occ.})\\
  	\geq& 1-\sum_{v\in \gamma}\sum_{w\in \gamma'}\prob(v \conn w \text{ off }(\gamma\cup \gamma')\setminus\{v,w\} ).\nn
	\end{align}

The lower bound on $\Pi^{\ssc[1],\iota,\kappa}_p$ is proven in a similar way.
\iflongversion
 As this is not very insightful we omit it here and give it in the extended version in Appendix \ref{Appendix-ToTheProofs-LowerBound}.
In the same part of the appendix, we give the proof of Lemma \ref{lemmapercboundXi0minus1}.
 \else
As this is elaborate and not very insightful we omit the details. These details and the
proof of Lemma \ref{lemmapercboundXi0minus1} are given in \cite[Appendix B.2]{FitHof13d-ext} of the extended version.
\fi
\end{proof}

\subsection{Proof of the bounds for $N\geq 2$}
\label{secProofBoundsNBig}
\subsubsection{Strategy of proof for the bounds}
In this section we sketch how to prove the bounds on the coefficients stated in Propositions \ref{PropBoundXiBig} and \ref{PropBoundXiIotaBig}.
The proofs are basically an adaptation of the techniques of the classical lace expansion, see e.g.\ \cite{Slad06}, in combination with a consideration of cases for the lengths of lines that are shared by two parts of the arising diagrams. The first author explains this in detail in his thesis( see \cite[Chapter 4]{Fit13}).

The first step is to prove a pointwise bound on the coefficients. In order to do this, we combine the building blocks to construct the bounding diagrams. For $b=0,1,2$ and $x,y\in\Zd$, let
\begin{eqnarray}
P^{\ssc[0],b}(x,y)=P^{{\sss \rm S},b}(x,y),\qquad \qquad R^{\ssc[0],b}(x,y)=P^{{\sss \rm E},b}(x,y),
\end{eqnarray}
and, for $N\geq 1$, we recursively define
\begin{align}
P^{\ssc[N],b}(u_{\sss N},w_{\sss N})=&\sum_{u_{\sss N-1},w_{\sss N-1}\in\Zd}\sum_{\kappa}\sum_{a=0}^2
P^{\ssc[N-1],b}(u_{\sss N-1},w_{\sss N-1}) B^{\kappa,a,b}(u_{\sss N-1},w_{\sss N-1},w_{\sss N},u_{\sss N}),\\
%P^{\ssc[N],\iota,b}(u_{\sss N},w_{\sss N})=&\sum_{u_{\sss N-1},w_{\sss N-1}\in\Zd}\sum_{\kappa}\sum_{a=0}^2
%P^{\ssc[N-1],\iota,b}(u_{\sss N-1},w_{\sss N-1})\nn\\
%&\qquad \qquad\qquad\times B^{\kappa,a,b}(u_{\sss N-1},w_{\sss N-1},w_{\sss N},u_{\sss N}),\\
R^{\ssc[N],a}(x,y)=&\sum_{u,v\in\Zd}\sum_{\kappa}\sum_{b=0}^2\bar B^{\kappa,a,b}(x,y,u,v) R^{\ssc[N-1],b}(u,v).
\end{align}
Further, recall the definition of $Q^{{\sss \rm S},a}$ and $Q^{{\sss \rm E},a}$ in \refeq{definition-Q1}-\refeq{definition-Q3}.
Then, we prove that these diagrams can be used to bound the coefficients as follows:
\begin{lemma}[$x$-space bounds]
\label{lemmaPercpointBound}
For every $x\in\Zd, N\geq 1$ and $0\leq M\leq N-1$,
\begin{align}
\Xi^{\ssc[N]}(x)\leq& \sum_{u_{\sss M},w_{\sss M},w_{\sss M+1},z_{\sss M+1}\in\Zd}\sum_{\kappa_{\sss M}}\sum_{a,b=1}^2
P^{\ssc[M],a}(u_{\sss M},w_{\sss M})\lbeq{Percolation-XiXspaceBoundOne} \\
&\qquad\quad \times \bar A^{\kappa_M,a,b}(u_{\sss M},w_{\sss M},w_{\sss M+1},z_{\sss M+1})R^{\ssc[N-M-1],b}(z_{\sss M+1}-x,w_{\sss M+1}-x),\nnb
\Xi^{\ssc[N]}(x)\leq& \sum_{u_{\sss N-1},w_{\sss N-1},t_N,z_N}\sum_{\kappa_{\sss N}}\sum_{a,b=0}^2 P^{\ssc[N-1],a}(u_{\sss N-1},w_{\sss N-1})
\lbeq{Percolation-XiXspaceBoundTwo}\\
&\qquad\quad \times A^{\kappa_N,a,b}(u_{\sss N-1},w_{\sss N-1},w_{\sss N},u_{\sss N})
Q^{{\sss \rm E},b}(u_{\sss N}-x,w_{\sss N}-x),\nnb
%\end{align}
%\begin{align}
\Xi^{\ssc[N]}(x)\leq& \sum_{w_0,u_0,w_1,z_1}\sum_{a,b=0}^2\sum_{\kappa_{\sss 1}}
Q^{{\sss \rm S},a} (u_0,w_0) A^{\kappa,a,b}(z_1,w_{1},u_0,w_0)
\lbeq{Percolation-XiXspaceBoundThree}
%\\&\qquad\quad \times
R^{\ssc[N-1],b}(u_1-x,w_1-x).
%\nn
\end{align}
\end{lemma}

Let us briefly discuss this in the example $N=2$, as displayed in Figure \ref{fig-Form-Xi-InBound}.
In Section \ref{secBoundsPercEvents}, we have bounded the coefficient in terms of simpler events and
have produced the bound \refeq{XiFs}:\footnote{Note that the measure ${\prob}^{\ssc[2]}$ enforces that the events
$E_{\text{vac}}(\tb_{\sss 0})_0$ and $E_{\text{vac}}(\tb_{\sss 1})_1$ occur, recall \refeq{PN-def} and \refeq{E'bd}.}
	\eqan{
 	 \Xi^{\ssc[2]}_p(x)\leq \sum_{\vec b, \vec w, \vec t.\vec z}  p^2 {\prob}^{\ssc[2]}\Big(& F_0(b_0,w_0,z_1)_0 \cap F(b_0,t_1,z_1,b_1, w_1, z_{2})_1
                        \cap F_{\sss 2}(b_1,t_{2},z_{2},x)_2\Big).
	}
We draw one possible contribution in Figure \ref{fig-Form-Xi-InBound}.
\begin{figure}[h]
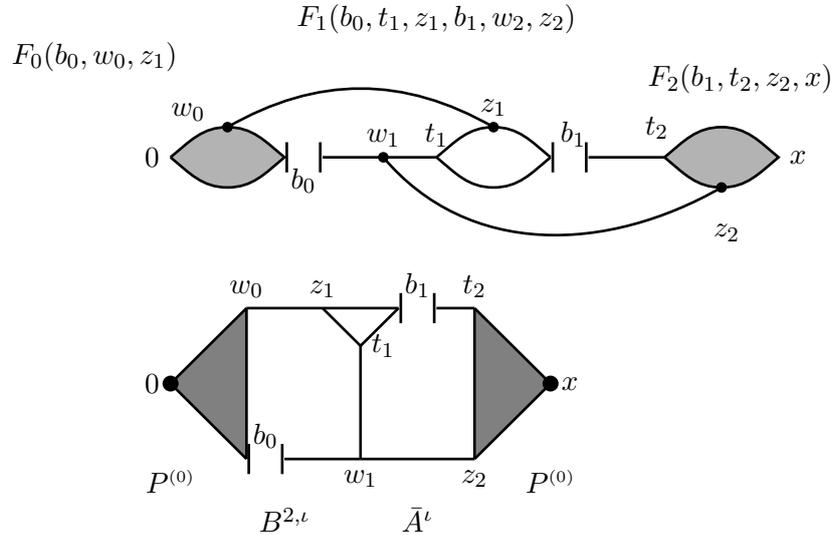

\begin{center}
\picXiNtwoStructure[1]
\caption{The combination of events that we have used to bound $\Xi_p^{\ssc[2]}(x)$ and the corresponding bounding diagram. Lines indicate disjoint connections. A filled triangle might be trivial. Note that in this case we can choose $z_1$ and $z_2$ such that the path from $w_i$ to $z_i$ intersects $\tilde \Ccal_i$ only at $z_i$.} \label{fig-Form-Xi-InBound}
\end{center}
\end{figure}
We define
\begin{itemize}
\item[$\rhd$] $a_0$ to be the length of the path in $\tilde \Ccal_0$ from $\bb_0$ to $w_0$ that does not pass the origin,
\item[$\rhd$] $a_1$ to be the length of the path in $\tilde \Ccal_1$ from $\bb_1$ to $w_1$ that does not pass $z_1$,
\item[$\rhd$] $a_2$ to be the length of the path in $\tilde \Ccal_2$ from $z_2$ to $t_2$ that does not pass $x$.
\end{itemize}
Performing a consideration of cases for $a_i$, as was done in the proof of Lemma \ref{lemmapercboundXi1}, we obtain the bounds stated in
\refeq{Percolation-XiXspaceBoundOne}. Which of the three bounds is obtained depends on where we let the connections
$\{\bb_0\conn w_0\}$, $\{\bb_1\conn w_1\}$, $\{z_2\conn t_2\}$ contribute.\\

To prove the bounds for all $N$ we use induction on $N$.
The proof for $\Xi^{\iota}$ differs only in the different initial block of the bounding diagram.\\

Once the $x$-wise bounds of Lemma \ref{lemmaPercpointBound} are proven, we use a split as demonstrated in \refeq{lemmapercboundXi1-1-summation} to conclude the bounds stated in Propositions \ref{PropBoundXiBig} and \ref{PropBoundXiIotaBig}.

For the bound on the weighted sum we first split the weights at the level of events using
\begin{align}
\lbeq{Weightsplit-end}
  \|x\|_2^2\leq J\sum_{i=1}^J \|x_i\|_2^2,\quad \text{ for $x_i$ such that }  x=x_i.
 \end{align}
For each of the $J$ terms we use one of the bounds stated in Lemma \ref{lemmaPercpointBound} and decompose the sums as in the unweighted case.\\
To be able to show the mean-field result in $d=11,12$, we improve the bounds for $N=2,3$ by considering the special case that the left- and/or right-most triangle are trivial. Doing this we reduce the leading factor $J$ originating from \refeq{Weightsplit-end}, by one or two.
Further, for $N=2$, we extract the leading contribution, consisting only of two trivial triangles and bound these manually.
\iflongversion
Details can be found in Appendix \ref{Appendix-ToTheProofs-DoubleTriangle}.
\else
Details can be found in \cite[Appendix B.1]{FitHof13d-ext}.
\fi

%=================================================ProofRelRes=================================================================
%==========================================================================================================================
\section{Proof related results: Proofs of Theorems \ref{thm-x-space}, \ref{thm-IIC-lim} and \ref{thm-one-arm}}
\label{sec-proof-rel-res}
In this section, we prove Theorems \ref{thm-x-space}, \ref{thm-IIC-lim} and \ref{thm-one-arm} one by one.
\bigskip

\noindent
{\it Proof of Theorem \ref{thm-x-space}.} The proof of Theorem \ref{thm-x-space} follows by using the $x$-space asymptotics proved by Takashi Hara in \cite{Hara08}. See in particular \cite[Proposition 1.3]{Hara08}.
In more detail we use that, by our numerical computations in dimension $d=11$,
	\eqn{
	\lbeq{Tbar-bd}
	\bar{T}^{\sss(0,0)}=\sup_{x\in \Zd}(\tau_{p_c}^{\star 3}(x)-\delta_{0,x})\leq 0.53562,\qquad
	T_{p_c}=2d p_c \sup_{x\in \Zd}(\tau_{p_c}^{\star 3}\star D)(x)\leq 0.28036.
	}
In particular, by a recent improvement of the bounds by Hara compared to \cite[Proposition 1.3]{Hara08}, it suffices to prove that 	
	\eqn{
	\lbeq{geom-cond}
	T_{p_c}(1+2\bar{T}^{\sss(0,0)})<1.
	}
This corresponds to the middle diagram in Figure \ref{fig-Form-Xi4-Decomposed}, which needs to be at most $1$ as it appears to the power
$N-1$ in Figure \ref{fig-Form-Xi4-Decomposed} and is being summed out over $N$.
The improvement in \refeq{geom-cond} follows by carefully inspecting which triangles can be trivial and which are not.
The first contribution in the middle diagram in Figure \ref{fig-Form-Xi4-Decomposed}
corresponds to the term $T_{p_c}$ in \refeq{geom-cond}, the other two contributions are each bounded by
$T_{p_c}\bar{T}^{\sss(0,0)}$.
The bound in \refeq{geom-cond} follows from \refeq{Tbar-bd} and the estimate on $p_c(11)\leq 0.048242$.
\qed
\bigskip

\noindent
{\it Proof of Theorem \ref{thm-IIC-lim}.}  Theorem \ref{thm-IIC-lim} is proved by the second author and J\'arai  \cite{HofJar04}. We use the more recent version in \cite{HeyHofHul14a}, where it was proved under the assumption that the classical lace expansion converges (see also the proof of Theorem \ref{thm-x-space}).
Thus, Theorem \ref{thm-IIC-lim} follows from Theorem \ref{thm-x-space} and the fact that $\hat{\Pi}_{p_c}^{\ssc[N]}(0)$ is exponentially small for $N$ large. We next show this latter claim. By \cite[(4.31) in Proposition 4.1]{BorChaHofSlaSpe05b}, for all $N\geq 1$,
	\eqn{
	\lbeq{PiN-bd}
	\hat{\Pi}_{p_c}^{\ssc[N]}(0)\leq T_{p_c}'[2T_{p_c}T_{p_c}']^{N-1},
	}
where
	\eqn{
	\lbeq{Tp-def}
	T_p'=\max_{x\in \Zd} (\tau_p\star\tau_p\star \tau_p)(x)\leq 1+\bar{T}^{\sss(0,0)}.
	}
An improvement alike the one used in \refeq{geom-cond} can improve the above by replacing \refeq{PiN-bd} by
	\eqn{
	\lbeq{PiN-bd-impr}
	\hat{\Pi}_{p_c}^{\ssc[N]}(0)\leq T_{p_c}'[T_{p_c}(1+2\bar{T}^{\sss(0,0)})]^{N-1},
	}
Therefore, it suffices to show that $T_{p_c}(1+2\bar{T}^{\sss(0,0)})<1$, which we have already proved above.
\qed
\bigskip

\noindent
{\it Proof of Theorem \ref{thm-one-arm}.}  Theorem \ref{thm-one-arm} is proved by Kozma and Nachmias \cite{KozNac08,KozNac11} under the assumption that there exist constants $c_1$ and $c_2$ with $0<c_1<c_2<\infty$ such that
	\eqn{
    	\lbeq{x-space-KN}
    	c_1\|x\|_2^{-(d-2)}\leq \tau_{p_c}(x)\leq c_2\|x\|_2^{-(d-2)}.
	}
The assumption in \refeq{x-space-KN} follows from Theorem \ref{thm-x-space}.
\qed

%=================================================Ack=================================================================
%==========================================================================================================================
\paragraph{Acknowledgements.}
This work was supported in part by the Netherlands Organisation for Scientific Research (NWO) through VICI grant 639.033.806 and the Gravitation {\sc Networks} grant 024.002.003. We thank
David Brydges, Takashi Hara and Gordon Slade for their constant encouragement, as well as for several stimulating discussions. This work builds upon the work by Takashi Hara and Gordon Slade. We particularly thank Takashi Hara for sharing his handwritten notes on the proof of mean-field behavior for $d\geq 19$, and explaining how this can be extended to $d\geq 15$. We have thoroughly enjoyed our animated discussions with Takashi in July 2013, which allowed us to compare notes and estimates on triangles, two-point functions, etc. Without these discussions, it would have been much harder to compare our results to the results by Hara and Slade. Finally, we are indebted to Takashi for his help in the proof of Theorem \ref{thm-x-space}, which relies on an improved version of this analysis in \cite{Hara08} that Takashi shared with us. The work of RF was performed in part at Stockholm University in the period September 2013 until September 2015. We further thank the referee for comments that significantly improved the presentation of the paper.

\clearpage
\appendix
%=================================================ListOfNotation=================================================================
%==========================================================================================================================
\section{Notation}
\begin{table}[th]
\begin{tabular}{|c|c|c|}
  \hline
  Notation & brief description & defined in\\
  \hline
  SRW & simple random walk  												                & Section \ref{sec-phil-proof}\\
  NBW & non-backtracking random walk  										                & Section \ref{sec-phil-proof}\\
  \hline
  & & \\[-2mm]
  $D$ & SRW step distribution                                                               & \refeq{def-Dhat} \\
  $\iota,\kappa$ & direction of a bond $\iota,\kappa\in\{\pm 1,\dots,\pm d \}$              & above \refeq{J-D-matrix-def}  \\
  $u,v,w,x,y$ & points on the lattice: $\Zd$                                                & \\
  $k$ & Fourier argument, so $k\in(-\pi,\pi)^d$ 								       	    & \refeq{def-FourTrans} \\
  $p$ & probability of a bond being occupied                                                & \\
  $f\star g, f^{\star n}$ & convolution of functions $f,g\mapsto \Zd$                       &  \refeq{definition-convolution}\\
  \hline
  & & \\[-2mm]
  $C_z, B_z$ &SRW and NBW two-point functions 	             				    			&  \refeq{genSRW}, \refeq{NBWGenSolved} \\
  $\mJ$ & permutation matrix with entries $(\mJ)_{\iota,\kappa}=\delta_{\iota,-\kappa}$ 	& \refeq{J-D-matrix-def}\\
  $\mD[k]$ & diagonal matrix with entries $(\mD[k])_{\iota,\kappa}=\delta_{\iota,\kappa} \e^{\ii k_\iota}.$ & \refeq{J-D-matrix-def}\\
  $\tau_p$ & percolation two-point function 									            & \refeq{def-tau}, \refeq{lace-exp-eq}\\
  $\tau^\iota_p$ & modified percolation two-point function 									& \refeq{tau-z-def}\\
  $\Xi_z, \Xi^{\iota}_z$ & coefficient of the NoBLE expansion 						        &  \\
  $\Psi^{\kappa}_z, \Pi^{\iota,\kappa}_z$ & coefficient of the NoBLE expansion 			    &  \\
  $\diagRepulsiveLetter{D},\diagRepulsiveLetter{B},\diagRepulsiveLetter{T},\diagRepulsiveLetter{S}$ & repulsive diagram used for the bounds 			    & \refeq{defpercRepDouble}-\refeq{defpercRepTriangle} \\
  $P^{{\sss \rm S}}$, $A$, $\bar B^{{\sss (2)},\iota}$ & building blocks used for the bounds 			    & Section \ref{secBoundsPercDefinitionBlocks}\\[2mm]
  \hline
  & & \\[-2mm]
   $f_1,f_2,f_3$ & bootstrap function                                                       &  \refeq{defFunc1}-\refeq{defFunc3}\\
   $\gamma_i,\Gamma_i$ & assumed/concluded bounds on the $f_i$                              &  \refeq{defFunc1}-\refeq{defFunc3}\\
  \hline
\end{tabular}
\caption{List of notation, that is used in at least two different sections.}
\end{table}
\clearpage

\bibliographystyle{plain}
\bibliography{NoBLEBiB}

 \iflongversion
%=================================================DetailsDefinitionBlocks=================================================================
%==========================================================================================================================

\section{Detailed definition of the bounding diagrams}
\label{app-bounds}
In the appendix we define the ingredients of the bounding diagrams. These bounds are stated in the form of several tables.
For an example of how to read the definitions via tables compare \refeq{Pa-exampledefinition} and Table \ref{PercBoundTableP1One}. For the diagrams we use $a=d_{\tilde\Ccal}(0,v)$ and $b=d_{\tilde\Ccal}(x,y)$.
%====================================================================
%====================================================================
%====================================================================

{\small \fourcolomntablePer{Diagrams and definition of $P^{b}(x,y)$ }{PercBoundTableP1One}{
  \multicolumn{1}{|c|} {\multirow{2}{*}{$\begin{array}{c} b=0\\ \Rightarrow x=y\end{array}$}} & $x=0$ &     \picTableCaseOne[0.6] & $\delta_{0,x}$    \\ \cline{2-4}
  \multicolumn{1}{|c|} {} & $x\neq 0 $ &     \picTableCaseThree[0.7] & $(1-\delta_{0,x})\prob(0\dbc x) $   \\ \cline{1-4}
  \multicolumn{1}{|c|} {\multirow{2}{*}{$\begin{array}{c} b=1\\ \Rightarrow x\neq 0\end{array}$}} & $y=0$ &    \picTableCaseFour[0.7] & $\delta_{0,y}\diagRepulsiveLetter{B}_{3,\underline 1}(x,0)$    \\ \cline{2-4}
  \multicolumn{1}{|c|} {} & $\begin{array}{c} 0\neq y \end{array}$ &  \picTableCaseSix[0.6] & $ \diagRepulsiveLetter{T}_{1,\underline 1,1}(x,y,0)$   \\ \cline{1-4}
   \multicolumn{1}{|c|} {\multirow{2}{*}{$\begin{array}{c} b\geq 2\\ \Rightarrow x\neq 0\end{array}$}} & $ y=0$ &   \picTableCaseEight[0.6] & $ \delta_{0,y}\diagRepulsiveLetter{D}_{2,2}(x)$  \\ \cline{2-4}
    \multicolumn{1}{|c|} {} & $y\neq 0 $ &     \picTableCaseNine[0.7] & $ \diagRepulsiveLetter{T}_{1,2,1}(x,y,0)$ \\
 }

 \fourcolomntablePer
{Diagrams and definition of $P^{\iota,b}(x,y)$ }
{PercBoundTableP1Iota}{
  \multicolumn{1}{|c|} {\multirow{1}{*}{$\begin{array}{c} b=0\\ \Rightarrow x=y\end{array}$}} &  &     \picPIotaZero[0.7] & $\tau_{3,p}(\ve[\iota]) \prob(\ve[\iota]\dbc x)$    \\ \cline{1-4}
  \multicolumn{1}{|c|} {\multirow{2}{*}{$\begin{array}{c} b=1\end{array}$}} & $y \text{ on sausage}$ &    \picPIotaOne[0.65] & $\begin{array}{c}
  \tau_{3,p}(\ve[\iota])  \\
\times \big(  \delta_{\ve[\iota],y}\diagRepulsiveLetter{B}_{3,\underline 1}(x-\ve[\iota],0)\\
+\diagRepulsiveLetter{T}_{1,\underline 1,1}(y-\ve[\iota],x-\ve[\iota],0) \big) \end{array} $
  \\ \cline{2-4}
  \multicolumn{1}{|c|} {} & $x=\ve[\iota]$ &\picPIotaOneTwo[0.7] & $\diagRepulsiveLetter{B}_{2,\underline 1}(y,\ve[\iota])$    \\ \cline{1-4}
   \multicolumn{1}{|c|} {\multirow{3}{*}{$\begin{array}{c} b\geq 2\end{array}$}} & $y \text{ on sausage}$ &     \picPIotaTwo[0.7] & $\tau_{3,p}(\ve[\iota]) P^{\ssc[0],2}(x-\ve[\iota],y-\ve[\iota])$   \\ \cline{2-4}
  \multicolumn{1}{|c|} {} & $x=\ve[\iota]$ &\picPIotaTwoTwo[0.7] & $\delta_{0,y}\tau_{3,p}(\ve[\iota])+\diagRepulsiveLetter{B}_{1,2}(y,x)$\\ \cline{2-4}
  \multicolumn{1}{|c|} {} & $\begin{array}{c}x\neq \ve[\iota]\\ y\text{ not}\\\text{ on sausage}\end{array}$ &\picPIotaTwoThree[0.65] & $\begin{array}{c}\left(\begin{array}{c} \diagRepulsiveLetter{B}_{1,1}(y,\ve[\iota])\\+ \delta_{0,y}\tau_{3,p}(\ve[\iota])\end{array}\right) \\ \times 	(1-\delta_{\ve[\iota],x})\prob(\ve[\iota]\dbc x)\end{array}$\\ }
}

\threecolomntablePer{Diagrams and definition of $A^{a,b}(0,v,x,y)$ }{PercBoundTableA}
{
$\begin{array}{c} a=b=0\\ \Rightarrow x=y,v=0\\ \Rightarrow v\neq y,x\neq e \end{array}$ & \picACaseTwo[0.5] & $
(1-\delta_{0,x})\prob(0\dbc x)$ \\ \cline{1-3}
$\begin{array}{c} a=0,b=1\\ \Rightarrow v=0,y\neq v\end{array}$&       \picACaseThree[0.5] &
$\begin{array}{c} \diagRepulsiveLetter{T}_{1,\underline 1,1}(x,y,0)\end{array}$\\ \cline{1-3}
  $\begin{array}{c} a=0,b\geq 2\\ \Rightarrow v=0,y\neq 0,\\
  \quad x\neq 0 \end{array}$ &       \picACaseFour[0.5] & $\diagRepulsiveLetter{T}_{1,2,1}(x,y,0)$ \\ \cline{1-3}
  $a=1,b=0$ &           \picACaseFive[0.5] & $\begin{array}{c} 2d D(v) \diagRepulsiveLetter{B}_{1,1}(x,v)\end{array}$ \\ \cline{1-3}
  $a=b=1$ &            \picACaseSix[0.5] & $\begin{array}{c}  2d D(v)\diagRepulsiveLetter{T}_{1,\underline 1,0}(x,y,v)\end{array}$ \\ \cline{1-3}
  $a=1,b\geq 2$ &      \picACaseSeven[0.5] & $ 2d D(v)\diagRepulsiveLetter{T}_{1,2,0}(x,y,v)$    \\ \cline{1-3}
  $a\geq 2,b=0$ &      \picACaseEight[0.5] & $\diagRepulsiveLetter{B}_{1,0}(x,v)$ \\ \cline{1-3}
  $a\geq 2,b=1$ &      \picACaseNine[0.5] & $\diagRepulsiveLetter{T}_{1,\underline 1,0}(x,y,v)$    \\ \cline{1-3}
  $a\geq 2,b\geq 2$ &  \picACaseTen[0.5] & $\diagRepulsiveLetter{T}_{1,2,0}(x,y,v)$  \\
}
\clearpage
\threecolomntablePer{Diagrams and definition of $A^{\iota,a,b}(0,v,x,y)$}{PercBoundTableAIota}
{
$\begin{array}{c} a=b=0\\ \Rightarrow x=y,v=0\\ \Rightarrow v\neq y,x\neq e \end{array}$ & \picAIotaCaseTwo[0.5] & $\begin{array}{c} \diagRepulsiveLetter{T}_{1,\underline 1,1}(\ve[\iota],x,0)\end{array}$ \\ \cline{1-3}
$\begin{array}{c} a=0,b=1\\ \Rightarrow v=0,y\neq v\end{array}$&       \picAIotaCaseThree[0.5] &
$\begin{array}{c}\delta_{x,\ve[\iota]} \diagRepulsiveLetter{T}_{\underline 1,\underline 1,2}(\ve[\iota],y,0)\\+\diagRepulsiveLetter{S}_{\underline 1,1,\underline 1,1}(\ve[\iota],x,y,0)\end{array}$\\ \cline{1-3}
  $\begin{array}{c} a=0,b\geq 2\\ \Rightarrow v=0,y\neq 0,\\
  x\neq 0 \end{array}$ &       \picAIotaCaseFour[0.5] & $\diagRepulsiveLetter{S}_{\underline 1,0,2,1}(\ve[\iota],x,y,0)$ \\ \cline{1-3}
  $\begin{array}{c} a=1,b=0\\  \Rightarrow x\neq 0 \end{array}$&      \picAIotaCaseFive[0.5] & $2d D(v) \diagRepulsiveLetter{T}_{\underline 1,1,0}(\ve[\iota],x,v)$ \\ \cline{1-3}
  $a=b=1$ &       \picAIotaCaseSeven[0.5] & $2d D(v)\diagRepulsiveLetter{S}_{\underline 1,0,\underline 1,0}(\ve[\iota],x,y,v)$ \\ \cline{1-3}
  $a=1,b\geq 2$ &      \picAIotaCaseEight[0.5] & $ 2d D(v)\diagRepulsiveLetter{S}_{\underline 1,0,2,0}(\ve[\iota],x,y,v)$    \\ \cline{1-3}
  $a\geq 2,b=0$ &       \picAIotaCaseSix[0.5] & $\diagRepulsiveLetter{T}_{\underline 1,1,0}(\ve[\iota],x,v)$ \\ \cline{1-3}
  $a\geq 2,b=1$ &   \picAIotaCaseTen[0.5] & $\diagRepulsiveLetter{S}_{\underline 1,0,\underline 1,0}(\ve[\iota],x,y,v)$    \\ \cline{1-3}
  $a\geq 2,b\geq 2$ & \picAIotaCaseNine[0.5] & $\diagRepulsiveLetter{S}_{\underline 1,0,2,0}(\ve[\iota],x,y,v)$  \\ \cline{1-3}
}
We define $A^{\iota,a,b,*}(0,v,x,y)$ alike $A^{\iota,a,b}(0,v,x,y)$, where we replace the repulsive\\
diagrams $\diagRepulsiveLetter{B}$, $\diagRepulsiveLetter{T}$, $\diagRepulsiveLetter{S}$ by the non-repulsive diagrams $\diagRepulsiveLetter{B}^*$, $\diagRepulsiveLetter{T}^*$, $\diagRepulsiveLetter{S}^*$ for $b\neq 0$.
\clearpage
\fourcolomntablePer{Diagrams and definition of $B^{{\sss (2)},\iota,a,b}(0,v,x,y)$}{PercBoundTableBNT}{
  \multicolumn{1}{|c|} {\multirow{3}{*}{$\begin{array}{c}a=0,b\geq 2\\0=v\neq w \end{array} $}} & $d_{\tilde\Ccal}(w,u)=1 $ &     \picBTwoPrimeCaseZeroTwoTwo[0.6] & \hspace{-5mm}$ \begin{array}{c} 2dD(w-u)\\
  \times \diagRepulsiveLetter{S}_{1,\underline 1,1,0}(w,u,x,\ve[\iota])\\
  \times\diagRepulsiveLetter{T}_{1,1,\underline 1}(u-y,w-y,0)
  \end{array}$    \\ \cline{2-4}
  \multicolumn{1}{|c|} {} & $d_{\tilde\Ccal}(w,u)\geq 2 $ &     \picBTwoPrimeCaseZeroTwoThree[0.6] & \hspace{-5mm}$ \begin{array}{c} \diagRepulsiveLetter{S}_{1,0,\underline 1,1}(x-u,\ve[\iota]-u,\\
  \qquad\qquad\quad-u,w-u)\\
  \times\diagRepulsiveLetter{T}_{2,1,1}(w-u,y-u,0)  \end{array}$    \\ \cline{1-4}
   \multicolumn{1}{|c|} {\multirow{3}{*}{$\begin{array}{c}a=1,\\b\geq 2\end{array}$}} &  $\begin{array}{c}  d_{\tilde\Ccal}(u,w)=1\end{array}$ &    \picBTwoPrimeCaseOneTwoThree[0.6] & \hspace{-8mm}
   $\begin{array}{c} \frac 1 p (2d)^2 D(v) D(w-u)   \\
  \times \diagRepulsiveLetter{T}_{1,1,\underline 1}(y-u,w-u,0) \\
  \times \diagRepulsiveLetter{P}_{\underline 1,0,1,\underline 1,0}(\ve[\iota],x,u,w,v) \end{array}$ \\ \cline{2-4}
    \multicolumn{1}{|c|} {} & $\begin{array}{c}d_{\tilde\Ccal}(u,w)\geq 2\end{array}$ &     \picBTwoPrimeCaseOneTwoFour[0.6] & \hspace{-5mm} $\begin{array}{c}
 \frac 1 p\diagRepulsiveLetter{T}_{1,1,2}(y-u,w-u,0) \\
  \times \diagRepulsiveLetter{P}_{1,0,\underline 1,\underline 1,0}(x-u,\ve[\iota]-u,\\
  \qquad\quad-u,v-u,w-u) \end{array}$ \\ \cline{1-4}
  \multicolumn{1}{|c|} {\multirow{3}{*}{$a\geq 2,b\geq 2$}} & $d_{\tilde\Ccal}(u,w)=1$ &    \picBTwoPrimeCaseTwoTwoTwo[0.7] & $\hspace{-5mm}  \begin{array}{c}
  \frac 1 p 2dD(w-u)\\  \times\diagRepulsiveLetter{T}_{1,1,\underline 1}(u-y,w-y,0)\\
  \times \diagRepulsiveLetter{P}_{\underline 1,0,1,\underline 1,0}(\ve[\iota],x,u,w,v)
  \end{array}$   \\ \cline{2-4}
  \multicolumn{1}{|c|} {} & $d_{\tilde\Ccal}(u,w)\geq 2$&     \picBTwoPrimeCaseTwoTwoThree[0.7] & $ \begin{array}{c}
   \diagRepulsiveLetter{B}_{1,1}(u-y,w-y)\\ \times
   \diagRepulsiveLetter{P}_{\underline 1,0,1,2,0}(\ve[\iota],x,u,w,v)
  \end{array}$  \\
 }

\fourcolomntablePer
{Diagram of the different cases of $\bar B^{{\sss (2)},\iota,a,b}(0,v,x,y)$ }
{PercBoundTableBNTTwoTwo}
{
    \multicolumn{1}{|c|} {\multirow{3}{*}{$\begin{array}{c} a\geq 2,b=0\\ \Rightarrow v=0 \end{array}$}} & $u=y$ & \picBbarPrimeTwoZeroOne[0.7] & $\hspace{-2mm}\begin{array}{c}\prob(y\dbc w) \\
    \times \diagRepulsiveLetter{S}_{\underline 1, 1,0,2}(\ve[\iota],-w,\\
    \qquad x-w,y-w)
    \end{array}$ \\ \cline{2-4}
  \multicolumn{1}{|c|} {} & $\begin{array}{c} d_{\tilde\Ccal}(u,y)\geq 1\\ d_{\tilde\Ccal}(w,u)=1\end{array}$ &  \picBbarPrimeTwoZeroTwo[0.7] & $ \hspace{-2mm} \begin{array}{c} \frac 1 p \diagRepulsiveLetter{T}_{1,1,\underline 1}(y-w,u-w,0)\\
  \times \diagRepulsiveLetter{P}_{0,1,\underline 1,\underline 1,1}(x,u,w,w+\ve[\iota],0)\end{array}$ \\ \cline{2-4}
  \multicolumn{1}{|c|} {} & $\begin{array}{c} d_{\tilde\Ccal}(u,y)\geq 1\\ d_{\tilde\Ccal}(w,u)\geq 2\end{array}$ &  \picBbarPrimeTwoZeroThree[0.7] & $\hspace{-2mm}\begin{array}{c} \diagRepulsiveLetter{T}_{1,1,2}(y-w,u-w,0)  \\ \times \diagRepulsiveLetter{S}_{\underline 1,1,0,1}(\ve[\iota],-w,x-w,\\
  \qquad\qquad\qquad u-w)  \end{array}$ \\ \cline{1-4}
    \multicolumn{1}{|c|} {\multirow{3}{*}{$\begin{array}{c} a\geq 2,b=1\end{array}$}} & $u=y$ & \picBbarPrimeTwoOneOne[0.7] & $\hspace{-2mm}
   \begin{array}{c} \frac 1 p \prob(y\dbc w)  \\
    \times  \diagRepulsiveLetter{P}_{\underline 1, 0,\underline 1,0,2}(\ve[\iota],v-w,-w,\\
    \qquad\qquad \quad x-w,y-w) \end{array} $ \\ \cline{2-4}
  \multicolumn{1}{|c|} {} & $\begin{array}{c} d_{\tilde\Ccal}(u,y)\geq 1\\ d_{\tilde\Ccal}(w,u)=1\end{array}$ &  \picBbarPrimeTwoOneThree[0.7] & $\hspace{-3mm} \begin{array}{c} \frac {2d D(v)} {p} \diagRepulsiveLetter{T}_{1,1,\underline 1}(y-w,u-w,0)\\
  \times \diagRepulsiveLetter{P}_{0,1,\underline 1,\underline 1,0}(x,u,w,w+\ve[\iota],v)\end{array}$ \\ \cline{2-4}
  \multicolumn{1}{|c|} {} & $\begin{array}{c} d_{\tilde\Ccal}(u,y)\geq 1\\ d_{\tilde\Ccal}(w,u)\geq 2\end{array}$ &  \picBbarPrimeTwoOneFour[0.7] & $\hspace{-3mm} \begin{array}{c} \frac{2d D(v)} p\diagRepulsiveLetter{T}_{1,1,2}(y-w,u-w,0)  \\ \times \diagRepulsiveLetter{P}_{\underline 1,0,\underline 1,0,1}(\ve[\iota],v-w,-w,\\
  \qquad\qquad x-w,u-w) \end{array}$ \\ \cline{1-4}
    \multicolumn{1}{|c|} {\multirow{3}{*}{$\begin{array}{c} a\geq 2,b\geq 2 \end{array}$}} & $u=y$ & \picBbarPrimeTwoTwoOne[0.7] & $\begin{array}{c}\prob(w\dbc y) \\
    \times \diagRepulsiveLetter{B}_{\underline 1,0}(\ve[\iota],v-w)  \\
    \times \diagRepulsiveLetter{B}_{0,2}(x,y) \end{array}$   \\ \cline{2-4}
    \multicolumn{1}{|c|} {} & $\begin{array}{c} d_{\tilde\Ccal}(u,y)\geq 1\\ d_{\tilde\Ccal}(w,u)=1\end{array}$ &  \picBbarPrimeTwoTwoThree[0.7] &.\vspace{-1cm} $ \begin{array}{c} \frac 1 p \diagRepulsiveLetter{T}_{1,1,\underline 1}(y-w,u-w,0)\\
    \times \diagRepulsiveLetter{P}_{0,1,\underline 1,\underline 1,0}(x,u,w,\\
    \qquad \qquad w+\ve[\iota],v)\end{array}$  \\ \cline{2-4}
        \multicolumn{1}{|c|} {} & $\begin{array}{c} d_{\tilde\Ccal}(u,y)\geq 1\\ d_{\tilde\Ccal}(w,u)\geq 2\end{array}$ &
        \picBbarPrimeTwoTwoFour[0.7] & $\begin{array}{c}\diagRepulsiveLetter{B}_{\underline 1,0}(\ve[\iota],v-w)
        \diagRepulsiveLetter{B}_{0,1}(x,u)  \\ \diagRepulsiveLetter{T}_{1,1,2}(y-w,u-w,0) \end{array}$  \\ }
\clearpage

\fourcolomntablePer{Diagram of the different cases of $\bar B^{{\sss (2)},\iota,a,b}(0,v,x,y)$(continued)}
{PercBoundTableBNTTwoOne}{
  \multicolumn{1}{|c|} {\multirow{1}{*}{$\begin{array}{c} a=1,b=0 \\ \Rightarrow y=u,v=w+\ve[\iota]\neq 0 \end{array}$}} &  &     \picBbarPrimeOneZeroTwo[0.7] &  $\hspace{-8mm} \begin{array}{c} \prob(y\dbc w) \\
  \times  \diagRepulsiveLetter{S}_{\underline 1,0,1,\underline 1}(x-y,-y,\\
  \qquad \qquad w+\ve[\iota]-y,w-y)\end{array}$  \\ \cline{1-4}
  \multicolumn{1}{|c|} {\multirow{1}{*}{$\begin{array}{c} a=1,b=1\\ \Rightarrow y=u \end{array}$}} &  & \picBbarPrimeOneOneTwo[0.7] &
  $\hspace{-8mm}\begin{array}{c} \frac 1 p \prob(y\dbc w) \\
  \times \diagRepulsiveLetter{P}_{\underline 1,0,\underline 1,0,\underline 1}(x-y, -y,\\
  \qquad v-y,w+\ve[\iota]-y,w-y)\end{array}$\\ \cline{1-4}
  \multicolumn{1}{|c|} {\multirow{1}{*}{$\begin{array}{c} a=1,b\geq2 \\ \Rightarrow y=u\end{array}$}} &  & \picBbarPrimeOneTwoTwo[0.7]
  & $\hspace{-2mm}\begin{array}{c} \prob(y\dbc w) \\\times \diagRepulsiveLetter{B}_{\underline 1,0}(\ve[\iota],v-w)\\
  \times\diagRepulsiveLetter{B}_{0,\underline 1}(x,y)\end{array}$  \\
}

%======================================================================
%======================================================================
%======================================================================
%======================================================================
\noindent
We define the diagrams $B^{{\sss (2)},\iota,a,b}, \bar B^{{\sss (2)},\iota,a,b}$ for the cases of $a,b$ not defined in the Tables \ref{PercBoundTableBNT}-\ref{PercBoundTableBNTTwoTwo} to be zero, i.e., for $b=0,1$ we let $B^{{\sss (2)},\iota,a,b}(0,v,x,y)=0$ and for $a=0$ we let $\bar B^{{\sss (2)},\iota,a,b}(0,v,x,y)=0$.\\
We have explained this in Section \ref{secBoundsPercEvents}. We define the double-open triangle $\bar A^{\iota,a,b}(0,v,x,y)$ to be
	\begin{eqnarray}
	\bar A^{\iota,a,0}(0,v,x,y)&=&A^{\iota,a,0}(0,v,x,y) \qquad\quad \text{ for }a=0,1,2,\\
	\bar A^{\iota,a,1}(0,v,x,y)&=&\frac {1}{p} A^{\iota,a,1}(0,v,x,y) \qquad \text{ for }a=0,1,2,\\
	\bar A^{\iota,0,2}(0,v,x,y)&=&\delta_{0,v}
	\diagRepulsiveLetter{T}^*_{1,\underline 1,0}(-y,\ve[\iota]-y,x-y),\\
	\bar A^{\iota,1,2}(0,v,x,y)&=&\frac 1 p
	\diagRepulsiveLetter{S}^*_{0,\underline 1,\underline 1,0}(v-y,-y,	\ve[\iota]-y,x-y),\\
	\bar A^{\iota,2,2}(0,v,x,y)&=& p\tau_{0,p}(x-\ve[\iota])\tau_{0,p}(v-y).
	\end{eqnarray}

%====================================================================
\subsection*{Building blocks with weight.}
To bound $\sum_x \|x\|_2^2 \Xi^{\ssc[N]}_z(x)$ we define weighted diagrams.
These are diagrams in which one line has the weight $\|x\|_2^2$.
In Table \ref{InformalDefinitonOfBlockDelta} we give as small overview of these diagrams.
We use the building blocks defined in the previous section to define the weighted diagrams. For $a,b\in\{0,1,2\}$, let
	\begin{align}
	H^{{\sss (1)},a,b}(u,v,x,y)=&\|u-x\|_2^2 \bar A^{a,b}(u,v,x,y), \\
	H^{{\sss (2)},\iota,a,b}(u,v,x,y)=&\|u-x\|_2^2 \bar A^{\iota,a,b,*}(u,v,x,y),\\
	H^{{\sss (3)},\iota,a,b}(u,v,x,y)=&\|v-y\|_2^2 \bar A^{\iota,a,b,*}(u,v,x,y),
	\end{align}
and
	\begin{align}
	\lbeq{DefBlockc1-Percolation}
	C^{{\sss (1)},\iota,\kappa,a,b}(0,v,x,y)=
	&\sum_{c=0}^2\sum_{w,u}B^{\iota,a,c}(0,v,w,u)\bar A^{\kappa,c,b}(u,w,y,x)\|w\|_2^2,\\
	\lbeq{DefBlockc1-Percolation-two}
	C^{{\sss (2)},\iota,\kappa,a,b}(0,v,x,y)=
	&\sum_{c=0}^2\sum_{w,u}B^{\iota,a,c}(v,0,u,w)\bar A^{\kappa,c,b}(w,u,x,y) \|u\|_2^2.
	\end{align}

%=================================================AdditionalBoundForLongVersion=================================================================
%==========================================================================================================================
\section{Additional details for the bounds on the NoBLE coefficients}
\label{Appendix-ToTheProofs}
In this section we give some additional details of the proof of the bounds on the NoBLE coefficients that we omit in the article version of this document.

\subsection{Additional details for the bounds on $\Xi^{\ssc[1]}$ and $\Xi^{\ssc[2]}$}
\label{Appendix-ToTheProofs-DoubleTriangle}
In this section, we show how we improve the bound on the weighted coefficients $\Xi^{\ssc[1]}$ and $\Xi^{\ssc[2]}$ by considering special cases in which the diagram consists of two connected triangles.\\
We aim to improve the bound on $\Xi^{\ssc[1]}$, by replacing the sum of terms that have a coefficient 2 in the first line on the right-hand side of \refeq{XiIota-Deltabounds-versions1}, i.e.,
%Our aim is to replace the terms by the sharper bound
	\begin{align}
	2\vec u^T  {\bf H}^{\sss (3)}\big(\vec P^{\sss \rm E}-\vec u\big)
	+2\vec u^T {\bf A}^\iota\vec h^{\sss \rm E}
	+2\vec h^{{\sss \rm S}}{\bf A}^\iota \vec u
	+2\big(\vec P^{{\sss \rm S}}-\vec u^T\big)  {\bf H}^{\sss (3)}\vec u,
	\end{align}
by a better bound. These terms correspond to the cases in which either the left or the right diagram are trivial, see Figure \ref{fig-Form-Xi1}. We use the notation used there. We discuss the case in which the right triangle is trivial, i.e., $t=z=x$ and consider the following four cases for this diagram:
a) $u=w\neq 0$, b) $w=0, u\neq w,$ c) $u,w$ are directly connected by a bond and d) the remaining cases.\\
{\bf Case a) $u=w\neq 0$.} We use $\|x\|_2^2=\|w\|_2^2+\|x-w\|_2^2+w^T (x-w)$ and spatial symmetry to obtain
	\begin{align}
	\sum_{w,\iota,x}\|x\|_2^2&
	\prob(0\dbc w)\diagRepulsiveLetter{T}_{\underline 1,1,1}(\ve[\iota],x-w,0)  \nnb
	=&\sum_{w,\iota,x}(\|w\|_2^2+\|x-w\|_2^2)
	\prob(0\dbc w)\diagRepulsiveLetter{T}_{\underline 1,1,1}(\ve[\iota],x-w,0)
	\nnb
	\leq& \diagRepulsiveLetter{H}^{\sss D}_{1}
	\sum_{x}\diagRepulsiveLetter{T}_{\underline 1,1,1}(\ve[\iota], x,0)
	+\diagRepulsiveLetter{H}_{2}(0) \sum_x \diagRepulsiveLetter{D}_{1,1}(x),
	\lbeq{XiIota-Delta-Step1}
	\end{align}
where we recall \eqref{Hi-defs} for the definitions of $\diagRepulsiveLetter{H}^{\sss D}_{1}$ and $\diagRepulsiveLetter{H}_{2}(x)$.\\
{\bf Case b) $w=0, u\neq w$.} We bound this contribution by
	\begin{align}
	\sum_{u,\iota,x}\|x\|_2^2
	\diagRepulsiveLetter{T}_{1,1,\underline 1}(x,u+\ve[\iota],u)  \prob(0\dbc u)
	\leq \sup_{x\neq 0}\diagRepulsiveLetter{H}_{2}(x) \sum_u \diagRepulsiveLetter{D}_{1,1}(u).
	\lbeq{XiIota-Delta-Step2}
	\end{align}
{\bf Case c) $u,w$ are directly connected.} We use $\|x\|_2^2\leq 2\|w\|_2^2+2\|x-w\|_2^2$ and rename $u=w+\ve[\kappa]$ to bound the diagram by
	\begin{align}
	p\sum_{v,\kappa,\iota,x}&\|x\|_2^2
	\diagRepulsiveLetter{B}_{1,1}(-w,\ve[\kappa])
	\diagRepulsiveLetter{T}_{1,1,\underline 1}(x-w,\ve[\kappa]+\ve[\iota],\ve[\kappa]) \nnb
	\leq& 4d p \diagRepulsiveLetter{H}_{1}(\ve[1])
	\sum_{\iota,x}\diagRepulsiveLetter{T}_{1,1,\underline 1}(x,\ve[1]+\ve[\iota],\ve[1])
	+4d p \diagRepulsiveLetter{H}_{2}(\ve[1])
	\sum_{x}\diagRepulsiveLetter{B}_{1,1}(x,\ve[1])
	\lbeq{XiIota-Delta-Step3}
	\end{align}
%\RvdH{This can be improved to}
%	\begin{align}
%	&p\sum_{v,\kappa,\iota,x}\|x\|_2^2
%	\diagRepulsiveLetter{B}_{1,1}(-w,\ve[\kappa])
%	\diagRepulsiveLetter{T}_{1,1,\underline 1}(x-w,\ve[\kappa]+\ve[\iota],\ve[\kappa]) \nnb
%	&\qquad= \frac{1}{2d}\sum_{v,\kappa,\iota,\rho, x}\|x\|_2^2
%	\diagRepulsiveLetter{B}_{1,1}(-w,\ve[\rho])
%	\diagRepulsiveLetter{T}_{1,1,\underline 1}(x-w,\ve[\kappa]+\ve[\iota],\ve[\kappa]) \nnb
%	&\qquad=\sum_{v,\kappa,\iota,\rho, x}(\|w\|_2^2 +\|x-w\|_2^2)
%	\diagRepulsiveLetter{B}_{1,1}(-w,\ve[\rho])
%	\diagRepulsiveLetter{T}_{1,1,\underline 1}(x-w,\ve[\kappa]+\ve[\iota],\ve[\kappa])
%	\lbeq{XiIota-Delta-Step3}
%	\end{align}
{\bf Case d) Remaining cases.} We bound the remaining cases by
	\begin{align}
	&2 \sum_{u,w,\iota,x}\|w\|_2^2
	\diagRepulsiveLetter{B}_{1,1}(-w,u-w)
	\diagRepulsiveLetter{S}_{\underline 1, 1,1,2}(\ve[\iota],x-u,w-u,0) \nnb
	&+2 \sum_{u,w,\iota,x}\|x-w\|_2^2
	\diagRepulsiveLetter{T}_{1,2,1}(w,u,0)
	\diagRepulsiveLetter{T}_{\underline 1,1,1}(\ve[\iota],x-u,w-u)  \nnb
	\leq& 2\sup_{x\neq 0}\diagRepulsiveLetter{H}_{1}(x)
	\sum_{\iota,x,y} \diagRepulsiveLetter{S}_{\underline 1, 1,1,2}(\ve[\iota],x,y,0)
	+2\sup_{x\neq 0}\diagRepulsiveLetter{H}_{2}(x)
	\sum_{x,y} \diagRepulsiveLetter{T}_{1,2, 1}(x,y,0).
	\lbeq{XiIota-Delta-Step4}
	\end{align}
We bound the diagram in which $u=w=0$, so that the left diagram is trivial, in the same way, with the exception that the special case $z=x$, alike \refeq{XiIota-Delta-Step2}, does not need to be considered, as we can choose $t=x$ in this case. The improved bound on the sum of terms that have a coefficient 2 in the first line on the right-hand side of \refeq{XiIota-Deltabounds-versions1} follows by summing all the bounds obtained above.\\

To be able to show the mean-field result in $d=11,12$, we extract the leading contribution of $\Xi^{\ssc[2]}$, see Figure
\ref{fig-Form-Xi2-leading}, and bound these manually in a similar way as for the diagrams discussed for $\Xi^{\ssc[1]}$ above. Here, the leading contribution arises when all {\em three} triangles that could be present in the diagram $\Xi^{\ssc[2]}$ are trivial.

\begin{figure}[h]
\begin{center}
\picXiNTwoMain[1]
\caption{Diagrammatic representations of main contribution to $\Xi_p^{\ssc[2]}(x)$.
The solid line is a connections in $\tilde \Ccal^{b_0}_{\sss 0}(0)$,
the dashed lines represent connections in $ \tilde\Ccal^{b_1}_1\subseteq \Zd\setminus\{\bb_0,\tb_1\}$, and the dotted line a connection in $ \Ccal_2(y+\ve[\kappa])\subseteq \Zd\setminus\{\bb_1\}$.}
\label{fig-Form-Xi2-leading}
\end{center}
\end{figure}

\noindent
For these bounds we defined an adaptation of $\diagRepulsiveLetter{H}_{n}(x)$, see \eqref{Hi-defs}:
\begin{align}
  	\diagRepulsiveLetter{H}'_{n}(x)&:=
	\max \big\{\sum_{e,y} \|y\|_2^2 \diagRepulsiveLetter{T}_{\underline{1},0,n}(e,y,x),
	\sum_{e,y} \|y\|_2^2 \diagRepulsiveLetter{T}_{0,\underline{1},n}(y-e,y,x),\\
	&\qquad\qquad\qquad\qquad
	\sum_{\iota,\kappa,y} \|y\|_2^2
	\diagRepulsiveLetter{S}_{\underline{1},1,n-1,\underline{1}}(e_{\iota},y,x+e_{\kappa},x)\big\}.\nn
	\label{Hidash-defs}
	\end{align}
In $\diagRepulsiveLetter{H}_{n}(x)$ the weight $\|y\|_2^2$ is \emph{carried} by a single connection/path, while the weight is in
 $\diagRepulsiveLetter{H}'_{n}(x)$ is along the combination of a pivotal edge and a connection.
At the end of this section we discuss how we bound $\diagRepulsiveLetter{H}_{n}(x)$ and $\diagRepulsiveLetter{H}'_{n}(x)$ numerically.\\
Now we bound the diagram pictured in Figure \ref{fig-Form-Xi2-leading}, by considering the five cases
a) $t=x$, b) $t=\tb_0$, c) $\bb_1=t$, d) $\bb_1$ and $t$ being directly connected by an bond and e) the remaining cases.\\
	
\noindent
{\bf Case a) $t=x$.} We bound this contribution by
	\begin{align}
	\sum_{\iota,\kappa,y, x}&
	\|x\|_2^2\diagRepulsiveLetter{T}_{0,\underline 1,1}(\ve[\iota]-x,-x,y-x)
	\diagRepulsiveLetter{T}_{1,1,\underline 1}(\ve[\kappa],x-y,0)  \nnb
	&\leq
	\sum_{\kappa,y}\diagRepulsiveLetter{T}_{\underline 1,1,1}(\ve[\kappa],y,0)   	
	\Big(\sup_{x}\diagRepulsiveLetter{H}'_{1}(x)\Big).
	\end{align}
{\bf Case b) $t=\bb_0$.} We bound this contribution by
	\begin{align}
	\sum_{\iota,\kappa,y, x}&
	\|x\|^2 \diagRepulsiveLetter{B}_{1,1}(y,\ve[\iota])
	\diagRepulsiveLetter{S}_{\underline{1},1,1,\underline 1}(\ve[\iota],x,y+\ve[\kappa],y)
	\leq \sum_{e,y}\diagRepulsiveLetter{B}_{1,1}(y,e)
	\Big(\sup_{x}\diagRepulsiveLetter{H}'_{2}(x)\Big).
	\end{align}
{\bf Case c) $\bb_1=y=t$.} We split again as in \refeq{XiIota-Delta-Step1} and use symmetry to obtain
	\begin{align}
	\sum_{\iota,\kappa,t, x}&\|x\|_2^2
	\diagRepulsiveLetter{T}_{\underline 1,1,1}(\ve[\iota],t,0)
	\diagRepulsiveLetter{T}_{\underline 1,1,1}(\ve[\kappa],x-t,0)  \\
	&=\sum_{\iota,\kappa,t, x}(\|t\|_2^2+\|t-x\|_2^2)
	\diagRepulsiveLetter{T}_{\underline 1,1,1}(\ve[\iota],t,0)
	\diagRepulsiveLetter{T}_{\underline 1,1,1}(\ve[\kappa],x-t,0)
	\leq 2\diagRepulsiveLetter{H}_{2}(0)
	\sum_{x,t}\diagRepulsiveLetter{T}_{\underline 1,1,1}(x,y,0).\nn
\end{align}
{\bf Case d) $\bb_1=y$ and $t$ are directly connected.} We bound as in \refeq{XiIota-Delta-Step3}
	\begin{align}
	&2dp\sum_{\iota,\kappa,y,t, x}
	2\left(\|y\|_2^2+\|x-y\|_2^2\right) D(y-t)
	\diagRepulsiveLetter{T}_{1,\underline 1,1}(-y,\ve[\iota]-y,t-y)
	\diagRepulsiveLetter{T}_{\underline 1,1, 1}(\ve[\kappa],x-y,t-y)  \nnb
	&\qquad=
	4d p \left( \diagRepulsiveLetter{H}'_{1}(\ve[1])  + \diagRepulsiveLetter{H}_{2}(\ve[1]) \right)
	\sum_{\iota,x}\diagRepulsiveLetter{T}_{\underline 1,1,1}(\ve[\iota],x,\ve[1]).
	\end{align}
{\bf Case e) $\bb_1=y\neq t$ that are not directly connected.} We proceed as in \refeq{XiIota-Delta-Step4} and bound the diagram by
	\begin{align}
	&2\sum_{\iota,\kappa,y,t, x}\|y\|_2^2
	\diagRepulsiveLetter{T}_{1,\underline 1,1}(-y,\ve[\iota]-y,t-y)
	\diagRepulsiveLetter{S}_{\underline 1,1,1,2}(\ve[\kappa],x-y,t-y,0)  \nnb
	&+2\sum_{\iota,\kappa,y,t, x}\|x-y\|_2^2
	\diagRepulsiveLetter{S}_{\underline 1,1,2,1}(\ve[\iota],t,y,0)
	\diagRepulsiveLetter{T}_{\underline 1,1,1}(\ve[\kappa], x-y,t-y)\nnb
	&\qquad=
	2\sup_{x\neq 0} \diagRepulsiveLetter{H}_{2}(x)
	\sum_{x,y,t}\diagRepulsiveLetter{S}_{\underline 1,1,1,2}(x,y,t,0)
	+2\sup_{x\neq 0} \diagRepulsiveLetter{H}'_{1}(x)
	\sum_{x,y,t}\diagRepulsiveLetter{S}_{\underline 1,1,2,1}(x,y,t,0).
	\end{align}
This completes the derivation of the bound for the special case of contribution to $\Xi^{\ssc[2]}$ that have a form as shown in Figure \ref{fig-Form-Xi2-leading}.

\paragraph{How to bound repulsive weighted diagrams.}
Here we explain how we can bound $\diagRepulsiveLetter{H}_{n}(x)$ and $\diagRepulsiveLetter{H}'_{n}(x)$. We start with $\diagRepulsiveLetter{H}_{n}(x)$.

Dropping the repulsiveness constraint present in $\diagRepulsiveLetter{H}_{n}(x)$ and using $\tau_n(x)\leq (2dp)^n (D\star \tau)(x)$, we obtain
	\begin{align}
 	\diagRepulsiveLetter{H}_{n}(x)\leq (2dp)^n\sum_{y} \tau(y)\|y\|_2^2 (D\star \tau)(x)
	\end{align}
for all three terms in \eqref{Hi-defs}. The bound on the bootstrap function $f_3(p)\leq \Gamma_3$ directly implies a numerical bound on the right-hand-side. In \cite[Section 5.3.3]{FitHof13b}, we discuss in detail how we improve this bound for $x=0$, by extracting short explicit contributions. For $x=e$ being a neighbor of the origin, we use that $\diagRepulsiveLetter{H}_{n}(e)=\tfrac 1 {2d} (D\star \diagRepulsiveLetter{H}_{n})(0)$ to gain an extra factor $D$.

In $\diagRepulsiveLetter{H}'_{n}(x)$, the weight $\|y\|^2$ is on a line that combines an edge $(0,e)$ and a path $e\conn y$. We first bound $\diagRepulsiveLetter{H}'_{n}(x)$ by
	\begin{align}
  	\lbeq{Bound-Hprime-step1}
 	\diagRepulsiveLetter{H}'_{n}(x)
	\leq\sum_{\iota, \rho,y} \|y\|_2^2 \diagRepulsiveLetter{B}_{\underline{1},1}(e_{\iota},y)\diagRepulsiveLetter{B}_{n-1,\underline{1}} (x-y-e_{\rho},x-y).
	\end{align}
Unfortunately, we can not bound this using $f_3(p)\leq\Gamma_3$ directly, as $\diagRepulsiveLetter{B}_{\underline{1},1}$ is not obviously bounded by $\tau$. To derive a bound we define the event
	\begin{align}
	E^\iota(y)=(\{(0,\ve[\iota]) \text{ is occ.}\}\cap \{ \ve[\iota]\conn y \text{ off } (0,\ve[\iota])\}
	\end{align}
and note that, by inclusion-exclusion,
	\begin{align}
	\prob(0\conn y) =\prob\Big(\bigcup_\iota  E^\iota(y) \Big)
	\geq \sum_\iota \prob (E^\iota )
	- \tfrac12 \sum_{\kappa\neq\iota}  \prob(E^\iota(y) \cap E^\kappa(y)).
  	\lbeq{Bound-Hprime-step2}
	\end{align}
For the special case $\diagRepulsiveLetter{B}_{\underline{1},1}$ this first connection is only a single bond, so that
	\begin{align}
	\prob (E^\iota )=\diagRepulsiveLetter{B}_{\underline{1},1}(\ve[\iota],y).
  	\lbeq{Bound-Hprime-step3}
	\end{align}
For the second term we note that $E^\iota(y) \cap E^\kappa(y)$ implies the existence of a point $w$ for which $E^\iota(w) \circ E^\kappa(w)\circ \{w\conn y\}$ occurs, so that
	\begin{align}
	\tfrac12 \sum_{\kappa,\iota}  \prob(E^\iota(y) \cap E^\kappa(y))&\leq
	\tfrac12 \sum_{\kappa, \iota} \sum_w
	\prob(E^\iota(w) \circ E^\kappa(w)\circ \{w\conn y\})\nnb
	&\leq
	\tfrac12 \sum_{\kappa,\iota}
	\sum_w \prob(E^\iota(w) \circ E^\kappa(w)) \tau(y-w).
	\lbeq{Bound-Hprime-step4}
	\end{align}
Combining \refeq{Bound-Hprime-step2}-\refeq{Bound-Hprime-step4} leads to
	\eqn{
	\sum_{\iota} \diagRepulsiveLetter{B}_{\underline{1},1}(e_\iota,y)
	\leq \tau(y)+\tfrac12 \sum_{\kappa,\iota} \sum_w
	\diagRepulsiveLetter{S}_{\underline{1},0,0,\underline{1}}(e_\iota,w,e_\kappa,0) \tau(y-w),
	}
which, when substituted into \refeq{Bound-Hprime-step1} leads to
	\begin{align}
 	\diagRepulsiveLetter{H}'_{n}(x)
	&\leq \sum_{\rho,y}\|y\|_2^2\tau(y)\diagRepulsiveLetter{B}_{n-1,\underline{1}} (x-y-e_{\rho},x-y)\\
	&\qquad+
 	\tfrac12  \sum_{\rho,\kappa,\iota}
	\sum_{y,w\neq 0} \|y\|_2^2
	\diagRepulsiveLetter{S}_{\underline{1},0,0,\underline{1}}(e_\iota,w,e_\kappa,0) \tau(y-w)
	\diagRepulsiveLetter{B}_{n-1,\underline{1}} (x-y-e_{\rho},x-y).\nn
  	\end{align}
The first part is bounded using $f_3(p)\leq \Gamma_3$. The second, numerically smaller, term can be bounded as
	\begin{align}
 	&\shift\shift \tfrac12  \sum_{\rho,\kappa,\iota}
	\sum_{y,w\neq 0} \|y\|_2^2\diagRepulsiveLetter{S}_{\underline{1},0,0,\underline{1}}(e_\iota,w,e_\kappa,0)
	\tau (y-w) \diagRepulsiveLetter{B}_{n-1,\underline{1}} (x-y-e_{\rho},x-y)\\
	&\leq  \tfrac12  \sum_{\kappa,\iota} \sum_{y,w\neq 0} 2(\|w\|_2^2+\|y-w\|_2^2)
	\diagRepulsiveLetter{S}_{\underline{1},0,0,\underline{1}}(e_\iota,w,e_\kappa,0)
	\tau(y-w)\diagRepulsiveLetter{B}_{n-1,\underline{1}} (x-y-e_{\rho},x-y)\nnb
	&\leq \diagRepulsiveLetter{H}'_{1}(0)
	\sup_{w'}\sum_{y,\rho} \tau_{0}(y-w')\diagRepulsiveLetter{B}_{n-1,\underline{1}} (x-y-e_{\rho},x-y)\nnb
	&\qquad\qquad
	+\sum_{\iota,\kappa,w} \diagRepulsiveLetter{S}_{\underline{1},0,0,\underline{1}}(e_\iota,w,e_\kappa,0)
	\sup_{w'}  \sum_{t,\rho} \|y-w'\|_2^2\tau_{0}(y-w')\diagRepulsiveLetter{B}_{n-1,\underline{1}} (x-y-e_{\rho},x-y).\nn
	\end{align}
The last term can be numerically bounded using the usual techniques. For $x=0$ and $l=1$, the above gives a linear relation, where the coefficient of $\diagRepulsiveLetter{H}'_{1}(0)$ on the right hand side is strictly smaller than one. Thus, this gives a bound on $\diagRepulsiveLetter{H}'_{1}(0)$, which can then be used to obtain bounds for general $x\neq 0$.

\subsection{Additions to the proof of Lemma \ref{lemmapercboundLowerBounds}}
\label{Appendix-ToTheProofs-LowerBound}
Here we prove lower bounds on $\Pi^{\ssc[1],\iota,\kappa}_p(x)$ for $x\sim 0$ and $x\sim e_{\iota}$. We do this by explicitly identifying contributions involving a small number of bonds, and bounding these contributions from below.
A special role is played by the $E'(x,y;A)$ event, which requires the last sausage from $x$ to $y$ to be cut through by $A$ (in our examples $A$ will be given by $\tilde{\Ccal}_{\sss 0}$), as well as the statement that there is no previous pivotal $b'$ for $x\conn y$ such that $x\ct{A}\bb'$ occurs. We refer to the latter as the `no previous pivotal' requirement. Further, in $\Pi^{\ssc[1],\iota,\kappa}_p(x)$, the indicator that $x+e_{\kappa}\not\in \tilde{\Ccal}_{\sss 1}$ appears, which we will refer to as the `no backtracking' requirement. We will deal with the `no previous pivotal' and `no backtracking' requirements by using inclusion-exclusion and explicitly bounding the contributions where these requirements are violated from above.

We start with the definition of $\Pi^{\ssc[1],1,\kappa}(x)$ in \refeq{tau-tauj-LEC-ident2}, which contains two terms.
For the first term in \refeq{tau-tauj-LEC-ident2}, we restrict to the case that the first pivotal bond starts at $\ve[1]=\bb_0$. The second term is simplified using the notation introduced in \refeq{Abbreviation-for-extra-bond}.
This yields
	\begin{align}
  	\Pi^{\ssc[1],1,\kappa}(x)\geq &\ p^2\sum_{\iota}
  	\expec^{(0,\ve[1])}_{\sss 0 }\Big(\indic{0 \conn \ve[1]} \indicwo{{\mathcal T}_\iota}\prob^{\ve[1]}_{\sss 1}
   	\left(E'(\ve[1]+\ve[\iota],x;\tilde{\Ccal}_{\sss 0})\cap\{x+\ve[\kappa]\nin\tilde \Ccal_{\sss 1}\}\right)\Big)\nnb
   	&+  p^2\expect_{\sss 0}^{(0,\ve[1])} \left[\indicwo{\ve[1]\nin\tilde\Ccal_0 }
	\prob^{0}_{\sss 1}\big(E'(\ve[1],x;\tilde \Ccal^{(0,\ve[1])}_{\sss 0}(0))
	\cap \{x+\ve[\kappa]\nin \tilde \Ccal_{\sss 1}\}\big)\right].
   	\lbeq{lemmaLowerbound-3-1}
	\end{align}
Let us first discuss the second term as it is simpler. We consider two cases for $x=\ve[2]$ and $x=\ve[2]+\ve[1]$. For
$x=\ve[2]$ and $\kappa\neq -1,-2$, we compute
	\begin{align}
 	&\shift\shift p^2\expect_{\sss 0}^{(0,\ve[1])}
	\left[\indic{\ve[1]\nin\tilde\Ccal_0}
	\prob^{0}_{\sss 1}\big(E'(\ve[1],\ve[2];\tilde \Ccal^{(0,\ve[1])}_{\sss 0}(0))
	\cap \{\ve[2]+\ve[\kappa]\nin \tilde \Ccal_{\sss 1}\}\big)\right]
	\nnb
 	\geq& p^2 \expect_{\sss 0}^{(0,\ve[1])}
	\left[\indic{\ve[1],\ve[1]+\ve[2]\nin \tilde \Ccal^{(0,\ve[1])}_{\sss 0}(0)}\indic{(0,\ve[2]) \text{ is occ.}}
	\prob_{\sss 1}^{0}\big((\ve[1],\ve[1]+\ve[2]),(\ve[1]+\ve[2],\ve[2])
	\text{ are occ.}, \ve[2]+\ve[\kappa]\nin\tilde\Ccal_1\big)\right].\nn
	\end{align}
Next, we use inclusion-exclusion on the event $\ve[2]+\ve[\kappa]\nin\tilde\Ccal_1$, as we did in \refeq{Lower-Bounds-tau3-creation},
and bound first term as described in \refeq{Lower-Bounds-not-connected-paths} as
	\begin{align}
  	& p^2 \prob_{\sss 0}^{(0,\ve[1])} \big(\ve[1],\ve[1]+\ve[2]\nin \tilde \Ccal^{(0,\ve[1])}_{\sss 0}(0),
	(0,\ve[2]) \text{ is occ.}\big)
	\prob_{\sss 1}^{0}\big((\ve[1],\ve[1]+\ve[2]),(\ve[1]+\ve[2],\ve[2]) \text{ are occ.}\big)\nnb
  	&-  p^2 \expec_{\sss 0}^{(0,\ve[1])} \Big[\indic{\ve[1],\ve[1]+\ve[2]\nin \tilde \Ccal^{(0,\ve[1])}_{\sss 0}(0)}
	\indic{(0,\ve[2]) \text{ is occ.}} \prob_{\sss 1}^{0}
	\big((\ve[1],\ve[1]+\ve[2]),(\ve[1]+\ve[2],\ve[2])
	\text{ are occ.},\ve[2]+\ve[\kappa]\in\tilde\Ccal_1\big)\Big]\nnb
  	\geq & p^5 \left(1-\prob_{\sss 0}^{(0,\ve[1])} ( \ve[1]\in \Ccal(0))
	-\prob_{\sss 0}^{\{\ve[1],\ve[2]\}}( \ve[1]+\ve[2]\in \Ccal(0))
	-\prob_{\sss 0}^{\{0,\ve[1]+\ve[2]\}} ( \ve[1]\in \Ccal(\ve[2])) \right)\nnb
 	& -p^5\prob_{\sss 0} (\ve[1]+\ve[2]\in \Ccal(\ve[2]) ) -p^5 \prob_{\sss 0}^{(\ve[2],\ve[2]+\ve[\kappa])}
	\left(\ve[2]+\ve[\kappa]\in\Ccal(\ve[2])\mid (\ve[1],\ve[1]+\ve[2]),(\ve[1]+\ve[2],\ve[2]) \text{ are occ.}\right)\nnb
 	\geq & p^5 (1-p-2\tau_{3,p}(\ve[1])-2 \tau_{4,p}^1(\ve[1]+\ve[2]))
	-p^5(\tau_{3,p}(\ve[1])+\tau_{2,p}^1(\ve[1]+\ve[\kappa])+\tau_{3,p}^1(\ve[1]+\ve[2]+\ve[\kappa])).\nn
  	\end{align}
 In the same way we obtain, for $x=\ve[1]+\ve[2]$,
	\begin{align}
 	&\shift\shift p^2\expect_{\sss 0}^{(0,\ve[1])} \left[\indic{\ve[1]\nin\tilde\Ccal_0}
	\prob^{0}_{\sss 1}(E'(\ve[1],\ve[2];\tilde \Ccal^{(0,\ve[1])}_{\sss 0}(0))
	\cap \{\ve[1]+\ve[2]+\ve[\kappa]\nin\tilde\Ccal_{\sss 1}(\ve[2])\})\right]\nnb
 	\geq & p^2 \expect_{\sss 0}^{(0,\ve[1])} \left[\indic{\ve[1]\nin\tilde\Ccal_0}
 	\indic{(0,\ve[2]),(\ve[2],\ve[1]+\ve[2]) \text{ occ.}} \prob_{\sss 1}^{0}
	((\ve[1],\ve[1]+\ve[2])) \text{ occ.}, \ve[1]+\ve[2]+\ve[\kappa]\nin\tilde \Ccal_1(e_1))
	\right]\nnb
 	\geq & p^5 (1-p-2\tau_{3,p}(\ve[1])-\tau_{4,p}^{\sss 1}(\ve[1]+\ve[2])
  	-\tau_{3,p}(\ve[1])-\tau_{2,p}^{\sss 1}(\ve[2]+\ve[\kappa])),
  	\lbeq{lemmaLowerbound-3-3}
	\end{align}
and note that $\kappa=-1$ and $\kappa=-2$ do not contribute.
Using symmetry, we conclude that
	\begin{align}
	\shift\shift\shift\sum_{\kappa,\rho}&\left[\Pi^{\ssc[1],1,\kappa}(\ve[\rho])+ \Pi^{\ssc[1],1,\kappa}(\ve[1]+\ve[\rho])\right]\nnb
	&\geq (2d-1)(2d-2) p^5(1-p-3\tau_{3,p}(\ve[1])-\theta_2-\theta_4)-(2d-2)p^5(\theta_4 +\vartheta),
  	\lbeq{lemmaLowerbound-3-4}
  	\end{align}
with
	\begin{align}
	\theta_2&=\max_{\iota=1,2}\tau_{2,p}^{\sss 1}(\ve[1]+\ve[\iota]),\quad
	\theta_4=\max_{\iota=1,2}\tau_{4,p}^{\sss 1}(\ve[1]+\ve[\iota])\\
	\vartheta&=\frac {d^2}{(d-1)(d-2)}(D^{\star 3}\star \tau_{5,p})(0).
	\end{align}

Now we bound the first term in  \refeq{lemmaLowerbound-3-1}. For a lower bound we restrict to the case where $0\conn \ve[1]$ is realised via the path $\gamma_\rho$, $b_0=(\ve[1],\ve[1]+\ve[\iota])$, and either
\begin{enumerate}[(i.)]
\item $x=\ve[1]+\ve[\rho]$, which is connected to $\tb=\ve[1]+\ve[\iota]$ in $\tilde \Ccal_1$ via the two-bond path
\begin{align}
\gamma'_{\iota,\rho}=(\ve[1]+\ve[\iota],\ve[1]+\ve[\iota]+\ve[\rho],\ve[1]+\ve[\rho]),
\end{align}
\item $x=\ve[1]+\ve[\rho]+\ve[\iota]$, the bond $(\ve[1]+\ve[\rho],x)$ is occupied in $\expec^{(0,\ve[1])}_{\sss 0 }$ and
 $(\ve[1]+\ve[\iota],x)$ is occupied in $\expec^{(0,\ve[1])}_{\sss 1 }$,
\item $x=\ve[1]+\ve[t]$, the bond $(\ve[1],\ve[1]+\ve[t])$ is occupied in $\expec^{(0,\ve[1])}_{\sss 0 }$ and $x$ is connected to $\tb=\ve[1]+\ve[\iota]$ in $\tilde \Ccal_1$ via the two-bond path $\gamma'_{\iota,t}$.
\end{enumerate}
We note that only $\rho,\iota,t$, with $|\rho|\neq 1$, $\iota\neq \rho,-\rho,-1$ and $t\neq \rho,1,-1$ contribute.
For the explanation below we fix $\rho, \iota$ and $t$, and only sum later.\\
{\bf Case (i.).} We bound
\begin{align}
&\expec^{(0,\ve[1])}_{\sss 0 }\Big(\indic{0 \conn \ve[1]} \indicwo{\mathcal{T}_\iota}\prob^{\ve[1]}_{\sss 1}
   \big(E'(\ve[1]+\ve[\iota],\ve[1]+\ve[\rho];\tilde{\Ccal}_{\sss 0})  \cap \{\ve[1]+\ve[\rho]+\ve[\kappa]\nin\tilde \Ccal_1\}\big)\Big)\nnb
&\geq  \sum_{\iota,t}\expec^{(0,\ve[1])}_{\sss 0 }\Big[\indicwo{\gamma_\rho\text{ occ.}}\Big(\prod_{\rho'\neq \rho}\indicwo{\gamma_{\rho'}\text{ vac.}}\Big)
 \indic{\ve[1]+\ve[\iota],\ve[1]+\ve[\iota]+\ve[\rho] \nin \tilde \Ccal_{\sss 0}}
% \nnb &\qquad\qquad
\prob^{\ve[1]}_{\sss 1}  \big(\gamma'_{\iota,\rho}\text{ occ.}, \ve[1]+\ve[\rho]+\ve[\kappa]\nin\tilde \Ccal_1\big)\Big] .
\end{align}
We bound the probabilities of this in the way explained in \refeq{Lower-Bounds-tau3-creation}, \refeq{Lower-Bounds-not-connected-paths}, as
\begin{align}
\prob^{\ve[1]}_{\sss 1}
\big(\gamma'_{\iota,\rho}\text{ occ.},\ve[1]+\ve[\rho]+\ve[\kappa]\nin\tilde \Ccal_1\big)
&\geq p^2(1-\tau_{3,p}(\ve[1])-\tau_{2,p}^{\sss \rho}(\ve[\rho]+\ve[\kappa])-\tau_{5,p}(\ve[\rho]+\ve[\kappa]-\ve[\iota])),
\end{align}
and
\begin{align}
\expec^{(0,\ve[1])}_{\sss 0 }\Big[\indicwo{\gamma_\rho\text{ occ.}}&
\Big(\prod_{\rho'\neq \rho}\indicwo{\gamma_{\rho'}\text{ vac.}}\Big)
 \indic{\ve[1]+\ve[\iota],\ve[1]+\ve[\iota]+\ve[\rho] \nin \tilde \Ccal_{\sss 0}}\Big]\nnb
\geq& p^3 (1-p^3)^{2d-3}(1-p-2\tau_{3,p}(\ve[1])-2\tau_{4,p}^{\sss 1}(\ve[1]+\ve[2])
-2\tau_{2,p}^{\sss \iota}(\ve[1]+\ve[\iota]))\nnb
&-2p^3 (1-p^3)^{2d-3}(p^3+\tau_{5,p}^{\sss 1}(\ve[1]+\ve[\rho]+\ve[\iota])).
\end{align}
This means that
\begin{align}
   \sum_{\kappa,\rho}  \Pi^{\ssc[1],1,\kappa}(\ve[1]+\ve[\rho])\geq &\
(2d-2)^2(2d-3)p^7\left( 1-p-2p^2-2p^3-2\tau_{3,p}(\ve[1]) -4\theta_4-2\vartheta\right)\nnb
&\ \times \left( 1-\tau_{3,p}(\ve[1]) -\theta_2-\vartheta\right).
 \lbeq{lemmaLowerbound-3-5}
\end{align}
{\bf Case (ii.).} We consider $x=\ve[1]+\ve[\iota]+\ve[\rho]$ is a neighbor of $\ve[1]+\ve[\iota]$, and note that since $\ve[1]+\ve[\iota]\not\in \tilde{\Ccal}_{\sss 0}$ the `no previous pivotal' requirement in $E'(\ve[1]+\ve[\iota],x;\tilde{\Ccal}_{\sss 0})$ is automatically satisfied. Therefore,
	\begin{align}
	\lbeq{one-bond-no-previous}
	&\shift\shift\expec^{(0,\ve[1])}_{\sss 0 }\Big(\indic{0 \conn \ve[1]}
	\indicwo{\mathcal{T}_\iota} \prob^{\ve[1]}_{\sss 1}
   	\big(E'(\ve[1]+\ve[\iota],x;\tilde{\Ccal}_{\sss 0})  \cap \{x+\ve[\kappa]\nin\tilde\Ccal_{\sss 1}\}\big)\Big)\nnb
	\geq&\   \expec^{(0,\ve[1])}_{\sss 0}
	\Big(\indic{\gamma_\rho,(\ve[1]+\ve[\rho],x)\text{ occ.}}
	\indic{\ve[1]+\ve[\iota]\nin \tilde \Ccal_{\sss 0}}
	\prob^{\ve[1]}_{\sss 1} \big((\ve[1]+\ve[\iota],x)\text{ occ.}, x+\ve[\kappa]\nin \tilde \Ccal_{\sss 1}\big)\Big)  \nnb
	\geq&\ p^5(1-p-2\tau_{3,p}(\ve[1])-\tau_{2,p}^{\sss 1}(\ve[1]+\ve[\iota])
	-\tau_{4,p}^{\sss \rho}(\ve[\rho]+\ve[\iota])
	-\tau^{\sss 1}_{5,p}(\ve[1]+\ve[\iota]-\ve[\rho]))\nnb
	&\quad\times(1-\tau_{3,p}(\ve[1])-\tau_{2,p}^{\iota}(\ve[\iota]+\ve[\kappa])),
	\end{align}
and conclude
	\begin{align}
 	\sum_{\kappa,\rho,\iota} \Pi^{\ssc[1],1,\kappa}(\ve[1]+\ve[\iota]+\ve[\rho])
  	\geq&(2d-2)^2(2d-3)p^7(1-p^3)^{2d-3}(1-p-p^2-2\tau_{3,p}(\ve[1])-2\theta_4-\vartheta)\nnb
	&\ \times  (1-\tau_{3,p}(\ve[1])-\theta_2).
 	\lbeq{lemmaLowerbound-3-6}
	\end{align}
{\bf Case (iii.).} We consider $x=\ve[1]+\ve[t]$ and use the bound
	\begin{align}
	&\shift\shift\expec^{(0,\ve[1])}_{\sss 0 }\Big(\indic{0 \conn \ve[1]}
	\indicwo{\mathcal{T}_\iota}\prob^{\ve[1]}_{\sss 1}
   	\big(E'(\ve[1]+\ve[\iota],\ve[1]+\ve[t];\tilde{\Ccal}_{\sss 0})
	\cap \{\ve[1]+\ve[t]+\ve[\kappa]\nin\tilde \Ccal_{\sss 1}\}\big)\Big)\\
	\geq&  \expec^{(0,\ve[1])}_{\sss 0 }\Big[\indicwo{\gamma_\rho,(\ve[1],\ve[1]+\ve[t])\text{ occ.}}
	\Big(\prod_{\rho'\neq \rho}\indicwo{\gamma_{\rho'}\text{ vac.}}\Big)
 	\indic{\ve[1]+\ve[\iota],\ve[\iota]+\ve[t]+\ve[\rho] \nin \tilde \Ccal_{\sss 0}}
 	%\nnb &\qquad\qquad
 	\prob^{\ve[1]}_{\sss 1}   \big(\gamma'_{\iota,t}\text{ occ.},
 	\ve[1]+\ve[t]+\ve[\kappa]\nin\tilde \Ccal_{\sss 1}\big)\Big]. \nn
	\end{align}
We bound the probabilities of this as explained in \refeq{Lower-Bounds-tau3-creation}, \refeq{Lower-Bounds-not-connected-paths}, to obtain
	\begin{align}
	\prob^{\ve[1]}_{\sss 1}   \big(\gamma'_{\iota,t}\text{ occ.},
 	\ve[1]+\ve[t]+\ve[\kappa]\nin\tilde \Ccal_{\sss 1}\big)
	&\geq p^2(1-\tau_{3,p}(\ve[1])-\tau_{2,p}^{\sss t}(\ve[t]+\ve[\kappa])
	-p^3-\tau_{5,p}(\ve[\iota]+\ve[\kappa]-\ve[t])),\nn
	\end{align}
and
	\begin{align}
	&\expec^{(0,\ve[1])}_{\sss 0 }\Big[\indicwo{\gamma_\rho,(\ve[1],\ve[1]+\ve[\rho])\text{ occ.}}
	\Big(\prod_{\rho'\neq \rho}\indicwo{\gamma_{\rho'}\text{ vac.}}\Big)
 	\indic{\ve[1]+\ve[\iota],\ve[\iota]+\ve[t]+\ve[\rho] \nin \tilde \Ccal_{\sss 0}}\Big]  \nnb
 	&\qquad \geq p^4 (1-p^3)^{2d-3}(1-p-2\tau_{3,p}(\ve[1])-2\tau_{4,p}^{\sss 1}(\ve[1]+\ve[2])
	-\tau_{2,p}^{\sss \iota}(\ve[t]+\ve[\iota])-\tau_{2,p}^{\sss 1}(\ve[1]+\ve[t]))\nnb
	&\qquad \quad-p^4 (1-p^3)^{2d-3}(3\tau_{5,p}^{\sss 1}(\ve[1]+\ve[t]+\ve[\iota])+
	\tau_{4,p}^{\sss 1}	(\ve[1]+\ve[\rho]+\ve[t]+\ve[\iota])).
	\end{align}
This implies that
	\begin{align}
   	\lbeq{lemmaLowerbound-3-7}
   	\sum_{\kappa,\rho}  \Pi^{\ssc[1],1,\kappa}(\ve[1]+\ve[\rho])
	\geq &(2d-2)^3(2d-3)p^8(1-p^3)^{2d-3}
	\left( 1-2\tau_{3,p}(\ve[1]) -\theta_2-\vartheta\right)\\
	&\ \times \left( 1-p-2p^2-2\tau_{3,p}(\ve[1]) -4\theta_4
	-3\vartheta-(2dp)^4 (D^{\star 4}\star \tau_{4,p})(0)\right).\nn
	\end{align}
We add the bounds \refeq{lemmaLowerbound-3-4}, \refeq{lemmaLowerbound-3-5}, \refeq{lemmaLowerbound-3-6}, \refeq{lemmaLowerbound-3-7} to obtain the bound stated in \refeq{lemmaLowerbound-3}.

\subsection{Proof of Lemma \ref{lemmapercboundXi0minus1}}
\label{Appendix-ToTheProofs-Difference}
We begin the proof of Lemma \ref{lemmapercboundXi0minus1}
by noting that \refeq{Differencebound-1} holds trivially as both terms are zero by definition, see \refeq{Def-Xi-Split} and \refeq{Def-XiOne-Split}.
We prove the other bounds stated in Lemma \ref{lemmapercboundXi0minus1} by deriving upper and lower bounds on the coefficients.
\paragraph{Upper bounds.}
We begin with the upper bounds, as they are simpler.
By definition,
  	\begin{align}
	\Xi^{\ssc[0]}_{\alpha,p}(\ve[1])&=\prob_p(\{0\dbc \ve[1]\}\cap\{(0,\ve[1])\text{ is occ.}\} )\nnb
	&=p \prob^{(0,\ve[1])}(0\conn \ve[1])=p\tau_{3,p}(\ve[1])\leq (2d-2)p^4+p\tau_{5,p}(\ve[1]).
	\lbeq{proof-Difference-Xi0Upper}
	\end{align}
For $\Psi^{\ssc[0],\kappa}_{\alpha,{\sss II},p}$, defined in \refeq{Def-Psi-Split-II}, we first note that $\Psi^{\ssc[0],-1}_{\alpha,{\sss II},p}(\ve[1])=0$ and then compute that
	\begin{align}
	\sum_{\kappa}\Psi^{\ssc[0],\kappa}_{\alpha,{\sss II},p}(\ve[1])&\leq
	\frac p \aap \sum_{\kappa\neq -1}\prob_p(\{0\connLe{\underline 3} \ve[1]\} \cap\ \mathcal{T}_\kappa )+ 	
	\prob_p(\{0\connLe{5} \ve[1]\} \cap\ \mathcal{T}_\kappa )\nnb
	&\stackrel{\refeq{Harris-bd}} \leq
	(2d-1)(2d-2) \frac {p^5}{\aap} +(2d-1)p\tau_{5,p}(\ve[1]).
	\lbeq{proof-Difference-PsiII0Upper}
	\end{align}
For $\Psi^{\ssc[0],1}_{\alpha,{\sss I},p}(\ve[1]+\ve[2])$, we compute
	\begin{align}
	\Psi^{\ssc[0],1}_{\alpha,{\sss I},p}(\ve[1]+\ve[2])=&\frac {p}{\aap}
	\prob_p(\{0\dbc \ve[1]+\ve[2]\}\cap\{0\connLe{\underline 2}
	\ve[1]+\ve[2]\}\cap\{(2\ve[1]+\ve[2])\nin\tilde \Ccal^{(\ve[1]+\ve[2],2\ve[1]+\ve[2])}(0)\})\nnb
	\leq &\ \frac {p}{\aap} \prob_p (\{ (0,\ve[1]),(\ve[1],\ve[1]+\ve[2]),(0,\ve[2] ),(\ve[2],\ve[1]+\ve[2])
	\text{ are occ.}\} )\nnb
	&+2\frac {p}{\aap} \prob_p (\{ (0,\ve[1]),(\ve[1],\ve[1]+\ve[2]) \text{ are occ.}\}
	\circ \{0\connLe{4} \ve[1]+\ve[2]\} )\nnb
	\leq &\ \frac {p^5}{\aap}+ \frac {2p^3}{\aap} \tau_{4,p}^{\sss 1 }(\ve[1]+\ve[2])+\frac{2p^3}\aap\tau_{3,p}(\ve[1])^2.
	\lbeq{proof-Difference-PsiI0Upper-temp}
	\end{align}
The factor $2$ of the second term is present as there exists two $2$-step paths from $0$ to $\ve[1]+\ve[2]$ and the $\tau_{3,p}(\ve[1])^2$ arises due to paths that use $\ve[1]$, but not the bonds $ (0,\ve[1]),(\ve[1],\ve[1]+\ve[2])$.
By symmetry, we can use this bound for $\Psi^{\ssc[0],1}_{\alpha,{\sss I},p}(\ve[1]+\ve[\iota])$ for all $\iota$ with $|\iota|\neq 1$. For $\iota=-1$, we note that $\Psi^{\ssc[0],\kappa}_{\alpha,{\sss I},p}(0)=0$, while, for $\iota=1$ we do not extract any special contributions as there exists only one $2$-step path from $0$ to $2\ve[1]$, and conclude that
	\begin{align}
	\Psi^{\ssc[0],1}_{\alpha,{\sss I},p}(2\ve[1])\leq
	&\frac {p^3}{\aap}\left(\tau_{4,p}^{\sss 1 }(2\ve[1])+\tau_{3,p}(\ve[1])^2\right).
	\end{align}
The sum over $\iota$ leads to
	\begin{align}
	\lbeq{proof-Difference-PsiI0Upper}
	\sum_\iota \Psi^{\ssc[0],1}_{\alpha,{\sss I},p}(\ve[1]+\ve[\iota]) \leq&
	(2d-2)\frac {p^5}{\aap}\\
	&+ \frac {p^3}{\aap} \left( 2(2d-2)\tau_{4,p}^{\sss 1 }(\ve[1]+\ve[2])
	+\tau_{4,p}^{\sss 1 }(2\ve[1])+(4d-3)\tau_{3,p}(\ve[1])^2\right),\nn
	\end{align}
where our numerical analysis shows that the first term constitutes around $90\%$ of the numerical bound.
This completes the upper bound for $N=0$. We bound $\Xi^{\ssc[1]}_{\alpha,p}(\ve[1]) $ as
	\begin{align}
    	\Xi^{\ssc[1]}_{\alpha,p}(\ve[1]) &= p\sum_{e} \expec_{\sss 0}
	\left(\indic{e\nin \tilde \Ccal^{(0,e)}_{\sss 0}(0) }  \indic{(0,\ve[1])\text{ is occ.}}
	\prob_{\sss 1}^{0} \big(E'(e,\ve[1];\tilde \Ccal^{(0,e)}_{\sss 0}(0))\big)\right)\nnb
	&\leq p^2 \sum_{\iota\neq 1} \tau_{2,p}^{-\iota}(\ve[1]-\ve[\iota])
	\leq  p^2 \left( (2d-2)\tau^{\sss 1}_{2,p}(\ve[1]+\ve[2])+\tau^{\sss 1}_{2,p}(2\ve[1])\right).
	\lbeq{proof-Difference-Xi1Upper}
	\end{align}
To create this bound we first drop the constraint $e\nin \tilde \Ccal^{(0,e)}_{\sss 0}(0)$. Then, we bound $E'(e,\ve[1];\tilde \Ccal^{(0,e)}_{\sss 0}(0))$ by the event  $\{\ve[\iota]\conn\ve[1]\text { off 0}\}$, which has probability $\tau_{2,p}^{-\iota}(\ve[1]-\ve[\iota])$.  In the last step, we have used the spatial symmetry.
For $\Psi^{\ssc[1],\kappa}_{\alpha,{\sss II},p}(x)$, we use an explicit bound for
contributions in which the loop present consists of four bonds.
Applying this to $x=e_1$ and bounding the contributions where the loop consists of more than four step as in
\refeq{proof-Difference-Xi1Upper}, we obtain
	\begin{align}
	\sum_{\kappa} \Psi^{\ssc[1],\kappa}_{\alpha,{\sss II},p}(\ve[1])
	\leq (2d-1)\frac { p^3}{\aap} \left((2d-2)p^2+ (2d-2)\tau^{\sss 1}_{4,p}(\ve[1]+\ve[2])
	+\tau^{\sss 1}_{4,p}(2\ve[1])\right),
	\lbeq{proof-Difference-PsiII1Upper}
	\end{align}
where the first contribution is due to loops of 4 bonds, and the others due to longer loops.

For the last remaining upper bound needed to prove Lemma \ref{lemmapercboundXi0minus1}, we first simplify the representation of $\Psi^{\ssc[1],\kappa}_{\alpha,{\sss I},p}(x)$ in \refeq{Def-PsiOne-SplitOne} for the special case $\kappa=1$ and $\|x-\ve[1]\|_2=1$ as
	\begin{align}	
	\Psi^{\ssc[1],1}_{\alpha,{\sss I},p}(x)=
	\frac {p^2}  {\aap} 	\sum_{ e } \expec_{\sss 0} \big( &\indic{e\nin \tilde \Ccal_0}\indic{x\in \tilde \Ccal_0}
	    	\prob_{\sss 1}^{0} \big(E'(e,x;\tilde \Ccal_0)\cap\{\ve[1]+x \nin \tilde \Ccal^{(x,\ve[1]+x)}_1\}\cap\{x\sim e,
	(x,e) \text{ occ.}\} \big)\big),
	\end{align}
with $\tilde \Ccal_{\sss 0}=\tilde \Ccal^{(0,e)}_{\sss 0}(0)$ and $\tilde \Ccal_{\sss 1}=\tilde \Ccal^{(x,x+\ve[1])}_{\sss 1}(e)$. Thus, $x$ and $e$ are of the form $x=\ve[1]+\ve[\iota], e\in\{\ve[1],\ve[\iota]\}$,  with $\iota\neq -1$ as $0\nin \tilde \Ccal_1$ by definition of $\expec_{\sss 1}^{0}$.
Since, $\tb=e\nin \tilde \Ccal_0$ is directly connected to $x=\ve[1]+\ve[\iota]$ the event $E'(e,x;\tilde \Ccal_0)$ always occurs, so that
	\begin{align}	
	\Psi^{\ssc[1],1}_{\alpha,{\sss I},p}(\ve[1]+\ve[\iota])=
	\frac {p^2}  {\aap} 	\sum_{e\in\{\ve[1],\ve[\iota]\}} \expec_{\sss 0} \big( \indic {e\nin \tilde  \Ccal_0}
	\indic{\ve[1]+\ve[\iota]\in \tilde\Ccal_0}
	    	\prob_{\sss 1}^{0} \big(&2\ve[1]+\ve[\iota]\nin \tilde \Ccal_1,(\ve[1]+\ve[\iota], e )\text{ is occ.}\big)\big).
	    	\lbeq{proof-Difference-PsiI1Rewrite}
	\end{align}
For the upper bound, we drop the condition $2\ve[1]+\ve[\iota]\nin \tilde \Ccal_1$ and obtain
	\begin{align}	
	\sum_{\iota}\Psi^{\ssc[1],1}_{\alpha,{\sss I},p}(\ve[1]+\ve[\iota])	    	
	\leq& 	\frac {p^3}{\aap}\left( 2(2d-2) 	\tau_{2,p}^{\sss 1}(\ve[1]+\ve[2])+\tau_{2,p}^{\sss 1}(2\ve[1])\right).
	    	\lbeq{proof-Difference-PsiI1Upper}
	\end{align}

\paragraph{Lower bounds.} For $\Xi^{\ssc[0]}_{\alpha,p}(\ve[1])$, defined in
\refeq{Def-Xi-Split}, we see that
	\eqan{
  	\Xi^{\ssc[0]}_{\alpha,p}(\ve[1])&=p \prob_p^{\sss (0,\ve[1])} (0\conn \ve[1])
	}
and use \refeq{Lower-Bounds-tau3} for a lower bound on the probability. For $\Psi^{\ssc[0],\kappa}_{\alpha,{\sss II},p}$, we note that
\begin{align}
\Psi^{\ssc[0],\kappa}_{\alpha,{\sss II},p}(\ve[1])=&\frac {p}{\aap}
\prob_p(\{0\connLe{3} \ve[1]\}\cap\{(0,\ve[1])\text{ is occ.}\}\cap\mathcal{T}_\kappa )
\lbeq{proof-Difference-PsiII0Lower-temp}
\end{align}
and use \refeq{Abbreviation-for-extra-bond} to obtain a lower bound.
% Comment: just for the latex code to have the complete lower bound stated somewhere
%\begin{align}
%\sum_{\kappa}&\Psi^{\ssc[0],\kappa}_{\alpha,{\sss II},p}(\ve[1])
%\geq &  (2d-2)\frac {p^5}{\aap}(1-p^3)^{2d-3}
%\left( 1-\tau_{3,p}(\ve[1])-\tau^{\sss 1}_{2,p}(\ve[1]+\ve[2])-\tau^{\sss 1}_{6,p}(2\ve[1]+\ve[2])\right) \nnb
%&-\frac {p^5}{\aap} (1-p^3)^{2d-3}\left( (2d-4)\tau^{\sss 1}_{2,p}(\ve[1]+\ve[2])+2\tau^{\sss 1}_{2,p}(2\ve[1])\right)
%\end{align}
For a lower bound on $ \Psi^{\ssc[0],1}_{\alpha,{\sss I},p}(\ve[1]+\ve[\iota])$, we only consider $|\iota|\neq 1$ and set $\iota=2$ for our discussion. We recall
\refeq{proof-Difference-PsiI0Upper-temp} and use \refeq{Lower-Bounds-not-connected-paths} to compute
\begin{align}
\Psi^{\ssc[0],1}_{\alpha,{\sss I},p}(\ve[1]+\ve[2])
\geq &\ \frac {p}{\aap} \prob_p (\{ (0,\ve[1]),(\ve[1],\ve[1]+\ve[2]),(0,\ve[2] ),(\ve[2],\ve[1]+\ve[2])\text{ are occ.}\}\nnb
&\qquad \qquad \qquad \cap\{(2\ve[1]+\ve[2])\nin\tilde \Ccal^{(\ve[1]+\ve[2],2\ve[1]+\ve[2])}(0)\} )\nnb
\geq &\ \frac {p^5}{\aap}(1-\tau_{3,p}(\ve[1])-\tau_{2,p}^{\sss 1}(\ve[1]+\ve[2])-\tau_{2,p}^{\sss 1}(2\ve[1])-\tau_{5,p}^{\sss 1}(2\ve[1]+\ve[2]))
\lbeq{proof-Difference-PsiI0Lower}
\end{align}
The lower bound on $\Xi^{\ssc[1]}_{\alpha,p}(\ve[1])$ is obtained in a similar way as
\begin{align}
    	\Xi^{\ssc[1]}_{\alpha,p}(\ve[1]) &= p\sum_{e} \expec_{\sss 0}
	\left(\indic{e\nin \tilde \Ccal^{(0,e)}_{\sss 0}(0)}  \indic{(0,\ve[1])\text{ is occ.}}
	\prob_{\sss 1}^{0} \big(E'(e,\ve[1];\tilde \Ccal^{(0,e)}_0(0))\big)\right)\nnb
	&\geq p\sum_{\iota\colon |\iota|\neq 1} \expec_{\sss 0} 	\left(\indic{\ve[\iota],\ve[\iota]+\ve[1]\nin \tilde \Ccal^{(0,\ve[\iota])}_{\sss 0}(0)}  \indic{(0,\ve[1])\text{ is occ.}}
	\prob_{\sss 1}^{0} \big((\ve[\iota],\ve[\iota]+\ve[1]), (\ve[\iota]+\ve[1],\ve[1]) \text{ occ.} \big)\right)\nnb
	%&=(2d-2)p^4 \prob(\ve[\iota],\ve[\iota]+\ve[1] \nin \tilde \Ccal^{(0,e)}_0(0) |(0,\ve[1]) \text{ occ.},(0,\ve[\iota]) \text{ vac.})(1-p)\nnb
	&\geq (2d-2)p^4 (1-p-2\tau_{3,p}(\ve[1])-2\tau_{4,p}(\ve[1]+\ve[2])).
\lbeq{proof-Difference-Xi1Lower}
\end{align}

In the same way, we obtain
\begin{align}
\lbeq{proof-Difference-Psi1Lower}
\sum_{\kappa} 	\Psi^{\ssc[1], \kappa}_{\alpha,{\sss II},p}(\ve[1])\geq&
(2d-1)(2d-2)\frac {p^5}\aap( 1-p-2\tau_{3,p}(\ve[1])-2\theta_4)
\left(1-\tau_{3,p}(\ve[1])-\theta_2-\vartheta\right),
\end{align}
where the second factor is due to the constraint $\ve[1]+\ve[\kappa]\nin \tilde \Ccal^{(\ve[1],\ve[1]+\ve[\kappa])}_{\sss 1}(\ve[1])$, while $(\ve[\iota],\ve[\iota]+\ve[1])$ and $(\ve[1],\ve[\iota]+\ve[1])$ are occupied.
For the lower bound on $\Psi^{\ssc[1],1}_{\alpha,{\sss I},p}(\ve[1]+\ve[\iota])$, we start
from \refeq{proof-Difference-PsiI1Rewrite} and restrict to the case that the connection $0\conn x=\ve[1]+\ve[\iota]$ is created in two steps:
\begin{align}	
	\Psi^{\ssc[1],1}_{\alpha,{\sss I},p}(\ve[1]+\ve[\iota])\geq
		\frac {p^2}  {\aap} \sum_{
            \stackrel{  \rho,\kappa\in\{1,\iota\}}
		 {\kappa\neq \rho }
		} \expec_{\sss 0}& \Big( \indic{\ve[\rho]\nin \tilde \Ccal^{(0,\ve[\rho])}_0(0)}\indic{(0,\ve[\kappa]),(\ve[\kappa],\ve[1]+\ve[\iota])\text{ occ.}}\nnb
	    	&\prob_{\sss 1}^{0} \Big(2\ve[1]+\ve[\iota]\nin \tilde\Ccal^{(\ve[1]+\ve[\iota],2\ve[1]+\ve[\iota])}_{\sss 1},(\ve[1]+\ve[\iota] ),(\ve[\rho]+\ve[\iota],\ve[1])\text{ occ.}\Big)\Big)
\end{align}
Using the technique explained in the proof of Lemma \ref{lemmapercboundLowerBounds}, see
\refeq{Lower-Bounds-not-connected-paths}, we bound this by
\begin{align}	
	\Psi^{\ssc[1],1}_{\alpha,{\sss I},p}(\ve[1]+\ve[\iota])\geq
	 \sum_{\rho=1,\iota} \frac {p^5}  {\aap}(1-2\tau_{3,p}(\ve[1])-\tau_{2,p}^{\sss 1}(\ve[\iota]+\ve[\rho]))(1-\tau_{3,p}(\ve[1])-\tau_{2,p}^{\sss 1}(\ve[1]+\ve[\rho]))
\end{align}
Summing over $\iota$ with $|\iota|\neq 1$ we obtain the bound
\begin{align}	
\sum_{\iota}\Psi^{\ssc[1],1}_{\alpha,{\sss I},p}(\ve[1]+\ve[\iota])\geq
	 (2d-2) \frac {p^5}  {\aap}&(1-2\tau_{3,p}(\ve[1])-\tau_{2,p}^{\sss 1}(\ve[1]+\ve[2]))\nnb
	 &\times(2-2\tau_{3,p}(\ve[1])-\tau_{2,p}^{\sss 1}(2\ve[1])-\tau_{2,p}^{\sss 1}(\ve[1]+\ve[2]))
\lbeq{proof-Difference-PsiI1Lower}
\end{align}
Combining the upper and lower bounds creates the bounds stated in \refeq{Differencebound-2}-\refeq{Differencebound-7}.
\qed

\subsection{Some difficult weighted blocks}
\label{secBoundNBigSpecial}
As most building blocks are defined using just one simple diagrams, it is rather straightforward to bound them using our numerical estimates and the bootstrap functions. See e.g.\ \cite[Section 4.3]{FitHof13b}, where this is explained in detail.
However, the bounds on the blocks $C^{{\sss (1)},\iota,\kappa,a,b}, C^{{\sss (2)},\iota,\kappa,a,b}$ and the entries of $\vec h^{\iota},\vec h^{\iota,{\rm \sss II}}$ require some additional ideas. As these have not been discussed in \cite{FitHof13b}, we do so here.

\subsubsection{Bounds on $C^{\sss (1)}$ and $C^{\sss (2)}$}
In this section, we bound $C^{{\sss (1)},\iota,\kappa,a,b}, C^{{\sss (2)},\iota,\kappa,a,b}$ as shown in Figure \ref{fig-Form-C1-Block-perc}-\ref{fig-Form-C2-Block-perc}.
%We need bound on
%\begin{align}
%({\bf C}^{1})_{a,b}=&\sup_{v,y}\sum_{c=0}^2\sum_{\iota,\kappa,u,w,x}B^{\iota,a,c}(0,v,w,u)\bar H^{2,\kappa,c,b}(u,w,x+y,x)\|w\|_2^2,\\
%({\bf C}^{2})_{a,b}=&\sup_{v,y}\sum_{c=0}^2\sum_{\iota,\kappa,u,w,x} B^{\iota,a,c}(v,0,u,w)\bar A^{\kappa,c,b}(w,u,x,x+y)\|u\|_2^2,
%\end{align}
\begin{figure}[h]
\begin{center}
\picWeightConstructOne[0.7]
\caption{Diagrammatic representation of $C^{{\sss (1)},\iota,\kappa,a,b}$.}
\label{fig-Form-C1-Block-perc}
\picWeightConstructTwo[0.7]
\caption{Diagrammatic representation of $C^{{\sss (2)},\iota,\kappa,a,b}$.}
\label{fig-Form-C2-Block-perc}
\end{center}
\end{figure}
\noindent
For the bound on ${\bf C}^{\sss (1)}$, Figure \ref{fig-Form-C1-Block-perc},  we consider three cases.
First the left diagram, in which we apply $\|w\|_2^2\leq 2(\|t\|_2^2+\|t-w\|_2^2)$ and bound the result by
$2 \left(\left[{\bf H}^{\sss (2)}{\bf A}+{\bf A}^{\iota,*} {\bf H}^{\sss (1)}\right] {\bf A}^\iota\right)_{a,b}$.
Secondly, the right diagram in which, the small triangle is trivial. In this case, the bound is ${\bf H}^{\sss (2)}{\bf A}^\iota_{a,b}$. In the last case, we bound the diagram by a combination of $H^{\sss (2)}$ and $\bar B^{\sss (2)}$.
In this way, we obtain the bound
	\begin{align}
	({\bf C}^{\sss (1)})_{a,b}\leq&
 	\left(\left[2{\bf H}^{\sss (2)}{\bf A}+2{\bf A}^{\iota,*} {\bf H}^{\sss (1)}+{\bf H}^{\sss (2)}\right]
	{\bf A}^\iota\right)_{a,b}
 	%\nnb&
 	+\sum_{c=0}^2 {\bf H}^{\sss (2)}_{a,c}\sup_{v}\sum_{\iota,x,y} \bar B^{{\sss (2)},\iota,c,b}(0,v,x,y).\nn
	\end{align}
For the bound on $({\bf C}^{\sss (2)})_{a,b}$, we extract the case that the triangle $(z,t,u)$ is non-trivial, as this is by far the most difficult term.
For this, we define $C^{{\sss (3)},\iota,a,b}(0,v,x,y)$ to be the diagram $C^{{\sss (2)},\iota,a,b}(0,v,x,y)$ in which we replace the weight $\|u\|_2^2$ by $\|z-u\|_2^2$ (see Figure \ref{fig-Form-C2-Block-perc}), and define
	\begin{align}
  	\lbeq{definition-supC3}
	({\bf C}^{\sss (3)})_{a,b}=&\sup_{v,y} \sum_{\iota,\kappa,x}C^{{\sss (3)},\iota,\kappa,a,b}(0,v,x,x+y).
	\end{align}
Then, we bound ${\bf C}^{\sss (2)}$ similarly to ${\bf C}^{\sss (1)}$,
using $\|u\|_2^2\leq 2(\|z\|_2^2+\|u-z\|_2^2)$, to obtain the bound
	\begin{align}
	({\bf C}^{\sss (2)})_{a,b}\leq&
 	\left(\left[2{\bf H}^{\sss (3)}{\bf A}+2{\bf A}^{\iota,*} {\bf H}^{\sss (1)}+
	{\bf H}^{\sss (3)}\right] {\bf A}^\iota +2  {\bf C}^{\sss (3)}\right)_{a,b}\nnb
 	&+\sum_{c=0}^2 {\bf H}^{\sss (3)}_{a,c}\sup_{v}\sum_{\iota,x,y} \bar B^{(2),\iota,c,b}(0,v,x,y).
	\end{align}
To bound $({\bf C}^{\sss (3)})_{a,b}$, we bound the underlying diagram $C^{(3),\iota,\kappa}$ using simple diagrams:
	\begin{align}
  	\lbeq{C3-first-Bound}
	\sum_{\iota,\kappa}C^{{\sss (3)},\iota,\kappa}&(0,v,x,x+y)
	\leq \sum_{t,u,w,z\in\Zd}
	\diagRepulsiveLetter{S}_{0,1,1,1}(z,t,w,v) \tau_p(z-u) \|z-u\|_2^2\\
	&\qquad\qquad \qquad \qquad\qquad\times \tau_{1,p}(u-x)\tau_{1,p}(t-u) \tau_{0,p}(w-x-y)\nnb
	&\qquad\qquad\qquad+\sum_{u,w,z\neq u}\shift \diagRepulsiveLetter{T}_{0,1,1}(z,w,v)
	\prob_p(z\dbc u ) \|z-u\|_2^2\tau_{1,p}(u-x)\tau_{0,p}(w-x-y),
	\nn
	\end{align}
compare this with the right diagram in Figure \ref{fig-Form-C2-Block-perc}.  We use this bound to compute \refeq{definition-supC3}.
As the order in which the bounds are performed is quite delicate, we show each step explicitly:
	\begin{align}
	& \sup_{v,y} \sum_{x,t,u,w,z\in\Zd}
	\diagRepulsiveLetter{S}_{0,1,1,1}(z,t,w,v) \tau_z(z-u) \|z-u\|_2^2\tau_{1,p}(u-x)\tau_{1,p}(t-u) \tau_{0,p}(w-x-y)\nnb
	\leq &\big(\sup_{r\in\Zd}\tau_p(r)\|r\|_2^2\big)
	\sup_{v,y\in\Zd} \sum_{t,u,w,z,x\in\Zd}\Big[\diagRepulsiveLetter{S}_{0,1,1,1}(z,t,w,v)
	\tau_{1,p}(t-u)\tau_{1,p}(u-x) 	\tau_{0,p}(w-x-y)\Big]\nnb
	=&\big(\sup_{r\in\Zd}\tau_p(r)\|r\|_2^2\big)
	\sup_{v,y\in\Zd} \sum_{t,w',z\in\Zd}\Big[ \diagRepulsiveLetter{S}_{0,1,1,1}(z,t,t+w',v) \nnb
	& \qquad\qquad\qquad \qquad\qquad \qquad\qquad
	\times \sum_{u',x'\in\Zd }\tau_{1,p}(u')\tau_{1,p}(u'-x') \tau_{0,p}(w'-x'-y)\Big].
	\end{align}
We relabel $w'=w-t, x'=x-t, u'=u-t$. Then, we take the supremum over $w'$ and obtain the bound
	\begin{align}
	\leq&\big(\sup_{r\in\Zd}\tau_p(r)\|r\|_2^2\big)
 	\sup_{v\in\Zd}\sum_{t,w,z\in\Zd} \diagRepulsiveLetter{S}_{0,1,1,1}(z,t,w,v)
  	\sup_{y,w\in\Zd} \sum_{x'}(\tau_{1,p}^{\star 2})(u')\tau_{0,p}(y-x'-w)
	\end{align}
with
	\begin{align}
 	\sup_{y,w\in\Zd} \sum_{x'}(\tau_{1,p}^{\star 2})(x')\tau_{0,p}(y-x'-w)
	= \sup_{y\in\Zd} (\tau_{1,p}^{\star 2}\star \tau_{0,p})(y)=\sup_{y}\sum_{w,u}\diagRepulsiveLetter{T}^*_{1,1,0}(u,w,y)
	\end{align}
We bound the second term in \refeq{C3-first-Bound} in a similar manner. We relabel $u'=u-z$ and rewrite this term as
	\begin{align}
	&\sup_{v,y}
	\sum_{x,u',w,z\colon u'\neq 0}
	\diagRepulsiveLetter{T}_{0,1,1}(z,w,v) \prob_p(0\dbc u' ) \|u'\|_2^2\
	\tau_{1,p}(u+z-x)\tau_{0,p}(w-x-y)\\
	&\leq \sum_{u'\neq 0}\prob_p(0\dbc u' ) \|u'\|_2^2
	\left[\sup_{v}\sum_{w,z}\diagRepulsiveLetter{T}_{0,1,1}(z,w,v)\right]
	\left[ \sup_{u',w,y,z}\sum_{x}\tau_{1,p}(u'+z-x)\tau_{0,p}(w-x-y)\right],\nn
	\end{align}
with
	\begin{align}
  	\sup_{u',w,y,z}\sum_{x}\tau_{1,p}(u'+z-x)\tau_{0,p}(w-x-y)&=\sup_{u',w,y,z}\sum_{x}\tau_{1,p}(u'+z-x)\tau_{0,p}(x+y-w)\nnb
	&=\sup_{y\in\Zd} (\tau_{1,p}\star \tau_{0,p})(y).
	\end{align}
These bounds hold for all $a,b$. For $a,b\in \{0,1\},$ we can further improve these bounds, for example, by using that the complete square $\diagRepulsiveLetter{S}_{0,1,1,1}(z,t,w,v)$ contains at least four steps. \\[2mm]

\noindent
{\bf Remark.} We have bounded ${\bf C}^{\sss (3)}$ using a square. Thus, this bound can only be used for $d\geq 9$ as the square is infinite in $d=7,8$. It is possible to bound ${\bf C}^{\sss (3)}$ using only triangles and weighted bubbles, which are finite in $d\geq 7$. However, this requires that we decompose the diagram of the coefficient differently. This decomposition would require a second set of building blocks. As we cannot prove mean-field behaviour in $d=7,8$ anyway, we simply use the bound derived above.

\subsubsection{Bounds on $\vec h^{\iota}$ and $\vec h^{\iota,{\rm \sss II}}$}
For $\vec h^{\iota,{\rm \sss II}}$, it follows from its definition in \refeq{defHiotaII} and a simple step that
	\begin{align}
	(\vec h^{\iota,{\rm \sss II}})_{b}\leq (\vec h^{\iota})_{b}+2 (\vec h^{\iota}{\bf A}^{\iota})_{b}
	+2 (\vec P^\iota{\bf A}^\iota {\bf H}^{\sss (2)})_b.
	\end{align}
The term $\vec h^{\iota}$ consists of three contributions that we display in Figure \ref{fig-Form-hiota}.

\begin{figure}[h]
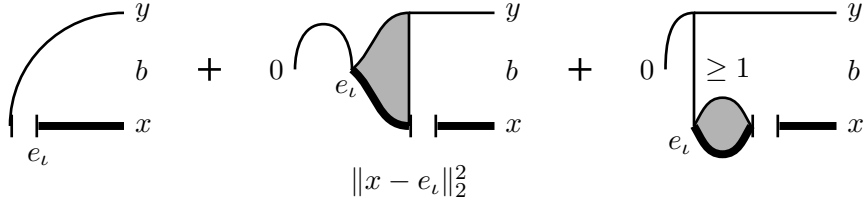

\begin{center}
\picWeightPiIota[0.75]
\caption{The possible forms of a diagram in $\vec h^{\iota}(k)$. For the bounds we take the supremum over $x-y$ and sum over $x$ and $\iota$. We divide the result by $2d$ to average over $\iota$.}
\label{fig-Form-hiota}
\end{center}
\end{figure}
\noindent

In the first diagram only $y\neq 0$ contributes, since $y\in\tilde \Ccal_{\sss 0} \cap \tilde \Ccal_{\sss 1}$. This means that the connection $\ve[\iota]\to 0\to y$ consists of at least two steps. We bound this contribution by $p \sum_{\iota,\kappa}H^{{\sss (3)},\kappa,1,b}(\ve[\iota],0,y,x)$, where the factor $p$ is created by the connection $\ve[\iota]\conn 0$ that does not contribute to $H^{{\sss (3)},\kappa,1,b}$.
The second and third diagrams are decomposed as shown in Figures \ref{fig-Form-hiota-decomp1} and \ref{fig-Form-hiota-decomp2}.

\begin{figure}[h]
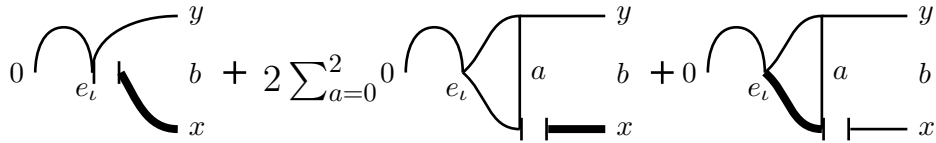

\begin{center}
\picHIotaDecompositionOne[0.75]
\caption{The decomposition of the second diagram of Figure \ref{fig-Form-hiota}.}
\label{fig-Form-hiota-decomp1}
\end{center}
\end{figure}
\begin{figure}[h]
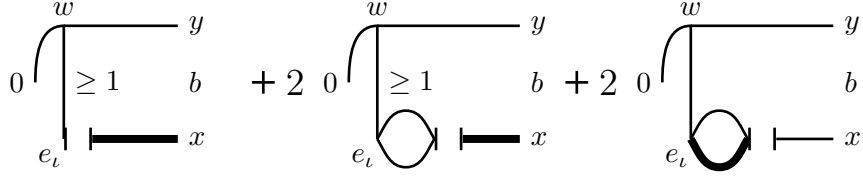

\begin{center}
\picHIotaDecompositionTwo[0.75]
\caption{The decomposition of the third diagram of Figure \ref{fig-Form-hiota}.}
\label{fig-Form-hiota-decomp2}
\end{center}
\end{figure}

We bound the diagrams shown in Figure \ref{fig-Form-hiota-decomp1} by
\begin{align}
& \tau_{3,p}(\ve[\iota])({\bf H}^{\sss (2)})_{0,b}
+2\tau_{3,p}(\ve[\iota])((\vec P^{\sss \rm S}-(1,0,0)^T){\bf H}^{\sss (2)})_{b}+2\tau_{3,p}(\ve[\iota])({\bf H}^{\sss (1)}{\bf A}^\iota)_{0,b}.
\lbeq{bound-hiota-partial-one}
\end{align}
The diagrams of Figure \ref{fig-Form-hiota-decomp2} are bounded by
\begin{align}
& \sum_{w}\diagRepulsiveLetter{B}_{2,\underline 1 }(w,\ve[\iota])({\bf H}^{\sss (2)})_{1,b}
+\sum_{w}\diagRepulsiveLetter{B}_{0,2}(w,\ve[\iota])({\bf H}^{\sss (2)})_{2,b}\nnb
&+2\sum_{w} (\diagRepulsiveLetter{B}_{2,\underline 1}(w,\ve[\iota])+\diagRepulsiveLetter{B}_{0,2}(w,\ve[\iota]))
({\bf A}^\iota)_{0,0} \max\left\{({\bf H}^{\sss (2)})_{1,b},({\bf H}^{\sss (2)})_{2,b}\right\}\nnb
&+2\sum_{w} (\diagRepulsiveLetter{B}_{2,\underline 1}(w,\ve[\iota])+\diagRepulsiveLetter{B}_{0,2}(w,\ve[\iota]))({\bf H}^{\sss (1)})_{0,0}
\max\left\{({\bf A}^\iota)_{1,b},({\bf A}^\iota)_{2,b}\right\}.
\lbeq{bound-hiota-partial-two}
\end{align}
Combining these bounds we obtain
	\begin{align}
	(\vec h^{\iota})_{b}\leq& p({\bf H}^{\sss (3)})_{1,b}+\refeq{bound-hiota-partial-one}
	+ \refeq{bound-hiota-partial-two},
	\end{align}
where the two line numbers denote the terms given in the corresponding lines.

\fi

\end{document}